%% file: carpets.tex
\documentclass{memo-l}

\input{macro.tex}

\usepackage{imakeidx}
\makeindex

\begin{document}

\frontmatter

\title{Potential theory on Sierpi\'nski carpets with applications to uniformization}


\author{Dimitrios Ntalampekos}
\address{Department of Mathematics, University of California, Los Angeles, CA 90095, USA}
\email{dimitrisnt@math.ucla.edu}
\thanks{The author was partially supported by NSF grant DMS-1506099.}

\date{\today}



\dedicatory{To my parents...\\[2pt] who taught me how to love what I do.}

\begin{abstract}
This research is motivated by the study of the geometry of fractal sets and is focused on uniformization problems: transformation of sets to canonical sets, using maps that preserve the geometry in some sense. More specifically, the main question addressed is the uniformization of planar Sierpi\'nski carpets by square Sierpi\'nski carpets, using methods of potential theory on carpets.

We first develop a potential theory and study harmonic functions on planar Sierpi\'nski carpets. We introduce a discrete notion of Sobolev spaces on Sierpi{\'n}ski carpets and use this to define harmonic functions. Our approach differs from the classical approach of potential theory in metric spaces discussed in \cite{HeinonenKoskelaShanmugalingamTyson:Sobolev} because it takes the ambient space that contains the carpet into account. We prove basic properties such as the existence and uniqueness of the solution to the Dirichlet problem, Liouville's theorem, Harnack's inequality, strong maximum principle, and equicontinuity of harmonic functions. 

Then we utilize this notion of harmonic functions to prove a uniformization result for Sierpi\'nski carpets. Namely, it is proved that every planar Sierpi\'nski carpet whose peripheral disks are uniformly fat, uniform quasiballs can be mapped to a square Sierpi\'nski carpet with a map that preserves carpet modulus. If the assumptions on the peripheral circles are strengthened to uniformly relatively separated, uniform quasicircles, then the map is a quasisymmetry. The real part of the uniformizing map is the solution of a certain Dirichlet-type problem. Then a harmonic conjugate of that map is constructed using the methods of Rajala \cite{Rajala:uniformization}.
\end{abstract}

\maketitle

\tableofcontents


\mainmatter

\input{introduction.tex}

\input{harmonic.tex}

\input{uniformization.tex}

\appendix

\backmatter

\bibliographystyle{amsalpha}
\bibliography{carpets.bbl}

\printindex

\end{document}

%% file: macro.tex
\usepackage{amssymb,amsthm,amsmath}

\usepackage{graphicx,tikz-cd}
\usetikzlibrary{arrows,patterns,calc,backgrounds}

\usepackage[percent]{overpic}
\usepackage[shortlabels]{enumitem}
\usepackage{mathtools}
\usepackage{verbatim}

\usepackage{hyperref}

\usepackage{enumitem}

\usepackage{cite}

\makeatletter
\newcommand*{\mint}[1]{%
  \mint@l{#1}{}%
}
\newcommand*{\mint@l}[2]{%
  \@ifnextchar\limits{%
    \mint@l{#1}%
  }{%
    \@ifnextchar\nolimits{%
      \mint@l{#1}%
    }{%
      \@ifnextchar\displaylimits{%
        \mint@l{#1}%
      }{%
        \mint@s{#2}{#1}%
      }%
    }%
  }%
}
\newcommand*{\mint@s}[2]{%
  \@ifnextchar_{%
    \mint@sub{#1}{#2}%
  }{%
    \@ifnextchar^{%
      \mint@sup{#1}{#2}%
    }{%
      \mint@{#1}{#2}{}{}%
    }%
  }%
}
\def\mint@sub#1#2_#3{%
  \@ifnextchar^{%
    \mint@sub@sup{#1}{#2}{#3}%
  }{%
    \mint@{#1}{#2}{#3}{}%
  }%
}
\def\mint@sup#1#2^#3{%
  \@ifnextchar_{%
    \mint@sup@sub{#1}{#2}{#3}%
  }{%
    \mint@{#1}{#2}{}{#3}%
  }%
}
\def\mint@sub@sup#1#2#3^#4{%
  \mint@{#1}{#2}{#3}{#4}%
}
\def\mint@sup@sub#1#2#3_#4{%
  \mint@{#1}{#2}{#4}{#3}%
}
\newcommand*{\mint@}[4]{%
  \mathop{}%
  \mkern-\thinmuskip
  \mathchoice{%
    \mint@@{#1}{#2}{#3}{#4}%
        \displaystyle\textstyle\scriptstyle
  }{%
    \mint@@{#1}{#2}{#3}{#4}%
        \textstyle\scriptstyle\scriptstyle
  }{%
    \mint@@{#1}{#2}{#3}{#4}%
        \scriptstyle\scriptscriptstyle\scriptscriptstyle
  }{%
    \mint@@{#1}{#2}{#3}{#4}%
        \scriptscriptstyle\scriptscriptstyle\scriptscriptstyle
  }%
  \mkern-\thinmuskip
  \int#1%
  \ifx\\#3\\\else_{#3}\fi
  \ifx\\#4\\\else^{#4}\fi  
}
\newcommand*{\mint@@}[7]{%
  \begingroup
    \sbox0{$#5\int\m@th$}%
    \sbox2{$#5\int_{}\m@th$}%
    \dimen2=\wd0 %
    \let\mint@limits=#1\relax
    \ifx\mint@limits\relax
      \sbox4{$#5\int_{\kern1sp}^{\kern1sp}\m@th$}%
      \ifdim\wd4>\wd2 %
        \let\mint@limits=\nolimits
      \else
        \let\mint@limits=\limits
      \fi
    \fi
    \ifx\mint@limits\displaylimits
      \ifx#5\displaystyle
        \let\mint@limits=\limits
      \fi
    \fi
    \ifx\mint@limits\limits
      \sbox0{$#7#3\m@th$}%
      \sbox2{$#7#4\m@th$}%
      \ifdim\wd0>\dimen2 %
        \dimen2=\wd0 %
      \fi
      \ifdim\wd2>\dimen2 %
        \dimen2=\wd2 %
      \fi
    \fi
    \rlap{%
      $#5%
        \vcenter{%
          \hbox to\dimen2{%
            \hss
            $#6{#2}\m@th$%
            \hss
          }%
        }%
      $%
    }%
  \endgroup
}

\newcommand{\x}{\scalebox{1.2}{$\chi$} } 

\newcommand{\br}{\overline}
\newcommand{\R}{\mathbb R}
\newcommand{\C}{\mathbb C}
\newcommand{\D}{\mathbb D}
\newcommand{\Z}{\mathbb Z}
\newcommand{\N}{\mathbb N}
\newcommand{\Q}{\mathbb Q}

\newcommand{\UHP}{\mathbb H}

\theoremstyle{plain}
\newtheorem{theorem}{Theorem}
\newtheorem{lemma}[theorem]{Lemma}
\newtheorem{conjecture}{Conjecture}
\newtheorem{prop}[theorem]{Proposition}
\newtheorem{corollary}[theorem]{Corollary}

\newtheorem{innercustomthm}{Theorem}			
\newenvironment{customthm}[1]
  {\renewcommand\theinnercustomthm{#1}\innercustomthm}
  {\endinnercustomthm}
  
\newtheorem{innercustomprop}{Proposition}		
\newenvironment{customprop}[1]
  {\renewcommand\theinnercustomprop{#1}\innercustomprop}
  {\endinnercustomprop} 

\theoremstyle{definition}
\newtheorem{definition}[theorem]{Definition}
\newtheorem*{claim}{Claim}

\theoremstyle{remark}
\newtheorem{remark}[theorem]{Remark}

\newtheorem{example}[theorem]{Example}

\DeclareMathOperator{\dist}{\textup{\text{dist}}}
\DeclareMathOperator{\diam}{\textup{\text{diam}}}

\DeclareMathOperator{\inter}{\textup{\text{int}}}

\DeclareMathOperator{\re}{\textup{\text{Re}}}
\DeclareMathOperator{\im}{\textup{\text{Im}}}

\DeclareMathOperator{\md}{\textup{mod}}

\DeclareMathOperator{\loc}{\textup{loc}}
\DeclareMathOperator*{\osc}{\textup{osc}}

\numberwithin{equation}{chapter}
\numberwithin{theorem}{chapter}
\numberwithin{figure}{chapter}

%% file: introduction.tex
\chapter{Introduction}\label{Chapter:Intro}

One of the main problems in the field of Analysis on Metric Spaces is to find geometric conditions on a given metric space, under which the space can be transformed to a ``canonical" space with a map that preserves the geometry. In other words, we wish to \textit{uniformize} the metric space by a canonical space. For example, the Riemann mapping theorem gives a conformal map from any simply connected proper subregion of the plane onto the unit disk, which is the canonical space in this case.

In the setting of metric spaces we search instead for other types of maps, such as bi-Lipschitz, quasiconformal or quasisymmetric maps. One method for obtaining such a map is by solving \textit{minimization problems}, such as the problem of minimizing the \textit{Dirichlet energy} $\int_{\Omega} |\nabla u|^2$ in an open set $\Omega\subset \C$ among Sobolev functions $u\in W^{1,2}(\Omega)$ that have some certain boundary data. 

To illustrate the method, we give an informal example. Let $\Omega \subset \C$ be a quadrilateral, i.e., a Jordan region with four marked points on $\partial \Omega$ that define a topological rectangle. Consider two opposite sides $\Theta_1,\Theta_3 \subset \partial \Omega$ of this topological rectangle. We study the following minimization problem:
\begin{align}\label{Intro:MinimizationProblem}
\inf \biggl\{ \int_{\Omega} |\nabla u|^2 : u\in W^{1,2}(\Omega),\, u\big|_{\Theta_1}=0,\, u\big|_{\Theta_3}=1 \biggr\}.
\end{align}
One can show that a minimizer $u$ with the right boundary values exists and is harmonic on $\Omega$. Let $D(u)\coloneqq \int_\Omega |\nabla u|^2$ be the Dirichlet energy of $u$, and $\Theta_2,\Theta_4\subset \partial \Omega$ be the other opposite sides of the quadrilateral $\partial \Omega$, numbered in a counter-clockwise fashion. Now, we consider the ``dual" problem
\begin{align*}
\inf \biggl\{ \int_{\Omega} |\nabla v|^2 : v\in W^{1,2}(\Omega), v\big|_{\Theta_2}=0, v\big|_{\Theta_4}=D(u) \biggr\}.
\end{align*}
Again, it turns out that a minimizer $v$ with the right boundary values exists and is harmonic. In fact, $v$ is the harmonic conjugate of $u$. Then the pair $f\coloneqq (u,v)$ yields a conformal map from $\Omega$ onto the rectangle $(0,1)\times (0,D(u))$; see \cite{Courant:Dirichlet} for background on classical potential theory and construction of conformal maps.

This example shows that in the plane  harmonic functions that minimize the Dirichlet energy and solve certain boundary value problems can be very handy in uniformization theory. Namely, there exist more minimization problems whose solution $u$ can be paired with a harmonic conjugate $v$ as above to yield a conformal map $(u,v)$ that transforms a given region to a {canonical} region. For example, one can prove in this way the Riemann mapping theorem, the uniformization of annuli by round annuli, and the uniformization of planar domains by slit domains; see \cite{Courant:Dirichlet}.

A natural question is whether such methods can be used in the abstract metric space setting in order to obtain uniformization results. Hence, one would first need to develop a harmonic function theory. Harmonic functions have been studied in depth in the abstract metric space setting. Their definition was based on a suitable  notion of Sobolev spaces in metric measure spaces. The usual assumptions on the intrinsic geometry of the metric measure space  is that it is doubling and supports a Poincar\'e inequality. Then, harmonic functions are defined as local energy minimizers, among Sobolev functions with the same boundary data. We direct the reader to \cite{Shanmugalingam:harmonic} and the references therein for more background. However, to the best of our knowledge, this general theory has not been utilized yet towards a uniformization result. As we see from the planar examples mentioned previously, a crucial ingredient in order to obtain such a result is the existence of a harmonic conjugate in the 2-dimensional setting. Constructing harmonic conjugates turns out to be an extremely challenging task in the metric space setting.

Very recently this was achieved by K.~Rajala \cite{Rajala:uniformization}, who solved a minimization problem  on metric spaces $X$ homeomorphic to $\R^2$, under some geometric assumptions. The minimization procedure yielded a ``harmonic" function $u$, which was then paired with a ``harmonic conjugate" $v$ to provide a quasiconformal homeomorphism $(u,v)$ from $X$ to $\R^2$. The construction of a harmonic conjugate, which is one of the most technical parts of Rajala's work, is very powerful. As a corollary, he obtained the Bonk-Kleiner theorem \cite{BonkKleiner:quasisphere}, which asserts that a metric sphere that is \textit{Ahlfors 2-regular} and \textit{linearly locally contractible} is quasisymmetrically equivalent to the standard sphere.

Another minimization problem in similar spirit is Plateau's problem; see the book of Courant \cite{Courant:Dirichlet}. This has also recently been extended to the metric space setting \cite{LytchakWenger:parametrizations} and its solution provides canonical quasisymmetric embeddings of a metric space $X$ into $\R^2$, under some geometric assumptions. Thus, \cite{LytchakWenger:parametrizations} provides  an alternative proof of the Bonk-Kleiner theorem.

The development of uniformization  results for metric spaces homeomorphic to $\R^2$ or to the sphere would provide some insight towards the better understanding of hyperbolic groups, whose boundary at infinity is a 2-sphere. A basic problem in geometric group theory is  finding relationships between the algebraic properties of a finitely generated group $G$  and the geometric properties of its Cayley graph. For each \textit{Gromov hyperbolic} group there is a natural metric space, called \textit{boundary at infinity}, and denoted by $\partial_\infty G$. This metric space is equiped with a family of \textit{visual metrics}. The geometry of $\partial_\infty G$ is very closely related to the asymptotic geometry of the group $G$. A major conjecture by Cannon \cite{Cannon:Conjecture}  is the following: when $\partial_\infty G$ is homeomorphic to the 2-sphere, then $G$ admits a discrete, cocompact, and isometric action on the hyperbolic 3-space $\UHP^3$. By a theorem of Sullivan \cite{Sullivan:Cannon}, this conjecture is equivalent to the following conjecture:
\begin{conjecture}\label{Conjecture:Cannon}
If $G$ is a Gromov hyperbolic group and $\partial_\infty G$ is homeomorphic to the $2$-sphere, then $\partial_\infty G$, equipped with a visual metric, is quasisymmetric to the $2$-sphere.
\end{conjecture}

We now continue our discussion on applications of potential theory and minimization problems to uniformization. We provide an example from the discrete world. In \cite{{Schramm:Tiling}}, using again an energy minimization procedure, Schramm proved the following fact. Let $\Omega$ be a quadrilateral, and $T$ a finite triangulation of $\Omega$ with vertex set $\{v\}_{v\in I}$. Then there exists a square tiling $\{Z_v\}_{v\in I}$ of a rectangle $R$ such that each vertex $v$ corresponds to a square $Z_v$, and two squares $Z_u,Z_v$ are in contact whenever the vertices $u,v$ are adjacent in the triangulation. In addition, the vertices corresponding to squares at corners of $R$ are at the corners of the quadrilateral $\Omega$. In other words, triangulations of quadrilaterals can be transformed to square tilings of rectangles. Of course, we are not expecting any metric properties for the correspondence between vertices and squares, since we are not endowing the triangulation with a metric and we are only taking into account the adjacency of vertices.

Hence, it is evident that potential theory is a precious tool that is also available in metric spaces and can be used to solve uniformization problems. Furthermore, from the aforementioned results we see that harmonic functions and energy minimizers interact with quasiconformal and quasisymmetric maps in metric spaces. We now switch our discussion to Sierpi\'nski carpets and related uniformization problems.

A planar \textit{Sierpi\'nski carpet} $S\subset \C$ is a locally connected continuum with empty interior that arises from a closed Jordan region $\br{\Omega}$ by removing countably many Jordan regions $Q_i$, $i\in \N$, from $\br \Omega$ such that the closures $\br Q_i $ are disjoint with each other and with $\partial \Omega$. The local connectedness assumption can be replaced with the assumption that $\diam(Q_i)\to 0$ as $i\to \infty$. The sets $\partial \Omega$ and $\partial Q_i$ for $i\in \N$ are called the \textit{peripheral circles} of the carpet $S$ and the Jordan regions $Q_i$, $i\in \N$, are called the \textit{peripheral disks}. According to a theorem of Whyburn \cite{Whyburn:theorem} all such continua are homeomorphic to each other and, in particular, to the \textit{standard Sierpi\'nski carpet}, which is formed by removing the middle square of side-length $1/3$ from the unit square $[0,1]^2$ and then proceeding inductively in each of the remaining eight squares.

The study of uniformization problems on carpets was initiated by Bonk in \cite{Bonk:uniformization}, where he proved that every Sierpi\'nski carpet in the sphere $\widehat \C$ whose peripheral circles are uniform quasicircles and they are also uniformly relatively separated  is quasisymmetrically equivalent to a \textit{round} Sierpi\'nski carpet, i.e., a carpet all of whose peripheral circles are geometric circles. The method that he used does not rely on any minimization procedure, but it uses results from complex analysis, and, in particular, Koebe's theorem that allows one to map conformally a finitely connected domain in the plane to a circle domain. 

A partial motivation for the development of uniformization results for carpets is another conjecture from geometric group theory, known as the Kapovich-Kleiner conjecture. The conjecture asserts that if a Gromov hyperbolic group $G$ has a boundary at infinity $\partial_\infty G$ that is homeomorphic to a Sierpi\'nski carpet, then $G$ admits a properly discontinuous, cocompact, and isometric action on a convex subset of the hyperbolic 3-space $\UHP^3$ with non-empty totally geodesic boundary. The Kapovich-Kleiner conjecture \cite{KapovichKleiner:Conjecture}  is equivalent to the following uniformization problem, similar in spirit to Conjecture \ref{Conjecture:Cannon}:
\begin{conjecture}\label{Conjecture:KK}
If $G$ is a Gromov hyperbolic group and $\partial_\infty G$ is a Sierpi\'nski carpet, then $\partial_\infty G$ can be quasisymmetrically embedded into the $2$-sphere.
\end{conjecture}

The main focus in this work is to prove a uniformization result for planar Sierpi\'nski carpets by using an energy minimization method. We believe that these methods can be extended to some non-planar carpets and therefore provide some insight to the problem of embedding these carpets into the plane. The canonical spaces in our setting are square carpets, which arise naturally as the extremal spaces of a minimization problem. A \textit{square carpet} here is a planar carpet all of whose peripheral circles are squares, except for the one that separates the rest of the carpet from $\infty$, which could be a rectangle. Also, the sides of the squares and the rectangle are required to be parallel to the coordinate axes.

Under some geometric assumptions, we obtain the following main result:
\begin{customthm}{1}\label{Thm:Main}\index{quasiconformal!discrete}
Let $S\subset \C$ be a Sierpi\'nski carpet of measure zero. Assume that the peripheral disks of $S$ are uniformly fat, uniform quasiballs. Then there exists a ``quasiconformal" map (in a discrete sense) from $S$ onto a square carpet.
\end{customthm}

The precise definitions of the geometric assumptions and of the notion of quasiconformality that we are employing are given in Chapter \ref{Chapter:Uniformization}, Section \ref{unif:Section Introduction}; see Theorem \ref{unif:Main theorem}. Roughly speaking, fatness prevents outward pointing cusps in a uniform way. The quasiball assumption says that in large scale the peripheral disks  $Q_i$ look like balls, in the sense that for each $Q_i$ there exist two concentric balls, one contained in $Q_i$ and one containing $Q_i$, with uniformly bounded ratio of radii. For example, if the peripheral disks are John domains with uniform constants, then they satisfy the assumptions; see \cite{SmithStegenga:HolderPoincare} for the definition of a John domain. The uniformizing map is ``quasiconformal" in the sense that it almost preserves carpet-modulus, a discrete notion of modulus suitable for Sierpi\'nski carpets. 

If one strengthens the assumptions, then one obtains a quasisymmetry:

\begin{customthm}{2}[Theorem \ref{unif:Main theorem-quasisymmetric}]
Let $S\subset \C$ be a Sierpi\'nski carpet of measure zero. Assume that the peripheral circles of $S$ are uniformly relatively separated, uniform quasicircles. Then there exists a quasisymmetry from $S$ onto a square carpet.
\end{customthm} 

These are the same assumptions as the ones used in \cite{Bonk:uniformization}, except for the measure zero assumption, which is essential for our method. The assumption of uniform quasicircles is necessary both in our result and in the uniformization by round carpets result of \cite{Bonk:uniformization}, because this property is preserved under quasisymmetries, and squares and circles share it. The uniform relative separation condition prevents large peripheral circles to be too close to each other. This is essentially the best possible condition one could hope for:

\begin{customprop}{1}[Proposition \ref{unif:Proposition Equivalence of square carpet}]
A round carpet is quasisymmetrically equivalent to a square carpet if and only if the uniform relative separation condition holds.
\end{customprop}
 
The map in Theorem \ref{Thm:Main} is the pair of a certain  \textit{carpet-harmonic} function $u$ with its ``harmonic conjugate" $v$. Recall that the carpet $S$ is equal to $\br \Omega\setminus \bigcup_{i\in \N} Q_i$, where $\Omega$ is a Jordan region. We wish  to view $\partial \Omega$ as a topological rectangle with sides $\Theta_1,\dots,\Theta_4$ and consider a discrete analog of the minimization problem \eqref{Intro:MinimizationProblem}. This is the problem that will provide us with the real part $u$ of the uniformizing map. Then, adapting the methods of \cite{Rajala:uniformization} we construct a harmonic conjugate $v$ of $u$. This is discussed in Chapter \ref{Chapter:Uniformization}.

Hence, in order to proceed, we need to make sense of a Sobolev space $\mathcal W^{1,2}(S)$ and of carpet-harmonic functions. This is the content of Chapter \ref{Chapter:Harmonic}.

Before providing a sketch of our definition of Sobolev spaces and carpet-harmon\-ic functions, we recall the definition of Sobolev spaces---also called Newtonian spaces---and harmonic functions on metric spaces, following \cite{Shanmugalingam:newtonian} and \cite{Shanmugalingam:harmonic}. Roughly speaking, a function $u\colon X\to \R$ lies in the Newtonian space $N^{1,p}(X)$ if $u\in L^p(X)$, and there exists a function $g\in L^p(X)$ with the property that
\begin{align*}
|u(x)-u(y)|\leq \int_\gamma g \, ds
\end{align*} 
for \textit{almost every} path $\gamma$ and all points $x,y\in \gamma$. Here, ``almost every" means that a family of paths with $p$-modulus zero has to be excluded; see sections \ref{harmonic:2-Carpet modulus} and \ref{harmonic:Section Existence of paths} for a discussion on modulus and non-exceptional paths. The function $g$ is called a \textit{weak upper gradient} of $u$. Let $I(u)=\inf_g \|g\|_{p}$ where the infimum is taken over all weak upper gradients of $u$. A $p$-harmonic function in an open set $\Omega\subset X$ with boundary data $f\in N^{1,p}(X)$ is a function  that minimizes the \textit{energy functional} $I(u)$ over functions $u\in N^{1,p}(X)$ with $u\big|_{X\setminus \Omega}\equiv f \big|_{X\setminus \Omega}$. As already remarked, the usual assumptions on the space $X$ for this theory to go through is that it is doubling and supports a Poincar\'e inequality.

In our setting, we follow a slightly different approach and we do not use measure and integration in the carpet $S$ to study Sobolev functions, but we rather put the focus on studying the ``holes" $Q_i$ of the carpet. Hence, we do not make any assumptions on the intrinsic geometry of the carpet $S$, other than it has Lebesgue measure zero, but we require that the holes $Q_i$ satisfy some uniform geometric conditions; see Section \ref{harmonic:1-Basic Assumptions}. In particular, we do not assume that the carpet $S$ supports a Poincar\'e inequality or a doubling measure. What is special about the theory that we develop is that Sobolev funcions and harmonic functions will acknowledge in some sense the existence of the ambient space, where the carpet lives.

The precise definitions will be given later in Sections \ref{harmonic:3-Sobolev} and \ref{harmonic:Section Carpet Harmonic} but here we give a rough sketch. A function $u\colon S\to \R$ lies in the Sobolev space $\mathcal W^{1,2}(S)$ if it satisfies a certain $L^2$-integrability condition and it has  an \textit{upper gradient} $\{\rho(Q_i)\}_{i\in \N}$, which is a square-summable sequence with the property that 
\begin{align*}
|u(x)-u(y)|\leq \sum_{i:Q_i\cap \gamma\neq \emptyset}\rho(Q_i)
\end{align*} 
for \textit{almost every} path $\gamma \subset \Omega$ and points $x,y\in \gamma\cap S$. We remark here that the path $\gamma$ will also travel through the ambient space $\Omega$, and does not stay entirely in the carpet $S$. Here, ``almost every" means that we exclude a family of ``pathological" paths of \textit{carpet modulus} equal to zero; see Section \ref{harmonic:2-Carpet modulus} for the definition. This is necessary, because there exist (a lot of) paths $\gamma$ that are entirely contained in the carpet $S$ without intersecting any peripheral disk $\br Q_i$. For such paths the sum $\sum_{i:Q_i\cap \gamma\neq \emptyset} \rho(Q_i)$ would be $0$, and thus a function $u$ satisfying the upper gradient inequality for \textit{all} paths would be constant.

In order to define a \textit{carpet-harmonic} function, one then minimizes the \textit{energy functional} $\sum_{i\in \N}\rho(Q_i)^2$ over all Sobolev functions that have given boundary data. This energy functional corresponds to the classical Dirichlet energy $\int |\nabla u|^2$ of a classical Sobolev function in the plane.

We will develop this theory for a generalization of Sierpi\'nski carpets called \textit{relative Sierpi\'nski carpets}. The difference to a Sierpi\'nski carpet is that here we will actually allow the set $\Omega$ to be an arbitrary (connected) open set in the plane, and not necessarily a Jordan region. So, we start with an open set $\Omega\subset \C$ and we  remove the  countably many {peripheral disks} $Q_i$ from $\Omega$ as in the definition of a Sierpi\'nski carpet; see Section \ref{harmonic:1-Basic Assumptions} for definition. This should be regarded as a generalization of relative Schottky sets studied in \cite{Merenkov:relativeSchottky}, where all peripheral disks $Q_i$ are round disks. This generalization allows us, for example, to set $\Omega=\C$ and obtain an analog of Liouville's theorem, that bounded carpet-harmonic functions are constant	.

Under certain assumptions on the geometry of the peripheral disks $Q_i$ (see Section \ref{harmonic:1-Basic Assumptions}) we obtain the following results (or rather discrete versions of them) for carpet-harmonic functions:
\begin{itemize}
\item Solution to the Dirichlet problem; see Section \ref{harmonic:Section Carpet Harmonic}.
\item Continuity, maximum principle, uniqueness of the solution to the Dirichlet problem, comparison principle; see Section \ref{harmonic:Section Properties of Harmonic}.
\item Caccioppoli inequality; see Section \ref{harmonic:Section Caccioppoli}.
\item Harnack's inequality, Liouville's theorem, strong maximum principle; see Section \ref{harmonic:Section Harnack}.
\item Local equicontinuity and compactness of harmonic functions; see Section \ref{harmonic:Section Equicontinuity}.
\end{itemize}

\section*{Acknowledgments}

This work is part of my PhD thesis at UCLA. I would like to thank my advisor, Mario Bonk, for suggesting the study of potential theory on carpets and guiding me throughout this research at UCLA. He has been a true teacher, sharing his expertise in the field and constantly providing deep insight to my questions, while being always patient and supportive. Moreover, I thank him for his thorough reading of this work and for his comments and corrections, which substantially improved the presentation.


%% file: harmonic.tex
\chapter{Harmonic functions on Sierpi\'nski carpets}\label{Chapter:Harmonic}

\section{Introduction}
In this chapter we introduce and study notions of Sobolev spaces and harmonic functions on Sierpi\'nski carpets. We briefly describe here some of the applications of carpet-harmonic functions, and then the organization of the current chapter.

In Chapter \ref{Chapter:Uniformization}, carpet-harmonic functions are applied towards a uniformization result. In particular, it is proved there that Sierpi\'nski carpets, under the geometric assumptions  described in Section \ref{harmonic:1-Basic Assumptions}, can be uniformized by \textit{square} carpets. This is done by constructing a ``harmonic conjugate" of a certain carpet-harmonic function, and modifying the methods used in \cite{Rajala:uniformization}. The uniformizing map is not quasisymmetric, in general, but it is quasiconformal in a discrete sense. If the assumptions on the {peripheral circles} are strengthened to \textit{uniformly relatively separated} (see Remark \ref{harmonic:Uniform relative separation} for the definition), \textit{uniform quasicircles}, then the map is actually a quasisymmetry. 

Carpet-harmonic functions also seem to be useful in the study of rigidity problems for quasisymmetric or bi-Lipschitz maps between square Sierpi\'nski carpets. The reason is that the real and imaginary parts of such functions are  carpet-harmonic, under some conditions; see Corollary \ref{harmonic:Pullback square carpets}. Such a rigidity problem is studied in \cite{BonkMerenkov:rigidity}, where it is shown that the only quasisymmetric self-maps of the standard Sierpi\'nski carpet are Euclidean isometries. In Theorem \ref{harmonic:Rigidity} we use the theory of carpet-harmonic functions to show an elementary rigidity result, which was already established in \cite[Theorem 1.4]{BonkMerenkov:rigidity}, for mappings between square Sierpi\'nski carpets that preserve the sides of the unbounded peripheral disk. It would be very interesting to find a proof of the main result in \cite{BonkMerenkov:rigidity} using carpet-harmonic functions.

The sections of the chapter are organized as follows. In Section \ref{harmonic:1-Basic Assumptions} we introduce our notation and our basic assumptions on the geometry of the peripheral disks.

In Section \ref{harmonic:2-Carpet modulus} we discuss notions of \textit{carpet modulus} that will be useful in studying path families in Sierpi\'nski carpets, and, in particular, in defining families of modulus zero which contain ``pathological" paths that we wish to exclude from our study. In Section \ref{harmonic:Section Existence of paths} we prove the existence of paths with certain properties that avoid the ``pathological" families of modulus zero. 

In Section \ref{harmonic:3-Sobolev} we finally introduce Sobolev spaces, starting first with a preliminary notion of a discrete Sobolev function, and then deducing the definition of a Sobolev function. We also study several properties of these functions and give examples.

Section \ref{harmonic:Section Carpet Harmonic} discusses the solution to the Dirichlet problem on carpets. Then in Section \ref{harmonic:Section Properties of Harmonic} we establish several classical properties of harmonic functions, including the continuity, the maximum principle, the uniqueness of the solution to the Dirichlet problem, and the comparison principle. We also prove a discrete analog of the Caccioppoli inequality in Section \ref{harmonic:Section Caccioppoli}.

Some more fine properties of carpet-harmonic functions are discussed Section \ref{harmonic:Section Harnack}, where we show Harnack's inequality, the analog of Liouville's theorem, and the strong maximum principle. We finish this chapter with Section \ref{harmonic:Section Equicontinuity}, where we study equicontinuity and convergence properties of carpet-harmonic functions.

\section{Basic assumptions and notation}\label{harmonic:1-Basic Assumptions}
We denote $\widehat \R=\R \cup \{-\infty,+\infty\}$, and $\widehat \C= \C \cup \{\infty\}$. A function that attains values in $\widehat \R$ is called an extended function. We use the standard open ball notation $B(x,r)=\{y\in \R^2: |x-y|<r\}$ and $\br B(x,r)$ is the closed ball. If $B=B(x,r)$ then $cB=B(x,cr)$. Also, $A(x;r,R)$ denotes the annulus $B(x,R)\setminus \br B(x,r)$, for $0<r<R$. All the distances will be in the Euclidean distance of $\C\simeq \R^2$. A point $x$ will denote most of the times a point of $\R^2$ and rarely we will use the notation $(x,y)$ for coordinates of a point in $\R^2$, in which case $x,y\in \R$. Each case will be clear from the context.

Let $\Omega  \subset \C$ be a connected open set, and let $\{Q_i\}_{i\in \N}$ be a collection of (open) Jordan regions compactly contained in $\Omega$, with disjoint closures, such that the set $S\coloneqq \Omega \setminus \bigcup_{i\in \N} Q_i$ has empty interior and is locally connected. The latter will be true if and only if for every ball $B(x,r)$ that is compactly contained in $\Omega$ the Jordan regions $Q_i$ with $Q_i\cap B(x,r)\neq \emptyset $ have diameters shrinking to $0$. We call the pair $(S,\Omega)$ a \textit{relative Sierpi\'nski carpet}\index{Sierpi\'nski carpet!relative}. We will often drop $\Omega$ from the notation, and just call $S$ a relative Sierpi\'nski carpet. The Jordan regions $Q_i$ are called the \textit{peripheral disks}\index{peripheral disk} of $S$, and the boundaries $\partial Q_i$ are the \textit{peripheral circles}\index{peripheral circle}. Note here that $\partial \Omega \cap S=\emptyset$. The definition of a relative Sierpi\'nski carpet is motivated by the fact that if $\Omega$ is a Jordan region, then $\br S$ is a Sierpi\'nski carpet in the usual sense, as defined in the Introduction. See Figure \ref{harmonic:fig:boundary} for a Sierpi\'nski carpet, and Figure \ref{harmonic:fig:square} for a relative Sierpi\'nski carpet, in which $\Omega$ has two boundary components.

We will impose some further assumptions on the geometry of the peripheral disks $\{Q_i\}_{i\in \N}$. First, we assume that they are \textit{uniform quasiballs}\index{quasiball}, i.e., there exists a uniform constant $K_0\geq 1$ such that for each $Q_i$ there exist concentric balls 
\begin{align}\label{harmonic:1-Quasi-balls}
B(x,r) \subset Q_i \subset B(x,R),
\end{align}
with $R/r\leq K_0$. 

Second, we assume that the peripheral disks are \textit{uniformly fat sets}\index{fat set}, i.e., there exists a uniform constant $K_1>0$ such that for every $Q_i$ and for every ball $B(x,r)$ centered at some $x\in Q_i$ with $r<\diam (Q_i)$ we have 
\begin{align}\label{harmonic:1-Fat sets}
\mathcal H^2( B(x,r)\cap Q_i) \geq K_1 r^2,
\end{align}
where by $\mathcal H^m$ we denote the $m$-dimensional Hausdorff measure, normalized so that it agrees with the $m$-dimensional Lebesgue measure, whenever $m\in \N$. 

A Jordan curve $J\subset \R^2$ is a $K$-quasicircle\index{quasicircle} for some $K>0$, if for any two points $x,y\in J$ there exists an arc $\gamma \subset J$ with endpoints $x,y$ such that $|x-y|\leq K \diam (\gamma)$. Note that if the peripheral circles $\partial Q_i$ are uniform quasicircles (i.e., $K$-quasicircles with the same constant $K$), then they are both uniform quasiballs and uniformly fat sets. The first claim is proved in \cite[Proposition 4.3]{Bonk:uniformization} and the second in \cite[Corollary 2.3]{Schramm:transboundary}, where the notion of a fat set appeared for the first time in the study of conformal maps. Another example of Jordan regions being quasiballs and fat sets are \textit{John domains}; see \cite{SmithStegenga:HolderPoincare} for the definition and properties of John domains. We remark that the boundary of such a domain has strictly weaker properties than those of a quasicircle.

Finally, we assume that $\mathcal H^2(S)=0$. In the following, a relative Sierpi\'nski carpet $(S,\Omega)$ (which will also be denoted by $S$ if $\Omega$ is implicitly understood) will always be assumed to have area zero and peripheral disks that are uniform quasiballs and uniformly fat sets. These will also be referred to as the \textit{standard assumptions}\index{standard assumptions}. We say that a constant $c>0$ \textit{depends on the data of the carpet $S$}, if it depends only on the quasiball and fatness constants $K_0$ and $K_1$, respectively.

The notation $V\subset \subset \Omega$ means that $\br V$ is compact and is contained in $\Omega$. Alternatively, we say that $V$ \textit{is compactly contained in} $\Omega$. For a set $E\subset \R^2$ and $\delta>0$ we denote by $N_\delta(E)$ the open  $\delta$-neighborhood of $E$
\begin{align*}
\{x\in \R^2 : \dist(x,E)<\delta\}.
\end{align*}
A \textit{continuum} $E\subset \R^2$ is a compact and connected set. A continuum $E$ is \textit{non-trivial} if it contains at least two points. Making slight abuse of notation and for visual purposes, we use $\br Q_i$ to denote the closure of $Q_i$, instead of using $\br{Q_i}$.

A \textit{path} or \textit{curve}\index{path}\index{curve} $\gamma$ is a continuous function $\gamma \colon I\to \R^2$, where $I\subset \R$ is a bounded interval, such that $\gamma$ has a continuous extension $\br \gamma:\br I\to \R^2$, i.e., $\gamma$ has endpoints. A \textit{closed} path $\gamma$ is a path with $I=[0,1]$ and an \textit{open} path $\gamma$ is a path with $I=(0,1)$. We will also use the notation $\gamma \subset \R^2$ for the image of the path as a set. A \textit{subpath} or \textit{subcurve} of a path $\gamma\colon I\to \R^2$ is the restriction of $\gamma$ to a subinterval of $I$. A Jordan curve is a homeomorphic image of the unit circle $S^1$, and a Jordan arc is homeomorphic to $[0,1]$. 

We denote by $S^\circ$ the points of the relative Sierpi\'nski carpet $S$ that do not lie on any peripheral circle $\partial Q_i$. For an open set $V\subset \Omega$ define $\partial_*V= S\cap \partial V$; see Figure \ref{harmonic:fig:boundary}. For a set $V$ that intersects the relative Sierpi\'nski carpet $S$  we define the index set $I_V=\{i\in \N: Q_i\cap V\neq \emptyset\}$.

In the proofs we will denote constants by $C,C',C'',\dots$, where the same symbol can denote a different constant if there is no ambiguity. 

\section{Notions of carpet modulus}\label{harmonic:2-Carpet modulus}
The carpet modulus is a generalization of the transboundary modulus introduced by Schramm in \cite{Schramm:transboundary}. Several properties of the carpet modulus were studied in \cite[Section 2]{BonkMerenkov:rigidity}.

Let $(S,\Omega)$ be a relative Sierpi\'nski carpet with the standard assumptions, and let $\Gamma$ be a family of paths in $\Omega$.

Let us recall first the definition of \textit{conformal modulus} or \textit{$2$-modulus}\index{modulus!conformal modulus}\index{modulus!$2$-modulus} of a path family $\Gamma$ in $\Omega$. A non-negative Borel function $\lambda$ on $\Omega$ is \textit{admissible} for the conformal modulus $\md_2(\Gamma)$ if 
\begin{align*}
\int_\gamma \lambda \,ds\geq 1
\end{align*} 
for all locally rectifiable paths $\gamma\in \Gamma$. If a path $\gamma$ is not locally rectifiable, we define $\int_\gamma \lambda \,ds=\infty$, even when $\lambda\equiv 0$. Hence, we may require that the above inequality holds for all $\gamma\in \Gamma$. Then $\md_2(\Gamma)= \inf  \int \lambda^2 \,d\mathcal H^2$ where the infimum is taken over all admissible functions.

A sequence of non-negative numbers $\{\rho(Q_i)\}_{i\in \N}$ is admissible for the \textit{weak (carpet) modulus}\index{modulus!weak carpet modulus} $\md_w(\Gamma)$ if there exists a path family $\Gamma_0 \subset \Gamma$ with $\md_2(\Gamma_0)=0$ such that
\begin{align}\label{harmonic:2-weak carpet modulus}
\sum_{i:Q_i\cap \gamma\neq \emptyset} \rho(Q_i)\geq 1
\end{align} 
for all $\gamma \in \Gamma\setminus \Gamma_0$. Note that in the sum each peripheral disk is counted once, and we only include the peripheral disks whose interior is intersected by $\gamma$, and not just the boundary. Then we define $\md_w(\Gamma)= \inf \sum_{i\in \N} \rho(Q_i)^2$ where the infimum is taken over all admissible weights $\rho$. 

Similarly we define the notion of \textit{strong (carpet) modulus}\index{modulus!strong carpet modulus}. A sequence of non-negative numbers $\{\rho(Q_i)\}_{i\in \N}$ is admissible for the \textit{strong carpet modulus} $\md_s(\Gamma)$ if 
\begin{align}\label{harmonic:2-strong carpet modulus}
\sum_{i:Q_i\cap \gamma\neq \emptyset} \rho(Q_i)\geq 1
\end{align}
for all $\gamma \in \Gamma$ that satisfy $\mathcal H^1(\gamma\cap S)=0$. Note that the path $\gamma$ could be non-rectifiable inside some $Q_i$. Then $\md_s(\Gamma)\coloneqq\inf \sum_{i\in \N} \rho(Q_i)^2$ where the infimum is taken over all admissible weights $\rho$.

For properties of the conformal modulus see \cite[Section 4.2, p.~133]{LehtoVirtanen:quasiconformal}. It can be shown as in the conformal case that both notions of carpet modulus satisfy monotonicity and countable subadditivity, i.e., if $\Gamma_1\subset \Gamma_2$ then $\md (\Gamma_1)\leq \md(\Gamma_2)$ and 
$$\md\left(\bigcup_{i\in \N}\Gamma_i \right)\leq \sum_{i\in \N} \md(\Gamma_i),$$
where $\md$ is either $\md_w$ or $\md_s$. 

The following lemma provides some insight for the relation between the two notions of carpet modulus.

\begin{lemma}\label{harmonic:2-weak modulus less than strong}
For any path family $\Gamma$ in $\Omega$ we have
\begin{align*}
\md_w(\Gamma)\leq \md_s(\Gamma).
\end{align*}
\end{lemma}
\begin{proof} Let $\{\rho(Q_i)\}_{i\in \N}$ be admissible for $\md_s(\Gamma)$, so $\sum_{i:Q_i\cap \gamma\neq \emptyset} \rho(Q_i)\geq 1$ for all $\gamma \in \Gamma$ with $\mathcal H^1(\gamma\cap S)=0$. Define $\Gamma_0=\{\gamma\in \Gamma: \mathcal H^1(\gamma\cap S)>0\}$. Then the function $\lambda= \infty \cdot \x_{S}$ is admissible for $\md_2(\Gamma_0)$. Since $\mathcal H^2(S)=0$, it follows that $\md_2(\Gamma_0)=0$. Hence, $\sum_{i:Q_i\cap \gamma\neq \emptyset}\rho(Q_i)\geq 1$ for all $\gamma \in \Gamma\setminus \Gamma_0$, which shows the admissibility of $\{\rho(Q_i)\}_{i\in \N}$ for the weak carpet modulus $\md_w(\Gamma)$.
\end{proof}

A version of the next lemma can be found in \cite[Lemma 2.2]{BonkMerenkov:rigidity} and \cite{Bojarski:inequality}.

\begin{lemma}\label{harmonic:Bojarski}
Let $\kappa\geq 1$ and $I$ be a countable index set. Suppose that $\{B_i\}_{i\in I}$ is a collection of balls in $\R^2$, and $a_i$, $i\in I$, are non-negative real numbers. Then there exists a constant $C>0$ depending only on $\kappa$ such that
\begin{align*}
\biggl \| \sum_{i\in I} a_i\x_{\kappa B_i} \biggr\|_2 \leq C \biggl\| \sum_{i\in I} a_i \x_{B_i} \biggr\|_2.
\end{align*}
\end{lemma}

Here $\|\cdot\|_2$ denotes the $L^2$-norm with respect to planar Lebesgue measure. We will also need the next lemma.

\begin{lemma}\label{harmonic:2-weak modulus zero implies two modulus zero}
For a path family $\Gamma$ in $\Omega$ we have the equivalence
\begin{align*}
\md_w(\Gamma)=0 \quad\textrm{if and only if}\quad \md_2(\Gamma)=0.
\end{align*}
\end{lemma}

Before starting the proof, we require the following consequence of the fatness assumption.

\begin{lemma}\label{harmonic:Fatness consequence}
Let $E$ be a compact subset of $\R^2$. Then for each $\varepsilon>0$ there exist at most finitely many  peripheral disks $Q_i$ intersecting $E$ with diameter larger than $\varepsilon$. Moreover, the spherical diameters of the peripheral disks $Q_i$ converge to $0$.
\end{lemma}
\begin{proof}
Note first that by an area argument no ball $B(0,R)$ can \textit{contain} infinitely many peripheral disks $Q_i$ with $\diam(Q_i)>\varepsilon>0$. Indeed, the fatness condition implies that 
\begin{align*}
\mathcal H^2(Q_i) \geq K_1 \diam(Q_i)^2.
\end{align*}
Since the peripheral disks $Q_i$ are disjoint we have 
\begin{align*}
\mathcal H^2(B(0,R)) \geq \sum_{i:Q_i\subset B(0,R)} \mathcal H^2(Q_i) \geq K_1 \sum_{i:Q_i\subset B(0,R)} \diam(Q_i)^2.
\end{align*}
Hence, we see that only finitely many of them can satisfy $\diam(Q_i)>\varepsilon$.

If there exist infinitely many $Q_i$ intersecting $E$ with $\diam(Q_i)>\varepsilon$ then we necessarily have $Q_i \to \infty$. In particular there exists a ball $B(0,R)\supset E$ such that there are infinitely many $Q_i$ intersecting both $\partial B(0,R)$ and $\partial B(0,2R)$. The fatness assumption implies now that for all such $Q_i$ we have 
\begin{align*}
\mathcal H^2(Q_i \cap (B(0,2R)\setminus B(0,R))) \geq CR^2 \geq C\varepsilon^2
\end{align*}
for a constant $C>0$ depending only on $K_1$; see Remark \ref{harmonic:Remark:Fatness implication} below. Hence, an area argument as before yields the conclusion.

For the final claim, suppose that there exist infinitely many peripheral disks with spherical diameters bounded below. Then there exists a compact set $E\subset \R^2$ intersecting infinitely many of these peripheral disks. In a neighborhood of $E$ the spherical metric is comparable to the Euclidean metric, hence there are infinitely many peripheral disks intersecting $E$ with Euclidean diameters bounded below. This contradicts the previous part of the lemma.
\end{proof}

\begin{remark}\label{harmonic:Remark:Fatness implication}
In the preceding proof we used the fact that if $\br Q_i$ intersects two circles $\partial B(x,r)$ and $\partial B(x,R)$ with $0<r<R$, then
\begin{align*}
\mathcal H^2(Q_i \cap (B(x,R)\setminus B(x,r))) \geq C (R-r)^2
\end{align*}
for a constant $C>0$ depending only on $K_1$. To see that, by the connectedness of $Q_i$ there exists a point $y\in Q_i \cap \partial B(x, (r+R)/2)$. Then $B(y, (R-r)/2) \subset B(x,R)\setminus B(x,r)$, so
\begin{align*}
\mathcal H^2(Q_i \cap (B(x,R)\setminus B(x,r)))\geq \mathcal H^2(Q_i\cap B(y,(R-r)/2)) \geq K_1 \frac{(R-r)^2}{4},
\end{align*}
by the fatness condition \eqref{harmonic:1-Fat sets}. 
\end{remark}

\begin{corollary}\label{harmonic:Fatness corollary}
Let $E$ be a compact subset of $\R^2$. Then 
\begin{align*}
\sum_{i\in I_E}\diam(Q_i)^2<\infty.
\end{align*}
\end{corollary}
Recall that $I_E=\{i\in \N: Q_i\cap E\neq \emptyset\}$.

\begin{proof}
Let $B(0,R)$ be a large ball that contains $E$. Then there are finitely many peripheral disks intersecting $B(0,R)$ and having diameter greater than $R/2$, by Lemma \ref{harmonic:Fatness consequence}. Hence, it suffices to show 
\begin{align*}
\sum_{i:Q_i\subset B(0,2R)} \diam(Q_i)^2<\infty.
\end{align*}
Using the fatness, one can see that this sum is bounded above by a multiple of $\mathcal H^2(B(0,2R))$, as in the proof of Lemma \ref{harmonic:Fatness consequence}.
\end{proof}

\begin{proof}[Proof of Lemma \ref{harmonic:2-weak modulus zero implies two modulus zero}] One direction is trivial, namely if $\md_2(\Gamma)=0$ then $\md_w(\Gamma)=0$, since the weight $\rho(Q_i)\equiv 0$ is admissible.

For the converse, note first that $\Gamma$ cannot contain constant paths if $\md_w(\Gamma)=0$. Indeed, assume that $\Gamma$ contains a constant path $\gamma =x_0 \in \Omega$, and $\{\rho(Q_i)\}_{i\in \N}$ is an admissible weight for $\md_w(\Gamma)$. Then there exists an exceptional family $\Gamma_0\subset \Gamma$ with $\md_2(\Gamma_0)=0$ such that 
\begin{align*}
\sum_{i:Q_i\cap \gamma\neq \emptyset} \rho(Q_i)\geq 1
\end{align*}
for all $\gamma\in \Gamma\setminus \Gamma_0$. The constant path $\gamma=x_0$ cannot lie in $\Gamma_0$, otherwise we would have $\md_2(\Gamma_0)=\infty$ because no function would be admissible for $\md_2(\Gamma_0)$. Hence, we must have
\begin{align*}
\sum_{i:x_0\in Q_i} \rho(Q_i)\geq 1.
\end{align*}
If $x_0\in S$, then this cannot happen since the sum is empty, so the only possibility is that $x_0\in Q_{i_0}$ for some $i\in \N$. In this case we have $\rho(Q_{i_0})\geq 1$. Hence, $\sum_{i\in \N} \rho(Q_i)^2\geq 1$, which implies that $\md_w(\Gamma)\geq 1$, a contradiction.

We now proceed to showing the implication. By the subadditivity of $2$-modulus, it suffices to show that the family $\Gamma_\delta$ of paths in $\Gamma$ that have diameter bounded below by $\delta>0$ has conformal modulus zero. Indeed, this will exhaust all paths of $\Gamma$, since $\Gamma$ contains no constant paths. For simplicity we denote  $\Gamma_\delta=\Gamma$, and note that we have $\md_w(\Gamma)=0$, using the monotonicity of modulus.

For $\varepsilon >0$ let $\{\rho(Q_i)\}_{i\in \N}$ be a weight such that $\sum_{i\in \N}\rho(Q_i)^2<\varepsilon$ and 
$$\sum_{i:Q_i\cap \gamma \neq \emptyset}\rho(Q_i)\geq 1$$
for $\gamma\in \Gamma\setminus \Gamma_0$, where $\Gamma_0$ is a path family with $\md_2(\Gamma_0)=0$. Using $\varepsilon =1/2^n$ and summing the corresponding weights $\rho$, as well as, taking the union of the exceptional families $\Gamma_0$, we might as well obtain a weight $\{\rho(Q_i)\}_{i\in \N}$ and an exceptional family $\Gamma_0$ such that $\sum_{i\in \N}\rho(Q_i)^2 <\infty$ and 
\begin{align}\label{harmonic:2-weak modulus Proof}
\sum_{i:Q_i\cap \gamma \neq \emptyset}\rho(Q_i)=\infty
\end{align}
for all $\gamma\in \Gamma\setminus \Gamma_0$, where $\md_2(\Gamma_0)=0$. 

We construct an admissible function $\lambda \colon\C \to [0,\infty]$ for $\md_2(\Gamma)$ as follows. Since the peripheral disks $Q_i$ are uniform quasiballs, there exist balls $B(x_i,r_i)\subset Q_i\subset B(x_i,R_i)$ with $R_i/r_i\leq K_0$. We define
\begin{align*}
\lambda= \sum_{i\in \N} \frac{\rho(Q_i)}{R_i}\x_{B(x_i,2R_i)}.
\end{align*}
Note that if $\gamma$ intersects some $Q_i$ with $4R_i<\delta$, then $\gamma$ must exit $B(x_i,2R_i)$, so $\int_\gamma \x_{B(x_i,2R_i)} \, ds\\ \geq R_i$. If $\gamma$ is a bounded path, i.e., it is contained in a ball $B(0,R)$, then there are only finitely many peripheral disks $Q_i$ intersecting $\gamma$ and satisfying $4R_i\geq \delta$. This follows from Lemma \ref{harmonic:Fatness consequence} and the fact that $\diam(Q_i)\geq r_i\geq R_i/K_0$ from the quasiballs assumption. Thus, we have
$$\sum_{\substack{i: Q_i\cap \gamma\neq \emptyset\\ 4R_i\geq \delta}} \rho(Q_i)<\infty,$$
since it is a finite sum. This implies that
\begin{align*}
\int_\gamma \lambda \,ds\geq \sum_{\substack{i:Q_i\cap \gamma\neq \emptyset \\ 4R_i<\delta}} \frac{\rho(Q_i)}{R_i}\int_\gamma \x_{B(x_i,2R_i)} \,ds\geq \sum_{\substack{i:Q_i\cap \gamma\neq \emptyset \\ 4R_i<\delta}} \rho(Q_i) =\infty
\end{align*}
by \eqref{harmonic:2-weak modulus Proof}, whenever $\gamma \in \Gamma\setminus \Gamma_0$. Now, if $\gamma \in \Gamma\setminus \Gamma_0$ is an unbounded path, then $\gamma$ always exits $B(x_i,2R_i)$ whenever $Q_i\cap \gamma\neq \emptyset$, so in this case we also have
\begin{align*}
\int_\gamma \lambda \,ds =\infty.
\end{align*}

Using Lemma \ref{harmonic:Bojarski} we obtain
\begin{align*}
\|\lambda\|_2 &\leq C \biggl\|\sum_{i\in \N} \frac{\rho(Q_i)}{R_i}\x_{B(x_i,R_i/K_0)} \biggr\|_2 \leq C \biggl\|\sum_{i\in \N} \frac{\rho(Q_i)}{R_i}\x_{B(x_i,r_i)} \biggr\|_2 \\
&\leq C'\left(\sum_{i\in \N} \rho(Q_i)^2 \right)^{1/2}<\infty.
\end{align*}
since the balls $B(x_i,r_i)$ are disjoint. This implies that $\md_2(\Gamma\setminus \Gamma_0)=0$. Thus, 
\[
\md_2(\Gamma)\leq \md_2(\Gamma\setminus \Gamma_0)+\md_2(\Gamma_0)=0. \qedhere
\]
\end{proof}

\begin{remark}Observe that families of paths passing through a single point $p\in Q_i$ would have conformal modulus and thus weak carpet modulus equal to zero, but this is not the case when we use the strong modulus. Thus, the notion of strong modulus is more natural for carpets, rather than the weak. In what follows we will study in parallel the two notions, pointing out the differences whenever they occur.
\end{remark}

Finally, we recall Fuglede's lemma\index{Fuglede's lemma} in this setting:

\begin{lemma}\label{harmonic:Fuglede}
Let $\{\rho(Q_i)\}_{i\in \N}$ and $\{\rho_n(Q_i)\}_{i\in \N}$ for $n\in \N$ be non-negative weights in $\ell^2(\N)$ such that $\rho_n \to \rho$ in $\ell^2(\N)$, i.e.
\begin{align*}
\sum_{i\in \N}|\rho_n(Q_i)-\rho(Q_i)|^2 \to 0
\end{align*}
as $n\to\infty$. Then there exists a subsequence $\{\rho_{k_n}(Q_i)\}_{i\in \N}$, $n\in \N$, and an exceptional family $\Gamma_0$ with $\md_s(\Gamma_0)=0$ such that for all paths $\gamma\subset \Omega$ with $\gamma\notin \Gamma_0$ we have
\begin{align*}
\sum_{i:Q_i\cap \gamma\neq \emptyset} |\rho_{k_n}(Q_i)-\rho(Q_i)|\to 0
\end{align*}
as $n\to\infty$.
\end{lemma}
The proof is a simple adaptation of the conformal modulus proof but we include it here for the sake of completeness. The argument is essentially contained in the proof of \cite[Proposition 2.4, pp.~604--605]{BonkMerenkov:rigidity}.

\begin{proof}
Without loss of generality, we may assume that $\rho_n\geq 0$ and $\rho_n\to 0 $ in $\ell^2(\N)$. We consider a subsequence $\rho_{k_n}$ such that
\begin{align*}
\sum_{i\in \N} \rho_{k_n}(Q_i)^2< \frac{1}{2^n}
\end{align*}
for all $n\in \N$. By the subadditivity of strong modulus, it suffices to show that for each $\delta>0$ the path family
\begin{align*}
\Gamma_0\coloneqq \{ \gamma \subset \Omega: \limsup_{n\to \infty}\sum_{i:Q_i\cap \gamma\neq \emptyset} \rho_{k_n}(Q_i) > \delta \}
\end{align*}
has strong modulus zero. 

Let 
\begin{align*}
\lambda\coloneqq \sum_{n=1}^\infty \rho_{k_n}
\end{align*}
and note that
\begin{align*}
\sum_{i:Q_i\cap \gamma\neq \emptyset}\lambda(Q_i) = \sum_{n=1}^\infty \sum_{i:Q_i\cap \gamma \neq \emptyset} \rho_{k_n}(Q_i) =\infty
\end{align*}
for all $\gamma\in \Gamma_0$. On the other hand,
\begin{align*}
\sum_{i\in \N} \lambda(Q_i)^2 &=\|\lambda(\cdot) \|_{\ell^2( \{Q_i:i\in \N\})}= \biggl \|\sum_{n=1}^\infty \rho_{k_n}(\cdot) \biggr\|_{\ell^2( \{Q_i:i\in \N\})}\\
&\leq \sum_{n=1}^\infty \| \rho_{k_n}(\cdot) \|_{\ell^2( \{Q_i:i\in \N\})}\leq \sum_{n=1}^\infty  \frac{1}{2^n}<\infty.
\end{align*}
Since $\varepsilon\cdot \lambda$ is admissible for $\md_s(\Gamma_0)$ for all $\varepsilon>0$, it follows that $\md_s(\Gamma_0)=0$.
\end{proof}

\section{Existence of paths}\label{harmonic:Section Existence of paths}
In this section we will show the existence of paths that avoid given families of (weak, strong, conformal) modulus  equal to $0$. These paths will therefore be ``good" paths for which we can apply, e.g., Fuglede's lemma. We will use these good paths later to prove qualitative estimates, such as continuity of carpet-harmonic functions.

First we recall some facts. The \textit{co-area formula}\index{co-area formula} and \textit{area formula}\index{area formula} in the next proposition are contained in \cite[Theorem 3.2.12]{Federer:gmt} and \cite[Theorem 3.2.3]{Federer:gmt}. 

\begin{prop}\label{harmonic:Coarea}
Let $T\colon\R^2\to \R$ be an $L$-Lipschitz function and $g$ be a non-negative measurable function on $\R^2$. Then the function $x\mapsto  \int_{T^{-1}(x)} g(y) \, d\mathcal H^1(y)$ is measurable, and there is a constant $C>0$ depending only on $L$ such that:
\begin{align*}
\tag{Co-area formula} \int_\R \left(\int_{T^{-1}(x)} g(z) \,d\mathcal H^1(z) \right) dx \leq  C \int_{\R^2} g(z)\,d\mathcal H^2(z),
\end{align*}
and
\begin{align*}
\tag{Area formula} \int_\R \sum_{z\in T^{-1}(x)}g(z) \, dx\leq C \int_{\R^2}g(z) \,d\mathcal H^1(z).
\end{align*}
\end{prop}

We say that a path $\alpha$ \textit{joins} or \textit{connects} two sets $E,F$ if $\br \alpha$ intersects both $E$ and $F$. The following proposition asserts that perturbing a curve yields several nearby curves; see \cite[Theorem 3]{Brown:distancesets}.

\begin{prop}\label{harmonic:Paths-distance function}
Let $\alpha \subset \R^2$ be a closed path that joins two non-trivial, disjoint continua $E,F \subset \R^2$. Consider the distance function $\psi(x)=\dist(x,\alpha )$. Then there exists $\delta>0$ such that for a.e.\ $s\in (0,\delta)$ there exists a simple path $\alpha_s\subset \psi^{-1}(s)$ joining $E$ and $F$.  
\end{prop}

Using that, we show:

\begin{lemma}\label{harmonic:Paths joining continua}
Let $\alpha\subset \Omega$ be a closed path joining two non-trivial, disjoint continua $E,F\subset\subset  \Omega$, and let $\Gamma$ be a given family of (weak, strong, conformal) modulus zero. Then, there exists $\delta>0$ such that for a.e.\ $s\in (0,\delta)$ there exists a simple path $\alpha_s\subset \psi^{-1}(s)$ that lies in $\Omega$, joins the continua $E$ and $F$, and lies outside the family $\Gamma$. Furthermore, if $A\subset \Omega$ is a given set with $\mathcal H^1(A)=0$, then for a.e.\ $s\in (0,\delta)$ the path $\alpha_s$ does not intersect $A$.
\end{lemma}
\begin{proof}
Note that if $\Gamma$ has strong or weak modulus zero, then it actually has conformal modulus zero, by Lemma \ref{harmonic:2-weak modulus less than strong} and Lemma \ref{harmonic:2-weak modulus zero implies two modulus zero}. Hence, it suffices to assume that $\md_2(\Gamma)=0$.

For $\varepsilon>0$ there exists an admissible function $\lambda$ such that $\int_\gamma \lambda \,ds\geq 1$ for all $\gamma\in \Gamma$, and $\|\lambda\|_2<\varepsilon$. Consider a small $\delta>0$ such that $N_\delta(\alpha)\subset \subset \Omega$ and the conclusion of Proposition \ref{harmonic:Paths-distance function} is true. Let $J$ be the set of $s\in (0,\delta)$ such that $\alpha_s\in \Gamma$, and $J'$ be the set of $s\in (0,\delta)$ such that $\int_{\psi^{-1}(s)}\lambda\, d\mathcal H^1 \geq 1$. It is clear that $J\subset J'$, and $J'$ is measurable by Proposition \ref{harmonic:Coarea}, since the function $\psi$ is $1$-Lipschitz. Thus, applying the co-area formula in Proposition \ref{harmonic:Coarea} and the Cauchy-Schwarz inequality, we have
\begin{align*}
\mathcal H^1(J) & \leq \mathcal H^1(J')= \int_{J'} \,d\mathcal H^1(s) \leq \int_{J'} \left(\int_{\psi^{-1}(s)} \lambda \,d\mathcal H^1 \right)d\mathcal H^1(s) \leq C \int_{N_\delta(\alpha)} \lambda\, d\mathcal H^2 \\
&\leq C\mathcal H^2(N_\delta(\alpha))^{1/2} \|\lambda \|_2  <C'\varepsilon.
\end{align*}
Letting $\varepsilon\to 0$, we obtain $\mathcal H^1(J)=0$, and this completes the proof.

Finally, we show the latter claim. Here we will use the area formula in Proposition \ref{harmonic:Coarea}. For $g(z)= \x_{A}\x_{N_\delta(\alpha)}$ we have
\begin{align*}
\int_0^\delta \#\{\psi^{-1}(s)\cap A\} \,ds \leq C \int_{N_\delta(\alpha)} \x_A\, d\mathcal H^1 =0.
\end{align*}
Here, $\#$ is the counting measure. Hence, $\#\{ \psi^{-1}(s)\cap A\} =0$ for a.e.\ $s\in (0,\delta)$, and the conclusion for $\alpha_s \subset \psi^{-1}(s)$ follows immediately.
\end{proof}

We also need a ``boundary version" of the above lemma:

\begin{lemma}\label{harmonic:Paths boundary}
Let $\alpha \subset \Omega$ be an open path with $\br \alpha \cap \partial \Omega\neq \emptyset$, and let $\Gamma$ be a given family of (weak, strong, conformal) modulus zero. Assume that $x \in \br \alpha \cap \partial \Omega$ lies in a non-trivial component of $\partial \Omega$. Then, for every $\varepsilon>0$ there exists a $\delta>0$ such that for a.e.\ $s\in (0,\delta)$ there exists an open path $\alpha_s\subset \psi^{-1}(s)$ that lies in $\Omega$, lands at a point $x_s\in B(x,\varepsilon)\cap \partial \Omega$, and avoids the path family $\Gamma$. Furthermore, if $A\subset \Omega$ is a given set with $\mathcal H^1(A)=0$, then for a.e.\ $s\in (0,\delta)$ the path $\alpha_s$ does not intersect $A$.
\end{lemma}
\begin{proof}
We only sketch the part of the proof related to the landing point of $\alpha_s$, since the rest is the same as the proof of Lemma \ref{harmonic:Paths joining continua}.

Note that for small $\varepsilon>0$ there exists a connected subset $E$ of $\partial \Omega$ that connects $x$ to $\partial B(x,\varepsilon)$. Hence, if we apply Lemma \ref{harmonic:Paths-distance function} to the path $\br \alpha$ we can obtain paths in $\psi^{-1}(s)\subset \R^2$ that land at $E$; here $F\subset \Omega$ can be any continuum that intersects $\br \alpha$. However, these paths do not lie necessarily in $\Omega$, so in this case we have to truncate them at the ``first time" that they meet $\partial \Omega$. If $\delta>0$ is chosen sufficiently small, then for a.e.\ $s\in (0,\delta)$ these paths will land at a point $x_s\in B(x,\varepsilon) \cap \partial \Omega$.
\end{proof}

Next, we switch to a special type of curves $\alpha_s$, namely circular arcs. A sequence of weights $\{h(Q_i)\}_{i\in \N}$ is \textit{locally square-summable} if for each $x\in S$ there exists a ball $B(x,r)\subset \Omega$ such that 
\begin{align*}
\sum_{i\in I_B} h(Q_i)^2<\infty.
\end{align*}

\begin{remark}\label{harmonic:Remark:summable}
Let $\{h(Q_i)\}_{i\in \N}$ be a sequence of non-negative weights with $\sum_{i\in \N} h(Q_i) <\infty$, and let $x\in S^\circ\cup \partial \Omega$, i.e., $x$ does not lie on the boundary of any peripheral disk. Then 
\begin{align*}
\sum_{i\in I_{B(x,r)}} h(Q_i) \to 0
\end{align*} 
as $r\to 0$. This is because the ball $B(x,r)$ cannot intersect any given peripheral disk $Q_i$ for arbitrarily small $r>0$.
\end{remark}

\begin{remark}\label{harmonic:3-remark-subpaths modulus zero}
Let $\Gamma$ be a path family in $\R^2$ with (weak, strong, conformal) modulus zero. Then the family $\Gamma_0$ of paths in $\R^2$ that contain a subpath lying in $\Gamma$ also has (weak, strong, conformal) modulus zero.
\end{remark}

\begin{lemma}\label{harmonic:3-curves gamma_r}
Let $\{h(Q_i)\}_{i\in \N}$ be a locally square-summable sequence. Consider a set $A\subset \Omega $ with $\mathcal H^1(A)=0$, and a path family $\Gamma$ in $\Omega$ that has (weak, strong) modulus equal to zero. 
\begin{enumerate}[\upshape (a)]
\item If $x\in S^\circ\cup \partial \Omega$ then for each $\varepsilon>0$ we can find an arbitrarily small $r>0$ such that the circular path $\gamma_r(t)=x+re^{it}$, as well as all of its subpaths, does not lie in $\Gamma$, it does not intersect $A$, and
\begin{align*}
\sum_{i:Q_i\cap \gamma_r\neq \emptyset} h(Q_i) <\varepsilon.
\end{align*} 
If $x\in \partial Q_{i_0}$ for some $i_0\in \N$, then the same conclusion is true, if we exclude the peripheral disk $Q_{i_0}$ from the above sum. 

\item If $x,y\in \partial Q_{i_0}$, then for each $\varepsilon>0$ we can find an arbitrarily small $r>0$ and a path $\gamma_0\subset Q_{i_0}$ that joins the circular paths $\gamma_{r}^{x}(t)=x+re^{it}$, $ \gamma_{r}^{y}(t)=y+re^{it}$ with the following property: any simple path $\gamma$ contained in the concatenation of the paths $\gamma_0,\gamma_{r}^{x},\gamma_{r}^{y}$ does not lie in $\Gamma$, $\gamma$ does not intersect $A$, and 
\begin{align*}
\sum_{\substack{i:Q_i\cap \gamma\neq \emptyset\\ i\neq i_0}} h(Q_i)<\varepsilon.
\end{align*}
\end{enumerate}
\end{lemma}
\begin{proof} 
We may assume that $\md_2(\Gamma)=0$, by lemmas \ref{harmonic:2-weak modulus less than strong} and \ref{harmonic:2-weak modulus zero implies two modulus zero}.

(a) Note that the circular path $\gamma_r$ centered at $x$ lies in the set $\psi^{-1}(r)$, where $\psi(z)=|z-x|$ is a $1$-Lipschitz function. As in the proof of Lemma \ref{harmonic:Paths joining continua} one can show that there exists $\delta>0$ such that for a.e.\ $r\in (0,\delta)$ the path $\gamma_r$ avoids $\Gamma$ and the set $A$. Remark \ref{harmonic:3-remark-subpaths modulus zero} implies that all subpaths of $\gamma_r$ also avoid $\Gamma$. Assume that $x\in S^\circ\cup \partial \Omega$ and fix $\varepsilon>0$. In order to show the statement, it suffices to show that for arbitrarily small $\delta>0$, there exists a set $J\subset (\delta/2,\delta)$ of positive measure such that 
\begin{align*}
\sum_{i:Q_i\cap \gamma_r\neq \emptyset}h(Q_i) <\varepsilon
\end{align*}
for all $r\in J$. Assume that this fails, so there exists a small $\delta>0$ such that the reverse inequality holds for a.e.\ $r\in (\delta/2,\delta)$. Noting that the function $r\mapsto \x_{Q_i\cap \gamma_r}$ is measurable and integrating over $r\in (\delta/2,\delta)$ we obtain
\begin{align}
\label{harmonic:3-curves gamma_r proof}\varepsilon \delta/2 &\leq \int_{\delta/2}^\delta\sum_{i:Q_i\cap \gamma_r\neq \emptyset}h(Q_i) \,dr \leq \sum_{i\in I_{B(x,\delta)}} h(Q_i)\int_0^\delta \x_{Q_i\cap \gamma_r} \,dr\\
\notag &= \sum_{i\in I_{B(x,\delta)}} h(Q_i)d(Q_i),
\end{align}
where $d(Q_i)\coloneqq \mathcal H^1( \{r \in [0,\delta]:   Q_i\cap \gamma_r\neq \emptyset \})$. The fatness of the peripheral disks implies that there exists some uniform constant $K>0$ such that $d(Q_i)^2 \leq K\mathcal H^2(Q_i\cap B(x,\delta))$; see Remark \ref{harmonic:Remark:Fatness implication}. Using the Cauchy-Schwarz inequality and this fact in \eqref{harmonic:3-curves gamma_r proof} we obtain
\begin{align*}
\varepsilon^2 \delta^2/4 &\leq \sum_{i\in I_{B(x,\delta)}}h(Q_i)^2 \sum_{i\in I_{B(x,\delta)}}d(Q_i)^2\\
&\leq K\sum_{i\in I_{B(x,\delta)}}h(Q_i)^2\sum_{i\in \N} \mathcal H^2(Q_i\cap B(x,\delta))\\
&=K\sum_{i\in I_{B(x,\delta)}}h(Q_i)^2 \cdot \mathcal H^2(B(x,\delta))\\
&=C\delta^2  \sum_{i\in I_{B(x,\delta)}}h(Q_i)^2.
\end{align*}
Hence, if $\delta$ is sufficiently small so that $\sum_{i\in I_{B(x,\delta)}}h(Q_i)^2 <\varepsilon^2/4C$  (see Remark \ref{harmonic:Remark:summable}) we obtain a contradiction. 

In the case that $x\in \partial Q_{i_0}$ the same computations work if we exclude the index $i_0$ from the sums, since eventually we want to make $\sum_{i\in I_{B(x,\delta)}\setminus \{i_0\}} h(Q_i)^2$ arbitrarily small.

(b) Arguing as in part (a) we can find a small $\delta>0$ such that the balls $B(x,\delta)$, $B(y,\delta)$ are disjoint, they are contained in $\Omega$, and there exists a set $J\subset (\delta/2,\delta)$ of positive $1$-measure such that 
\begin{align*}
\sum_{\substack{ i:Q_i\cap (\gamma_r^x \cup \gamma_r^y) \neq \emptyset \\ i\neq i_0}} h(Q_i)<\varepsilon
\end{align*}
for all $r\in J$. We may also assume that for all $r\in J$ the paths $\gamma_r^x,\gamma_r^y$ avoid the given set $A$ with $\mathcal H^1(A)=0$. 

Let $\gamma_0\subset \subset Q_{i_0}$  be a path that connects $\partial B(x,\delta/2) \cap Q_{i_0}$ to $\partial B(y,\delta/2) \cap Q_{i_0}$. Consider the function $\psi(\cdot)=\dist(\cdot,\gamma_0)$. Since $\dist(\gamma_0,\partial Q_{i_0})>0$, by Proposition \ref{harmonic:Paths-distance function} there exists $s_0>0$ such that for a.e.\ $s\in [0,s_0]$ there exists a path $\gamma_s\subset \psi^{-1}(s)\cap Q_{i_0}$ connecting $\partial B(x,\delta/2) \cap Q_{i_0}$ to $\partial B(y,\delta/2) \cap Q_{i_0}$. Then for all $r\in J$ and a.e.\ $s\in [0,s_0]$ the path $\gamma_s$ connects the circular paths $\gamma_r^x$ and $\gamma_r^y$. We claim that for a.e.\ $(r,s)\in J\times [0,s_0]$ all simple paths contained in the concatenation $\gamma_{r,s}$ of $\gamma_r^x$, $\gamma_s$, $\gamma_r^y$ avoid a given path family $\Gamma$ with $\md_2(\Gamma)=0$. Note here that $J\times [0,s_0]$ has positive $2$-measure.

For each $\eta>0$ we can find a function $\lambda$ that is admissible  for $\Gamma$ with $\|\lambda\|_2<\eta$. Let $T \subset J\times [0,s_0]$ be the set of $(r,s)$ for which $\gamma_{r,s}$ has  a simple subpath lying in $\Gamma$. Using the co-area formula in Proposition \ref{harmonic:Coarea} we have
\begin{align*}
\mathcal H^2(T) &=\int_T 1\,d\mathcal H^2 \leq \int_J \int_0^{s_0} \left( \int_{\gamma_r^x} \lambda \,d\mathcal H^1 + \int_{\gamma_r^y} \lambda \,d\mathcal H^1 + \int_{\gamma_{s}} \lambda \,d\mathcal H^1 \right)dsdr\\
&\leq C \|\lambda\|_2 \leq C\eta.
\end{align*}
Letting $\eta\to 0$ we obtain that $\mathcal H^2(T)=0$, as desired. This completes the proof of part (b).
\end{proof}

\begin{remark}\label{harmonic:Circular paths remark}
The proof of part (a) shows the following stronger conclusion: there exists a constant $C>0$ such that if $x\in S^\circ $, then there exists some $r\in [\delta/2,\delta]$ such that $\gamma_r(t)=x+re^{it}$ has the desired properties and 
\begin{align*}
\sum_{i:Q_i\cap \gamma_r\neq \emptyset}h(Q_i) \leq C \left( \sum_{i\in I_{B(x,\delta)}} h(Q_i)^2\right)^{1/2}.
\end{align*}
\end{remark}
\begin{remark}
The uniform fatness of the peripheral disks $Q_i$ was crucial in the proof. In fact, without the assumption of uniform fatness, one can construct a relative Sierpi\'nski carpet for which the conclusion of the lemma fails.
\end{remark}

We also include a topological lemma:

\begin{lemma}\label{harmonic:Paths in S^o}
The following statements are true:
\begin{enumerate}[\upshape(a)]
\item For each  peripheral disk $Q_{i_0}$, there exists a Jordan curve $\gamma \subset S^\circ$ that contains $Q_{i_0}$ in its interior and lies arbitrarily close to $Q_{i_0}$. In particular, $S^\circ$ is dense in $S$.
\item For any $x,y\in S$ there exists an open path $\gamma\subset S^\circ$ that joins $x,y$. Moreover, for each $r>0$, if $y$ is sufficiently close to $x$, the path $\gamma$ can be taken so that $\gamma\subset B(x,r)$.
\end{enumerate}
\end{lemma}

The proof is an application of Moore's theorem \cite{Moore:theorem} and can be found in \cite[Proof of Theorem 5.2, p.~4331]{Merenkov:localrigiditySchottky}, in case $\Omega$ is a Jordan region. We include a proof of this more general statement here. We will use the following decomposition theorem, which is slightly stronger than Moore's theorem:

\begin{theorem}[Corollary 6A, p.\ 56, \cite{Daverman:decompositions}]\label{harmonic:Decomposition}\index{Moore's theorem}
Let $\{Q_i\}_{i\in \N}$ be a sequence of Jordan regions in the sphere $S^2$ with disjoint closures and diameters converging to $0$, and consider an open set $U \supset \bigcup_{i\in \N} \br Q_i$. Then there exists a continuous, surjective map $f:S^2 \to S^2$ that is the identity outside $U$, and it induces the decomposition of $S^2$ into the sets $\{\br Q_i\}_{i\in \N}$ and points. In other words, there are countably many distinct points $p_i$, $i\in \N$, such that $f^{-1}(p_i)= \br Q_i$ for $i\in \N$, and $f$ is injective on $S^2\setminus \bigcup_{i\in \N}\br Q_i$ with  $f( S^2\setminus \bigcup_{i\in \N} \br Q_i)= S^2\setminus \{p_i:i\in \N\}$.
\end{theorem}

\begin{proof}[Proof of Lemma \ref{harmonic:Paths in S^o}]
By Lemma \ref{harmonic:Fatness consequence} the spherical diameters of the peripheral disks $Q_i$ converge to $0$. Hence, we may apply the decomposition theorem with $U=\Omega$ and obtain the collapsing map $f:S^2\to S^2$. A given peripheral disk $\br Q_{i_0}$ is mapped to a point $p_{i_0}\in \Omega$. Arbitrarily close to $p_{i_0}$ we can find round circles that avoid the countably many points that correspond to the collapsing of the other peripheral disks. The preimages of these round circles under $f$ are Jordan curves $\gamma \subset {S^2}$ that are contained in $\Omega$, lie in $S^\circ$, and are contained in small neighborhoods of $\br Q_{i_0}$. This completes the proof of part (a). For part (b), we consider three cases:

\textbf{Case 1:} Suppose first that $x,y\in S^\circ$. Using the decomposition theorem with $U=\Omega$, we obtain points $\tilde x=f(x)$ and $\tilde y=f(y)$ in $\Omega \setminus \{p_i:i\in \N\}$. We connect $\tilde x$ and $\tilde y$ with a path $\tilde \alpha \subset \Omega$ (recall that $\Omega$ is connected). By perturbing this path, we may obtain another path, still denoted by $\tilde \alpha$, that connects $\tilde x$ and $\tilde y$ in $\Omega$ but avoids the points $\{p_i:i\in \N\}$. Then $f^{-1}(\tilde \alpha)$, by injectivity, yields the desired path in $S^\circ$ that connects $x$ and $y$. In fact, if $y$ is sufficiently close to $x$, the path $f^{-1}(\tilde \alpha)$ can be taken to lie in a small neighborhood of $x$. To see this, first note that if $y$ is close to $x$ then $\tilde y$ is close to $\tilde x$ by the continuity of $f$. We can then find a path $\tilde \alpha \subset \Omega \setminus \{p_i: i\in \N\}$ connecting $\tilde x$ and $\tilde y$, and lying arbitrarily close to the line segment $[\tilde x,\tilde y]$. The lift $f^{-1}(\tilde \alpha)$, by continuity, has to be contained in a small neighborhood of $x$.

If $x$ or $y$ lies on a peripheral circle, then we need to modify the preceding argument to obtain an open path $\gamma\subset S^\circ$ with endpoints $x$ and $y$:

\textbf{Case 2:} Suppose that $x\in \partial Q_{i_0}$ and $y\in S^\circ$. Then in the application of the decomposition theorem we do not collapse the peripheral disk $\br Q_{i_0}$. We set $U= \Omega\cup Q_{i_0}$ and we collapse $\{\br Q_i: i\in \N \setminus \{i_0\} \}$ to points with a map $f:S^2\to S^2$ that is the identity outside $U$, as in the statement of the decomposition theorem. Then we consider an open path $\tilde \alpha\subset \Omega \setminus \br Q_{i_0}$ connecting $\tilde x$ and $\tilde y$; in order to find such a path one can assume that $\br Q_{i_0}$ is a round disk by using a homeomorphism of $S^2$.  Now, the path $\tilde \alpha$ can be modified to avoid the countably many points corresponding to the collapsed peripheral disks. This modified path lifts under $f$ to the desired path, and as before, it can be taken to lie arbitrarily close to $x$, provided that $y\in S^\circ$ is sufficiently close to $x$. This proves (b) in this case.

\textbf{Case 3:} Finally, suppose that $x\in \partial Q_{i_0}$ and $y\in \partial Q_{i}$ for some $i_0,i\in \N$. Then we can connect the points $x,y$ to points $x',y'\in S^\circ$ with open paths $\gamma_x,\gamma_y\subset S^\circ$ by Case 2. Case 1 implies that the points $x',y'$ can be connected with a path $\gamma\subset S^\circ$. Concatenating $\gamma_x,\gamma$, and $\gamma_y$ yields an open path in $S^\circ$ that connects $x$ and $y$.

It remains to show that if $y$ is sufficiently close to $x$, then there exists an open path $\gamma\subset S^\circ$ connecting $x$ and $y$ that lies near $x$. By the density of $S^\circ$ in $S$ (from part (a)) we may find arbitrarily close to $y$ points $z\in S^\circ$. By Case 2, for each $r>0$ there exists $\delta>0$ such that if $|z-x|<\delta$ and $z\in S^\circ$, then there exists an open path $\gamma \subset B(x,r)\cap S^\circ$ that connects $x$ and $z$. We assume that $|y-x|<\delta/2$ and $y\in \partial Q_i$ for some $i\in \N$. Then using the conclusion of Case 2 we consider a point $z\in S^\circ \in B(x,\delta)$ close to $y$ and an open path $\gamma_1 \subset B(x,r)\cap S^\circ$ that connects $y$ to $z$. Since $z\in B(x,\delta)$, there exists another open path $\gamma_2\subset B(x,r)\cap S^\circ$  connecting $z$ to $x$. Concatenating $\gamma_1$ and $\gamma_2$ provides the desired path.
\end{proof}

Finally, we include a technical lemma:

\begin{lemma}\label{harmonic:lemma:paths zero hausdorff}
Suppose that $\gamma \subset \C$ is a non-constant path with $\mathcal H^1(\gamma\cap S)=0$. 
\begin{enumerate}[\upshape(a)]
\item If $x\in \gamma\cap S^\circ$, then arbitrarily close to $x$ we can find peripheral disks $Q_i$ with $Q_i\cap \gamma\neq \emptyset$. 
\item If $\gamma$ is an open path that does not intersect a peripheral disk $ Q_{i_0}$, $i_0\in \N$, and $x\in \br \gamma\cap \partial Q_{i_0}$, then arbitrarily close to $x$ we can find peripheral disks $Q_i$, $i\neq i_0$, with $Q_i\cap \gamma\neq \emptyset$. 
\end{enumerate}
\end{lemma}
\begin{proof}
If either of the two statements failed, then there would exist a small ball $B(x,\varepsilon)$, not containing $\gamma$, such that all points $y\in \gamma\cap B(x,\varepsilon)$ lie in $S$. Since $\gamma$ is connected and it exits $B(x,\varepsilon)$, there exists a continuum $\beta\subset \gamma \cap B(x,\varepsilon) \cap S$ with $\diam(\beta)\geq \varepsilon/2$. Then $\mathcal H^1(\gamma\cap S) \geq \diam(\beta)>0$, a contradiction.
\end{proof}

\section{Sobolev spaces on relative Sierpi\'nski carpets}\label{harmonic:3-Sobolev}
In this section we treat one of the main objects of the chapter, the Sobolev spaces on relative Sierpi\'nski carpets. This is the class of maps among which we would like to minimize a type of \textit{Dirichlet energy}, in order to define carpet-harmonic functions.
 
We will start with a preliminary version of our Sobolev functions, namely with \textit{discrete Sobolev functions}. Using them, we define a continuous analog, and, as it turns out, these are just two sides of the same coin, since the two spaces of functions will be isomorphic. The reason for introducing the discrete Sobolev functions is because, in practice, it is much easier to check, as well as, limiting theorems are proved with less effort.

\subsection{Discrete Sobolev spaces}\label{harmonic:Subsection-Discrete Sobolev}\index{Sobolev space!discrete Sobolev space}
Let $\hat f \colon\{Q_i\}_{i\in \N} \to \R$ be a map defined on the set of peripheral disks $\{Q_i\}_{i\in \N}$. We say that the sequence $\{\rho(Q_i)\}_{i\in \N}$ is a \textit{weak (strong) upper gradient}\index{upper gradient}\index{upper gradient!weak upper gradient}\index{upper gradient!strong upper gradient} for $\hat f$ if there exists an exceptional family $\Gamma_0$ of paths in $\Omega$ with $\md_w(\Gamma_0)=0$ $(\md_s(\Gamma_0)=0)$ such that for all paths $\gamma\subset \Omega$ with $\gamma \notin \Gamma_0$ and all peripheral disks $Q_{i_j}$ with $Q_{i_j}\cap \gamma\neq \emptyset$, $j=1,2$, we have
\begin{align}\label{harmonic:3-Weak upper gradient}
|\hat f(Q_{i_1}) -\hat f(Q_{i_2})| \leq \sum_{i:Q_i\cap \gamma\neq \emptyset} \rho(Q_i).
\end{align}

Using the upper gradients we define:

\begin{definition}\label{harmonic:3-Discrete Sobolev Definition}Let $\hat f \colon\{Q_i\}_{i\in \N} \to \R$ be a map defined on the set of peripheral disks $\{Q_i\}_{i\in \N}$. We say that $\hat f$ lies in the \textit{local weak (strong) Sobolev space}\index{Sobolev space} $\widehat{\mathcal W}_{w,\loc}^{1,2} (S)$ $(\widehat{\mathcal W}_{s,\loc}^{1,2}(S))$ if there exists a {weak (strong) upper gradient} $\{\rho(Q_i)\}_{i\in \N}$ for $\hat f$ such that for every ball $B\subset\subset  \Omega $ we have 
\begin{align}
\label{harmonic:3-Sobolev L2}&\sum_{i\in I_B} \hat f(Q_i)^2 \diam(Q_i)^2<\infty, \quad \textrm{and}\\
\label{harmonic:3-Sobolev L2 gradient}& \sum_{i\in I_B} \rho(Q_i)^2<\infty.
\end{align}
Furthermore, if these conditions hold for the full sums over $i\in \N$, then we say that $\hat f$ lies in the \textit{weak (strong) Sobolev space}  $\widehat{\mathcal W}_{w}^{1,2}(S)$  $(\widehat{\mathcal W}_{s}^{1,2}(S))$.
\end{definition}

\begin{remark}
Conditions \eqref{harmonic:3-Sobolev L2} and \eqref{harmonic:3-Sobolev L2 gradient} are equivalent to saying that the sequences $\{\hat f(Q_i)\diam(Q_i)\}_{i\in \N}$ and $\{\rho(Q_i)\}_{i\in \N}$ are locally square-summable. Also, if, e.g., $\hat f(Q_i)$ is bounded then \eqref{harmonic:3-Sobolev L2} holds since we have $\sum_{i\in I_B} \diam(Q_i)^2<\infty$ for all $B\subset\subset \Omega$, by Corollary \ref{harmonic:Fatness corollary}.
\end{remark}

\begin{remark}
This definition resembles the definition of the classical Sobolev spaces $W_{\loc}^{1,2}( \R^n)$ and $W ^{1,2}(\R^n)$, using weak upper gradients. In fact, this definition is motivated by the \textit{Newtonian spaces} $N^{1,p}(\R^n)$, which contain \textit{good} representatives of Sobolev functions, rather than equivalence classes of functions; see \cite{HeinonenKoskelaShanmugalingamTyson:Sobolev} for background on weak upper gradients and Newtonian spaces.
\end{remark}

If $\md_s(\Gamma_0)=0$, then $\md_w(\Gamma_0)=0$ by Lemma \ref{harmonic:2-weak modulus less than strong}. This shows that if $\{\rho(Q_i)\}_{i\in \N}$ is a strong upper gradient for $\hat f$, then it is also a weak upper gradient for $\hat f$. Thus,
\begin{align}\label{harmonic:Inclusion of hat Sobolev spaces}
\widehat{\mathcal W}_{s,\loc}^{1,2}(S) \subset \widehat{\mathcal W}_{w,\loc}^{1,2}(S).
\end{align}

\begin{remark}\label{harmonic:Uniform relative separation}We have not been able to show the reverse inclusion, which probably depends on the geometry of the peripheral disks and their separation. A conjecture could be that the two spaces agree if the peripheral circles of the carpet are uniform quasicircles and they are \textit{uniformly relatively separated}\index{uniform relative separation}, i.e., there exists a uniform constant $\delta>0$ such that the \textit{relative distance}
\begin{align*}
\Delta(Q_i,Q_j) = \frac{\dist(Q_i,Q_j)}{\min \{ \diam(Q_i),\diam (Q_j) \}}
\end{align*}
satisfies $\Delta(Q_i,Q_j)\geq \delta$ for all $i\neq j$.
\end{remark}

For the rest of the section we fix a function $\hat f\in \widehat{\mathcal W}_{w,\loc}^{1,2}(S)$  $(\widehat{\mathcal W}_{s,\loc}^{1,2}(S))$. Our goal is to construct a function $f$ which is defined on ``most" of the points of the carpet $S$ and is a ``continuous" version of $\hat f$. Let $\mathcal {G}$ be the family of \textit{good paths}\index{good paths} $\gamma$ in $\Omega$ with
\begin{enumerate}[\upshape (1)]
\item $\mathcal H^1(\gamma\cap S)=0$,
\item the upper gradient inequality \eqref{harmonic:3-Weak upper gradient} is satisfied for all subpaths of $\gamma$,
\item $\sum_{i:Q_i\cap \gamma'\neq \emptyset}\rho(Q_i)<\infty$ for all subpaths $\gamma'$ of $\gamma$ which are compactly contained in $\Omega$,
\item[(3*)] in case $\sum_{i\in \N} \rho(Q_i)^2 <\infty$ we require $\sum_{i:Q_i\cap \gamma\neq \emptyset}\rho(Q_i)<\infty$.
\end{enumerate}
It is immediate to see that all subpaths of a path $\gamma\in \mathcal G$ also lie in $\mathcal G$. Note that $\mathcal G$ depends both on $\{\hat f(Q_i)\}_{i\in \N}$ and on $\{\rho(Q_i)\}_{i\in \N}$. Property (1) is crucial and will allow us to apply Lemma \ref{harmonic:lemma:paths zero hausdorff}.

\begin{lemma}\label{harmonic:Good family properties}
Suppose that $\hat f\in \widehat{\mathcal W}_{w,\loc}^{1,2}(S)$ $(\widehat{\mathcal W}_{s,\loc}^{1,2}(S))$. Then the complement of $\mathcal G$ (i.e., all curves in $\Omega$ that do not lie in $\mathcal G$) has weak (strong) modulus equal to $0$. 
\end{lemma}
In other words, $\mathcal G$ contains ``almost every" path $\gamma\subset \Omega$.

\begin{proof}
By the subadditivity of modulus, it suffices to show that the family of curves for which one of the conditions (1),(2),(3), or (3*) is violated has weak (strong) modulus equal to $0$.

We first note that $\mathcal H^1(\gamma\cap S)=0$ holds for all paths outside a family of strong carpet modulus equal to $0$, so the weak carpet modulus is also equal to $0$ by Lemma \ref{harmonic:2-weak modulus less than strong}. 

Moreover, Remark \ref{harmonic:3-remark-subpaths modulus zero} implies that the family of paths that have a subpath lying in a family of weak (strong) carpet modulus equal to $0$ has itself weak (strong) carpet modulus zero. This justifies that the upper gradient inequality holds for \textit{all} subpaths of paths $\gamma$ that lie outside an exceptional family of weak (strong) modulus zero.

Finally, since for any set $V\subset\subset \Omega$ we have $\sum_{i\in I_V} \rho(Q_i)^2<\infty$ by \eqref{harmonic:3-Sobolev L2 gradient}, it follows that the family of paths in $\Omega$ that have a subpath $\gamma'$, compactly contained in $\Omega$, with $\sum_{i: Q_i\cap {\gamma'}\neq \emptyset} \rho(Q_i)=\infty$, has strong (and thus weak) carpet modulus equal to zero. The case (3*) has the same proof, since $\sum_{i\in \N} \rho(Q_i)^2<\infty$ there.
\end{proof}

We say that a point $x\in S$ is \textit{``accessible" by a curve $\gamma_0 \in \mathcal G$}\index{``accessible'' point}, if $x\in \gamma_0\cap S^\circ$, or if there exists an open subcurve $\gamma$ of $\gamma_0$ with $x\notin \gamma$, $x\in \br\gamma$, and $\gamma$ does not meet the (interior of the) peripheral disk $Q_{i_0}$ whenever $x\in \partial Q_{i_0}$; see Figure \ref{harmonic:fig:accessible}. In the first case we set $\gamma=\gamma_0$. By Lemma \ref{harmonic:lemma:paths zero hausdorff}, if $x$ is ``accessible" by $\gamma_0$, then arbitrarily close to $x$ we can find peripheral disks $Q_i$ with $Q_i\cap \gamma\neq \emptyset$. We say that a point is \textit{``accessible"} if the family of paths we are using is implicitly understood. For a point $x\in S$  that is ``accessible" by $\gamma_0$ we define 
\begin{align}\label{harmonic:3-definition of f}
f(x)\coloneqq \liminf_{\substack{Q_i\to x\\Q_i\cap \gamma \neq \emptyset }} \hat f(Q_i),
\end{align}
where $\gamma$ is a subpath of $\gamma_0$ as above. Often, we will abuse terminology and we will be using the above definition if $x$ is ``accessible" by $\gamma$, without mentioning that $\gamma$ is a subpath of $\gamma_0$; see e.g.\ the statement of Lemma \ref{harmonic:3-well-defined}. Using Lemma \ref{harmonic:Paths joining continua} it is easy to construct non-exceptional paths in $\mathcal G$ passing through a given continuum. In fact, it can be shown that for each peripheral circle $\partial Q_{i_0}$ there is a dense set of points which are ``accessible'' by good paths $\gamma\in \mathcal G$. For ``non-accessible" points $x\in \partial Q_{i_0}$ we define $f(x)$ as $\liminf f(y)$, as $y$ approaches $x$ through ``accessible" points $y\in \partial Q_{i_0}$. For the other points of the carpet $S$ (which must belong to $S^\circ$) we define $f(x)$ as $\liminf f(y)$, as $y$ approaches $x$ through ``accessible" points.

\begin{figure}
  \centering
  \input{accessible.tikz}
  \caption{A point $x\in \partial Q_{i_0}$ that is ``accessible" by $\gamma_0$ (left) and ``non-accessible" (right).}
  \label{harmonic:fig:accessible}
\end{figure}
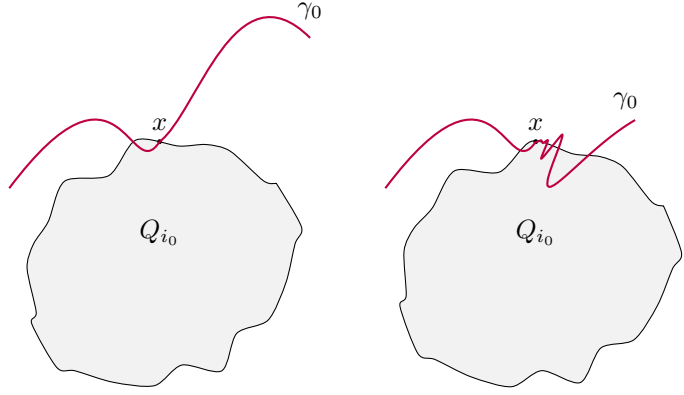

First we show that the map $f\colon S \to \widehat \R$ is well-defined, i.e., the definition does not depend on the curve $\gamma$ from which the point $x$ is ``accessible''.

\begin{lemma}\label{harmonic:3-well-defined}
If $\gamma_1,\gamma_2\in \mathcal G$ are two paths from which the point $x\in S$ is ``accessible'', then
\begin{align*}
\liminf_{\substack{Q_i\to x\\Q_i\cap \gamma_1 \neq \emptyset }} \hat f(Q_i)= \liminf_{\substack{Q_i\to x\\Q_i\cap \gamma_2 \neq \emptyset }} \hat f(Q_i).
\end{align*} 
\end{lemma}
As remarked before, here ``accessible" by $\gamma_1$ means that there exists an open subpath of $\gamma_1$ that we still denote by $\gamma_1$ such that $x\in \br \gamma_1\setminus \gamma_1$, and $\gamma_1$ does not intersect the peripheral disk $Q_{i_0}$, in case $x\in \partial Q_{i_0}$.
\begin{proof}
We may assume that $\gamma_1,\gamma_2$ are compactly contained in $\Omega$, otherwise we consider  subpaths of them. We fix $\varepsilon>0$ and consider peripheral disks $Q_{i_j}$ very close to $x$ such that $Q_{i_j}\cap \gamma_j\neq \emptyset$ (recall that $\mathcal H^1(\gamma_j\cap S)=0$ and see Lemma \ref{harmonic:lemma:paths zero hausdorff}) and such that the truncated paths $\gamma_j'\subset \gamma_j$ that join $Q_{i_j}$ to $x$ have short $\rho$-length, i.e., 
\begin{align}\label{harmonic:3-well-defined-short rho length}
\sum_{i: Q_i\cap \gamma_j'\neq \emptyset} \rho(Q_i)<\varepsilon
\end{align}
for $j=1,2$. This can be done since $\gamma_j$ is compactly contained in $\Omega$ and $\gamma_j\in \mathcal G$ which implies that $\sum_{i:Q_i\cap \gamma_j\neq \emptyset}\rho(Q_i)<{\infty}$ for  $j=1,2$. 

We first assume that $x\in S^\circ$. Using Lemma \ref{harmonic:3-curves gamma_r}(a), we can find a curve $\gamma_r$ lying in $\mathcal G$ with $r$ smaller than the distance of $Q_{i_j}$ and $x$ for $j=1,2$, such that $\gamma_r$ avoids the sets $\gamma_j'\cap S$ which have Hausdorff $1$-measure zero, and 
\begin{align}\label{harmonic:3-well-defined-gamma r}
\sum_{i:Q_i\cap \gamma_r\neq \emptyset}\rho(Q_i)<\varepsilon.
\end{align}
It follows that $\gamma_r$ has to meet some $Q_{i_j}'$ that is intersected by $\gamma_j'$, $j=1,2$, so by the triangle inequality and the upper gradient inequality we have
\begin{align*}
|\hat f(Q_{i_1})-\hat f(Q_{i_2})|&\leq |\hat f(Q_{i_1})-\hat f(Q_{i_1}')| +|\hat f(Q_{i_1}')-\hat f(Q_{i_2}')|+|\hat f(Q_{i_2}')-\hat f(Q_{i_2})|\\
&\leq \sum_{i:Q_i\cap \gamma_1'\neq \emptyset} \rho(Q_i)+ \sum_{i:Q_i\cap \gamma_r\neq \emptyset}\rho(Q_i)+\sum_{i:Q_i\cap \gamma_2'\neq \emptyset} \rho(Q_i)\\
&< 3\varepsilon
\end{align*}
by \eqref{harmonic:3-well-defined-short rho length} and \eqref{harmonic:3-well-defined-gamma r}. The conclusion follows in this case.

\begin{figure}
	\centering
	\begin{overpic}[width=.75\linewidth]{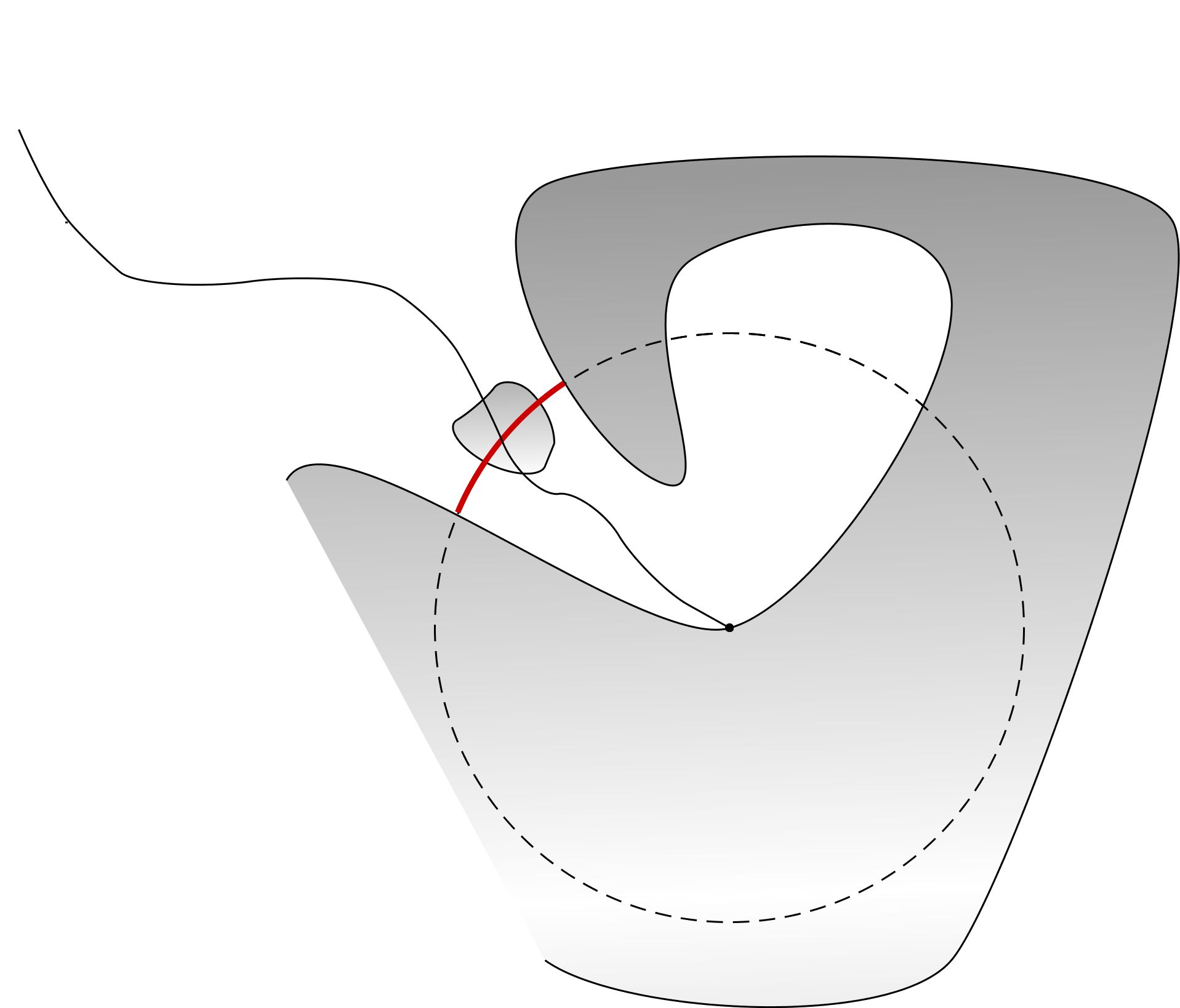}
 	\put (62,30) {$x$}
 	\put (85,50) {$Q_{i_0}$}
 	\put (25,65) {$\gamma_j'$}
 	\put (33,50) {$Q_{i_j}'$}
 	\put (40,42) {$\gamma_r'$}
  \end{overpic}
  \caption{The crosscut defined by $\gamma_r$.}
  \label{harmonic:fig:crosscut}
\end{figure}
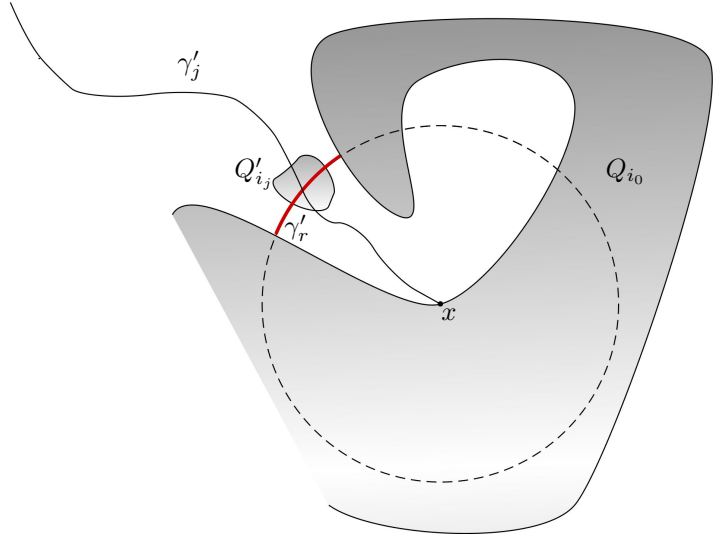

In the case that $x\in \partial Q_{i_{0}}$ for some $i_0$, the only modification we have to make is in the application of Lemma \ref{harmonic:3-curves gamma_r}(a), which now yields a path $\gamma_r$ for sufficiently small $r$ such that
\begin{align*}
\sum_{\substack{i:Q_i\cap \gamma_r\neq \emptyset \\ i\neq i_0}} \rho(Q_i) <\varepsilon.
\end{align*}
Since the curves $\gamma_1,\gamma_2$ have $x$ as an endpoint but they avoid $Q_{i_0}$, there exists a subarc $\gamma_r'$ of $\gamma_r$ that defines a \textit{crosscut}\index{crosscut} separating $\infty$ from $x$ in $\R^2\setminus  Q_{i_0}$ and having the following properties (see Figure \ref{harmonic:fig:crosscut}): $\gamma_r'$ meets peripheral disks $Q_{i_j}'$ which are also intersected by $\gamma_j'$, $j=1,2$, $\gamma_r'\cap Q_{i_0}=\emptyset$ but the endpoints of $\gamma_r'$ lie on $\partial Q_{i_0}$, and 
\begin{align*}
\sum_{i:Q_i\cap \gamma_r'\neq \emptyset} \rho(Q_i)<\varepsilon.
\end{align*}
We direct the reader to \cite[Chapter 2.4]{Pommerenke:conformal} for a discussion on crosscuts. If we use this in the place of \eqref{harmonic:3-well-defined-gamma r}, the argument in the previous paragraph yields the conclusion. Here, we remark that $\gamma_r'$ is a good path lying in $\mathcal G$, since it is a subpath of $\gamma_r\in \mathcal G$; see Lemma \ref{harmonic:Good family properties}.
\end{proof}

Next, we prove that we have $|f(x)|<\infty$ for all points $x\in S$ that lie on a peripheral circle $\partial Q_{i_0}$. 

Let $x,y\in \partial Q_{i_0}$ be points that are ``accessible" by $\gamma_x,\gamma_y$, respectively, and let $Q_{i_x},Q_{i_y}$, $i_x,i_y\neq i_0$, be peripheral disks intersected by $\gamma_x,\gamma_y$, close to $x,y$, respectively. Applying Lemma \ref{harmonic:3-curves gamma_r}(b) for small $\varepsilon>0$, we obtain that there exists a small $r<\min\{\dist(x,Q_{i_x}),\\ \dist(y,Q_{i_y})\}$ such that the concatenation $\gamma$ of the circular paths $\gamma_{r}^x,\gamma_{r}^y$ with a path $\gamma_0\subset Q_{i_0}$ intersects peripheral disks $Q_{i_x}',Q_{i_y}'$ on $\gamma_x,\gamma_y$, respectively, and
\begin{align*}
\sum_{i:Q_i\cap \gamma\neq \emptyset} \rho(Q_i)= \sum_{\substack{i:Q_i\cap \gamma\neq \emptyset \\ i\neq i_0}} \rho(Q_i)+\rho(Q_{i_0})<\varepsilon+\rho(Q_{i_0}).
\end{align*}
We replace $\gamma$ by a simple subpath of it connecting $Q_{i_x}'$ with $Q_{i_y}'$, and we still denote this  subpath by $\gamma$. By Lemma \ref{harmonic:3-curves gamma_r}(b) $\gamma$ is a good path in $\mathcal G$, so by the upper gradient inequality we have
\begin{align*}
|\hat f(Q_{i_x})-\hat f(Q_{i_y}) | &\leq |\hat f(Q_{i_x})-\hat f(Q_{i_x}')|+|\hat f(Q_{i_x}')-\hat f(Q_{i_y}')|+ |\hat f(Q_{i_y}')-\hat f(Q_{i_y})|\\
&\leq \sum_{i:Q_i\cap \gamma_x\neq \emptyset}\rho(Q_i)+\sum_{i:Q_i\cap \gamma\neq \emptyset} \rho(Q_i) + \sum_{i:Q_i\cap \gamma_y\neq \emptyset}\rho(Q_i)\\
&\leq \sum_{i:Q_i\cap \gamma_x\neq \emptyset}\rho(Q_i)+\rho(Q_{i_0})+\varepsilon + \sum_{i:Q_i\cap \gamma_y\neq \emptyset}\rho(Q_i).
\end{align*}  
By taking $Q_{i_x},Q_{i_y}$  to be sufficiently close to $x,y$ and  also, by truncating $\gamma_x,\gamma_y$ to smaller subpaths $\gamma_x',\gamma_y'$ that join $Q_{i_x},Q_{i_y}$ to $x,y$, respectively, we may assume that the right hand side is smaller than $2\varepsilon+\rho(Q_{i_0})$. Thus, taking limits as $Q_{i_x} \to x$, $Q_{i_y}\to y$, we obtain
\begin{align}\label{harmonic:3-bound for f(x)-f(y)}
|f(x)- f(y)| \leq \rho(Q_{i_0}).
\end{align}
One can use Lemma \ref{harmonic:3-curves gamma_r}(a) and argue in the same way to derive that for all ``accessible" points $x\in \partial Q_{i_0}$ we have
\begin{align}\label{harmonic:3-bound for f(x)}
|f(x)-\hat f(Q_{i_0})|\leq \rho(Q_{i_0}).
\end{align}
This shows that in fact $f(x)$ is uniformly bounded for ``accessible" points $x\in \partial Q_{i_0}$. Since ``accessible" points are dense in $\partial Q_{i_0}$, the bound in \eqref{harmonic:3-bound for f(x)} also holds on ``non-accessible" points $x$ of the peripheral circle $\partial Q_{i_0}$, by the definition of $f$. We have proved the following: 
\begin{corollary}\label{harmonic:3-Finiteness of f}
For each peripheral disk $Q_{i}$ the quantities
\begin{align*}
&M_{Q_{i}}(f)\coloneqq \sup_{x\in \partial Q_i} f(x),\quad  m_{Q_{i}}(f)\coloneqq \inf_{x\in \partial Q_i} f(x), \quad \textrm{and}\\
& \osc_{Q_{i}} (f)\coloneqq  M_{Q_{i}}(f)-m_{Q_{i}}(f)
\end{align*}
are finite, and in particular, 
$$\osc_{Q_i}(f)\leq \rho(Q_i)\quad \textrm{and} \quad |M_{Q_i}(f)- \hat f(Q_i)| \leq \rho(Q_i).$$
\end{corollary}

Before proceeding we define:

\begin{definition}\label{harmonic:3-Upper gradient Definition}
Let $g\colon S\to \widehat \R$ be an extended function. We say that $\{\rho(Q_i)\}_{i\in \N}$ is a \textit{weak (strong) upper gradient}\index{upper gradient}\index{upper gradient!weak upper gradient}\index{upper gradient!strong upper gradient} for $g$ if there exists an exceptional family $\Gamma_0$ of paths in $\Omega$ with $\md_w (\Gamma_0)=0$ ($\md_s(\Gamma_0)=0$) such that for all paths $\gamma\subset \Omega$ with $\gamma\notin \Gamma_0$ and all points $x,y\in \gamma\cap S$ we have  $g(x),g(y)\neq \pm\infty$ and
\begin{align*}
|g(x)-g(y)|\leq \sum_{i: Q_i\cap \gamma\neq \emptyset} \rho(Q_i).
\end{align*}
\end{definition}

The map $f$ that we constructed inherits the properties of $\hat f$ in the weak (strong) Sobolev space definition. Namely:

\begin{prop}\label{harmonic:3-f is Sobolev}
The sequence $\{\osc_{Q_i}(f)\}_{i\in \N}$ is a weak (strong) upper gradient for $f$. Moreover, for every ball $B\subset \subset \Omega$ we have
\begin{align*}
&\sum_{i\in I_B} M_{Q_i}(f)^2 \diam(Q_i)^2 <\infty, \textrm{ and }\\
&\sum_{i\in I_B}\osc_{Q_i}(f)^2<\infty.
\end{align*}
\end{prop}
\begin{proof}
Note that the latter two claims follow immediately from  Corollary \ref{harmonic:3-Finiteness of f}, if we observe that $\diam(Q_i)$ is bounded for $i\in I_B$, where $B$ is a ball compactly contained in $\Omega$. The latter is incorporated in the definition of a relative Sierpi\'nski carpet; see Section \ref{harmonic:1-Basic Assumptions}.

For the first claim, we will show that for every $\gamma \in \mathcal G$ and  points $x,y\in \gamma\cap S$ we have
\begin{align*}
|f(x)-f(y)|\leq \sum_{i:Q_i\cap \gamma\neq \emptyset } \osc_{Q_i}(f).
\end{align*}
Fix a good path $\gamma\in \mathcal G$ that has $x,y$ as its endpoints, and assume it is parametrized so that it travels from $x$ to $y$. If $x \in \partial Q_{i_x}$, let $x'\in \partial Q_{i_x}$ be the point of last exit of $\gamma$ from $Q_{i_x}$. Similarly, consider the point $y' \in \partial Q_{i_y}$ of first entry of $\gamma$ in $Q_{i_y}$ (after $x'$), in the case $y\in \partial Q_{i_y}$. Observe that $x'$ and $y'$ would be ``accessible" points by $\gamma$. Since $|f(x)-f(x')|\leq \osc_{Q_{i_x}}(f)\leq \rho(Q_{i_x})$ and $|f(y)-f(y')|\leq \osc_{Q_{i_y}}(f)\leq \rho(Q_{i_y})$, it suffices to prove the statement for $x',y'$ and the subpath of $\gamma$ connecting them (as described), instead. 

In particular, we assume that $\gamma\in \mathcal G$ is a good path that connects $x,y$ but it does not hit $\partial Q_{i_x}$ or $\partial Q_{i_y}$ in case $x\in \partial Q_{i_x}$ or $y\in \partial Q_{i_y}$. Hence, the path $\gamma$ can be used to define both $f(x)$ and $f(y)$ by \eqref{harmonic:3-definition of f}. If $Q_{i_1},Q_{i_2}$ are peripheral disks intersected by $\gamma$ close to $x,y$, respectively, then by the upper gradient inequality for $\hat f$ we have
\begin{align*}
|\hat f(Q_{i_1})-\hat f(Q_{i_2})| \leq \sum_{i: Q_{i}\cap \gamma\neq \emptyset} \rho(Q_i).
\end{align*} 
Taking limits as $Q_{i_1}\to x$ and $Q_{i_2}\to y$ along $\gamma$ (which also shows that $f(x),f(y)\neq \pm\infty$), we have
\begin{align}\label{harmonic:3-intermediate upper gradient}
|f(x)-f(y)|\leq \sum_{i:Q_i\cap \gamma\neq \emptyset} \rho(Q_i).
\end{align}
By our preceding remarks, this also holds if $\gamma$ intersects the peripheral disks that possibly contain $x$ and $y$ in their boundary. However, we would like to prove this statement with $\osc_{Q_i}(f)$ in the place of $\rho(Q_i)$. Before doing so, we need the next topological lemma that we prove later.

\begin{lemma}\label{harmonic:3-Split path}
Let $\gamma\subset \Omega$ be a path in $\mathcal G$, and let $J\subset \N$ be a finite index set. Assume that $\gamma$ has endpoints $x,y\in S$, but $\gamma$ does not intersect the peripheral disks that possibly contain $x$ or $y$ on their boundary. Then there exist finitely many  subpaths $\gamma_1,\dots,\gamma_m$ of $\gamma$ having endpoints in $S$ with the following properties:
\begin{enumerate}[\upshape(i)]
\item $\gamma_i\in \mathcal G$ for all $i\in \{1,\dots,m\}$,
\item $\gamma_i$ intersects only peripheral disks that are intersected by $\gamma$, for all $i\in \{1,\dots,m\}$,
\item $\gamma_i$ and $\gamma_j$ intersect disjoint sets of peripheral disks for $i\neq j$,
\item $\gamma_i$ does not intersect peripheral disks $ Q_j$, $j\in J$, for all $i\in \{1,\dots,m\}$,
\item $\gamma_1$ starts at $x_1=x$, $\gamma_m$ terminates at $y_m=y$, and in general the path $\gamma_i$ starts at $x_i$ and terminates at $y_i$ such that for each $i\in \{1,\dots,m-1\}$ we either have
\begin{itemize}
\item $y_i=x_{i+1}$, i.e., $\gamma_i$ and $\gamma_{i+1}$ have a common endpoint, or
\item $y_i,x_{i+1}\in \partial Q_{j_i}$ for some $j_i\in \N$, i.e., $\gamma_i$ and $\gamma_{i+1}$ have an endpoint on some peripheral circle $\partial Q_{j_i}$. 
\end{itemize}
The peripheral disks $Q_{j_i}$ that arise from the second case are distinct and they are all intersected by the original curve $\gamma$.
\end{enumerate}
\end{lemma}
Note that properties (i) and (ii) hold automatically for subpaths of $\gamma$, so (iii),(iv), and (v) are the most crucial properties.

Since $\gamma\in \mathcal G$ and it is compactly contained in $\Omega$, we have $\sum_{i:Q_i\cap \gamma\neq \emptyset} \rho(Q_i)<\infty$ so for fixed $\varepsilon>0$ there exists a finite index set $J\subset \N$ such that 
\begin{align}\label{harmonic:3-J epsilon bound}
\sum_{\substack{i:Q_i\cap \gamma\neq \emptyset \\ i\in \N\setminus J }} \rho(Q_i)<\varepsilon.
\end{align}

We consider curves $\gamma_1,\dots,\gamma_m$ as in the lemma with their endpoints, as denoted in the lemma. If the points $y_{k},x_{k+1}$ lie on the same peripheral circle $\partial Q_{j_k}$, we have $|f(y_{k})-f(x_{k+1})|\leq \osc_{Q_{j_{k}}} (f)$. Otherwise, by the first alternative in (v), we have $y_k=x_{k+1}$, so $|f(y_{k})-f(x_{k+1})|=0$. Also, note that the curves $\gamma_k$ lie in $\mathcal G$ (by (i)) and \eqref{harmonic:3-intermediate upper gradient} holds for their endpoints.  Therefore,
\begin{align*}
|f(x)-f(y)|& \leq \sum_{k=1}^m |f(y_k)-f(x_k)|+ \sum_{k=1}^{m-1} |f(y_{k})-f(x_{k+1})|\\
&\leq \sum_{k=1}^m \sum_{i: Q_{i}\cap \gamma_k\neq \emptyset} \rho(Q_i) + \sum_{j_k} \osc_{Q_{j_{k}}}(f).
\end{align*}
Using (iii) and (iv), we see that the curves $\gamma_k$ intersect disjoint sets of peripheral disks $Q_j$, $j\notin J$, which are all intersected by $\gamma$ (property (ii)). Thus, the first term can be bounded by the expression in \eqref{harmonic:3-J epsilon bound}, and hence by $\varepsilon$. The second term is just bounded by the full sum of $\osc_{Q_i}(f)$ over $\gamma$ (since the peripheral disks $Q_{j_k}$ are distinct and $\gamma$ intersects them by (v)), hence we obtain
\begin{align*}
|f(x)-f(y)|\leq \varepsilon + \sum_{i:Q_i\cap \gamma\neq \emptyset}\osc_{Q_{i}}(f).
\end{align*}
Letting $\varepsilon\to 0$ yields the result.
\end{proof}

Now, we move to the proof of Lemma \ref{harmonic:3-Split path}.

\begin{proof}[Proof of Lemma \ref{harmonic:3-Split path}]
By ignoring some indices of $J$, we may assume that $Q_j\cap\gamma\neq \emptyset$ for all $j\in J$. The idea is to consider subpaths of $\gamma$ joining peripheral circles $\partial Q_j$, $j\in J$, without intersecting $Q_{j}$, $j\in J$, and truncate them whenever they intersect some common peripheral disk.

More precisely, we assume that $\gamma$ is parametrized as it runs from $x$ to $y$ we let $\tilde \gamma_1$ be the subpath of $\gamma$ from $x=x_1$ until the first entry point of $\gamma$ into the first peripheral disk $Q_{i_1}$, $i_1\in J$, that $\gamma$ meets, among the peripheral disks $Q_j$, $j\in J$. We let $x_2\in \partial Q_{i_1}$ be the point of last exit of $\gamma$ from $\partial Q_{i_1}$ as it travels towards $y$. Now, we repeat the procedure with $x_1$ replaced by $x_2$, $\gamma$ replaced by the subpath of $\gamma$ from $x_2$ to $y$, and $J$ replaced by $J\setminus \{i_1\}$. The procedure will terminate in the $m$-th step if the subpath of $\gamma$ from $x_m$ to $y$ does not intersect any  peripheral disk $Q_j$, $j\in J\setminus \{i_1,\dots,i_{m-1}\}$. This will be the path $\tilde \gamma_m$. Note that the indices $i_1,\dots,i_{m-1}$ are distinct.

By construction, the paths $\tilde \gamma_1,\dots,\tilde \gamma_m$ do not intersect peripheral disks $Q_j$, $j\in J$, but they might still intersect common peripheral disks. Thus, we might need to truncate some of them in order to obtain paths with the desired properties. We do this using the following general claim that we prove later:

\begin{claim}
There exist closed paths $\gamma_1',\dots,\gamma_{m'}'$, where $m'\leq m$, such that $\gamma'_i$ is a subpath of some $\tilde \gamma_j$ and the following hold:
\begin{enumerate}[\upshape(a)]
\item $\gamma_1'$ starts at the starting point $x_1=x$ of $\gamma_1$, $\gamma_{m'}'$ terminates at the terminating point $y_m=y$ of $\gamma_m$, and 
\item for each $i\in \{1,\dots,m'-1\}$ the path $\gamma_i'$ starts at $x_i'$ and terminates at $y_i'$ such that we either have
\begin{enumerate}[\upshape(b1)]
\item $y_i'=x_{i+1}'$, i.e., $\gamma_i'$ and $\gamma_{i+1}'$ have a common endpoint, or
\item there exists $j_i\in \{1,\dots,m-1\}$ such that $y_i'=y_{j_i}$ and $x_{i+1}'=x_{j_i+1}$, i.e., $\gamma_i'$ and $\gamma_{i+1}'$ have a common endpoint with $\tilde \gamma_{j_i}$ and $\tilde \gamma_{j_i+1}$, respectively.  
\end{enumerate}
The indices $j_i$ arising from the second case are distinct. 
\end{enumerate}
Moreover, the closed paths $ {\gamma_{i}'}$ are disjoint, with the exceptions of some constant paths and of the case (b1) in which ``consecutive" paths can share only one endpoint.
\end{claim}

Essentially, the conclusion of the claim is that we can obtain a family of disjoint (except at their endpoints) subpaths of $\tilde \gamma_1, \dots,\tilde \gamma_m$ that have the same properties as $\tilde \gamma_1,\dots,\tilde \gamma_m$.

Note that the paths $\gamma_i'$ do not intersect peripheral disks $Q_j$, $j\in J$, since they are subpaths of paths $\tilde \gamma_j$. Also, whenever (b2) occurs, the paths $\gamma_i'$ and $\gamma_{i+1}'$ have an endpoint on some peripheral circle $\partial Q_{j_i}$, $j_i\in J$, and the indices $j_i$ are distinct, by the construction of the curves $\tilde \gamma_i$. We discard the constant paths from the collection $\gamma_i'$. Then, after re-enumerating, we still have the preceding statement. It only remains to shrink the paths $\gamma_i'$ such that property (iii) in the statement of Lemma \ref{harmonic:3-Split path} holds. For simplicity we  denote the endpoints of the paths $\gamma_i'$ by $x_i$ and $y_i$. By replacing $ \gamma_i'$ with a subpath, we may assume that $ \gamma_i'$ does not return to $x_i$ or $y_i$ twice; in other words, if we parametrize $ \gamma_i':[0,1]\to \C$, then $ \gamma_i'(t) \neq x_i,y_i$ for all $t\in (0,1)$.

By the definition of a relative Sierpi\'nski carpet, the peripheral disks staying in a compact subset of $\Omega$ have diameters shrinking to zero. Hence, any point of ${ \gamma_i'}$ that has positive distance from a curve $ \gamma_j'$, $j\neq i$, has an open  neighborhood with the property that it only intersects finitely many peripheral disks among the peripheral disks that intersect both $ \gamma_i'$ and $ \gamma_j'$.

This observation implies that if we have two paths $\gamma_i'$ and $\gamma_j'$, $j> i$, that intersect some common peripheral disks, then we can talk about the ``first" such peripheral disk $Q_{i_0}$ that $\gamma_i'$ meets as in travels from $x_i$ to $y_i$. It is crucial here that $ {\gamma_i'}$ is disjoint from $ { \gamma_j'}$, except possibly for the endpoints $y_i$ and $x_j$ which could agree if $j=i+1$ and we are in the case (a) of the Claim; in particular, $x_i$ has positive distance from $ \gamma_j'$ and $y_j$ has positive distance from $ \gamma_i'$.

Another important observation from our Claim is that for each $i$, each of the endpoints of $\gamma_i'$ is either an endpoint of some $\tilde \gamma_{j_i}$, and therefore lies on  the boundary of some peripheral disk $Q_{j_i}$ that is intersected by $\gamma$, or it is a common endpoint of $\gamma_i'$ and $\gamma_{i+1}'$. In the latter case, if this endpoint does not lie in $S$, then it lies in some peripheral disk $Q_{i_0}$ that is intersected by both curves $\gamma_i'$ and $\gamma_{i+1}'$. The truncating procedure explained below  ensures that the paths $\gamma_i'$ are truncated suitably so that their endpoints lie in $S$, as required in the statement of Lemma \ref{harmonic:3-Split path}.

We now explain the algorithm that will yield the desired paths. We first test if $ \gamma_1'$ intersects  some common peripheral disk with $ \gamma_{m'}'$. If this is the case, then we consider the ``first" such peripheral disk $Q_{i_0}$, and truncate the paths, so that the two resulting paths, denoted by  $ \gamma_1$ and $\gamma_2$, have an endpoint on $\partial Q_{i_0}$, but otherwise intersect disjoint sets of peripheral disks and $\gamma_1$ does not intersect $Q_{i_0}$. Then the statement of the lemma is proved with $m=2$. The truncation is done in such a way that the left endpoint of $ \gamma_1'$ is the left endpoint of $\gamma_1$ and the right endpoint of $\gamma_{m'}'$ is the right endpoint of $\gamma_2$. Note here that subpaths of paths in $\mathcal G$ are also in $\mathcal G$.

If the above does not occur, then we test $ \gamma_1'$ against $ \gamma_{m-1}'$. If they intersect some common peripheral disk, then we truncate them as above to obtain paths $\gamma_1$ and $\gamma_2$. Then we test $\gamma_2$ against $ \gamma_m'$ in the same way. 

If the procedure does not stop we keep testing $\gamma_1'$ against all paths, up to $ \gamma_2'$. If the procedure still does not stop, we set $\gamma_1= \gamma_1'$, and we start testing $\gamma_2'$ against the other paths $ \gamma_i'$, $i=m,m-1,\dots,3$, etc.

To finish the proof, one has to observe that the implemented truncation does not destroy the properties (i),(ii),(iv), and (v) in the statement of Lemma \ref{harmonic:3-Split path} with were true for the paths $ \gamma_i'$. In particular, note that the peripheral disks in the second case of (v) have to be distinct by our algorithm.
\end{proof}

\begin{proof}[Proof of Claim]
Suppose we are given closed paths $\tilde \gamma_1,\dots,\tilde \gamma_m$ in the plane such that $\tilde \gamma_i$ starts at $x_i$ and terminates at $y_i$, with $x_1=x$ and $y_m=y$. The algorithm that we will use is very similar to the one used in the preceding proof of Lemma \ref{harmonic:3-Split path}.

To illustrate the algorithm we assume that $m=3$. We now check whether $\tilde \gamma_3$ intersects $\tilde \gamma_1$ or not. If $\tilde \gamma_3$ intersects $\tilde \gamma_1$, then we consider the first point $y_1' \in \tilde \gamma_3$ that $\tilde \gamma_1$ meets as it travels from $x=x_1$ to $y_1$. We call $\gamma_1'$ the subpath of $\tilde \gamma_1$ from $x_1$ to $y_1'$ (it could be that $\gamma_1'$ is a constant path if $x_1\in \tilde \gamma_3$). Then, we let $\gamma_2'$ be the subpath of $\tilde \gamma_3$ from $y_2'\coloneqq y_3=y$ to $x_2'\coloneqq y_1'$ (assuming that $\tilde \gamma_3$ is parametrized to travel from $y_3$ to $x_3$). The paths $\gamma_1'$ and $\gamma_2'$ share an endpoint but otherwise are disjoint, and they are the desired paths. Note that the alternative (b1) holds here.

If $\tilde \gamma_3$ does not intersect $\tilde \gamma_1$, then we check whether $\tilde \gamma_2$ intersects $\tilde \gamma_1$. If not, we set $\gamma_1'=\tilde \gamma_1$ and we also test if $\tilde \gamma_2$ intersects $\tilde \gamma_3$. Again, if this is not the case, then all three paths are disjoint, and thus we may set $\gamma_i'=\tilde \gamma_i$ for $i=2,3$ and the alternative (b2) holds. If $\tilde \gamma_2$ does intersect $\tilde \gamma_3$, we truncate $\tilde \gamma_2$ and $\tilde \gamma_3$ with the procedure we described in the previous paragraph, to obtain paths $\gamma_2'$ and $\gamma_3'$. This procedure keeps the left endpoint of $\tilde \gamma_2$ as a left endpoint of $\gamma_2'$ so the alternative (b2) holds between the paths $\gamma_1'$ and $\gamma_2'$. For the paths $\gamma_2'$ and $\gamma_3'$ we have the alternative (b1).

Now, suppose that $\tilde \gamma_2$ intersects $\tilde \gamma_1$. We perform the truncating procedure as before, to obtain paths $\gamma_1'$ and $\gamma_2'$ that share an endpoint so the alternative (b1) holds here. The truncating procedure keeps the right endpoint of $\tilde \gamma_2$. Then we also need to test $\gamma_2'$ against $\tilde \gamma_3$. If $\gamma_2'$ does not intersect $\tilde \gamma_3$, then we set $\gamma_3'=\tilde \gamma_3$ and (b2) holds between $\gamma_2'$ and $\gamma_3'$. Finally, if $\gamma_2'$ intersects $\tilde \gamma_3$ then we truncate them and and the alternative (b1) holds.
 
This completes all possible cases, to obtain the desired paths that are disjoint, with the exceptions of some constant paths and of ``consecutive" paths having a common endpoint. 
\end{proof}

\begin{remark}\label{harmonic:remark:split path}
The proof of Lemma \ref{harmonic:3-Split path} also yields the following variant of Lemma \ref{harmonic:3-Split path}:

Let $J\subset \N$ be a finite index set and $\gamma\subset \Omega$ be an \textit{open} path with endpoints $x,y\in \br S$ such that $\gamma$ does not intersect the \textit{closures} of the peripheral disks that possibly contain $x$ or $y$ on their boundary. Then there exist finitely many \textit{open} subpaths $\gamma_1,\dots,\gamma_m$ of $\gamma$ having endpoints in $\br S$ with the following properties:
\begin{enumerate}
\item[(iii)$'$] $\gamma_i$ and $\gamma_j$ intersect disjoint sets of \textit{closed} peripheral disks for $i\neq j$,
\item[(iv)$'$] $\gamma_i$ does not intersect peripheral disks $\br Q_j$, $j\in J$, for all $i\in \{1,\dots,m\}$,
\item[(v)$'$] $\gamma_1$ starts at $x_1=x$, $\gamma_m$ terminates at $y_m=y$, and in general the path $\gamma_i$ starts at $x_i$ and terminates at $y_i$ such that for each $i\in \{1,\dots,m-1\}$ we either have
\begin{itemize}
\item $y_i=x_{i+1}$, i.e., $\gamma_i$ and $\gamma_{i+1}$ have a common endpoint, or
\item $y_i,x_{i+1}\in \partial Q_{j_i}$ for some $j_i\in \N$, i.e., $\gamma_i$ and $\gamma_{i+1}$ have an endpoint on some peripheral circle $\partial Q_{j_i}$. 
\end{itemize}
The peripheral disks $\br Q_{j_i}$ that arise from the second case are distinct and they are all intersected by the original curve $\gamma$.
\end{enumerate} 
\end{remark}

Summarizing, in this section, we used a discrete Sobolev function $\hat f$ to construct a function $f\colon S \to \widehat \R$, defined on the entire carpet $S$, that also satisfies an upper gradient inequality.

\subsection{Sobolev spaces}\label{harmonic:Subsection Sobolev spaces}
Now, we proceed with the definition of the actual Sobolev spaces we will be using. Recall the Definition \ref{harmonic:3-Upper gradient Definition} of an upper gradient and also the definition of $M_{Q_i}$ from Corollary \ref{harmonic:3-Finiteness of f}.

\begin{definition}\label{harmonic:3-Def Sobolev space}
Let $g\colon S \to \widehat\R$ be an extended function. We say that $g$ lies in the \textit{local weak (strong) Sobolev space}\index{Sobolev space} $\mathcal W_{w,\loc}^{1,2}(S)$ $(\mathcal W_{s,\loc}^{1,2}(S))$ if for every ball $B\subset \subset  \Omega$ we have
\begin{align}
\label{harmonic:3-Def Sobolev spaces L2}&\sum_{i\in I_B} M_{Q_i}(g)^2\diam(Q_i)^2<\infty, \\
\label{harmonic:3-Def Sobolev spaces L2gradient}&\sum_{i\in I_B} \osc_{Q_i}(g)^2<\infty,
\end{align}
and $\{\osc_{Q_i}(g)\}_{i\in \N}$ is a weak (strong) upper gradient for $g$. Furthermore, if the above conditions hold for the full sums over $i\in \N$, then we say that $f$ lies in the \textit{weak (strong) Sobolev space} $\mathcal W_{w}^{1,2}(S)$ $ (\mathcal W_{s}^{1,2}(S))$.
\end{definition}
Note that part of the definition is that $|M_{Q_i}(g)|<\infty$ and $\osc_{Q_i}(g)<\infty$ for every peripheral disk $Q_i$, $i\in \N$. Also, each such function $g$ comes with a family of good paths $\mathcal G_g$ defined as in the previous section, with the following properties: $\mathcal H^1(\gamma\cap S)=0$, the upper gradient inequality for $g$ holds along all subpaths of $\gamma$, and $\sum_{i:Q_i\cap \gamma'\neq \emptyset}\osc_{Q_i}(g)<\infty$ for subpaths $\gamma'$ of $\gamma$ compactly contained in $\Omega$. Besides, in case $\sum_{i\in \N}\osc_{Q_i}(g)^2<\infty$ we require that $\sum_{i:Q_i\cap \gamma\neq \emptyset} \osc_{Q_i}(\gamma)<\infty$ for all $\gamma\in \mathcal G_g$.  

\begin{remark}\label{harmonic:3-Definition Remark}
Observe that \eqref{harmonic:3-Def Sobolev spaces L2} is equivalent to 
$$\sum_{i\in I_B} M_{Q_i}(|g|)^2\diam(Q_i)^2<\infty$$
if we assume \eqref{harmonic:3-Def Sobolev spaces L2gradient}. This is because $M_{Q_i}(|g|)\leq |M_{Q_i}(g)|+\osc_{Q_i}(g)$ and 
$$\sum_{i\in I_B} \osc_{Q_i}(g)^2\diam(Q_i)^2 <\infty.$$
\end{remark}

\begin{remark}\label{harmonic:3-Definition Remark rho-osc}
The proof of Proposition \ref{harmonic:3-f is Sobolev} shows that if there exists a locally square-summable sequence $\{\rho(Q_i)\}_{i\in \N}$ that is a weak (strong) upper gradient for $g$, then $\{\osc_{Q_i}(g)\}_{i\in \N}$ is also a weak (strong) upper gradient for $g$.
\end{remark}

\begin{remark}\label{harmonic:3-Weak contains strong}
As for the discrete Sobolev spaces we also have here that $\mathcal W^{1,2}_{s,\loc} (S)\subset \mathcal W^{1,2}_{w,\loc}(S)$; see \eqref{harmonic:Inclusion of hat Sobolev spaces}.
\end{remark}

Our discussion in Section \ref{harmonic:Subsection-Discrete Sobolev} shows that each function $\hat f\in \widehat{\mathcal W}_{w,\loc}^{1,2}(S)$ $(\widehat{\mathcal  W}_{s,\loc}^{1,2}(S))$ yields a  function $f\in \mathcal  W_{w,\loc}^{1,2}(S)$ $(\mathcal  W_{s,\loc}^{1,2}(S))$ in a canonical way. Conversely, for any $g\in \mathcal  W_{w,\loc}^{1,2}(S)$ $(\mathcal  W_{s,\loc}^{1,2}(S))$ one can construct a discrete function $\hat g\colon \{Q_{i}\}_{i\in \N } \to \R$ by setting $\hat g(Q_i) \coloneqq M_{Q_i}(g)$ for $i\in \N$. It is easy to check that $\hat g$ inherits the upper gradient inequality of $g$. Indeed, for a non-exceptional curve $\gamma\in \mathcal G_g$ that intersects $Q_{i_1},Q_{i_2}$, let $x\in \partial Q_{i_1}$ be the point of last exit of $\gamma$ from $Q_{i_1}$, and $y\in \partial Q_{i_2}$ be the point of first entry of $\gamma$ in $Q_{i_2}$. Then the subpath $\gamma'$ of $\gamma$ from $x$ to $y$ does not intersect $Q_{i_1}$ or $Q_{i_2}$, so
\begin{align}\label{harmonic:3-Oscillation of g hat}
\begin{aligned}
|\hat g(Q_{i_1})-\hat g (Q_{i_2})| &\leq |M_{Q_{i_1}}(g)-g(x)|+|g(x)-g(y)|+|g(y)-M_{Q_{i_2}}(g)|\\
 &\leq \osc_{Q_{i_1}}(g) +\sum_{i:Q_i\cap \gamma'\neq \emptyset}\osc_{Q_i}(g)+ \osc_{Q_{i_2}}(g)\\
 &\leq \sum_{i:Q_i\cap \gamma\neq \emptyset}\osc_{Q_i}(g).
\end{aligned}
\end{align}
Thus, $\hat g\in \widehat{\mathcal W}_{w,\loc}^{1,2}(S)$ $(\widehat{\mathcal W}_{s,\loc}^{1,2}(S))$ with upper gradient $\{\osc_{Q_i}(g)\}_{i\in \N}$, and there is a corresponding family of good paths $\mathcal G_{\hat g} \supset \mathcal G_g$. 

\begin{remark}
Note that equality $\mathcal G_{\hat g} =\mathcal G_g$ does not hold in general, since changing the definition of $g$ at one point $x\in S^\circ$ might destroy the upper gradient inequality of $g$ for curves passing through that point. It turns out that these curves have weak and strong modulus zero. However, $\hat g$ is not affected by this change.
\end{remark}

A question that arises is whether we can recover $g$ from $\hat g$ using our previous construction in Section \ref{harmonic:Subsection-Discrete Sobolev}. As a matter of fact, this is the case at least for points ``accessible" by good paths $\gamma\in \mathcal G_g$. To prove this, let $x$ be a point ``accessible" by a good path $\gamma$, and $Q_{i_0}$ be a peripheral disk with $Q_{i_0}\cap \gamma\neq \emptyset$; recall at this point Lemma \ref{harmonic:lemma:paths zero hausdorff} that allows us to approximate $x$ by peripheral disks intersecting $\gamma$. Consider $y\in \partial Q_{i_0}$ to be the point of first entry of $\gamma$ in $Q_{i_0}$, as $\gamma$ travels from $x$ to $Q_{i_0}$, and $\gamma'$ be the subpath of $\gamma$ from $x$ to $y$. Then
\begin{align*}
|g(x)-\hat g(Q_{i_0})|&=|g(x)-M_{Q_{i_0}}(g)|\leq |g(x)-g(y)|+|g(y)-M_{Q_{i_0}}(g)| \\
&\leq \sum_{i:Q_i\cap \gamma'\neq \emptyset}\osc_{Q_i}(g) + \osc_{Q_{i_0}}(g).
\end{align*}
As $Q_{i_0}\to x$ along $\gamma$, the right hand side converges to $0$. This is because 
$$\sum_{i:Q_i\cap \gamma\neq \emptyset}\osc_{Q_i}(g)<\infty,$$
and the subpath $\gamma'$ of $\gamma$ intersects fewer and fewer peripheral disks as $Q_{i_0}\to x$. Thus, 
\begin{align}\label{harmonic:3-g agrees with tilde g}
g(x)=\liminf_{\substack{Q_i\to x\\Q_i\cap \gamma\neq \emptyset}}\hat g(Q_i).
\end{align}
For ``non-accessible" points of $\partial Q_i$ we do not expect this, since $g$ could have any value there between $m_{Q_i}(g)$ and $M_{Q_i}(g)$. 

We now define a \textit{normalized version}\index{Sobolev space!normalized version} $\tilde g$ of $g$ by using the construction in Section \ref{harmonic:Subsection-Discrete Sobolev}. More precisely, we define 
\begin{align}\label{harmonic:Normalized definition}
\tilde g(x)= \liminf_{\substack{Q_i\to x\\ Q_i\cap \gamma\neq \emptyset}} \hat g(Q_i),
\end{align}
whenever $x\in S$ is ``accessible" by a path $\gamma\in \mathcal G_{\hat g}$, as in \eqref{harmonic:3-definition of f}. For the other points of $S$ we define $\tilde g$ as in the paragraph following \eqref{harmonic:3-definition of f}. Note that the definition of $\tilde g$ depends on the good family $\mathcal G_{\hat g}$, which in turn depends on $\mathcal G_g$. By the discussion in Section \ref{harmonic:Subsection-Discrete Sobolev}, we have $\tilde g \in \mathcal W^{1,2}_{w,\loc}(S)$ $( \mathcal W^{1,2}_{s,\loc}(S))$ with upper gradient $\{\osc_{Q_i}(\tilde g)\}_{i\in \N}$, and the upper gradient inequality holds along paths of $\mathcal G_{\hat g}\supset \mathcal G_g$.

The function $\tilde g$ agrees with $g$ for all points ``accessible" by paths $\gamma\in \mathcal G_g$, as \eqref{harmonic:3-g agrees with tilde g} shows. We remark that by \eqref{harmonic:3-Oscillation of g hat} $\{\osc_{Q_i}(g)\}_{i\in \N}$ is an upper gradient of $\hat g$, and by Corollary \ref{harmonic:3-Finiteness of f} we obtain that for the normalized version $\tilde g$ of $g$ we always have
\begin{align}\label{harmonic:3-Normalized oscillation}
\osc_{Q_i}(\tilde g)\leq \osc_{Q_i}(g)
\end{align}
for all $i\in \N$. The intuitive explanation is that the ``jumps" at ``non-accessible" points are precisely what makes a function non-normalized. Thus the process of normalization cuts these ``jumps" and reduces the oscillation of the function. We summarize the above discussion in the following lemma:

\begin{lemma}\label{harmonic:Normalized main lemma}
For each function $g\in \mathcal W^{1,2}_{w,\loc}(S)$ $( \mathcal W^{1,2}_{s,\loc}(S))$ there exists a function $\tilde g \in \mathcal W^{1,2}_{w,\loc}(S)$ $( \mathcal W^{1,2}_{s,\loc}(S))$ such that
\begin{enumerate}[\upshape(i)]
\item $\tilde g$ is defined by \eqref{harmonic:Normalized definition}, where $\hat g(Q_i)\coloneqq M_{Q_i}(g)$ for $i\in \N$,  
\item $\tilde g$ agrees with $g$ on all points that are ``accessible" by paths of $\mathcal G_g$, and
\item $\osc_{Q_i}(\tilde g)\leq \osc_{Q_i}(g)$ for all $i\in \N$.
\end{enumerate}
Moreover, the upper gradient inequality of $\tilde g$ holds along paths of $\mathcal G_{\hat g} \supset \mathcal G_g$.
\end{lemma}

The most important property of the normalized version $\tilde g$ of $g$ is the following continuity property:

\begin{lemma}\label{harmonic:3-Normalized Property}
Let $x\in S$ be a point (not necessarily ``accessible"). Then there exist peripheral disks $Q_i$ contained in arbitrarily small neighborhoods of $x$ such that $M_{Q_i}(\tilde g)$ approximates $\tilde g(x)$.
\end{lemma}
\begin{proof}
It suffices to show that we can replace $\hat g(Q_i)=M_{Q_i}(g)$ in the definition \eqref{harmonic:Normalized definition} of $\tilde g$ by $M_{Q_i}(\tilde g)$. Indeed, if $x$ is an ``accessible" point by $\gamma\in \mathcal G_{\hat g}$, then there exist arbitrarily small peripheral disks near $x$, intersecting $\gamma$, such that $\hat g(Q_i)$ approximates $\tilde g(x)$, by the definition of $\tilde g$. If $x$ is ``non-accessible" then $\tilde g(x)$ is defined through approximating $x$ by ``accessible" points, and hence we can find again small peripheral disks near $x$ with the desired property.

Now we prove our claim. Let $x$ be a point that is ``accessible" by $\gamma\in \mathcal G_{\hat g}$, so
\begin{align*}
\tilde g(x)=\liminf_{\substack{Q_i\to x\\ Q_i\cap \gamma \neq \emptyset}} M_{Q_i}(g).
\end{align*}
We fix $i\in \N$. By an application of Lemma \ref{harmonic:Paths joining continua}, we can find a point $y\in \partial Q_i$ that is ``accessible" by a curve lying in $\mathcal G_g$. By Lemma \ref{harmonic:Normalized main lemma}(ii) we have $\tilde g(y)=g(y)$. Hence, using  Lemma \ref{harmonic:Normalized main lemma}(iii) we have
\begin{align*}
|M_{Q_i}(g)-M_{Q_i}(\tilde g)|&\leq |M_{Q_i}(g)-g(y)|+|\tilde g(y)-M_{Q_i}(\tilde g)| \\
&\leq \osc_{Q_i}(g)+ \osc_{Q_i}(\tilde g)\leq 2\osc_{Q_i}(g).
\end{align*}
Since $\{\osc_{Q_i}(g)\}_{i\in \N}$ is locally square-summable near $x$, and the peripheral disks $Q_i$ become arbitrarily small as $Q_i\to x$ and $Q_i\cap \gamma\neq \emptyset$, it follows that
\begin{align*}
\limsup_{\substack{Q_i\to x\\ Q_i\cap \gamma \neq \emptyset}}|M_{Q_i}(g)-M_{Q_i}(\tilde g)|\leq \limsup_{\substack{Q_i\to x\\ Q_i\cap \gamma \neq \emptyset}}2\osc_{Q_i}(g)=0.
\end{align*}
This shows that we can indeed define
\[
\tilde g(x)= \liminf_{\substack{Q_i\to x\\ Q_i\cap \gamma \neq \emptyset}} M_{Q_i}(\tilde g). \qedhere
\]
\end{proof}

This discussion suggests that we identify the functions $g$ of the space $\mathcal  W_{w,\loc}^{1,2}(S)$ $(\mathcal  W_{s,\loc}^{1,2}(S))$ that have the ``same" normalized version $\tilde g$. This will be made more precise with the following lemma:

\begin{lemma}\label{harmonic:3-Identification lemma}
Let $f,g\in \mathcal  W_{w,\loc}^{1,2}(S)$ $(\mathcal  W_{s,\loc}^{1,2}(S))$, each of them coming with a family of good paths $\mathcal G_f$ and $\mathcal G_g$, respectively. The following are equivalent:
\begin{enumerate}[\upshape (a)]
\item There exists a family $\mathcal G_0$ that contains almost every path in $\Omega$ (i.e., the complement of $\mathcal G_0$ has weak (strong) modulus zero) such that $f(x)=g(x)$ for all points $x$ that are ``accessible" by paths $\gamma \in \mathcal G_0$.
\item For the normalized versions $\tilde f$ and $\tilde g$ and for all $i\in \N$ we have
\begin{align*}
M_{Q_i}(\tilde f)&=M_{Q_i}(\tilde g), \quad m_{Q_i}(\tilde f)=m_{Q_i}(\tilde g), \quad \textrm{and} \quad \osc_{Q_i}(\tilde f)= \osc_{Q_i}(\tilde g).
\end{align*}
\end{enumerate}
\end{lemma}
\begin{proof}
First, we assume that there exists a family $\mathcal G_0$ such that $f(x)=g(x)$ for all points $x$ ``accessible" by $\gamma\in \mathcal G_0$. Then for points $x$ ``accessible" by paths in $\mathcal G\coloneqq  \mathcal G_f\cap \mathcal G_g\cap \mathcal G_0$ we have $\tilde f(x)=f(x)=g(x)=\tilde g(x)$ by Lemma \ref{harmonic:Normalized main lemma}(ii). Note that $\mathcal G$ contains almost every path, by the subadditivity of modulus. Fix a peripheral disk $Q_{i_0}$, $\varepsilon>0$, and a point $x\in \partial Q_{i_0}$ with $|M_{Q_{i_0}}(\tilde f)-\tilde f(x)|<\varepsilon$; recall the definition of $M_{Q_{i_0}}(\tilde f)$. Then, consider a point $x'\in \partial Q_{i_0}$ near $x$ that is ``accessible" by a curve $\gamma\in \mathcal G_{\hat f}$, such that $|\tilde f(x)- \tilde f(x')|<\varepsilon$; this can be done by the definition of $\tilde f$ on ``non-accessible" points. Now, Lemma \ref{harmonic:3-curves gamma_r}(a) yields a circular arc $\gamma_r$ around $x'$ with a small $r>0$ (see also Figure \ref{harmonic:fig:crosscut}) such that:
\begin{enumerate}[\upshape (1)]
\item $\gamma_r\cap Q_{i_0}=\emptyset$ and $\gamma_r$ has its endpoints on $\partial Q_{i_0}$, so that $\gamma_r$ defines a crosscut separating $x'$ from $\infty$ in $\R^2 \setminus Q_{i_0}$,
\item $\gamma_r$ avoids the set $\gamma\cap S$ and intersects a peripheral disk $Q_{i_1},i_1\neq i_0$, that is also intersected by $\gamma$,
\item $\gamma_r \in \mathcal G$,
\item $\sum_{i:Q_i\cap \gamma_r\neq \emptyset} \osc_{Q_i}(\tilde f)<\varepsilon$.
\end{enumerate}
Let $y\in \partial Q_{i_0}\cap  \gamma_r $ a point ``accessible" by $\gamma_r$, i.e., an endpoint of $\gamma_r$. Then $\tilde f(y)=\tilde g(y)$, because $\gamma_r\in \mathcal G$. Let $\gamma'$ be the subpath of $\gamma$ from $Q_{i_1}$ to $x$. Using the upper gradient inequality of $\tilde f$, which holds along $\gamma'\in \mathcal G_{\hat f}$ and $\gamma_r\in \mathcal G\subset \mathcal G_f \subset \mathcal G_{\hat f}$, we have 
\begin{align*}
|\tilde f(x')-\tilde f(y)| \leq \sum_{i: Q_i\cap \gamma'\neq \emptyset} \osc_{Q_i}(\tilde f)+ \sum_{i:Q_i\cap \gamma_r\neq \emptyset} \osc_{Q_i}(\tilde f)\leq \sum_{i: Q_i\cap \gamma'\neq \emptyset} \osc_{Q_i}(\tilde f) +\varepsilon.
\end{align*}
If $Q_{i_1}$ is sufficiently close to $x$ (and thus $r>0$ is chosen to be smaller), then the last sum can be made less than $\varepsilon$. Putting all the estimates together we obtain
\begin{align*}
|M_{Q_{i_0}}(\tilde f)-\tilde g(y)|<4\varepsilon.
\end{align*}
This shows that $M_{Q_{i_0}}(\tilde f)\leq M_{Q_{i_0}}(\tilde g)$. Interchanging the roles of $\tilde f$ and $\tilde g$ we obtain the equality. The equality $m_{Q_{i_0}}(\tilde f)= m_{Q_{i_0}}(\tilde g)$ is proved by using the same argument.

For the converse, note that for all points $x$ that are ``accessible" by paths $\gamma\in \mathcal G_{\hat g}$ we have
\begin{align*}
\tilde g(x)= \liminf_{\substack{Q_i\to x\\ Q_i\cap \gamma\neq \emptyset}} M_{Q_i}(\tilde g),
\end{align*}
and an analogous statement is true for $\tilde f$; see Lemma \ref{harmonic:3-Normalized Property} and its proof. Hence, for all points $x$ ``accessible" by paths $\gamma\in \mathcal G_0\coloneqq \mathcal G_f\cap \mathcal G_g\subset \mathcal G_{\hat f} \cap \mathcal G_{\hat g}$ we have
\begin{align*}
\tilde g(x)= \liminf_{\substack{Q_i\to x\\ Q_i\cap \gamma\neq \emptyset}} M_{Q_i}(\tilde g)=\liminf_{\substack{Q_i\to x\\ Q_i\cap \gamma\neq \emptyset}} M_{Q_i}(\tilde f)=\tilde f(x).
\end{align*}
On the other hand, for such points we also have $\tilde f(x)=f(x)$ and $\tilde g(x)=g(x)$ by Lemma \ref{harmonic:Normalized main lemma}(ii). The conclusion follows, if one notes that the curve family $\mathcal G_0$ contains almost every path in $\Omega$, by the subadditivity of modulus.
\end{proof}

Hence, we can identify functions $f,g\in \mathcal W_{w,\loc}^{1,2}(S)$ $(\mathcal  W_{s,\loc}^{1,2}(S))$ whenever their normalized versions $\tilde f, \tilde g$ yield the same sequences $\{M_{Q_i}(\tilde f)\}_{i\in \N}, \{\osc_{Q_i}(\tilde f) \}_{i\in \N}$. The identification allows us to regard $\mathcal  W_{w,\loc}^{1,2}(S)$, $\mathcal  W_{s,\loc}^{1,2}(S)$ as subsets of a space of sequences, the (non-linear) correspondence being
\begin{align*}
g\mapsto ( \{M_{Q_i}(\tilde g)\diam(Q_i)\}_{i\in \N}, \{\osc_{Q_i}(\tilde g) \}_{i\in \N} ).
\end{align*}
These sequences are locally square-summable in the sense of \eqref{harmonic:3-Def Sobolev spaces L2} and \eqref{harmonic:3-Def Sobolev spaces L2gradient}. If $g$ was originally in one of the non-local Sobolev spaces instead, then we could identify $g$ with an element of $\ell^2(\N)\times \ell^2(\N)$. 

However, we will not use this identification in the next sections, and the Sobolev functions that we use are not necessarily normalized, unless otherwise stated.

\subsection{Examples}\label{harmonic:examples}
We give some examples of functions lying in the Sobolev spaces $\mathcal W_{w,\loc}^{1,2}(S)$ and $\mathcal W_{s,\loc}^{1,2}(S)$.

\begin{example}\label{harmonic:3-Example-Lipschitz}
Let $f\colon\Omega\to \R$ be a locally Lipschitz function, i.e., for every compact set $K\subset \Omega$ there exists a constant $L>0$ such that $|f(x)-f(y)|\leq L|x-y|$ for $x,y\in K$. Then $f\big|_{S} \in \mathcal W_{s,\loc}^{1,2}(S) \subset \mathcal W_{w,\loc}^{1,2}(S)$. In particular, this is true if $f\colon\Omega\to \R$ is smooth.

To see this, consider a compact exhaustion $\{K_n\}_{n\in \N}$, $K_n\subset K_{n+1}$ of $\Omega$, and an increasing sequence of Lipschitz constants $L_n$ for $f\big|_{K_n}$. Define $\rho(Q_i)=L_n\diam(Q_i)$ where $n\in \N$ is the smallest integer such that $Q_i\subset K_n$. If $B\subset \subset \Omega$, then $f$ is bounded on $B$, thus
\begin{align*}
\sum_{i\in I_B} M_{Q_i}(f)^2 \diam(Q_i)^2<\infty,
\end{align*}
by Corollary \ref{harmonic:Fatness corollary}. Also, $\osc_{Q_i}(f)\leq \rho(Q_i)$ by the Lipschitz condition, and if $B\subset K_N$, then 
\begin{align*}
\sum_{i\in I_B} \rho(Q_i)^2 \leq L_N^2 \sum_{i\in I_B}\diam(Q_i)^2 <\infty.
\end{align*}

Finally, let $\gamma$ be a curve in $\Omega$ with $\mathcal H^1(\gamma\cap S)=0$, and  $x,y\in \gamma\cap S$. We wish to show that 
\begin{align*}
|f(x)-f(y)|\leq \sum_{i:Q_i\cap \gamma \neq \emptyset}\rho(Q_i).
\end{align*}
If suffices to prove this for a closed subpath of $\gamma$ that connects $x$ and $y$, which we still denote by $\gamma$. This will imply that $\{\rho(Q_i)\}_{i\in \N}$ is a strong upper gradient for $f$, and thus, so is $\{\osc_{Q_i}(f)\}_{i\in \N}$; see Remark \ref{harmonic:3-Definition Remark rho-osc}.

For fixed $\varepsilon>0$ we cover the compact set $\gamma\cap S$ with finitely many balls $B_j$ of radius $r_j$ such that $\sum_{j}r_j<\varepsilon$. Furthermore, we may assume that $B_j\subset \subset \Omega$ for all $j$. Then there are at most finitely many peripheral disks that intersect $\gamma$ and are not covered by $\bigcup_j B_j$. Indeed, note that the closure of each of these peripheral disks must intersect both $\partial (\bigcup_{j}B_j)$ and $\gamma\cap S$. On the other hand, $\partial (\bigcup_{j}B_j)$ and $\gamma\cap S$ have positive distance, hence the peripheral disks whose closure intersects both of them have diameters bounded below. By Lemma \ref{harmonic:Fatness consequence} we conclude that these peripheral disks have to be finitely many. 

We let $A_k$, $k\in \{1,\dots,M\}$, be the joint collection of the balls $B_j$ and of the finitely many peripheral disks $Q_i$ not covered by $\bigcup_j B_j$. In other words, for each $k=1,\dots,M$ we have $A_k=B_j$ for some $j$, or $A_k$ is one of these peripheral disks, and the sets $A_k$ are distinct, i.e., we do not include a ball or a peripheral disk twice. Note that $x$ and $y$ lie in some balls, so after reordering we suppose that $x\in A_1$ and $y\in A_M$. The union $\bigcup_k A_k$ contains the curve $\gamma$, and thus it contains a connected chain of sets $A_k$, connecting $x$ to $y$. We consider a minimal sub-collection of $\{A_k\}_k$ that connects $x$ and $y$ (there are only finitely many sub-collections), and we still denote it in the same way. Then, using the minimality, we can order the sets $A_k$ in such a way, that $x\in A_1$, $y\in A_M$, and  $A_k\cap A_{l}\neq \emptyset$ if and only if $l=k\pm 1$. 

Now, we choose points $x_{k+1}\in A_k\cap A_{k+1}$ for $k=1,\dots,M-1$, and set $x_1=x\in A_1$ and $x_{M+1}=y\in A_M$. Suppose that $\bigcup_k A_k \subset K_N$ for some $N\in \N$. If $A_k$ is a ball $B_j$, then $|f(x_k)-f(x_{k+1})| \leq 2L_Nr_j$, and if $A_k$ is a peripheral disk $Q_i$, then $|f(x_k)-f(x_{k+1})|\leq \rho(Q_i)$; see definition of $\rho(Q_i)$. Hence,

\begin{align*}
|f(x)-f(y)|&\leq \sum_{k=1}^{M} |f(x_{k})-f(x_{k+1})| \\
&\leq 2L_N\sum_{j}r_j +\sum_{i:Q_i\cap \gamma\neq \emptyset} \rho(Q_i)\\
&\leq 2L_N\varepsilon +\sum_{i:Q_i\cap \gamma\neq \emptyset} \rho(Q_i).
\end{align*}
Letting $\varepsilon \to 0$ shows that $\{\rho(Q_i)\}_{i\in \N}$ is a strong upper gradient for $f$, as desired.
\end{example}

\begin{example}\label{harmonic:3-Example Sobolev Homeo}
Let $f\colon\Omega \to f(\Omega)\subset\C$ be a homeomorphism that lies in the classical space $W^{1,2}_{\loc}(\Omega)$. Then $f\big|_{S}\in \mathcal W_{s,\loc}^{1,2}(S)$ in the sense that this holds for the real and imaginary parts of $f\big|_{S}$. In particular, (locally) quasiconformal maps on $\Omega$ lie in $\mathcal W_{s,\loc}^{1,2}(S)$; see \cite{AstalaIwaniecMartin:quasiconformal} for the definition of a quasiconformal map and basic properties.    

Since $f$ is locally bounded, for any ball $B\subset \subset \Omega$ we have
\begin{align*}
\sum_{i\in I_B} M_{Q_i}(|f|)^2\diam(Q_i)^2 <\infty.
\end{align*}
We will show that $\rho(Q_i)\coloneqq  \diam (f(Q_i))$ is also locally square-summable, and it is a strong upper gradient for $f$, i.e., there exists a path family $\Gamma_0$ in $\Omega$ with $\md_s(\Gamma_0)=0$ such that
\begin{align*}
|f(x)-f(y)|\leq \sum_{i:Q_i\cap \gamma\neq \emptyset}\rho(Q_i)
\end{align*}
for $\gamma\subset \Omega$, $\gamma\notin \Gamma_0$, and $x,y\in \gamma\cap S$. Note that this suffices by Remark \ref{harmonic:3-Definition Remark rho-osc}, and it will imply that $\{\osc_{Q_i}(\re(f))\}_{i\in \N}$ and $\{\osc_{Q_i}(\im(f))\}_{i\in \N}$ are strong upper gradients for $\re(f)$ and $\im(f)$, respectively. Using a compact exhaustion, we see that it suffices to show this for paths $\gamma$ contained in an open set $V\subset \subset \Omega$. 

Let $U$ be a neighborhood of $V$ such that $V\subset\subset U\subset\subset \Omega$, and $U$ contains all peripheral disks that intersect $V$. By a recent theorem of Iwaniec, Kovalev and Onninen \cite[Theorem 1.2]{IwaniecKovalevOnninen:W1papproximation} there exist smooth homeomorphisms $f_n\colon U \to f_n(U)$ such that $f_n\to f$ uniformly on $U$ and in $W^{1,2}(U)$. For each $n\in \N$ the function $f_n$ is locally Lipschitz, so by Example \ref{harmonic:3-Example-Lipschitz} it satisfies 
\begin{align}\label{harmonic:3-Example Sobolev Homeo limit}
|f_n(x)-f_n(y)| \leq \sum_{i:Q_i\cap \gamma\neq \emptyset}\osc_{Q_i}(f_n) = \sum_{i:Q_i\cap \gamma\neq \emptyset} \diam(f_n(Q_i))
\end{align}
for all $\gamma \subset V$ with $\mathcal H^1(\gamma\cap S)=0$ and $x,y\in \gamma\cap S$. We claim that $\{\diam (f_n(Q_i))\}_{i\in I_V}$ converges in $\ell^2$. Then its limit must be the same as the pointwise limit, namely $\{\diam (f(Q_i))\}_{i\in I_V}$, so the latter is also square-summable. Taking limits in \eqref{harmonic:3-Example Sobolev Homeo limit} after passing to a subsequence and applying Fuglede's Lemma \ref{harmonic:Fuglede} we would have the desired
\begin{align*}
|f(x)-f(y)| \leq \sum_{i:Q_i\cap \gamma \neq \emptyset} \diam(f(Q_i))
\end{align*}
for all paths $\gamma \subset V$ outside an exceptional family $\Gamma_0$ with $\md_s(\Gamma_0)=0$, and $x,y\in \gamma\cap S$.

Now, we show our claim. By the quasiballs assumption, there exist balls $B(x_i,r_i)\subset Q_i\subset B(x_i,R_i)$, $i\in \N$, with $R_i/r_i\leq K_0$. Since $\dist(V,\partial U)>0$, there exist only finitely many peripheral disks $Q_i$, $ i\in I_V$, such that $B(x_i,2R_i)$ is not contained in $U$; see Lemma \ref{harmonic:Fatness consequence}. Let $J$ be the family of such indices. Since this is a finite set and $\diam(f_n(Q_i)) \to \diam (f(Q_i)) $ all $i\in J$, it suffices to show that $\{\diam(f_n(Q_i))\}_{i\in I_V\setminus J}$ converges to $\{\diam(f(Q_i))\}_{i\in I_V\setminus J}$ in $\ell^2$.

We fix $i\in I_V\setminus J$. For each $r\in [R_i,2R_i]$ and $x,y\in \partial Q_i \subset B(x_i,r),$ by the maximum principle applied to the homeomorphisms $f_n$ and the fundamental theorem  of calculus we have
\begin{align*}
|f_n(x)-f_n(y)| &\leq \osc_{\partial B(x_i,r)}(f_n) \leq  \int_{\partial B(x_i,r)} |\nabla f_n|\,ds,
\end{align*} 
where $\nabla f_n= ( \nabla \re(f_n),\nabla \im(f_n))$. Integrating over $r\in [R_i,2R_i]$ we obtain
\begin{align*}
|f_n(x)-f_n(y)|\leq C R_i \mint{-}_{B(x_i,2R_i)} |\nabla f_n| 
\end{align*}
for some constant $C>0$. Since $x,y\in \partial Q_i$ were arbitrary, we have
\begin{align}\label{harmonic:3-Example Sobolev Homeo domination}
\diam (f_n(Q_i)) \leq C R_i \mint{-}_{B(x_i,2R_i)} |\nabla f_n| 
\end{align}
for all $i\in I_V\setminus J$.

In the following, for simplicity we write $f$ to mean $f\x_U$, and by assumption $f_n \to f$ in $W^{1,2}(U)$. Using the uncentered maximal function $M(g)(x)= \sup_{x\in B} \mint{-}_B |g|$ to change the center of balls we have
\begin{align*}
\left| R_i \mint{-}_{B(x_i,2R_i)} |\nabla f_n| - R_i \mint{-}_{B(x_i,2R_i)} |\nabla f|\right|&\leq R_i \mint{-}_{B(x_i,2R_i)} |\nabla f_n-\nabla f|\\
&\leq CR_i \mint{-}_{B(x_i,r_i)} M(|\nabla f_n -\nabla f|),
\end{align*}
where the constant $C>0$ depends only on the quasiballs constant $K_0$. Thus,
\begin{align*}
\sum_{i\in I_V\setminus J} &\left| R_i \mint{-}_{B(x_i,2R_i)} |\nabla f_n| - R_i \mint{-}_{B(x_i,2R_i)} |\nabla f|\right|^2 \\
&\leq C' \sum_{i\in I_V\setminus J} R_i^2 \left(\mint{-}_{B(x_i,r_i)} M(|\nabla f_n -\nabla f|)\right)^2\\
&\leq C'\sum_{i\in I_V\setminus J} R_i^2 \mint{-}_{B(x_i,r_i)} M(|\nabla f_n -\nabla f|)^2 \\
&\leq C'' \sum_{i\in I_V \setminus J} \int_{B(x_i,r_i)} M(|\nabla f_n -\nabla f|)^2 \\
& \leq C''\int_{U} M(|\nabla f_n- \nabla f|)^2 \\
&\leq C''' \int_{U} |\nabla f_n-\nabla f|^2.
\end{align*}
This shows that $\{R_i \mint{-}_{B(x_i,2R_i)} |\nabla f_n|\}_{i\in I_V\setminus J}$ converges to $\{R_i \mint{-}_{2B(x_i,R_i)} |\nabla f|\}_{i\in I_V\setminus J}$ in $\ell^2$. Since this sequence dominates $\{\diam(f_n(Q_i))\}_{i\in I_V\setminus J}$ by \eqref{harmonic:3-Example Sobolev Homeo domination}, it follows that $\{\diam(f_n(Q_i))\}_{i\in I_V\setminus J}$ converges in $\ell^2$ as well, as claimed.
\end{example}

\begin{remark}
One could do the preceding proof directly for quasiconformal maps, by approximating them with smooth quasiconformal maps; this is a much more elementary result than the approximation of Sobolev homeomorphisms by smooth homeomorphisms. However, the use of the strong result \cite[Theorem 1.2]{IwaniecKovalevOnninen:W1papproximation} proves a more general result while the proof remains essentially the same.
\end{remark}

\begin{remark}\label{harmonic:Remark:quasiconformal}
The same proof shows that if $\br \Omega$ is compact and $f$ is quasiconformal in a neighborhood of $\Omega$ then $f\big|_{S}\in \mathcal W^{1,2}_{s}(S)$.
\end{remark}

\begin{remark}\label{harmonic:Remark:convergence}
From the above proof we see that if $f_n\in \mathcal W_{s,\loc}^{1,2}(S)$ converges to $f\colon S\to \R$ locally uniformly and for each $V\subset \subset \Omega$ the sequence $\{\osc_{Q_i}(f_n)\}_{i\in I_V}$ converges to $\{\osc_{Q_i}(f)\}_{i\in I_V}$ in $\ell^2$, then $f$ lies in $\mathcal W_{s,\loc}^{1,2}(S)$.
\end{remark}

\subsection{Pullback of Sobolev spaces}\index{Sobolev space!pullback}
Here we study the invariance of Sobolev spaces under quasiconformal maps between relative Sierpi\'nski carpets.

Let $(S,\Omega)$, $(S',\Omega')$ be two relative Sierpi\'nski carpets and let $F\colon \Omega'\to \Omega$ be a locally quasiconformal homeomorphism that maps the peripheral disks $Q_i'$ of $S'$ to the peripheral disks $Q_i=F(Q_i')$ of $S$. We have:

\begin{prop}\label{harmonic:Pullback Sobolev Proposition}If $g \in \mathcal W^{1,2}_{w,\loc}(S)$, then the pullback $g\circ F$ lies in $\mathcal W^{1,2}_{w,\loc}(S')$.
\end{prop}

\begin{proof}
Let $g \in \mathcal W^{1,2}_{w,\loc}(S)$, and note that $M_{Q_i'}(g\circ F)= M_{Q_i}(g)$ and $\osc_{Q_i'}(g\circ F)= \osc_{Q_i}(g)$. We only have to show that there exists a path family $\Gamma_0'$ in $\Omega'$ with weak modulus equal to zero such that
\begin{align}\label{harmonic:Pullback Sobolev proof}
|g\circ F(x)-g\circ F(y)|\leq \sum_{i:Q_i'\cap \gamma\neq \emptyset} \osc_{Q_i'}(g\circ F)= \sum_{i:Q_i\cap F(\gamma)\neq \emptyset} \osc_{Q_i}(g)
\end{align}
whenever $\gamma\subset \Omega'$, $\gamma\notin \Gamma_0'$,  and $x,y\in \gamma \cap S$. 

By our assumption on $g$, there exists a path family $\Gamma_0$ in $\Omega$ with $\md_w(\Gamma_0)=0$ such that the upper gradient inequality for $g$ holds along paths $\gamma\notin \Gamma_0$. By the equality in \eqref{harmonic:Pullback Sobolev proof} it suffices to show that for almost every $\gamma$ in $\Omega'$ the image $F(\gamma)$ avoids the exceptional family $\Gamma_0$. Equivalently, we show that the family $\Gamma_0'\coloneqq F^{-1}(\Gamma_0)$ has weak modulus equal to zero. 

Note that $\md_w(\Gamma_0)=0$ implies that $\md_2(\Gamma_0)=0$ by Lemma \ref{harmonic:2-weak modulus zero implies two modulus zero}. Since $F$ is locally quasiconformal, so is $F^{-1}$, and they preserve conformal modulus zero. Therefore, $\md_2(F^{-1}(\Gamma_0))=0$. Again, by Lemma \ref{harmonic:2-weak modulus zero implies two modulus zero} we have $\md_w(\Gamma_0')=0$, as desired. 
\end{proof}

\begin{corollary}\label{harmonic:Pullback Quasisymmetry}
Assume that the peripheral circles $\partial Q_i',\partial Q_i$ of $S',S$, respectively, are uniform quasicircles and let $F\colon S'\to S$ be a local quasisymmetry. Then for any $g\in \mathcal W^{1,2}_{w,\loc}(S)$ we have  $g\circ F\in \mathcal W^{1,2}_{w,\loc}(S')$.
\end{corollary}

See \cite[Chapters 10--11]{Heinonen:metric} for background on quasisymmetric maps. The proof follows immediately from Proposition \ref{harmonic:Pullback Sobolev Proposition} and the following lemma:

\begin{lemma}\label{harmonic:Pullback Extension}
Assume that the peripheral circles $\partial Q_i',\partial Q_i$ of $S',S$, respectively, are uniform quasicircles and let $F\colon S'\to S$ be a local quasisymmetry. Then $F$ extends to a locally quasiconformal map from $\Omega'$ onto $\Omega$.
\end{lemma}

We only provide a sketch of the proof.

\begin{proof}
The first observation is that since $F$ is a homeomorphism it maps each peripheral circle $\partial Q_i'$ of $S'$ onto a peripheral circle $\partial Q_i$ of $S$; see \cite[Lemma 5.5]{Bonk:uniformization} for an argument. Then one uses the well-known Beurling-Ahlfors extension to obtain a quasiconformal extension $F\colon Q_i'\to Q_i$ inside each peripheral disk. The resulting map $F\colon \Omega'\to \Omega$ will be locally quasiconformal and the proof can be found in \cite[Section 5]{Bonk:uniformization}, where careful quantitative estimates are also shown. In our case we do not need such careful estimates.
\end{proof}

\subsection{Properties of Sobolev spaces}

We record here some properties of Sobolev functions:
\begin{prop}\label{harmonic:3-Properties Sobolev}
The spaces $\mathcal W^{1,2}_{w,\loc}(S)$ and $\mathcal W^{1,2}_{s,\loc}(S)$ are linear. Moreover, if $u,v\in \mathcal W^{1,2}_{w,\loc}(S)$ $(\mathcal W^{1,2}_{s,\loc}(S))$, then  the following functions also lie in the corresponding Sobolev space:
\begin{enumerate}[\upshape (a)]
\item $|u|$, with $\osc_{Q_i}(|u|)\leq \osc_{Q_i}(u)$,\label{harmonic:3-Properties Sobolev-|f|}
\item $u\lor v\coloneqq \max(u,v)$, with $\osc_{Q_i}(u\lor v)\leq \max\{ \osc_{Q_i}(u),\osc_{Q_i}(v)\}$,\label{harmonic:3-Properties Sobolev-max}
\item $u\land v\coloneqq \min(u,v)$, with $\osc_{Q_i}(u\land v)\leq \max\{ \osc_{Q_i}(u),\osc_{Q_i}(v)\}$,\label{harmonic:3-Properties Sobolev-min}
\end{enumerate}
where $i\in \N$. Moreover, 
\begin{enumerate}[\upshape(d)]
\item if $u$ and $v$ are locally bounded in $S$, then $u\cdot v$ lies in the corresponding Sobolev space, with
\begin{align*}
\osc_{Q_i}(u\cdot v) \leq M_{Q_i}(|v|)\osc_{Q_i}(u)  + M_{Q_i}(|u|)\osc_{Q_i} (v).
\end{align*}
for all $i\in \N$.
\end{enumerate} 
Finally, if we set $f=u\lor v$ and $g=u\land v$, then we have the inequality
\begin{align}\label{harmonic:3-Properties Sobolev-inequality}
\osc_{Q_i}(f)^2+\osc_{Q_i}(g)^2\leq \osc_{Q_i}(u)^2+\osc_{Q_i}(v)^2
\end{align}
for all $i\in \N$.
\end{prop}
\begin{proof}
To prove that the spaces are linear, we note that if $u,v$ are Sobolev functions and $a,b\in \R$, then $\osc_{Q_i}(au+bv) \leq |a|\osc_{Q_i}(u)+|b|\osc_{Q_i}(v)$, which shows that the upper gradient inequalities of $u$ and $v$ yield an upper gradient inequality  for $au+bv$. The summability conditions \eqref{harmonic:3-Def Sobolev spaces L2} and \eqref{harmonic:3-Def Sobolev spaces L2gradient} in the definition of a Sobolev function  are trivial.

Part \ref{harmonic:3-Properties Sobolev-|f|} follows from the triangle inequality $||u(x)|-|u(y)||\leq |u(x)-u(y)|$, which shows that $|u|$ inherits its upper gradient inequality from $u$.

To show \ref{harmonic:3-Properties Sobolev-max} note that $u\lor v= (u+v+|u-v|)/2 $. Using the linearity of Sobolev spaces and part \ref{harmonic:3-Properties Sobolev-|f|} we obtain that $u\lor v$ also lies in the Sobolev space. To show the inequality, we only need to observe that
\begin{equation}\label{harmonic:3-Properties max/min of max}
\begin{aligned}
M_{Q_i}(u\lor v)&= \max\{M_{Q_i}(u),M_{Q_i}(v)\}\quad \textrm{and}\\
 m_{Q_i}(u\lor v)&\geq \max\{m_{Q_i}(u),m_{Q_i}(v)\}.
\end{aligned} 
\end{equation}
Part \ref{harmonic:3-Properties Sobolev-min} is proved in the exact same way, if one notes that
\begin{equation}\label{harmonic:3-Properties max/min of min}
\begin{aligned}
M_{Q_i}(u\land v)&\leq \min\{M_{Q_i}(u),M_{Q_i}(v)\}\quad \textrm{and}\\
 m_{Q_i}(u\land v)&=\min\{m_{Q_i}(u),m_{Q_i}(v)\}.
\end{aligned}
\end{equation}

For part (d) we note that the oscillation inequality is a straightforward computation. This, together with the local boundedness of $u$ and $v$ show immediately the summability conditions \eqref{harmonic:3-Def Sobolev spaces L2} and \eqref{harmonic:3-Def Sobolev spaces L2gradient}; see also Corollary \ref{harmonic:Fatness corollary}. We only have to show the upper gradient inequality. Suppose that $\gamma \subset\subset  \Omega$ is a path that is good for both $u$ and $v$, and connects $x,y\in S$. Since $\gamma\subset \subset \Omega$, there exists $M>0$ such that $|u|\leq M$, $|v|\leq M$ on $\gamma\cap S$. Then
\begin{align*}
|u(x)v(x) -u(y)v(y)| &\leq M |u(x)-u(y)|+M|v(x)-v(y)|\\
&\leq M\sum_{i: Q_i\cap \gamma \neq \emptyset}( \osc_{Q_i} (u)+\osc_{Q_i}(v)).
\end{align*}
Since this also holds for subpaths of $\gamma$, the proof of Proposition \ref{harmonic:3-f is Sobolev} shows that
\begin{align*}
|u(x)v(x) -u(y)v(y)|\leq \sum_{i:Q_i\cap \gamma \neq \emptyset} \osc_{Q_i}(u\cdot v), 
\end{align*}
as desired; see also Remark \ref{harmonic:3-Definition Remark rho-osc}.

To show inequality \eqref{harmonic:3-Properties Sobolev-inequality} we fix $i\in \N$, and for simplicity drop $Q_i$ from the notations $\osc_{Q_i},M_{Q_i},m_{Q_i}$. We now split into cases, and by symmetry we only have to check two cases. If $M(u)\geq M(v)$ and $m(u)\geq m(v)$, then by \eqref{harmonic:3-Properties max/min of max} and \eqref{harmonic:3-Properties max/min of min} we have
\begin{align*}
\osc(f)^2+\osc(g)^2 \leq (M(u)-m(u))^2 + (M(v)-m(v))^2= \osc(u)^2+\osc(v)^2.
\end{align*}
If $M(u)\geq M(v)$ and $m(u)\leq m(v)$, then using again \eqref{harmonic:3-Properties max/min of max} and \eqref{harmonic:3-Properties max/min of min} we obtain
\begin{align*}
\osc(f)^2+\osc(g)^2 &\leq ( M(u)-m(v))^2+(M(v)-m(u))^2\\
&=\osc(u)^2+\osc(v)^2-2(M(u)-M(v))(m(v)-m(u)).
\end{align*}
The last term in the above expression is non-negative by assumption, thus the expression is bounded by $\osc(u)^2+\osc(v)^2$, as claimed. 
\end{proof}

Next, we include a lemma that allows us to ``patch" together Sobolev functions. For an open set $V\subset \Omega$ recall that $\partial_*V = \partial V\cap S$; see Figure \ref{harmonic:fig:boundary}.

\begin{figure}
  \centering
  \input{boundaryV.tikz}
  \caption{An open set $V$ (pink) intersecting a round carpet $S$, and the set $\partial_*V$ that corresponds to $V$. Here, $\Omega$ has one boundary component, the largest circle. Also, $\br S$ is an actual Sierpi\'nski carpet, as defined in the Introduction.}
  \label{harmonic:fig:boundary}
\end{figure}
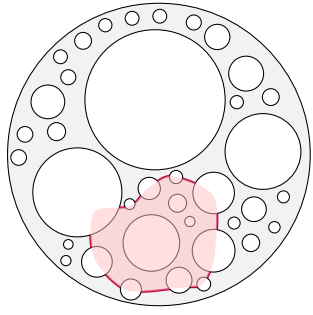

\begin{lemma}\label{harmonic:3-Lemma-removability} 
Let $V\subset \Omega$ be an open set such that $\partial_* V\neq \emptyset$. Let $\phi,\psi \in \mathcal W^{1,2}_{w,\loc}(S)$ $(\mathcal W^{1,2}_{s,\loc}(S))$ such that $\phi=\psi$ on $\partial_{*}V$. Then $h\coloneqq  \phi \x_{S\cap V}+ \psi \x_{S\setminus V}\in \mathcal W^{1,2}_{w,\loc}(S)$ $(\mathcal W^{1,2}_{s,\loc}(S)) $. 

Moreover, $\osc_{Q_i}(h)= \osc_{Q_i}(\phi)$ for $\partial Q_i\subset V$, $\osc_{Q_i}(h)=\osc_{Q_i}(\psi)$ for $\partial Q_i\subset \Omega\setminus \br{V}$, and $\osc_{Q_i}(h)\leq \osc_{Q_i}(\phi)+\osc_{Q_i}(\psi)$ otherwise.
\end{lemma}

\begin{proof}
We first show the oscillation relations. The first two are trivial and for the last one we note that if $\partial Q_i\not\subset V$ and $\partial Q_i\not\subset \Omega \setminus \br V$, then there exists a point $x\in \partial Q_i\cap \partial V \subset \partial_*V$, so $\phi(x)=\psi(x)$. Let $z,w\in \partial Q_i$ be arbitrary. If $h(z)=\phi(z)$ and $h(w)=\psi(w)$ then
\begin{align*}
|h(z)-h(w)|&\leq |h(z)-h(x)|+|h(x)-h(w)|\\
&= |\phi(z)-\phi(x)|+|\psi(x)-\psi(w)|\\
&\leq \osc_{Q_i}(\phi)+\osc_{Q_i}(\psi).
\end{align*}
The above inequality also holds trivially in case $h(z)=\phi(z)$ and $h(w)=\phi(w)$, or $h(z)=\psi(z)$ and $h(w)=\psi(w)$. This proves the claim. 

The summability condition \eqref{harmonic:3-Sobolev L2 gradient} follows immediately from the oscillation relations, and the summability condition \eqref{harmonic:3-Sobolev L2} follows from the fact that $M_{Q_i}(|h|) \leq  M_{Q_i}(|\phi|)+M_{Q_i}(|\psi|)$; see Remark \ref{harmonic:3-Definition Remark}. It remains to show the upper gradient inequality. By Remark \ref{harmonic:3-Definition Remark rho-osc}, it suffices to prove that there exists a locally square-summable sequence $\{\rho(Q_i)\}_{i\in \N}$, which is a weak (strong) upper gradient for $h$. This will imply that $\{\osc_{Q_i}(h)\}_{i\in \N}$ has the same property.

Let $\gamma$ be a path that is good for both $\phi$ and $\psi$ and joins $x,y\in \gamma\cap S$. We wish to prove that
\begin{align*}
|h(x)-h(y)|\leq \sum_{i:Q_i\cap \gamma\neq \emptyset}\rho(Q_i)
\end{align*}
for  points $x,y\in \gamma\cap S$, where $\rho(Q_i)$ is to be chosen. This will imply that the upper gradient inequality holds along the family of paths which are good for both $\phi$ and $\psi$, and thus for almost every path by the subadditivity of modulus.

If the endpoints $x,y$ of $\gamma$ lie in $S\cap V$, then we have
\begin{align*}
|h(x)-h(y)|\leq \sum_{i:Q_i\cap \gamma\neq \emptyset} \osc_{Q_i}(\phi),
\end{align*}
and if $x,y\in S\setminus V$, then
\begin{align*}
|h(x)-h(y)|\leq \sum_{i:Q_i\cap \gamma\neq \emptyset} \osc_{Q_i}(\psi).
\end{align*}

Now, suppose that $x\in S\cap V$ and $y\in S\setminus V$,  but the path $\gamma$ does \textit{not} intersect $\partial_*V$. This implies that $\gamma$ intersects some peripheral disk $Q_{i_0}$ with $\partial Q_{i_0}\cap \partial V\neq \emptyset$. Indeed, consider the set 
\begin{align*}
\widetilde V= V\cup \left( \bigcup_{i:\partial Q_i \subset V} Q_i\right),
\end{align*}
which has the properties that $\widetilde V\cap S=V\cap S$, and $\partial_{*}\widetilde V=\partial_*V$. Note that $\gamma$ intersects $\partial \widetilde V$ at a point $z$, since it has to exit $\widetilde V$. Furthermore, $z$ cannot lie in $S$ because $\gamma \cap \partial_*V=\gamma\cap \partial_{*}\widetilde V=\emptyset$, but it has to lie in some peripheral disk $Q_{i_0}$. We assume that $z$ is the first point of $\partial \widetilde V$ that $\gamma$ hits as it travels from $x$ to $y$. Let $x_0\in \partial Q_{i_0}\cap \gamma$ be the first entry point of $\gamma$ in $\partial Q_{i_0}$, and note that necessarily $x_0\in \widetilde V \cap S= V\cap S$. Since $\partial Q_{i_0} \not\subset V$ (otherwise $z\in \widetilde V$), we have $\partial Q_{i_0} \cap \partial V\neq \emptyset$, so we fix a point $w \in \partial Q_{i_0} \cap \partial V\subset \partial_*V$. We now have
\begin{align*}
|h(x)-h(y)|&= |\phi(x)-\psi(y)|\\
\notag &\leq |\phi(x)-\phi(x_0)|+ |\phi(x_0)-\phi(w)|\\
\notag &\quad \quad \quad+ |\psi(w)-\psi(x_0)|+ |\psi(x_0)-\psi(y)| \\
\notag &\leq 2\left(\sum_{i:Q_i\cap \gamma\neq \emptyset}\osc_{Q_i}(\phi)+\sum_{i:Q_i\cap \gamma\neq \emptyset} \osc_{Q_i}(\psi) \right),
\end{align*}
where we used that $\phi(w)=\psi(w)$ by the assumption that $\phi=\psi$ on $\partial_*V$.

Finally, we assume that $x\in S\cap V$ and $y\in S\setminus V$ and there exists a point $z\in \gamma\cap \partial_*V$. Here we have the estimate
\begin{align*}
|h(x)-h(y)|&\leq |h(x)-h(z)|+|h(z)-h(y)|\\
\notag&=|\phi(x)-\phi(z)| + |\psi(z)-\psi(y)|\\
\notag &\leq \sum_{i:Q_i\cap \gamma\neq \emptyset}\osc_{Q_i}(\phi)+\sum_{i:Q_i\cap \gamma\neq \emptyset} \osc_{Q_i}(\psi).
\end{align*}

Summarizing, we may choose $\rho(Q_i)=2(\osc_{Q_i}(\phi)+\osc_{Q_i}(\psi))$ for $i\in \N$. This is clearly locally square-summable, since $\{\osc_{Q_i}(\phi)\}_{i\in \N}$ and $\{\osc_{Q_i}(\psi)\}_{i\in \N}$ are.
\end{proof}

\begin{remark}
It is very crucial in the preceding statement that $\phi=\psi$ on $\partial_{*}V$, and we do not merely have $\phi(x)=\psi(x)$ for ``accessible" points $x\in \partial_{*}V$. Indeed, one can construct square Sierpi\'nski carpets for which the conclusion fails, if we use $\phi=0$ and $\psi=1$, and the ``interface" $\partial_{*}V=\partial V\cap S$ is too small to be ``seen" by carpet modulus; this is to say, that the curves passing through $\partial_*V$ have weak (strong) carpet modulus equal to $0$. See also the next lemma.
\end{remark}

For technical reasons, we also need the following modification of the previous lemma:

\begin{lemma}\label{harmonic:Removability3}
Let $V\subset \Omega$ be an open set. Let $\phi,\psi \in \mathcal W^{1,2}_{w,\loc}(S)$ $(\mathcal W^{1,2}_{s,\loc}(S))$ and suppose that there exists a path family $\mathcal G$ in $\Omega$ that contains almost every path, such that $\phi(x) =\psi(x)$ for all points $x\in S$ that are ``accessible" by paths of $\mathcal G$. Then $h\coloneqq  \phi \x_{S\cap V}+ \psi \x_{S\setminus V}\in \mathcal W^{1,2}_{w,\loc}(S)$ $(\mathcal W^{1,2}_{s,\loc}(S)) $. 

Moreover, $\osc_{Q_i}(h)= \osc_{Q_i}(\phi)$ for $\partial Q_i\subset V$, $\osc_{Q_i}(h)=\osc_{Q_i}(\psi)$ for $\partial Q_i\subset \Omega\setminus \br{V}$, and $\osc_{Q_i}(h)\leq \osc_{Q_i}(\phi)+\osc_{Q_i}(\psi)$ otherwise.
\end{lemma}

The statement that $\mathcal G$ contains almost every path is equivalent to saying that its complement has weak (strong) modulus equal to zero. By Lemma \ref{harmonic:3-Identification lemma}, the assumption of the lemma is equivalent to saying that $\phi$ and $\psi$ have the same normalized version. The conclusion is essentially that no matter how one modifies a function within its equivalence class, it still remains in the Sobolev space. 

\begin{proof}
The proof is elementary so we skip some steps. Consider the curve  family $\mathcal G_0$ which contains all curves that are good for both $\phi$ and $\psi$, and are contained in $\mathcal G$. 

For a fixed $i\in \N$ consider points $z,w\in \partial Q_i$. Using Lemma \ref{harmonic:Paths joining continua}, we may find a path $\gamma\in \mathcal G$ and a point $x\in \partial Q_i$ that is ``accessible" from $\gamma$. Therefore, $\phi(x)=\psi(x)$. Now, if $h(z)=\phi(z)$ and $h(w)=\psi(w)$, we have
\begin{align*}
|h(z)-h(w)|\leq |\phi(z)-\phi(x)|+|\psi(x)-\psi(w)|\leq \osc_{Q_i}(\phi)+ \osc_{Q_i}(\psi).
\end{align*}
This shows one of the claimed oscillation inequalities. The other cases are trivial.

For the upper gradient inequality, let $\gamma\in \mathcal G$ be a curve and $x,y\in \gamma\cap S$. If $x$ and $y$ are ``accessible" by $\gamma$, then by assumption $\phi(x)=\psi(x)=h(x)$ and $\phi(y)=\psi(y)=h(y)$, hence
\begin{align*}
|h(x)-h(y)|= |\phi(x)-\phi(y)|\leq \sum_{i:Q_i\cap \gamma \neq \emptyset}\osc_{Q_i}(\phi).
\end{align*}
If $x\in \partial Q_{i_x}$ is ``non-accessible" then we may consider the last exit point of $\gamma$ from $Q_{i_x}$ as it travels from $x$ to $y$, in order to obtain an additional contribution $\osc_{Q_{i_x}}(h)\leq \osc_{Q_{i_x}}(\phi)+ \osc_{Q_{i_x}}(\psi)$ in the above sum. The same comment applies if $y$ is ``non-accessible". Thus, in all cases
\begin{align*}
|h(x)-h(y)|\leq \sum_{i:Q_i\cap \gamma \neq \emptyset}(\osc_{Q_i}(\phi)+\osc_{Q_i}(\psi)).
\end{align*}
This shows that $\rho(Q_i)\coloneqq \osc_{Q_i}(\phi)+\osc_{Q_i}(\psi)$ is an upper gradient of $h$. Using Remark \ref{harmonic:3-Definition Remark rho-osc} we derive the desired conclusion.
\end{proof}

Finally, we need a special instance of Lemma \ref{harmonic:3-Lemma-removability}:

\begin{corollary}\label{harmonic:Removability4}
Let $V\subset \Omega$ be an open set such that $\partial_* V\neq \emptyset$. Let $\psi \in \mathcal W^{1,2}_{w,\loc}(S)$ $(\mathcal W^{1,2}_{s,\loc}(S))$ and $M\in \R$ be such that $\psi\leq M$ on $\partial_{*}V$. Then $h\coloneqq  (\psi \land M) \x_{S\cap V}+ \psi \x_{S\setminus V}\in \mathcal W^{1,2}_{w,\loc}(S)$ $(\mathcal W^{1,2}_{s,\loc}(S)) $. Moreover, $\osc_{Q_i}(h)\leq \osc_{Q_i}(\psi)$ for all $i\in \N$.
\end{corollary}
\begin{proof}
The function $\phi\coloneqq \psi\land M$ lies in $\mathcal W^{1,2}_{w,\loc}(S)$ $(\mathcal W^{1,2}_{s,\loc}(S)) $ by Lemma \ref{harmonic:3-Properties Sobolev}\ref{harmonic:3-Properties Sobolev-min}, with $\osc_{Q_i}(\phi)\leq \osc_{Q_i}(\psi)$. Since $\phi=\psi$ on $\partial_*V$, it follows that $h\in \mathcal W^{1,2}_{w,\loc}(S)$ $(\mathcal W^{1,2}_{s,\loc}(S)) $ by Lemma \ref{harmonic:3-Lemma-removability}. It remains to show the oscillation inequality.

We fix $i\in \N$. If $\partial Q_i$ is contained in $V$ or in $S\setminus V$, then there is nothing to show, since $\osc_{Q_i}(\phi)\leq \osc_{Q_i}(\psi)$. Hence, we assume that $\partial Q_i$ intersects $\partial V$, i.e., $\emptyset \neq \partial Q_i\cap \partial V\subset \partial_*V$. Using the assumption that $\psi\leq M$ on $\partial_{*}V$, we see that $m_{Q_i}(\psi) \leq M$. If $\psi \big|_{\partial Q_i\cap V} \leq M$, then there is nothing to show, since $h=\psi$ on $\partial Q_i$. Suppose that there exists $z\in \partial Q_i\cap V$ such that $\psi(z)>M$. Then 
\begin{align*}
m_{Q_i}(\psi)\leq M < M_{Q_i}(\psi).
\end{align*}
This implies that $\osc_{Q_i}(h)\leq \osc_{Q_i}(\psi)$, as desired.
\end{proof}

\section{Carpet-harmonic functions}\label{harmonic:Section Carpet Harmonic}
Throughout the section we fix a relative Sierpi\'nski carpet $(S,\Omega)$ with the standard assumptions.

\subsection{Definition of carpet-harmonic functions}
Let $V\subset \Omega$ be an open set, and $f\in \mathcal W_{w,\loc}^{1,2}(S)$ $(\mathcal W_{s,\loc}^{1,2}(S))$. Define the \textit{(Dirichlet) energy functional}\index{Dirichlet energy} by
\begin{align*}
D_V(f)= \sum_{i\in I_V } \osc_{Q_i}(f)^2\in [0,\infty].
\end{align*}
Using the energy functional we define the notion of a weak (strong) carpet-harmonic function.

\begin{definition}\label{harmonic:4-Def carpet harmonic}
A function $u\in \mathcal W_{w,\loc}^{1,2}(S)$ $(\mathcal W_{s,\loc}^{1,2}(S))$ is weak (strong) carpet-harmonic\index{carpet-harmonic} if for every open set $V\subset  \subset \Omega$ and each $\zeta \in \mathcal W_{w}^{1,2}(S)$ $(\mathcal W_{s}^{1,2}(S))$ with $\zeta\big|_{S\setminus V}\equiv 0$ we have
\begin{align*}
D_V(u)\leq D_V(u+\zeta).
\end{align*}
In other words, $u$ minimizes the energy functional $D_V$ over Sobolev functions with the same boundary values as $u$.
\end{definition}

The functions $u,\zeta$ in the above definition are not assumed to be normalized, in the sense of the discussion in Section \ref{harmonic:Subsection Sobolev spaces}. Later we will see that the normalized version of a carpet-harmonic function has to be continuous; see Theorem \ref{harmonic:4-continuous}.

\begin{example}\label{harmonic:Square carpets coordinate functions}Let $(S,\Omega)$ be a relative Sierpi\'nski carpet\index{Sierpi\'nski carpet!relative!square} such that all peripheral disks $Q_i$ are squares with sides parallel to the coordinate axes; see Figure \ref{harmonic:fig:square}. Then the coordinate functions $u(x,y)=x$ and $v(x,y)=y$ are both weak and strong carpet-harmonic.

Since $u,v$ are Lipschitz, Example \ref{harmonic:3-Example-Lipschitz} implies that they both lie in $\mathcal W^{1,2}_{s,\loc}(S) \subset \mathcal W^{1,2}_{w,\loc}(S)$.

Let $V\subset \subset \Omega$ be an open set and consider the open set $V'=V\cup \bigcup_{i\in I_V}Q_i \supset V$. This set contains all the peripheral disks that it intersects and it is also compactly contained in $\Omega$. Moreover, $V'$ satisfies $I_{V'}=I_V$, and thus $D_{V'}\equiv D_V$. We will show that $D_{V}(v) \leq D_{V}(g)$ for all $g\in \mathcal W_{w}^{1,2}(S)\supset \mathcal W_s^{1,2}(S)$ with $g=v$ outside $V'$. This suffices for harmonicity. Indeed, if $g\in \mathcal W_{w}^{1,2}(S)$ is arbitrary with $g=v$ outside $V$, then $g=v$ outside $V'\supset V$, so $D_{V}(v)\leq D_{V}(g)$, which shows harmonicity.

From now on, we denote $V'$ by $V$ and we will use the property that it contains the peripheral disks that it intersects. Let $g\in \mathcal W_{w}^{1,2}(S)$ with $g=v$ outside $V$. Note that for a.e.\ $x\in \R$ the vertical line $\gamma_x$ passing through $x$ (or rather its subpaths that lie in $\Omega$) is a good path for $g$, by an argument very similar to the proof of Lemma \ref{harmonic:Paths joining continua}. We fix $x$ such that $\gamma_x$ is good and $\gamma_x\cap V\neq \emptyset$. The intersection is an open subset of a line, so it can be written as an (at most) countable union of disjoint open intervals $J_j=\{x\}\times(a_j,b_j)$, $j\in \N$. The points $(x,a_j),(x,b_j)$ must lie in $\partial V$, and therefore, they lie in $\partial V\cap S$, by the assumption on $V$. By the fact that we have $g(x,y)=v(x,y)=y$ for $(x,y)\notin V$, and the upper gradient inequality we obtain
\begin{align*}
b_j-a_j= |g(x,b_j)-g(x,a_j)|\leq \sum_{i: Q_i\cap J_j\neq \emptyset} \osc_{Q_i}(g).
\end{align*}
Summing over all $j$ and noting that a square $Q_i$ cannot intersect two distinct sets $J_j$, we have
\begin{align*}
\mathcal H^1(\gamma_x\cap V) \leq \sum_{\substack{i:Q_i\cap \gamma_x\neq \emptyset \\ i\in I_V}} \osc_{Q_i}(g).
\end{align*}
Integrating over $x$ and using Fubini's theorem in both sides, we have
\begin{align*}
\mathcal H^2(V) \leq \sum_{i\in I_V} \osc_{Q_i}(g) \int_{Q_i\cap \gamma_x\neq \emptyset} \,dx= \sum_{i\in I_V} \osc_{Q_i}(g) \ell(Q_i),
\end{align*}
where $\ell(Q_i)$ is the sidelength of $Q_i$. Now, using the Cauchy-Schwarz inequality and the fact that $\mathcal H^2(V)= \sum_{i\in I_V} \ell(Q_i)^2$ (because $\mathcal H^2(S)=0$), we see that
\begin{align*}
\sum_{i\in I_V} \ell(Q_i)^2 \leq \sum_{i\in I_V} \osc_{Q_i}(g)^2.
\end{align*} 
On the other hand, it is easy to see that
\begin{align*}
D_V(v)= \sum_{i\in I_V} \osc_{Q_i}(v)^2=\sum_{i\in I_V} \ell (Q_i)^2,
\end{align*}
so indeed $v$ has the minimal energy. The computation for $u(x,y)=x$ is analogous.
\end{example}

\begin{figure}
	\centering
	\includegraphics[width=.6\linewidth]{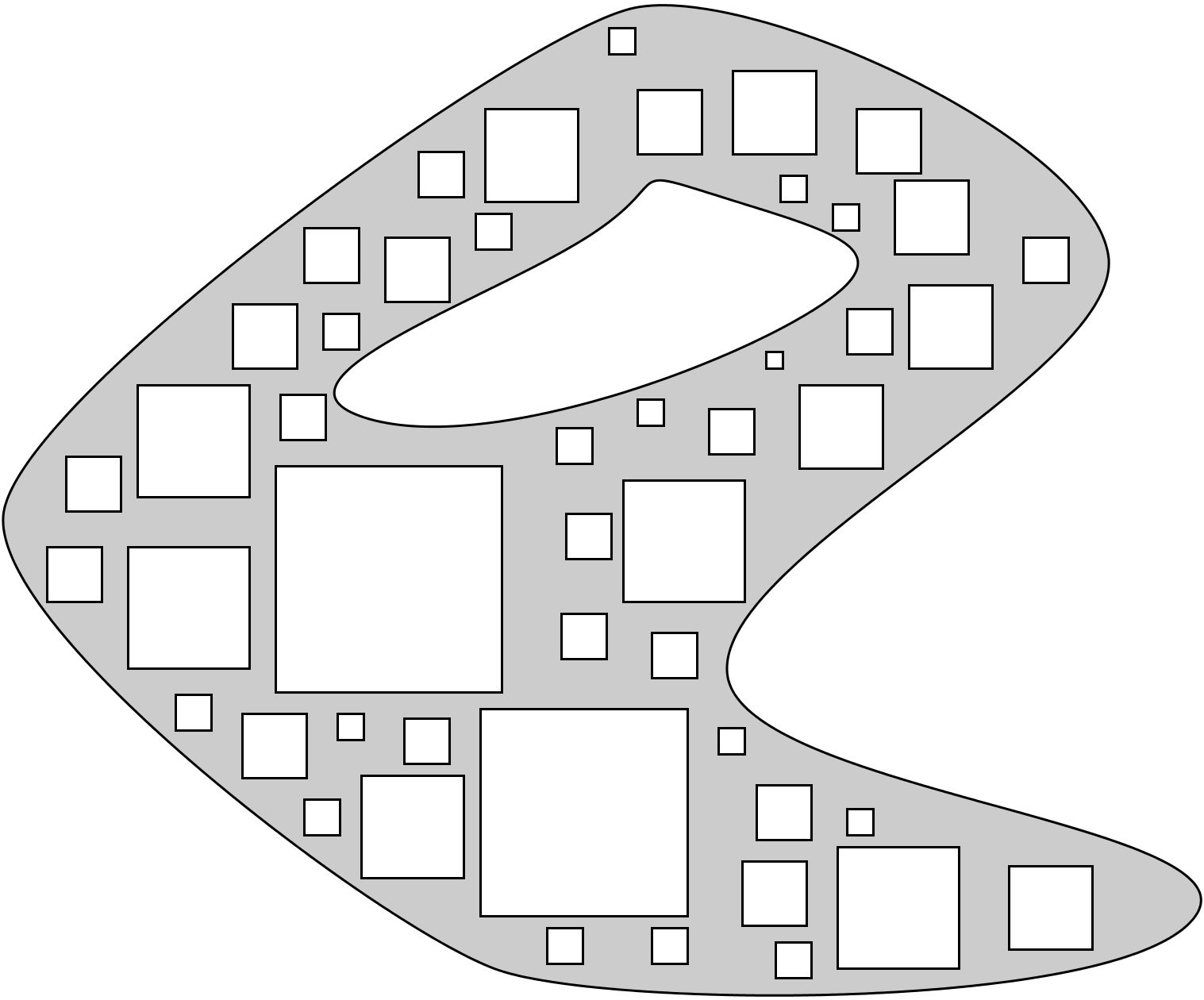}
  \caption{A square relative Sierpi\'nski carpet $(S,\Omega)$. Here $\Omega$ has two boundary components, the curves that are not squares.}
  \label{harmonic:fig:square}
\end{figure}

As we saw in Proposition \ref{harmonic:Pullback Sobolev Proposition}, locally quasiconformal maps preserve Sobolev spaces. Therefore, they must also preserve carpet-harmonic functions:
\begin{prop}\label{harmonic:Pullback of Carpet-harmonic}
Let $(S,\Omega)$, $(S',\Omega')$ be relative Sierpi\'nski carpets, and assume that $F\colon \Omega'\to \Omega$ is a locally quasiconformal map that maps the peripheral disks $Q_i'$ of $S'$ to the peripheral disks $Q_i=F(Q_i')$ of $S$. If $u\colon S\to \widehat \R$ is weak carpet-harmonic, then $u\circ F\colon  S'\to \widehat \R$ is also weak carpet-harmonic.
\end{prop}
\begin{proof}
Let $u\colon S\to \widehat \R$ be a weak carpet-harmonic function. Fix an open set $V'\subset \subset \Omega'$ and a function $\zeta' \in \mathcal W_{w}^{1,2}(S')$ such that $\zeta' \big|_{S'\setminus V'}\equiv 0$. Then $V\coloneqq F(V')$ is compactly contained in $\Omega$, and $\zeta\coloneqq  \zeta' \circ F^{-1} \in \mathcal W^{1,2}_{w}(S)$ with $\zeta\big|_{S\setminus V}\equiv 0$, by Proposition \ref{harmonic:Pullback Sobolev Proposition}. Thus, by the correspondence of the peripheral disks and the harmonicity of $u$ we have
\begin{align*}
D_{V'}(u\circ F+ \zeta')&= D_{V}(u\circ F\circ F^{-1} +\zeta'\circ F^{-1})\\
&= D_V(u+\zeta)\geq D_V(u) =D_{V'}(u\circ F).
\end{align*}  
This shows that $u\circ F$ is weak carpet-harmonic, as desired.
\end{proof}

\begin{corollary}\label{harmonic:Pullback quasisymmetry of carpet-harmonic}
Let $(S,\Omega)$, $(S',\Omega')$ be relative Sierpi\'nski carpets, and let $F\colon S'\to S$ be a local quasisymmetry. Furthermore, we assume that the peripheral circles of $S$ and $S'$ are uniform quasicircles. If $u\colon S\to \widehat \R$ is weak carpet-harmonic, then $u\circ F\colon  S'\to \widehat \R$ is also weak carpet-harmonic.
\end{corollary}

The proof follows immediately from the extension Lemma \ref{harmonic:Pullback Extension} and Proposition \ref{harmonic:Pullback of Carpet-harmonic}.

An interesting corollary of this discussion that relates carpet-harmonic functions to rigidity problems on square carpets is the following:
\begin{corollary}\label{harmonic:Pullback square carpets}
Let $(S,\Omega)$, $(S',\Omega')$ be relative Sierpi\'nski carpets, and let $F\colon S'\to S$ be a local quasisymmetry. We assume that the peripheral circles of $S'$ are uniform quasicirlces and that the peripheral circles of $S$ are squares with sides parallel to the coordinate axes. Then the coordinates $u,v$ of the map $F\coloneqq (u,v)$ are weak carpet-harmonic. 
\end{corollary}
\begin{proof}
By Example \ref{harmonic:Square carpets coordinate functions} the $x$,$y$-coordinate functions on $S$ are weak carpet-harmonic. Corollary \ref{harmonic:Pullback quasisymmetry of carpet-harmonic} implies that the pullbacks $u,v$ of the coordinates are weak carpet-harmonic.  
\end{proof}

\subsection{Solution to the Dirichlet problem}\label{harmonic:Section Dirichlet}\index{Dirichlet problem}
Let $(S,\Omega)$ be a relative Sierpi\'nski carpet such that $\partial \Omega$ consists of finitely many, non-trivial, and disjoint Jordan curves (recall that these are homeomorphic images of $S^1$). We fix a function $f \in \mathcal W_{w}^{1,2}(S)$ $(\mathcal W_{s}^{1,2}(S))$. Then we can define the boundary values\index{boundary values} of $f$ on points $x\in \partial \Omega$ that are ``accessible" by a good path $\gamma$, using an analog of Lemma \ref{harmonic:3-well-defined}. Namely, we consider a good open path $\gamma\subset \Omega$ such that $\br \gamma\cap \partial \Omega\neq \emptyset$, and for $x\in \br \gamma \cap \partial \Omega$ we define
\begin{align*}
f(x)= \liminf_{\substack{Q_i\to x \\Q_i\cap \gamma\neq \emptyset }} M_{Q_i}(f).
\end{align*}
By a variant of Lemma \ref{harmonic:3-well-defined} this definition does not depend on the path $\gamma$ with $x\in \br \gamma$. For the ``non-accessible" points  $x\in \partial \Omega$ we define $f(x)=\liminf_{y\to x}f(y)$ where $y\in \partial \Omega$ is ``accessible". Note that every point $x\in \partial \Omega$ is the landing point of a (not necessarily good) path $\gamma\subset \Omega$, by our assumptions on $\partial \Omega$. Perturbing $\gamma$ as in Lemma \ref{harmonic:Paths boundary} we obtain a point $y$ near $x$ that is ``accessible" by a good path. Hence, there is a dense set of points in $\partial \Omega$ which are ``accessible" by good paths, and this implies that the boundary values of $f$ are well-defined on all of $\partial \Omega$.

We say that a function $u\in \mathcal W_{w}^{1,2}(S)$ $(\mathcal W_{s}^{1,2}(S))$ \textit{has boundary values equal to $f$} if there exists a path family $\mathcal G_0$ in $\Omega$, whose complement has weak (strong) modulus zero, such that $u(x)=f(x)$ for all points $x\in \partial \Omega$ that are ``accessible" by paths $\gamma\in \mathcal G_0$.

\begin{theorem}\label{harmonic:4-Thm Solution to Dirichlet problem}\index{Dirichlet problem!solution}
Suppose that $\Omega$ is bounded and let $f\in \mathcal W_{w}^{1,2}(S)$ $(\mathcal W_{s}^{1,2}(S))$ be a function with bounded boundary values, i.e., there exists $M>0$ such that $|f(x)|\leq M$ for all $x\in \partial \Omega$. Then there exists a unique function $u \in \mathcal W_{w}^{1,2}(S)$ $(\mathcal W_{s}^{1,2}(S)) $ that minimizes $D_\Omega(g)$ over all $g\in \mathcal W_{w}^{1,2}(S)$ $(\mathcal W_{s}^{1,2}(S))$ with boundary values equal to $f$. The function $u$ is weak (strong) carpet-harmonic.
\end{theorem}
\begin{proof}The uniqueness part will be postponed until we have established several properties of carpet-harmonic functions; see Theorem \ref{harmonic:Uniqueness}. For the existence part, one has to minimize $D_\Omega(g)$ over all $g\in \mathcal W_{w}^{1,2}(S)$ $(\mathcal W_{s}^{1,2}(S))$ with boundary values equal to $f$. We call such functions $g$ \textit{admissible (for the Dirichlet problem)}. It is easy to show that if $g$ minimizes $D_\Omega(g)$ then it is carpet-harmonic. Indeed, for every  $\zeta\in \mathcal W_{w}^{1,2}(S)$ $(\mathcal W_{s}^{1,2}(S))$ that vanishes outside an open set $V\subset \subset \Omega$ we have 
\begin{align*}
D_{\Omega}(g)\leq D_{\Omega}(g+\zeta).
\end{align*}
Note that $\osc_{Q_i}(g+\zeta)=\osc_{Q_i}(g)$ for $i\notin I_V$. Canceling the common terms we obtain $D_V(g)\leq D_V(g+\zeta)$, so $g$ is carpet-harmonic.

Define $D\coloneqq  \inf D_\Omega(g)$ where the infimum is taken over all admissible functions, and is finite since $f$ is admissible. Let $g_n \in \mathcal W_{w}^{1,2}(S)$ $(\mathcal W_{s}^{1,2}(S))$ be a minimizing sequence of admissible functions, i.e., $D_\Omega(g_n) \to D$ as $n\to \infty$. Note that $G_n\coloneqq (g_n \land M)\lor(-M)$ is still a Sobolev function with $D_\Omega(G_n)\leq D_\Omega(g_n)$, by Proposition \ref{harmonic:3-Properties Sobolev}\ref{harmonic:3-Properties Sobolev-max},\ref{harmonic:3-Properties Sobolev-min}, since $M$ is a constant function. Thus, by replacing $g_n$ with $G_n$, we may assume that $|g_n|\leq M$.

Now, we have a minimizing sequence $ g_n$ that satisfies $|M_{Q_i}(g_n)|\leq M$ for all $i\in \N$. In particular, the sequences $\{M_{Q_i}( g_n)\diam(Q_i)\}_{i\in \N}$, $\{\osc_{Q_i}( g_n)\}_{i\in \N}$ are uniformly bounded in $\ell^2(\N)$. Here, it is crucial that $M_{Q_i}( g_n)\diam(Q_i) \leq M\diam(Q_i)$, and that 
\begin{align*}
\sum_{i\in \N} \diam(Q_i)^2 <\infty
\end{align*}
by the boundedness of $\Omega$ and the quasiballs assumption (or by Corollary \ref{harmonic:Fatness corollary}).

By passing to subsequences we may assume that for each $i\in \N$ we have $M_{Q_i}(g_n) \to M_{Q_i}$ and $\osc_{Q_i}( g_n) \to \rho(Q_i)$ for some real numbers $M_{Q_i},\rho(Q_i)$. By Fatou's lemma we have
\begin{align}\label{harmonic:3-Solution to Dirichlet problem Fatou}
\sum_{i\in \N}\rho(Q_i)^2 \leq \liminf_{n\to\infty} \sum_{i\in \N} \osc_{Q_i}(g_n)^2 =D.
\end{align}
If we show that $\rho(Q_i)$ corresponds to the oscillation of an admissible function $g$, then this will be the desired minimizer.

Applying the Banach-Alaoglu theorem, we assume that $\{M_{Q_i}( g_n)\diam(Q_i)\}_{i\in \N}$ and $\{\osc_{Q_i}( g_n)\}_{i\in \N}$ converge weakly in $\ell^2(\N)$ to $\{M_{Q_i}\diam(Q_i)\}_{i\in \N}$ and $\{\rho(Q_i)\}_{i\in \N}$, respectively. Since $\ell^2(\N)\times \ell^2(\N)$ is reflexive, by Mazur's lemma (see e.g.\ \cite[Theorem 2, p.~120]{Yosida:functional}) we have that there exist convex combinations 
\begin{align*}
M_{Q_i}^n\coloneqq  \sum_{j=1}^n \lambda_j^n M_{Q_i}(g_j), \quad \rho^n(Q_i)\coloneqq  \sum_{j=1}^n \lambda_j^n \osc_{Q_i} ( g_j)
\end{align*}
such that the sequences $\{M_{Q_i}^n\diam(Q_i)\}_{i\in \N}$ and $\{\rho^n(Q_i)\}_{i\in \N}$ converge strongly in $\ell^2(\N)$ to $\{M_{Q_i}\diam(Q_i)\}_{i\in \N}$ and $\{\rho(Q_i)\}_{i\in \N}$, respectively. 

We now show that $\{M_{Q_i}\}_{i\in \N}$ defines a discrete Sobolev function in the sense of Definition \ref{harmonic:3-Discrete Sobolev Definition} with upper gradient $\{\rho(Q_i)\}_{i\in \N}$. Consider $\mathcal G$ to be the family of curves that are good curves for the functions $ g_n$, $n\in \N$, and $f$, and also the boundary values of $g_n$ along paths $\gamma\in \mathcal G$ are equal to $f$ for $n\in \N$. Moreover, we assume that the paths of $\mathcal G$ are non-exceptional for Fuglede's Lemma \ref{harmonic:Fuglede}, applied to a subsequence of $\{\rho^n(Q_i)\}_{i\in \N}$, which we still denote in the same way. For peripheral disks $Q_{i_1},Q_{i_2}$ that intersect a curve $\gamma\in \mathcal G$ we have
\begin{align}\label{harmonic:3-Solution to Dirichlet problem Upper gradient M^n}
\begin{aligned}
|M_{Q_{i_1}}^n- M_{Q_{i_2}}^n|&\leq \sum_{j=1}^n \lambda_j^n |M_{Q_{i_1}}( g_j)-M_{Q_{i_2}}( g_j)|\\
&\leq \sum_{j=1}^n \lambda_j^n \sum_{i:Q_i\cap \gamma\neq \emptyset}\osc_{Q_i} (g_j)\\
&= \sum_{i:Q_i\cap \gamma\neq \emptyset} \rho^n(Q_i).
\end{aligned}
\end{align} 
Taking limits we obtain 
\begin{align*}
|M_{Q_{i_1}}-M_{Q_{i_2}}|\leq \sum_{i:Q_i\cap \gamma\neq \emptyset} \rho(Q_i).
\end{align*}
Hence, $\{M_{Q_i}\}_{i\in \N}$ is indeed a discrete Sobolev function having $\{\rho(Q_i)\}_{i\in \N}$ as an upper gradient.

By the discussion in Section \ref{harmonic:Subsection-Discrete Sobolev}, the discrete Sobolev function $\{M_{Q_i}\}_{i\in \N}$ yields a  Sobolev function $g$ with upper gradient $\{\osc_{Q_i}(g)\}_{i\in \N}$ that satisfies $\osc_{Q_i}(g)\leq \rho(Q_i)$ for all $i\in \N$; see Corollary \ref{harmonic:3-Finiteness of f}. Combining this with \eqref{harmonic:3-Solution to Dirichlet problem Fatou}, we see that $D_\Omega(g)\leq D$. If we prove that $g$ is admissible for the Dirichlet problem, then we will have
\begin{align*}
D_\Omega(g)=\sum_{i\in \N} \osc_{Q_i}(g)^2=\sum_{i\in \N}\rho(Q_i)^2=D,
\end{align*}
and in particular $\osc_{Q_i}(g)=\rho(Q_i)$ for all $i\in \N$.

It remains to show that $g$ is admissible for the Dirichlet problem. For this, it suffices to show that there exists a path family $\mathcal G_0$ that contains almost every path, such that for all points $x\in \partial \Omega$ which are ``accessible" by paths in $\mathcal G_0$ we have $g(x)=f(x)$. Let $\mathcal G_0$ be the path family that contains all paths $\gamma \in \mathcal G$ for which
\begin{align*}
\sum_{i:Q_i\cap \gamma\neq \emptyset}\rho(Q_i)<\infty.
\end{align*}
Note that the complement of $\mathcal G_0$ has weak (strong) modulus zero. If $x\in \partial \Omega$ is ``accessible" by a path $\gamma\in \mathcal G_0$, and $Q_{i_1},Q_{i_2}$ intersect $\gamma$, by \eqref{harmonic:3-Solution to Dirichlet problem Upper gradient M^n} we have
\begin{align*}
|M_{Q_{i_1}}^n -M_{Q_{i_2}}^n |\leq \sum_{i:Q_i\cap \gamma\neq \emptyset} \rho^n(Q_i).
\end{align*}
As we let $Q_{i_2} \to x$, the quantity $M_{Q_{i_2}}^n = \sum_{j=1}^n \lambda_j^n M_{Q_{i_2}}( g_j)$ converges to $f(x)$, because each term $M_{Q_{i_2}}( g_j)$ does so (recall that $\gamma$ is non-exceptional for each $g_j$, and they are admissible). Hence, we have
\begin{align*}
|M_{Q_{i_1}}^n -f(x) |\leq \sum_{i:Q_i\cap \gamma\neq \emptyset} \rho^n(Q_i).
\end{align*}
Now we let $n\to \infty$, and using Fuglede's Lemma \ref{harmonic:Fuglede} we obtain
\begin{align*}
|M_{Q_{i_1}}-f(x)|\leq \sum_{i:Q_i\cap \gamma\neq \emptyset} \rho(Q_i).
\end{align*}
We claim that $M_{Q_{i_1}}\to g(x)$ as $Q_{i_1}\to x$ and $Q_{i_1}\cap \gamma \neq \emptyset$. We assume this for the moment, and we have
\begin{align*}
|g(x)-f(x)|\leq \sum_{i:Q_i\cap \gamma\neq \emptyset} \rho(Q_i).
\end{align*}
This is also true for all subpaths of $\gamma$ landing at $x$, so if we shrink $\gamma$ to $x$, we obtain $g(x)=f(x)$, as desired.

To prove our claim, we note that, by the definition of boundary values, $g(x)$ can be approximated by $M_{Q_{i_1}}(g)=\sup_{x\in \partial Q_{i_1}}(g)$, as $Q_{i_1}\to x$ and $Q_{i_1}\cap \gamma \neq \emptyset$. On the other hand, by the last inequality in Corollary \ref{harmonic:3-Finiteness of f}, we have
\begin{align*}
| M_{Q_{i_1}} (g)-M_{Q_{i_1}}| \leq \rho(Q_{i_1}).
\end{align*}
Since $\rho$ is square-summable, as $Q_{i_1}\to x$ we have $\rho(Q_{i_1}) \to 0$. Hence, our claim is proved.
\end{proof}

\begin{remark}\label{harmonic:Dirichlet Normalized}
By the discussion in Section \ref{harmonic:Subsection Sobolev spaces}, there exists a normalized version $\tilde g$ of the solution $g$ to the Dirichlet problem. For the normalized version we have $\osc_{Q_i} (\tilde g)\leq \osc_{Q_i}(g)$ for all $i\in \N$ by Lemma \ref{harmonic:Normalized main lemma}(iii). Hence, $D_\Omega(\tilde g)\leq D_\Omega(g)$. If $\tilde g$ has boundary values equal to $f$, then we will have that $\tilde g$ is admissible, hence $D_{\Omega}(\tilde g)=D_{\Omega}(g)$, and $\tilde g$ is also a solution to the Dirichlet problem. 

However, $\tilde g$ agrees with $g$ at all points which are ``accessible" by paths $\gamma\in \mathcal G_g$ by Lemma \ref{harmonic:Normalized main lemma}(ii). This also holds for ``accessible" boundary points. Hence, indeed $\tilde g$ has boundary values equal to $f$.
\end{remark}

\section{Properties of harmonic functions}\label{harmonic:Section Properties of Harmonic}
To simplify the treatment, we drop the terminology weak/strong for the Sobolev functions and carpet-harmonic functions. We also drop the subscripts $w,s$ for the Sobolev spaces, so e.g., the non-local Sobolev space will be denoted by $\mathcal W^{1,2}(S)$. First we record a lemma that allows us to switch to the normalized version (see Section \ref{harmonic:Subsection Sobolev spaces}) of a harmonic function:

\begin{lemma}\label{harmonic:Switch to normalized}
Let $u:S\to \widehat{\R}$ be a carpet-harmonic function. Then its normalized version $\tilde u$ is also carpet-harmonic.
\end{lemma}
\begin{proof}
It suffices to prove that for each open set $V\subset \subset \Omega$ and each  function $\zeta\in \mathcal W^{1,2}(S)$ with $\zeta\big|_{S\setminus V}\equiv 0$ we have $D_V(\tilde u)\leq D_V(\tilde u+\zeta)$. We fix such a function $\zeta$. Recall that $\osc_{Q_i}(\tilde u)\leq \osc_{Q_i}(u)$ for all $i\in \N$, by Lemma \ref{harmonic:Normalized main lemma}(iii). Hence, $D_V(\tilde u)\leq D_V(u)$. On the other hand, by Lemma \ref{harmonic:Normalized main lemma}(ii) we have $\tilde u(x)=u(x)$ for all points $x\in S$ that are ``accessible" by a curve family that contains almost every curve. Hence, by Lemma \ref{harmonic:Removability3} and linearity, for each open set $W\subset \Omega$ the function $\eta\coloneqq (\tilde u-u)\x_{S\cap W}+\zeta$ lies in the Sobolev space $\mathcal W^{1,2}(S)$.

First, assume that $\partial Q_i\subset V$ whenever $i\in I_V$. We set $W=V$ and consider the function $\eta$ as above. Then $u+\eta= \tilde u+\zeta$ on $\partial Q_i \subset V$ for $i\in I_V$. Since $\eta$ vanishes outside $V$, by the harmonicity of $u$ we have
\begin{align*}
D_V(u)\leq D_V(u+\eta)=D_V(\tilde u+\zeta).
\end{align*}
Summarizing, $D_V(\tilde u)\leq D_V(\tilde u+\zeta)$, as desired.

Now, we treat the general case. We fix $\varepsilon>0$ and for each $i\in I_V$ we consider a number $\delta_i=\delta_i(\varepsilon)\in (0,\varepsilon)$ such that the open $\delta_i$-neighborhood of $Q_i$ intersects only peripheral disks having smaller diameter than that of $Q_i$; recall from Lemma \ref{harmonic:Fatness consequence} that the diameters of the peripheral disks shrink to $0$ in compact subsets of $\R^2$. We denote by $W_\varepsilon$ the union of $V$ with all these neighborhoods and $\eta$ is defined as before with $W=W_\varepsilon$. Note that $W_\varepsilon$ contains $\partial Q_i$, whenever $i\in I_V$. Therefore, as $\eta$ vanishes outside $W_\varepsilon\supset V$, we have
\begin{align*}
D_V(u)\leq D_{W_\varepsilon}(u) \leq D_{W_\varepsilon}(u+\eta) &=D_{W_\varepsilon}(\tilde u\x_{S\cap W_\varepsilon }+ u\x_{S\setminus W_\varepsilon}+\zeta)\\
&=D_{V}(\tilde u +\zeta) +\sum_{i\in I_{W_\varepsilon}\setminus I_V} \osc_{Q_i}( \tilde u\x_{S\cap W_\varepsilon}+ u\x_{S\setminus W_\varepsilon})^2\\
&\leq D_{V}(\tilde u +\zeta) +\sum_{i\in I_{W_\varepsilon}\setminus I_V} (\osc_{Q_i}(\tilde u) + \osc_{Q_i}(u))^2,
\end{align*}  
where we used the oscillation inequalities from Lemma \ref{harmonic:Removability3}. Since the last sum is finite, it will converge to $0$ and we will have the desired conclusion, provided that the set $I_{W_\varepsilon}$ shrinks to  $I_V$ as $\varepsilon\to 0$.

We argue by contradiction, assuming that there exists $i_0\notin I_V$ that lies in $I_{W_\varepsilon}$ infinitely often as $\varepsilon\to 0$, say along a sequence $\varepsilon_n\to 0$. For each $n\in \N$, there exists $i(n)\in I_V$ such that $Q_{i_0}$  intersects the $\delta_{i(n)}(\varepsilon_n)$-neighborhood of  $Q_{i(n)}$. Note that the set $\{i(n):n\in \N\}$ cannot be finite, since $\delta_{i(n)}(\varepsilon_n)\to 0$ as $n\to \infty$, and $Q_{i_0}$ has positive distance from each individual peripheral disk. Therefore, the set $\{i(n):n\in \N\}$ is infinite, and by passing to a subsequence we may assume that the indices $i(n)$ are distinct. However, $\delta_{i(n)}(\varepsilon_n)$ was chosen so that the $\delta_{i(n)}(\varepsilon_n)$-neighborhood of $Q_{i(n)}$ intersects only smaller peripheral disks than $Q_{i(n)}$. Since $\diam(Q_{i (n)})\to 0$ (by Lemma \ref{harmonic:Fatness consequence} since they all stay near $Q_{i_0}$), it follows that $Q_{i_0}$ cannot intersect these neighborhoods for large $n$. This is a contradiction.    
\end{proof}

Recall also by Remark \ref{harmonic:Dirichlet Normalized} that  the normalized version of the solution to the Dirichlet problem (Theorem \ref{harmonic:4-Thm Solution to Dirichlet problem}) has the same boundary values as the original solution. In what follows, we will always be using normalized versions of carpet-harmonic functions. In particular, by Lemma \ref{harmonic:3-Normalized Property} we may assume that for any $x\in S$ the value of $u(x)$ can be approximated by $M_{Q_i}(u)$ where $Q_i$ is a peripheral disk close to $x$.

\subsection{Continuity and maximum principle}

\begin{lemma}\label{harmonic:4-zero oscillation lemma}
Let $u\colon S\to \widehat \R$ be a normalized Sobolev function, and $V\subset  \Omega$ a connected open set.
\begin{enumerate}[\upshape (a)]
\item If $\osc_{Q_i}(u)=0$ for all $i\in I_V$, then $u$ is constant on $S\cap V$.
\item If $\osc_{Q_i}(u)=0$ for all $Q_i\subset V$, then $u$ is a constant on $S\cap W$, where $W$ is a component of $V\setminus \br{\bigcup_{i\in I_{\partial V}} {Q_i}}$.
\end{enumerate} 
\end{lemma}
\begin{proof}
(a) Note that $u$ is constant on (the boundary of) any given peripheral disk that intersects $V$. Thus we may assume that $V$ does not intersect only one peripheral disk. Let $i_1,i_2\in I_V$ be distinct. Since $V$ is path connected, by Lemma \ref{harmonic:Paths joining continua} we can find a non-exceptional path $\gamma\subset V$ joining $Q_{i_1},Q_{i_2}$ such that the upper gradient inequality holds along $\gamma$. If $x\in \partial Q_{i_1}\cap \gamma$ and $y\in \partial Q_{i_2}\cap \gamma$ then
\begin{align*}
|u(x)-u(y)|\leq \sum_{i:Q_i\cap \gamma\neq \emptyset}\osc_{Q_i}(u)=0.
\end{align*}
Thus $u$ is equal to the same constant $c$ on all peripheral circles $\partial Q_i$, $i\in I_V$. If $x\in S\cap V$ is arbitrary, then $u(x)$ can be approximated by $M_{Q_i}(u)$, where $Q_i\cap V\neq \emptyset$, thus $u(x)=c$ also here. 

(b) The set $V\setminus \br{\bigcup_{i\in I_{\partial V}} {Q_i}}$ is open. We fix a component $W$ of this set and observe that if $Q_i\cap W\neq \emptyset$ then $Q_i\subset W$; see the remark below. This implies that $I_W=\{i\in \N: Q_i\subset W\}$. Thus the conclusion follows immediately by an application of part (a) of the lemma.
\end{proof}

\begin{remark}\label{harmonic:Properties of W}
(a) Each component $W$ of $V\setminus \br{\bigcup_{i\in I_{\partial V}} {Q_i}}$ has the property that it contains all peripheral disks that it intersects. Moreover, if $S\cap W\neq \emptyset$ and $\partial_*V\neq \emptyset$, then we have $\br {S\cap W}\cap \partial_*V\neq \emptyset$. 

For the first claim, note that $E\coloneqq \br{\bigcup_{i\in I_{\partial V}} Q_i}$ contains all the peripheral disks that it intersects. This implies that if $Q_i\cap W\neq \emptyset$, then we necessarily have $Q_i\cap E=\emptyset$ and $Q_i\cap \partial V=\emptyset$. Hence, $Q_i\subset V\setminus E$ and $Q_i\subset W$ by the connectedness of $Q_i$.

For the second claim, by Lemma \ref{harmonic:Paths in S^o}, there exists an open path $\gamma\subset S^\circ$ connecting a point $x$ of $S \cap W$ to a point outside $V$, for example to a point of $\partial_*V$. Let $y\in \gamma\cap \partial V$ be the first point of $\partial V$ that $\gamma$ meets, assuming that it  is parametrized to start at $x$. We claim that $y\in \br {S\cap W}$. We consider the (smallest) open subpath of $\gamma$ that connects $x$ to $y$, and we still denote it by $\gamma$. Then $\gamma\subset V$, and also $\gamma\cap \br{\bigcup_{i\in I_{\partial V}} {Q_i}} =\emptyset$. Indeed, if $z\in \gamma\subset S^\circ$ is a limit point of ${\bigcup_{i\in I_{\partial V}} {Q_i}}$, then there exists a  sequence of $Q_i$, $i\in I_{\partial V}$, with diameters shrinking to $0$ and with $Q_i\to z$. This would imply that $z\in \partial V$, a contradiction. Hence, $\gamma \subset V\setminus \br{\bigcup_{i\in I_{\partial V}} {Q_i}} $, and in fact $\gamma\subset S\cap W$, which shows that $y\in \br \gamma\subset \br {S\cap W}$.

(b) The assumption $\partial_*V\neq \emptyset$ in the previous statement holds always, unless $V \supset S$ or $\C\setminus \br V\supset S$. Indeed, if $V$ is an open set and $S\setminus V\neq \emptyset$, $S\cap \br V\neq \emptyset$, then we can connect a point of $S\cap \br V$ to a point of $S\setminus V$ with an open path in $S^\circ$, by Lemma \ref{harmonic:Paths in S^o}. This path necessarily hits $\partial V\cap S=\partial_*V$.
\end{remark}

\begin{theorem}\label{harmonic:4-continuous}
Let $u\colon S\to \widehat \R$ be a carpet-harmonic function. Then $u$ is continuous.
\end{theorem}
\begin{proof}
Let $x\in S^\circ$. If $\osc_{Q_i}(u)=0$ for all $Q_i$ contained in a ball $B(x,r)$ then there exists $r'<r$ such that $\osc_{Q_i}(u)=0$ for all $Q_i$ intersecting the ball $B(x,r')$. This follows from Lemma \ref{harmonic:Fatness consequence} and the fact that no peripheral disk can intersect a ball $B(x,r')$ for arbitrarily small $r'>0$. Applying the previous lemma, we conclude that $u$ is constant in $B(x,r')\cap S$, so it is trivially continuous. 

We assume that arbitrarily close to $x$ there exists some $Q_i$ with $\osc_{Q_i}(u)>0$. Consider a circular arc $\gamma_r$ as in Lemma \ref{harmonic:3-curves gamma_r}(a) with $\sum_{i:Q_{i}\cap \gamma_r\neq \emptyset }\osc_{Q_i}(u)<\varepsilon$ and $B(x,r)\subset \Omega$. Also, let $y\in B(x,r)\cap S$. Since $u$ is normalized, there exist peripheral disks $Q_{i_x},Q_{i_y} \subset B(x,r)$ such that $|u(x)-M_{Q_{i_x}}(u)|<\varepsilon$ and $|u(y)-M_{Q_{i_y}}(u)|<\varepsilon$, so it suffices to show that $|M_{Q_{i_x}}(u)-M_{Q_{i_y}}(u)|$ is small. 

Since $\gamma_r$ is non-exceptional, the upper gradient inequality implies that the number $M\coloneqq \sup_{z\in \partial_*B(x,r)}u(z)= \sup_{z\in S\cap \gamma_r}u(z)$ is finite. We claim that $M_{Q_k}(u) \leq M $, for all $Q_k\subset B(x,r)$. Consider the function $h\coloneqq  u\cdot \x_{S\setminus B(x,r)} + u\land M \cdot \x_{S\cap B(x,r)}$. By Corollary \ref{harmonic:Removability4}, it follows that $h$ is a Sobolev function and $\osc_{Q_i}(h)\leq \osc_{Q_i}(u)$ for all $i\in \N$. Therefore, for the Dirichlet energy we  have $D_{B(x,r)}(h)\leq D_{B(x,r)}(u)$. 

Assume now that there exists some $Q_k\subset B(x,r)$ with $M_{Q_k}(u)>M$. If $\osc_{Q_k}(u)>0$, then it is easy to see that $\osc_{Q_k}(h)< \osc_{Q_k}(u)$, which implies that $D_{B(x,r)}(h)<D_{B(x,r)}(u)$, a contradiction to harmonicity. If $\osc_{Q_k}(u)=0$, then consider a good path $\gamma\subset B(x,r)$, given by Lemma \ref{harmonic:Paths joining continua}, that joins $Q_k$ to some $Q_l\subset B(x,r)$ with $\osc_{Q_l}(u)>0$. Using the upper gradient inequality one can then find a peripheral disk $Q_m\subset B(x,r)$ that intersects $\gamma$, such that $\osc_{Q_m}(u)>0$ and $M_{Q_m}(u)$ is arbitrarily close to $M_{Q_k}(u)$, so, in particular, $M_{Q_m}(u)>M$. By the previous case we obtain a contradiction.

With a similar argument, one shows that $\inf_{z\in \partial_*B(x,r)}u(z)\leq M_{Q_k}(u) $ for all $Q_k\subset B(x,r)$. Therefore by the upper gradient inequality we have 
\begin{align*}
|M_{Q_{i_x}}(u)-M_{Q_{i_y}}(u)| &\leq \sup_{z\in \partial_*B(x,r)}u(z)-\inf_{z\in \partial_*B(x,r)}u(z)\\
&\leq \sum_{i:Q_i\cap \gamma_r\neq \emptyset}\osc_{Q_i}(u)<\varepsilon,
\end{align*}
as desired.
 
Now, we treat the case $x\in \partial Q_{i_0}$ for some $i_0\in \N$. Consider a small ball $B(x,r)$ with $\partial B(x,r)\cap Q_{i_0}\neq \emptyset$. If $\osc_{Q_i}(u)=0$ for all $Q_i$ contained in $B\coloneqq B(x,r)$, then from Lemma \ref{harmonic:4-continuous}(b) for the component $W$ of $B(x,r) \setminus  \br{\bigcup_{i\in I_{\partial B}}{Q_i}}$ that contains $x$ in its boundary we have that $u\big|_{S\cap W}$ is a constant $c$. In fact $\partial Q_{i_0}\cap \partial W$ contains a non-trivial arc $\alpha$ that contains $x$ in its interior. Since $u$ is normalized, Lemma \ref{harmonic:3-Normalized Property} implies that the value of $u(y)$ for $y\in \alpha$ is approximated by $M_{Q_i}(u)=c$, where $Q_i\subset W$. This shows that $u\equiv c$ in a neighborhood of $x$, and thus, $u$ is continuous at $x$. 

Now, we assume that arbitrarily close to $x$ there exists some $Q_i$, $i\neq i_0$, with $\osc_{Q_i}(u)>0$. For a small $\varepsilon>0$ we apply again Lemma \ref{harmonic:3-curves gamma_r}(a) to obtain a circular arc $\gamma_r$ around $x$ such that 
\begin{align*}
\sum_{\substack{i:Q_i\cap \gamma_r\neq \emptyset \\i\neq i_0}}\osc_{Q_i}(u)<\varepsilon.
\end{align*}
As in the proof of  Lemma \ref{harmonic:3-well-defined} (see also Figure \ref{harmonic:fig:crosscut}), there exists a (closed) subarc $\gamma_r'$ of $\gamma_r$ with endpoints on $\partial Q_{i_0}$ such that $\gamma_r'\cap Q_{i_0}=\emptyset$ and $\gamma_r'$ defines a crosscut that separates $x$ from $\infty$ in $\R^2\setminus Q_{i_0}$. We consider an arc $\beta\subset \br Q_{i_0}$ whose endpoints are the endpoints of $\gamma_r'$, but otherwise it is contained in $Q_{i_0}$. Then $\beta\cup  \gamma_r'$ bounds a Jordan region $V$ that contains $x$ in its interior.

With a similar variational argument as in the case $x\in S^\circ$ we will show that for each $Q_k\subset V$ we have
\begin{align*}
\inf_{z\in S\cap \gamma_r'}u(z)\leq M_{Q_k}(u) \leq \sup_{z\in S\cap \gamma_r'}u(z).
\end{align*} 
Then continuity will follow as before, because $\sum_{i:Q_i\cap \gamma_r'\neq \emptyset} \osc(Q_i) <\varepsilon$.

We sketch the proof of the right inequality. Let $M\coloneqq \sup_{z\in S\cap \gamma_r'} u(z)$ (which is finite by the upper gradient inequality for the good path $\gamma_r'$), and consider the function $h=u\cdot \x_{S \setminus V}+ u\land M \cdot \x_{S\cap V}$. By Corollary \ref{harmonic:Removability4} the function $h$ is a Sobolev function which agrees with $u$ outside $V$ and $\osc_{Q_i}(h)\leq \osc_{Q_i}(u)$ for all $i\in \N$. Now, if there exists $Q_k\subset V$ with $M_{Q_k}(u)>M$ we derive a contradiction as in the previous case.
\end{proof}

The continuity implies, in particular, that $|u(x)|<\infty$ for every $x\in S$.

\begin{theorem}[Maximum principle]\label{harmonic:Maximum Principle}\index{maximum principle}
Let $u\colon S\to \R$ be a carpet-harmonic function and $V\subset \subset\Omega$ be an open set. Then 
\begin{align*}
\sup_{x\in S\cap \br{V}} u(x)= \sup_{x\in \partial_* V} u(x)\quad  \textrm{and} \quad \inf_{x\in S\cap \br{V}} u(x)= \inf_{x\in \partial_* V} u(x).
\end{align*}
\end{theorem}

Note that if $S\cap \br V\neq \emptyset$, then $\partial_*V \neq \emptyset$. This is because $V\subset \subset \Omega$; recall Remark \ref{harmonic:Properties of W}(b). Hence, all quantities above are defined, or they are simultaneously vacant. Moreover, by the continuity of $u$ all quantities are finite. 

\begin{proof}
We clearly have $\sup_{x\in S\cap \br{V}} u(x)\geq \sup_{x\in \partial_* V} u(x)\eqqcolon M$ because $S\cap \br V\supset \partial_*V$. 

Assume that there exists $x\in S\cap V$ such that $u(x)>M$. Since $S^\circ$ is dense in $S$ (this follows e.g.\ by Lemma \ref{harmonic:Paths in S^o}) and $u$ is continuous, we may assume that there exists $x\in S^\circ \cap V$ such that $u(x)>M$. Note that $x$ cannot lie in $ \br{ \bigcup_{i\in I_{\partial V}} {Q_i}}$ since all peripheral disks contained in a small neighborhood of $x$ have to lie in $V$. Let $W$ be the component of $V\setminus \br{ \bigcup_{i\in I_{\partial V}} {Q_i}}$ that contains $x$. By Remark \ref{harmonic:Properties of W}(a), we have $\br {S\cap W}\cap \partial_* V \neq \emptyset$ and  $W$ contains all peripheral disks that it intersects. 

If $\osc_{Q_i}(u)=0$ for all $Q_i\subset W$, then $u$ is constant in $S\cap W$ by Lemma \ref{harmonic:4-zero oscillation lemma}(a) and by continuity it is also constant and equal to $u(x)$ on $\br {S\cap  W}$. This contradicts the fact that $\br {S\cap W} \cap \partial_*V \neq \emptyset$, and that $u\leq M$ on $\partial_*V$.

Hence, there exists some $Q_i\subset W$ with $\osc_{Q_i}(u)>0$. Arbitrarily close to $x$ we can find a peripheral disk $Q_{i_x}$ with $M_{Q_{i_x}}(u)>M$. Arguing as in the proof of Theorem \ref{harmonic:4-continuous}, we can derive that there exists some $Q_{i_0}\subset W$ with $M_{Q_{i_0}}(u)>M$ and $\osc_{Q_{i_0}}(u)>0$. Consider the variation $h= u\cdot \x_{S\setminus V} + u\land M \cdot \x_{S\cap V}$ and note that $u\leq M$ on $\partial_* V$. By Corollary \ref{harmonic:Removability4} $h$ is a Sobolev function with $\osc_{Q_i}(h) \leq \osc_{Q_i}(u)$ for all $i\in \N$. However $\osc_{Q_{i_0}}(h)<\osc_{Q_{i_0}}(u)$ which contradicts the minimizing property of $u$.

The claim for the infimum follows by looking at $-u$. 
\end{proof}

\subsection{Uniqueness and comparison principle}

Here, we first recover the uniqueness part in Theorem \ref{harmonic:4-Thm Solution to Dirichlet problem}, and then a comparison principle for solutions to the Dirichlet problem. The standing assumption here is that the set $\Omega$ has boundary $\partial \Omega$ that consists of finitely many, non-trivial, and disjoint Jordan curves, so that we can define boundary values of Sobolev functions. 

\begin{theorem}[Uniqueness]\label{harmonic:Uniqueness}\index{Dirichlet problem!uniqueness}
Let $u,v\colon S \to \R$ be solutions to the Dirichlet problem given by Theorem \ref{harmonic:4-Thm Solution to Dirichlet problem} with boundary values equal to $f$ on $\partial \Omega$. Then $u=v$ on $S$.
\end{theorem}

\begin{proof}
Since both $u,v$ are solutions to the Dirichlet problem, it follows that $D\coloneqq D_{\Omega}(u)=D_{\Omega}(v)$. Recall that a function $g$ is admissible for the Dirichlet problem if $g\in \mathcal W^{1,2}(S)$ and $g$ has boundary values equal to $f$. 

For $s\in [0,1]$ the function $(1-s)u+sv$ is admissible, thus by the subadditivity of $\osc_{Q_i}(\cdot)$ (see e.g.\ the proof of Proposition \ref{harmonic:3-Properties Sobolev}) and the Cauchy-Schwarz inequality we have
\begin{align*}
D&\leq D_{\Omega}((1-s)u+sv) \\
&\leq (1-s)^2D_\Omega(u) +2s(1-s)\sum_{i\in \N} \osc_{Q_i}(u)\osc_{Q_i}(v) +s^2 D_\Omega(v)\\
&\leq  (1-s)^2 D+2s(1-s) D^{1/2}D^{1/2}+s^2D\\
&=D.
\end{align*}
Since we have equality, it follows that $\osc_{Q_i}(u)=\osc_{Q_i}(v)$ for all $i\in \N$. 

Consider the function $g=u\lor v$ which is a Sobolev function with $\osc_{Q_i}(g)\leq \osc_{Q_i}(u)$, by Proposition \ref{harmonic:3-Properties Sobolev}\ref{harmonic:3-Properties Sobolev-max}. Also, $g(x)=f(x)$ for all ``accessible" points $x\in \partial \Omega$, thus $g$ is admissible for the Dirichlet problem with boundary data $f$. It follows that
\begin{align*}
D\leq D_\Omega(g)=\sum_{i\in \N} \osc_{Q_i}(g)^2 \leq \sum_{i\in \N} \osc_{Q_i}(u)^2=D.
\end{align*}
This implies that $\osc_{Q_i}(g)=\osc_{Q_i}(u)$ for all $i\in \N$, and $g$ is also carpet-harmonic on $\Omega$, since it minimizes the Dirichlet energy. 

We assume that there exists $x_0\in \Omega$ such that $u(x_0)<v(x_0)$. Then $u(x_0)<g(x_0)$, and there exists a level $\Lambda\in \R$ such that $u(x_0)<\Lambda<g(x_0)$. We define the function
\begin{align*}
h= \begin{cases} g,& g\leq \Lambda \\  \Lambda,& u<\Lambda<g \\ u,& u\geq \Lambda .\end{cases}
\end{align*}
It is immediate to see that $h=(u\lor \Lambda)\land g$. Proposition \ref{harmonic:3-Properties Sobolev}\ref{harmonic:3-Properties Sobolev-max},\ref{harmonic:3-Properties Sobolev-min}  shows that $h$ is a Sobolev function with  $\osc_{Q_i}(h) \leq \osc_{Q_i}(u)=\osc_{Q_i}(g)$ for all $i\in \N$. It is also clear that $h$ is admissible for the Dirichlet problem on $\Omega$ with boundary data equal to $f$. It thus follows that $h$ is also carpet-harmonic and in fact $\osc_{Q_i}(h)=\osc_{Q_i}(u)$ for all $i\in \N$.

If the closure of $\{u<\Lambda <g\}$ relative to $S$ is the entire carpet $S$, then $h\equiv \Lambda$ is constant and $\osc_{Q_i}(h)=\osc_{Q_i}(u)=\osc_{Q_i}(v)=0$. Lemma \ref{harmonic:4-zero oscillation lemma}(a) implies that $u,v$ are constants, but then they cannot have the same boundary values, unless $u\equiv v$. This contradicts the assumption that $u(x_0)<v(x_0)$. Hence, we assume that there exists a point of $S$ that does not lie in $\{u<\Lambda <g\}$.

We will show that there exists a peripheral disk $Q_{i_0} \subset Z\coloneqq \{u<\Lambda<g\}$ with $\osc_{Q_{i_0}}(u)>0$. However, $h$ is constant in $S\cap Z$, thus $\osc_{Q_{i_0}}(h)=0$, which is again a contradiction, because $\osc_{Q_{i_0}}(h)=\osc_{Q_{i_0}}(u)$.

To prove our claim, note first that by the continuity of the carpet-harmonic functions $u,g$ the set $Z$ is the intersection of an open set $V$ in the plane with $S$, and $Z$ is non-empty, since it contains $x_0$. Since $S\setminus V=S\setminus Z\neq \emptyset$ and $S\cap V\neq \emptyset$, we have $\partial_{*}V\neq \emptyset$; see Remark \ref{harmonic:Properties of W}(b). Let $W$ be the component of $V\setminus \br{\bigcup_{i\in I_{\partial V}} {Q_i}}$ that contains $x_0$. Then $\br {S\cap W} \cap \partial_*V\neq \emptyset$, by Remark \ref{harmonic:Properties of W}(a). If $\osc_{Q_i}(u)=0$ for all $Q_i\subset W$, then $u$ is constant in $S\cap W$ by Lemma \ref{harmonic:4-zero oscillation lemma}(a) and by continuity it is also constant on $\br{S\cap W}\cap \partial_*V$. In particular, there exists a point  $z_0\in \partial_*V\subset \partial V$ with $u(z_0)=u(x_0)<\Lambda$ and $g(z_0)>\Lambda$. Since these inequalities hold in a neighborhood of $z_0$ we obtain a contradiction. 
\end{proof}

\begin{theorem}[Comparison principle]\label{harmonic:Comparison Principle}\index{comparison principle}
Assume that $u,v\colon S \to \R$ are solutions to the Dirichlet problem in $\Omega$ with boundary data $\alpha,\beta$, respectively. We assume that $\alpha(x)\geq \beta(x)$ for points $x\in \partial \Omega$ that are ``accessible" by paths $\gamma\in \mathcal G_0$, where $\mathcal G_0$ is a path family that contains almost every path. Then $u\geq v$ in $S$.  
\end{theorem}
\begin{proof}
Assume that the conclusion fails, so there exists $x_0\in S$ with $u(x_0)<v(x_0)$. Let $f= u\lor v$ which is a Sobolev function with boundary values $\alpha$ on $\partial \Omega$. Thus, $f$ is admissible for the Dirichlet problem on $\Omega$ with boundary values $\alpha$, so $D_{\Omega}(u)\leq D_\Omega(f)$. By the uniqueness of solutions to the Dirichlet problem in Theorem \ref{harmonic:Uniqueness}, it follows that
\begin{align}\label{harmonic:Comparison principle inequality}
D_\Omega(u)<D_\Omega(f).
\end{align}

Similarly, consider $g\coloneqq  u\land v $ which is admissible for the Dirichlet problem on $\Omega$ with boundary values $\beta$. As before, this implies that $D_\Omega(v)<D_\Omega(g)$. Adding this to \eqref{harmonic:Comparison principle inequality}, we obtain
\begin{align*}
\sum_{i\in \N}(\osc_{Q_i}(u)^2+ \osc_{Q_i}(v)^2) < \sum_{i\in \N} (\osc_{Q_i}(f)^2 +\osc_{Q_i}(g)^2).
\end{align*}
This, however, contradicts \eqref{harmonic:3-Properties Sobolev-inequality} in Proposition \ref{harmonic:3-Properties Sobolev}.
\end{proof}

\subsection{Continuous boundary data}
In this section we continue the treatment of the Dirichlet problem, proving that the solutions are continuous up to the boundary, if the boundary data is continuous. Here we assume, as usual, that the boundary $\partial \Omega$ consists of finitely many, non-trivial, and disjoint Jordan curves.

\begin{theorem}\label{harmonic:Continuous boundary data}
Assume that $u\colon S\to \R$ is the solution to the Dirichlet problem in $\Omega$ with continuous boundary data $f\colon \partial \Omega\to \R$. Then $u$ can be extended continuously to $\partial \Omega$. 
\end{theorem}
\begin{proof}
The proof is very similar to the proof of Theorem \ref{harmonic:4-continuous}, and uses, in some sense, a maximum principle near the boundary.

Recall that there exists a path family $\mathcal G_0$ that contains almost every path, such that for all $x\in \partial \Omega$ that are ``accessible" by paths $\gamma\in \mathcal G_0$ we have $u(x)=f(x)$. Furthermore, the fact that the boundary $\partial \Omega$ consists of finitely many Jordan curves implies that every $x \in \partial \Omega$ is the landing point of a path $\gamma \subset \Omega$ (not necessarily in $\mathcal G_0$). Perturbing $\gamma$ as in Lemma \ref{harmonic:Paths boundary} we obtain a point $y\in \partial \Omega$ near $x$ that is ``accessible" by a path $\gamma_0\in \mathcal G$. Hence, $u(y)=f(y)$ and this actually holds for a dense set of points in $\partial \Omega$. 

We fix a point $x_0 \in \partial \Omega$ and we wish to show that $u$ can extended at $x_0$ by $u(x_0)=f(x_0)$, so that $u\big|_{S\cup\{x_0\}}$ is continuous. If this is true for each $x_0\in \partial \Omega$, then $u$ will be continuous in $S\cup \partial \Omega$ by the continuity of $f$, as desired. 

If $\osc_{Q_i}(u)=0$ for all $Q_i$ contained in a neighborhood of $x_0$, then by Lemma \ref{harmonic:4-zero oscillation lemma}(b) $u$ is a constant $c$ near $x_0$. In particular, there exists an arc $\alpha\subset \partial \Omega$ with $x_0$ lying in the interior of $\alpha$ such that $u(y)=f(y)=c$  for a dense set of points $y\in \alpha$. This implies that $f(x_0)=c$ by continuity, and hence we may define $u(x_0)=c=f(x_0)$.

Now, suppose that arbitrarily close to $x_0$ we can find $Q_i$ with $\osc_{Q_i}(u)>0$. Consider a ball $B(x_0,r)$, where $r>0$ is so small that $B(x_0,r)$ intersects only one boundary component of $\partial \Omega$. The boundary $\partial B(x_0,r)$ defines a crosscut $\gamma_r'\subset \partial B(x_0,r)$, which bounds a region $W\subset \Omega$, together with a subarc $\alpha$ of $\partial \Omega$, such that $x_0 \in \partial W$; see Figure \ref{harmonic:fig:crosscut}. We fix $\varepsilon>0$ and take an even smaller $r$ so that 
\begin{align}\label{harmonic:Continuity on boundary-f inequality}
\sup_{y\in  \alpha}f(y)- \inf_{y\in \alpha} f(y)<\varepsilon
\end{align}
and $\sum_{i:Q_i\cap \gamma_r'\neq \emptyset}\osc_{Q_i}(u)<\varepsilon$, where the path $\gamma_r'\subset \gamma_r=\partial B(x_0,r)$ is non-excep\-tional, as in Lemma \ref{harmonic:3-curves gamma_r}(a). We wish to show that
\begin{align}\label{harmonic:Continuity on boundary-suffices}
|f(x_0)-u(z)|\leq 2\varepsilon
\end{align}
for all $z\in S\cap W$. This will show that $u$ can be continuously extended at $x_0$ by $u(x_0)=f(x_0)$.

Let $M\coloneqq \sup_{y\in \alpha} f(y)$ and $m\coloneqq \inf_{y\in \alpha}f(y)$. We claim that 
$$m-\varepsilon\leq M_{Q_k}(u)\leq M+\varepsilon$$ 
for all $Q_k\subset W$. This will imply that $|f(x_0)- M_{Q_k}(u)|<2\varepsilon$ by \eqref{harmonic:Continuity on boundary-f inequality}, and thus $|f(x_0)-u(z)|\leq 2\varepsilon$ for all $z\in S\cap W$, as desired; here we used the continuity of $u$ and the fact that near $z$ we can find arbitrarily small peripheral disks, and thus peripheral disks $Q_k\subset W$.

Observe that on the arc $\gamma_r'$ by the upper gradient inequality we have $|u(x)-u(y)|\leq \sum_{i:Q_i\cap \gamma_r'}\osc_{Q_i}(u) <\varepsilon$. Since $\gamma_r'$ is non-exceptional, we have $u(y)=f(y)$ for the endpoints of $\gamma_r'$. Hence
\begin{align*}
m-\varepsilon\leq u(x)\leq M+\varepsilon
\end{align*}
for all $x\in S\cap \gamma_r'=\partial_*W$. Consider the function $h=u \cdot \x_{S\setminus W}+ u\land (M+\varepsilon)\cdot \x_{S\cap W}$ which is a Sobolev function by Corollary \ref{harmonic:Removability4} with $\osc_{Q_i}(h)\leq \osc_{Q_i}(u)$ for all $i\in \N$. Note that $h$ is also admissible for the Dirichlet problem with boundary data $f$. If there exists some $Q_k \subset W$ with $M_{Q_k}(u)>M+\varepsilon$, then we can find actually some $Q_k\subset W$ with $M_{Q_k}(u)>M+\varepsilon$ and $\osc_{Q_k}(u)>0$; see proof of Theorem \ref{harmonic:4-continuous}. Then $\osc_{Q_k}(h) <\osc_{Q_k}(u)$ which contradicts the minimizing property of $u$. The inequality $m-\varepsilon\leq M_{Q_k}(u)$ for all $k\in \N$ is shown in the same way.
\end{proof}

So far we have treated the existence, uniqueness, and the comparison principle for solutions to the Dirichlet problem. A natural question that arises is whether every carpet-harmonic function can be realized at least locally as a solution to a Dirichlet problem, so that we can apply these principles. It turns out that this is the case.

\begin{prop}\label{harmonic:Dirichlet-Carpet}
Let $u\colon S\to \R$ be a carpet-harmonic function, and $V\subset \subset \Omega$ be an open set with the properties:
\begin{enumerate}[\upshape (1)]
\item $\partial V $ consists of finitely many, non-trivial, and disjoint Jordan curves, and
\item if $V\cap  Q_i\neq \emptyset$ then $\br Q_i \subset V$. This, in particular, implies that $\partial V\subset S$, and $(S\cap V,V)$ is a relative Sierpi\'nski carpet.
\end{enumerate}
Then $u$ agrees inside $S\cap V$ with the solution to the Dirichlet problem in $S\cap V$ with boundary values on $\partial V$ equal to $u$. 
\end{prop}
\begin{proof}
Let $v\colon S\cap V\to \R$ be the solution to the Dirichlet problem with boundary values equal to $u$, given by Theorem \ref{harmonic:4-Thm Solution to Dirichlet problem}. Since $u$ is continuous, by Theorem \ref{harmonic:Continuous boundary data} we have that $v$ has a continuous extension to $S\cap \br V$ that agrees with $u$ on $\partial V$.

Consider the function $\zeta\coloneqq  (v-u)\cdot \x_{S\cap V}+0\cdot \x_{S\setminus V}$. This is a Sobolev function on $S$, but we cannot apply Lemma \ref{harmonic:3-Lemma-removability} directly, since $v$ is not defined on all of $S$. The fact that $\zeta$ is a Sobolev function on $S$ follows from the following lemma that we prove right after: 
\begin{lemma}\label{harmonic:Removability2}
Let $V \subset\subset  \Omega$ be an open set as above, with $\partial V\subset S$ and $S\cap V$ being a relative Sierpi\'nski carpet. Consider functions $\phi,\psi\colon S\to \R$ such that $\phi\big|_{S\cap V} \in \mathcal W^{1,2}(S\cap V)$, $\psi \in \mathcal W_{\loc}^{1,2}(S)$, and $\phi=\psi $ on $\partial V$. Then $\zeta\coloneqq \phi \cdot \x_{S\cap V} + \psi \cdot \x_{S\setminus V}$ lies in $\mathcal W_{\loc}^{1,2}(S)$.
\end{lemma}
In our case we set $\psi\equiv 0$ and $\phi=(v-u)\cdot \x_{S\cap V}$, which is continuous on $S$. The harmonicity of $u$ implies that 
\begin{align*}
D_{V} (u) \leq D_V(u +\zeta)= D_V(v).
\end{align*}
However, $v$ is a minimizer for the Dirichlet energy in $V$, and this implies that $D_V(u)=D_V(v)$. The uniqueness in Theorem \ref{harmonic:Uniqueness} concludes that $u=v$ on $S\cap V$.
\end{proof}

\begin{proof}[Proof of Lemma \ref{harmonic:Removability2}]
Consider the families of good paths $\mathcal G_{\phi},\mathcal G_{\psi}$ for $\phi,\psi$, respectively. Note that the paths of $\mathcal G_{\phi}$ are contained in $V$. Let $\Gamma_0$ be the paths of $\mathcal G_{\psi}$ that have a subpath in $V$ which is not contained in $\mathcal G_{\phi}$. Then one can show that $\Gamma_0$ has weak (strong) modulus equal to zero with respect to the carpet $S$; see also Remark \ref{harmonic:3-remark-subpaths modulus zero}. We define $\mathcal G$ to be the family of paths in $\mathcal G_\psi$ that do not lie in $\Gamma_0$, and we shall show $\zeta$ has an upper gradient and the upper gradient inequality holds along these paths. 

Let $x,y\in \gamma\cap S$ and $\gamma\in \mathcal G$ be a path that connects them. If $x,y \notin V$, then 
\begin{align*}
|\zeta(x)-\zeta(y)| = |\psi(x)-\psi(y)|\leq \sum_{i:Q_i\cap \gamma\neq \emptyset}\osc_{Q_i}(\psi).
\end{align*}
If $x\in V$ and $y\notin V$, then $\gamma \cap \partial V\neq \emptyset$ and we can consider the point $z$ of first entry of $\gamma$ in $\partial V$, as it travels from $x$ to $y$. The point $z$ is ``accessible" by $\gamma$, so by the definition of the boundary values of $\phi$ we have 
\begin{align*}
\phi(z)= \liminf_{\substack{Q_i\to z\\ Q_i\cap \gamma \neq \emptyset,i\in I_V}} M_{Q_i}(\phi).
\end{align*}
The upper gradient inequality therefore holds up to the boundary $\partial V$ and we have
\begin{align*}
|\phi(x)-\phi(z)| \leq \sum_{\substack{i:Q_i\cap \gamma\neq \emptyset\\i\in I_V}} \osc_{Q_i}(\phi).
\end{align*}
Therefore,
\begin{align*}
|\zeta(x)-\zeta(y)|&\leq |\phi(x)-\phi(z)|+|\psi(z)- \psi(y)|\\
&\leq \sum_{i:Q_i\cap \gamma\neq \emptyset} (\osc_{Q_i}(\phi)\x_{I_V}(i) + \osc_{Q_i}(\psi)).
\end{align*}
The case $x,y\in V$ is treated similarly by considering also the point $w$ of first entry of $\gamma$ in $\partial V$, as it travels from $y$ towards $x$, thus obtaining the bound
\begin{align*}
|\zeta(x)-\zeta(y)|\leq \sum_{i:Q_i\cap \gamma\neq \emptyset} (2\osc_{Q_i}(\phi)\x_{I_V}(i) + \osc_{Q_i}(\psi)).
\end{align*}
This shows that $\{ 2\osc_{Q_i}(\phi)\x_{I_V}(i) + \osc_{Q_i}(\psi) \}_{i\in \N}$ is an upper gradient of $\zeta$, as desired; recall Remark \ref{harmonic:3-Definition Remark rho-osc}.
\end{proof}

\subsection{A free boundary problem}
In this section we mention some results on a different type of a boundary problem, in which boundary data is only present on part of the boundary.  The proofs are almost identical to the case of the Dirichlet problem so we omit them. These results are  used in Chapter \ref{Chapter:Uniformization} to prove a uniformization result for planar Sierpi\'nski carpets.

Let $\Omega\subset \C$ be a quadrilateral, i.e., a Jordan region with four marked ``sides" on $\partial \Omega$. Assume that the sides are closed and they are marked by $\Theta_1,\Theta_2,\Theta_3,\Theta_4$, in a counter-clockwise fashion, where $\Theta_1,\Theta_3$ are opposite sides. Consider a relative Sierpi\'nski carpet $(S,\Omega)$. In fact, in this case $\br{S}$ is an actual Sierpi\'nski carpet, as defined in the Introduction. We consider functions $g\in \mathcal W^{1,2}(S)$ with boundary data $g=0$ on $\Theta_1$ and $g=1$ on $\Theta_3$. Such functions are called \textit{admissible (for the free boundary problem)}.

\begin{theorem}\label{harmonic:Free boundary problem}\index{Dirichlet problem!free boundary problem}\index{free boundary problem}
There exists a unique carpet-harmonic function $u\colon S\to \R$ that minimizes the Dirichlet energy $D_\Omega(g)$ over all admissible functions $g\in \mathcal W^{1,2}(S)$. The function $u$ is has a continuous extension to $\partial \Omega$ (and thus to $\br S$), with $u=0$ on $\Theta_1$ and $u=1$ on $\Theta_3$.
\end{theorem}

Of course, if the class of admissible functions is the weak Sobolev class then $u$ is weak carpet-harmonic, and if the class of admissible functions is the strong Sobolev class, then $u$ is strong carpet-harmonic.

Since there is no boundary data on the interior of the arcs $\Theta_2,\Theta_4$, these arcs can be treated - in the proofs -  as subarcs of peripheral disks, and, in particular, as subsets of the carpet $S$. If $V\subset \C \setminus (\Theta_1\cup \Theta_3)$ is an open set we can define
$$\partial_{\bullet}V\coloneqq  \partial V\cap \br S.$$
This will be the ``boundary", on which the extremal values of the carpet-harmonic function $u$ are attained. Thus, the maximum principle reads as:
\begin{theorem}\label{harmonic:Free boundary problem - Maximum Principle}\index{maximum principle}
If $u\colon  \br S\to \R$ is the solution to the free boundary problem, then for any open set $V\subset \C \setminus (\Theta_1\cup \Theta_3)$ we have
\begin{align*}
\sup_{x\in  \br S\cap \br{V}} u(x)= \sup_{x\in \partial_\bullet V} u(x)\quad  \textrm{and} \quad \inf_{x\in  \br S\cap \br{V}} u(x)= \inf_{x\in \partial_\bullet V} u(x).
\end{align*}
\end{theorem}

This is a stronger maximum principle than the one in Theorem \ref{harmonic:Maximum Principle}. It says that the extremal values of $u$ on $\br S\cap \br V$ can be attained at the part of the boundary of $V\cap \Omega$ that is disjoint from the interiors of the ``free" arcs $\Theta_2$ and $\Theta_4$. However, this boundary could still intersect $\Theta_1$ and $\Theta_3$, where extremal values can be attained, and this is the reason that we look at sets $V\subset \C\setminus (\Theta_1\cup \Theta_3)$. See also the maximum principle as stated in \cite[Lemma 4.6]{Rajala:uniformization}. 

As an application, one can show a rigidity-type result for square Sierpi\'nski carpets\index{Sierpi\'nski carpet!square}\index{Sierpi\'nski carpet!rigidity}, which was established in \cite[Theorem 1.4]{BonkMerenkov:rigidity}:

\begin{theorem}\label{harmonic:Rigidity}Let $\Omega=(0,1)\times (0,A)$, $\Omega'=(0,1)\times (0,A')$, and consider relative Sierpi\'nski carpets $(S,\Omega)$,$(S'\Omega')$ such that all peripheral disks of $S,S'$ are squares with sides parallel to the coordinate axes, and the Hausdorff $2$-measure of $S$ and $S'$ is $0$. If $f\colon  \br S \to \br {S'}$ is a quasisymmetry that preserves the sides of $\Omega,\Omega'$ (i.e., $f( \{0\} \times [0,A])= \{0\}\times [0,A']$ etc.) then $A=A'$, $S=S'$, and $f$ is the identity.  
\end{theorem}

We remark that our proof here is simpler than the proof in \cite{BonkMerenkov:rigidity}, and relies  on the uniqueness of Theorem \ref{harmonic:Free boundary problem}; see also Theorem \ref{harmonic:Uniqueness}. In \cite{BonkMerenkov:rigidity}, the authors have to follow several steps, showing first that each square $Q_i$ and its image $Q_i'$ have the same sidelength, then that $Q_i=Q_i'$, and finally that $f$ is the identity. They pursue these steps using the absolute continuity of $f$ and modulus arguments. In our approach, these arguments seem to be incorporated in the properties of Sobolev spaces and in the uniqueness of the minimizer in Theorem \ref{harmonic:Free boundary problem}, which is a powerful statement.

\begin{proof}
Let $u$ be the solution to the free boundary problem on $(S,\Omega)$ with $u=0$ on $\Theta_1\coloneqq \{0\}\times [0,A]$ and $u=1$ on $\Theta_3\coloneqq \{1\}\times [0,A]$. For $y\in [0,A]$ consider the segment $\gamma_y= [0,1]\times \{y\}$ that is parallel to the $x$-axis. As in the proof of Lemma \ref{harmonic:Paths joining continua}, one can see that for a.e.\ $y\in [0,A]$ the path $\gamma_y$ is a good path for $u$, thus by the upper gradient inequality
\begin{align*}
1= |u(0,y)-u(0,1)|\leq \sum_{i:Q_i\cap \gamma_y\neq \emptyset} \osc_{Q_i}(u).
\end{align*}
Integrating over $y\in  [0,A]$ and applying the Cauchy-Schwarz inequality we obtain
\begin{align*}
A &\leq \sum_{i\in \N} \osc_{Q_i}(u) \ell(Q_i) \leq \left(\sum_{i\in \N} \osc_{Q_i}(u)^2\right)^{1/2} \left(\sum_{i\in \N} \ell(Q_i)^2 \right)^{1/2}.
\end{align*}
The latter sum is the area of $\Omega$, which is $A$, hence we obtain $A\leq D_\Omega(u)$. On the other hand, the function $g(x,y)=x$ is admissible for the free boundary problem, and it is easy to see that $A=D_\Omega(g)$. Since $u$ is a minimizer it follows that $D_\Omega(g)=D_\Omega(u)$ and by the uniqueness in Theorem \ref{harmonic:Free boundary problem} we have $u(x,y)=g(x,y)=x$ for all $(x,y)\in \br{S}$.

If $f\colon \br S\to \br {S'}$ is a quasisymmetry, then it extends to a quasiconformal homeomorphism $f=(u_0,v_0)\colon \br \Omega\to \br {\Omega'}$ (using $\R^2$ coordinates), by the extension results proved in \cite[Section 5]{Bonk:uniformization}, and in fact, it extends to a global quasiconformal map on $\widehat{\C}$. Example \ref{harmonic:3-Example Sobolev Homeo} and Remark \ref{harmonic:Remark:quasiconformal} show that $f$ restricts to a function in $\mathcal W_s^{1,2}(S)$. Since $f$ preserves the sides of $\Omega$, it follows that $u_0=0$ on $\Theta_1$ and $u_0=1$ on $\Theta_3$. Hence, $u_0$ is admissible for the free boundary problem in $\Omega$, and thus $A=D_\Omega(u)\leq D_\Omega(u_0)$. Note that $\osc_{Q_i}(u_0)= \ell(Q_i')$, where $Q_i'=f(Q_i)$. Hence
\begin{align*}
D_\Omega(u_0) =\sum_{i\in \N} \ell(Q_i')^2 = \mathcal H^2(\Omega')=A'.
\end{align*}
It follows that 
$$A=D_\Omega(u)\leq D_\Omega(u_0)=A'.$$
The same argument applied to $f^{-1}$ and the free boundary problem in $\Omega'$ yields $A'\leq A$. Thus $D_\Omega(u) =D_\Omega(u_0)=A=A'$. The uniqueness in Theorem \ref{harmonic:Free boundary problem} shows that $u(x,y)=u_0(x,y)=x$. 

The same argument applied to the dual free boundary problem $v=0$ on $\Theta_2 \coloneqq [0,1]\times\{0\}$ and $v=A$ on $\Theta_4\coloneqq [0,1]\times \{A\}$ yields $v_0(x,y)=y$.
\end{proof}

\section{The Caccioppoli inequality}\label{harmonic:Section Caccioppoli}\index{Caccioppoli inequality}
In this section, and also in the next, we assume that $(S,\Omega)$ is an arbitrary relative Sierpi\'nski carpet (with the standard assumptions). We still drop the terminology weak/strong, and carpet-harmonic functions are always assumed to be normalized, as in Section \ref{harmonic:Section Properties of Harmonic}. It will be convenient to call \textit{test function} a function $\zeta \in \mathcal W^{1,2}(S)$ that vanishes outside an open set $V\subset\subset \Omega$. 

\begin{theorem}[Caccioppoli inequality]\label{harmonic:Caccioppoli}
Let $\zeta\colon S\to \R$ be a non-negative test function, and $u\colon S\to \R$ be a carpet-harmonic function. Then
\begin{align*}
\sum_{i\in \N} M_{Q_i}(\zeta)^2 \osc_{Q_i}(u)^2 \leq C \sum_{i\in \N} \osc_{Q_i}(\zeta)^2 M_{Q_i}(|u|)^2,
\end{align*} 
where $C>0$ is some universal constant.
\end{theorem}

\begin{proof}
We can assume that $\zeta$ is bounded. Indeed, if $\zeta$ is unbounded, then for $M\in \R$ the function $\zeta\land M $ is a bounded Sobolev function. Moreover, we have $\osc_{Q_i}(\zeta \land M) \leq \osc_{Q_i}(\zeta)$ by Proposition \ref{harmonic:3-Properties Sobolev}\ref{harmonic:3-Properties Sobolev-min} and $M_{Q_i}(\zeta\land M)\to M_{Q_i}(\zeta)$ as $M\to\infty$, which show that the desired inequality holds for $\zeta$ if it holds for $\zeta\land M$.

By assumption, $\zeta=0$ outside a set $V\subset \subset \Omega$. For $\varepsilon>0$ consider $\eta\coloneqq \varepsilon \zeta^2$ and $h \coloneqq  u-\eta u$. The function $\eta$ is a Sobolev function by Lemma \ref{harmonic:3-Properties Sobolev}(d), and so is $\eta u$, by the same lemma and the local boundedness of the carpet-harmonic function $u$. Therefore, $h$ is a Sobolev function that is equal to $u$ outside $V$. It follows that $D_{V}(u) \leq D_{V}(h)$, by harmonicity. Now, we will estimate $\osc_{Q_i}(h)$. 

We recall the computational rule from Lemma \ref{harmonic:3-Properties Sobolev}(d), which is similar to the product rule for derivatives: for all $i\in \N$ and all functions $f_1,f_2 \colon S \to \R$ we have
\begin{align}\label{harmonic:Computational rule}
\osc_{Q_i}(f_1f_2) \leq M_{Q_i}(|f_2|)\osc_{Q_i}(f_1)  + M_{Q_i}(|f_1|)\osc_{Q_i} (f_2). 
\end{align}

Using this rule, for fixed $i\in \N$ we have
\begin{align*}
\osc_{Q_i}(h)= \osc_{Q_i}(u(1-\eta)) \leq M_{Q_i}(|1-\eta|)\osc_{Q_i}(u) + M_{Q_i}(|u|) \osc_{Q_i}(1-\eta). 
\end{align*}
Since $\zeta$ is bounded, for small $\varepsilon>0$ we have $m_{Q_i}(\eta)\leq M_{Q_i}(\eta)<1$ for all $i\in \N$. This implies that $M_{Q_i}(|1-\eta|)= 1-m_{Q_i}(\eta)$. Also, we trivially have $\osc_{Q_i}(1-\eta)=\osc_{Q_i}(\eta)$. Therefore, for all sufficiently small $\varepsilon>0$ and for all $i\in \N$
\begin{align*}
\osc_{Q_i}(h) \leq (1- m_{Q_i}(\eta)) \osc_{Q_i}(u) +M_{Q_i}(|u|) \osc_{Q_i}(\eta).
\end{align*}
Combining this with the inequality $D_V(u)\leq D_{V} (h)$ we obtain
\begin{align*}
\sum_{i\in I_V} \osc_{Q_i}(u)^2 & \leq \sum_{i\in I_V} \bigg[ (1-m_{Q_i}(\varepsilon \zeta^2))^2 \osc_{Q_i}(u)^2 + M_{Q_i}(|u|)^2 \osc_{Q_i}(\varepsilon \zeta^2)^2 \\
& \quad \quad \quad \quad +2(1-m_{Q_i}(\varepsilon \zeta^2) )\osc_{Q_i}(u) M_{Q_i}(|u|) \osc_{Q_i}(\varepsilon \zeta^2)  \bigg].
\end{align*}
Noting that $m_{Q_i}(\varepsilon \zeta ^2)= \varepsilon m_{Q_i}(\zeta^2)$, $\osc_{Q_i}(\varepsilon \zeta^2)=\varepsilon \osc_{Q_i}(\zeta^2)$, and doing cancellations yields
\begin{align*}
 0&\leq \sum_{i\in I_V}\bigg[ (-2\varepsilon m_{Q_i}(\zeta^2) +\varepsilon^2 m_{Q_i}(\zeta^2)^2)\osc_{Q_i}(u)^2+ \varepsilon^2 M_{Q_i}(|u|)^2\osc_{Q_i}(\zeta^2)^2\\
 &\quad \quad \quad 2\varepsilon(1-\varepsilon m_{Q_i}(\zeta^2)) \osc_{Q_i}(u) M_{Q_i}(|u|)\osc_{Q_i}(\zeta^2) \bigg].
\end{align*}
Dividing by $\varepsilon>0$ and letting $\varepsilon\to 0$ we obtain
\begin{align}\label{harmonic:Caccioppoli proof inequalities}
\sum_{i\in I_V} m_{Q_i}(\zeta^2)\osc_{Q_i}(u)^2 \leq  \sum_{i\in I_V} \osc_{Q_i}(u)M_{Q_i}(|u|)\osc_{Q_i}(\zeta^2).
\end{align}
Now, we use the inequalities 
\begin{align*}
\osc_{Q_i}(\zeta^2)&\leq 2M_{Q_i}(\zeta)\osc_{Q_i}(\zeta) \quad \textrm{and}\\
m_{Q_i}(\zeta^2)&=M_{Q_i}(\zeta^2)-\osc_{Q_i}(\zeta ^2) \geq M_{Q_i}(\zeta)^2 - 2M_{Q_i}(\zeta) \osc_{Q_i}(\zeta),
\end{align*}
where the first one follows from the computational rule \eqref{harmonic:Computational rule}. Together with \eqref{harmonic:Caccioppoli proof inequalities} they imply that
\begin{align*}
 \sum_{i\in I_V} M_{Q_i}(\zeta)^2\osc_{Q_i}(u)^2 - 2\sum_{i\in I_V}M_{Q_i}&(\zeta)\osc_{Q_i}(\zeta)\osc_{Q_i}(u)^2\\
 &\leq 2\sum_{i\in I_V} \osc_{Q_i}(u)M_{Q_i}(|u|)M_{Q_i}(\zeta)\osc_{Q_i}(\zeta).
\end{align*}
In the second term of the left hand side we first use the inequality $\osc_{Q_i}(u)\leq 2M_{Q_i}(|u|)$, and then apply the Cauchy-Schwarz inequality:
\begin{align*}
\sum_{i\in I_V} M_{Q_i}(\zeta)^2\osc_{Q_i}(u)^2 &\leq 6 \sum_{i\in I_V} \osc_{Q_i}(u)M_{Q_i}(|u|)M_{Q_i}(\zeta)\osc_{Q_i}(\zeta) \\
&\leq 6 \left(\sum_{i\in I_V} M_{Q_i}(\zeta)^2\osc_{Q_i}(u)^2 \right)^{1/2} \cdot \left( \sum_{i\in I_V} \osc_{Q_i}(\zeta)^2M_{Q_i}(|u|)^2 \right)^{1/2}.
\end{align*}
Hence
\begin{align*}
\sum_{i\in I_V} M_{Q_i}(\zeta)^2\osc_{Q_i}(u)^2  \leq 36 \sum_{i\in I_V} \osc_{Q_i}(\zeta)^2M_{Q_i}(|u|)^2.
\end{align*}
Since $\zeta=0$ outside $V$, we can in fact write the summations over $i\in \N$, and this completes the proof.
\end{proof}

We now record an application of the Caccioppoli inequality towards the proof of a weak version of Liouville's theorem:

\begin{theorem}\label{harmonic:Liouville Weak}\index{Liouville's theorem}
Let $(S,\C)$ be a relative Sierpi\'nski carpet, and $u\colon S\to \R$ a carpet-harmonic function such that $|u|$ is bounded. Then $u$ is constant.
\end{theorem}

\begin{proof}
Assume that $|u|\leq M$. We fix a ball $B(0,R_0)$ and we wish to construct a test function $\zeta$ such that $0\leq \zeta\leq 1$, $\zeta=1$ on $B(0,R_0)$, but $D_{\C}(\zeta)$ is arbitrarily small, not depending on $R_0$. Then by the Caccioppoli inequality we will have
\begin{align*}
\sum_{i\in I_{B(0,R_0)}} \osc_{Q_i}(u)^2 &\leq \sum_{i\in \N} M_{Q_i}(\zeta)^2\osc_{Q_i}(u)^2  \leq C \sum_{i\in \N} \osc_{Q_i}(\zeta)^2M_{Q_i}(|u|)^2 \leq CM^2 D_{\C}(\zeta).
\end{align*} 
Since $D_{\C}(\zeta)$ can be arbitrarily small, it follows that $\sum_{i:Q_i\cap B(0,R_0)\neq \emptyset} \osc_{Q_i}(u)^2=0$, and thus $\osc_{Q_i}(u)=0$ for all $Q_i$ that intersect $B(0,R_0)$.  Since $R_0$ is arbitrary we have $\osc_{Q_i}(u)=0$ for all $i\in \N$. Thus, $u$ is constant by Lemma \ref{harmonic:4-zero oscillation lemma}(a), as desired.

Now, we construct the test function $\zeta$ with the desired properties. Essentially, $\zeta$ will be a discrete version of the logarithm. We fix a large integer $N$ which will correspond to the number of annuli around $0$ that we will construct, and $\zeta$ will drop by $1/N$ on each annulus. Define $\zeta=1$ on $B(0,R_0)$, and consider $r_1\coloneqq R_0$, $R_1\coloneqq 2r_1$. In the annulus $A_1\coloneqq A(0;r_1,R_1)$ define $\zeta$ to be a radial function of constant slope $\frac{1}{Nr_1}$, so on the outer boundary of $A_1$ the function $\zeta$ has value $1-1/N$. Then consider $r_2>R_1$ sufficiently large and $R_2\coloneqq 2r_2$, so that no peripheral disk intersects both $A_1$ and $A_2\coloneqq A(0;r_2,R_2)$; recall Lemma \ref{harmonic:Fatness consequence}. In the ``transition" annulus $A(0;R_1,r_2)$ we define $\zeta$ to be constant, equal to $1-1/N$, and on $A_2$ we let $\zeta$ be a radial function with slope $\frac{1}{Nr_2}$. The last annulus will be $A_N=A(0;r_N,R_N)$ and the value of $\zeta$ will be $0$ on the outer boundary of $A_N$. We extend $\zeta$ to be $0$ outside $B(0,R_N)$. Note that $\zeta$ is locally Lipschitz, so its restriction to the carpet $S$ is a Sobolev function, by Example \ref{harmonic:3-Example-Lipschitz}.

We now compute the Dirichlet energy of $\zeta$. Let $d_j(Q_i)\coloneqq  \mathcal H^1( \{s \in [r_j,R_j]: \gamma_s \cap Q_i\neq \emptyset \})$ where $\gamma_s$ is the circle of radius $s$ around $0$. Since the peripheral disks $Q_i$ are fat, there exists a constant $K>0$ such that $d_j(Q_i)^2\leq K \mathcal H^2(Q_i\cap A_j)$; see Remark \ref{harmonic:Remark:Fatness implication}. Also, if $Q_i\cap A_j\neq \emptyset$, then $\osc_{Q_i}(\zeta)\leq d_j(Q_i)\frac{1}{Nr_j}$. By construction, each peripheral disk $Q_i$ can only intersect one annulus $A_j$, and if a peripheral disk $Q_i$ does not intersect any annulus $A_j$, then $\zeta$ is constant on $Q_i$, so $\osc_{Q_i}(\zeta)=0$. Thus
\begin{align*}
D_\C(\zeta)&=\sum_{i\in \N}\osc_{Q_i}(\zeta)^2= \sum_{j=1}^N \sum_{i:Q_i\cap A_j\neq \emptyset} \osc_{Q_i}(\zeta)^2\\
&\leq \frac{1}{N^2}\sum_{j=1}^N  \frac{1}{r_j^2}\sum_{i:Q_i\cap A_j\neq \emptyset}d_j(Q_i)^2\\
&\leq  \frac{K}{N^2}\sum_{j=1}^N  \frac{1}{r_j^2}\sum_{i:Q_i\cap A_j\neq \emptyset} \mathcal H^2(Q_i\cap A_j)\\
&\leq  \frac{K}{N^2}\sum_{j=1}^N  \frac{1}{r_j^2}\mathcal H^2(A_j)\\
&= \frac{\pi K}{N^2}\sum_{j=1}^N  \frac{R_j^2-r_j^2}{r_j^2}=\frac{\pi K}{N^2}\sum_{j=1}^N   \frac{3r_j^2}{r_j^2}  \\
&=\frac{3\pi K}{N},
\end{align*}  
which can be made arbitrarily small if $N$ is large.
\end{proof}

\begin{remark}
Liouville's theorem justifies that we do not define carpet-har\-monic functions on relative carpets in the whole sphere $\widehat \C$, i.e., $\Omega=\widehat \C$, as the carpets studied in \cite{Bonk:uniformization}, because by continuity the carpet-harmonic functions would then be bounded and thus constant. 
\end{remark}

The non-linearity of the theory does not allow us to apply the Caccioppoli inequality of Theorem \ref{harmonic:Caccioppoli} to linear combinations of harmonic functions. We record another version of the Caccioppoli inequality for differences of harmonic functions. This will be very useful in establishing convergence properties of harmonic functions in Section \ref{harmonic:Section Equicontinuity}.

\begin{theorem}\label{harmonic:Caccioppoli2}\index{Caccioppoli inequality!for differences}
Let $\zeta\colon S\to \R$ be a non-negative continuous test function, and $u,v\colon S\to \R$ be carpet-harmonic functions. Then
\begin{align*}
\sum_{i\in \N} m_{Q_i}(\zeta) (\osc_{Q_i}(u)-\osc_{Q_i}(v))^2 \leq C \sum_{i\in \N} \osc_{Q_i}(\zeta) (\osc_{Q_i}(u)+\osc_{Q_i}(v)) M_{Q_i}(|u-v|),
\end{align*} 
where $C>0$ is some universal constant.
\end{theorem}
\begin{proof}
Suppose that $\eta:S\to \R$ is a continuous function and $\osc_{Q_i}(u+\varepsilon \eta)= u(x_\varepsilon)+\varepsilon \eta(x_\varepsilon)- u(y_\varepsilon)-\varepsilon \eta(y_\varepsilon)$ for some $\varepsilon\in \R$ and points $x_\varepsilon,y_\varepsilon \in \partial Q_i$. Then as $\varepsilon \to 0$, the points $x_\varepsilon,y_\varepsilon$ subconverge to points $x,y$, respectively, with $\osc_{Q_i}(u)=u(x)-u(y)$. Here we used the continuity of $u$ and the boundedness of $\eta$ on $\partial Q_i$.

If $\eta$ is a continuous test function supported in $V\subset \subset \Omega$, we have $D_V(u)\leq D_V(u+\varepsilon \eta)$, which implies
\begin{align*}
\sum_{i\in I_V} \osc_{Q_i}(u)^2 &\leq \sum_{i\in I_V} (u(x_{i,\varepsilon})+\varepsilon \eta(x_{i,\varepsilon})- u(y_{i,\varepsilon})-\varepsilon \eta(y_{i,\varepsilon}))^2 \\
&\leq  \sum_{i\in I_V} \osc_{Q_i}(u)^2 +2\varepsilon \sum_{i\in I_V}(u(x_{i,\varepsilon})-u(y_{i,\varepsilon}))( \eta(x_{i,\varepsilon})- \eta(y_{i,\varepsilon}))+ O(\varepsilon^2),
\end{align*}
for some points $x_{i,\varepsilon}, y_{i,\varepsilon}\in \partial Q_i$. As $\varepsilon\to 0$, this yields
\begin{align}\label{harmonic:Caccioppoli2:variation}
\sum_{i\in I_V} \osc_{Q_i}(u) (\eta(x_i)-\eta(y_i))=0,
\end{align}
where $x_i,y_i$ are sublimits of $x_{i,\varepsilon}, y_{i,\varepsilon}$, respectively, and $\osc_{Q_i}(u)=u(x_i)-u(y_i)$.

We use $\eta= \zeta\cdot (u-v)$ in \eqref{harmonic:Caccioppoli2:variation}, where $\zeta$ is supported in $V\subset \subset \Omega$, and we obtain
\begin{align*}
0&=\sum_{i\in I_V} \osc_{Q_i}(u)( \zeta(x_i)(u(x_i)-v(x_i))- \zeta(y_i)(u(y_i)-v(y_i)))\\
&=\sum_{i\in I_V} \osc_{Q_i}(u)\cdot \biggl(\zeta(x_i)(u(x_i)-u(y_i)) -\zeta(x_i)(v(x_i)-v(y_i))\\
&\quad\quad\quad\quad\quad\qquad +(\zeta(x_i)-\zeta(y_i))(u(y_i)-v(y_i)) \biggr).
\end{align*}
Since $\osc_{Q_i}(u)=u(x_i)-u(y_i)$, $\osc_{Q_i}(v)\geq v(x_i)-v(y_i)$, and $\zeta\geq 0$, we have
\begin{align*}
\sum_{i\in I_V} \zeta(x_i) (\osc_{Q_i}(u)^2-\osc_{Q_i}(u)\osc_{Q_i}(v)) \leq \sum_{i\in I_V} \osc_{Q_i}(\zeta)  \osc_{Q_i}(u) M_{Q_i}(|u-v|).
\end{align*}
Interchanging the roles of $u$ and $v$, we obtain points $x_i'\in \partial Q_i$ such that
\begin{align*}
\sum_{i\in I_V} \zeta(x_i') (\osc_{Q_i}(v)^2-\osc_{Q_i}(u)\osc_{Q_i}(v)) \leq \sum_{i\in I_V} \osc_{Q_i}(\zeta)  \osc_{Q_i}(v) M_{Q_i}(|u-v|).
\end{align*}
Adding the two inequalities, we have
\begin{align*}
\sum_{i\in I_V} \zeta(x_i) (\osc_{Q_i}(u)-\osc_{Q_i}(v))^2 &\leq \sum_{i\in I_V} \osc_{Q_i}(\zeta)  (\osc_{Q_i}(u)+\osc_{Q_i}(v)) M_{Q_i}(|u-v|) \\
&\quad \quad + \sum_{i\in I_V} (\zeta(x_i)-\zeta(x_i')) (\osc_{Q_i}(v)^2-\osc_{Q_i}(u)\osc_{Q_i}(v)).
\end{align*}
The conclusion follows, upon observing that $\zeta(x_i)\geq m_{Q_i}(\zeta)$, $|\zeta(x_i)-\zeta(x_i')|\leq \osc_{Q_i}(\zeta)$, and 
\[|\osc_{Q_i}(u)-\osc_{Q_i}(v)|\leq \osc_{Q_i}(u-v)=M_{Q_i}(u-v)-m_{Q_i}(u-v)\leq 2M_{Q_i}(|u-v|).\qedhere\]
\end{proof}

\section{Harnack's inequality and consequences}\label{harmonic:Section Harnack}
\subsection{Harnack's inequality}\index{Harnack's inequality}
In this section the main theorem is:
\begin{theorem}[Harnack's inequality]\label{harmonic:Harnack}
There exists a constant $H>1$ such that: if $u\colon S\to \R$ is a non-negative carpet-harmonic function, then
\begin{align*}
\sup_{z\in S\cap B_0}u(z) \leq H \inf_{z\in S\cap B_0}u(z)
\end{align*}
for all balls $B_0\subset \Omega$ with the property that there exists a ball $B_1\subset \Omega$ such that 
\begin{align*}
\bigcup_{i:Q_i\cap c_oB_0\neq \emptyset}Q_i \subset B_1\subset c_1B_1\subset \subset\Omega,
\end{align*}
where $c_0,c_1>1$ are constants. The constant $H$ depends only on the data of the carpet $S$ and on $c_0,c_1$. The choice of the latter two constants can be arbitrary.  
\end{theorem}
The assumption in the theorem asserts that the peripheral disks that meet $c_0B_0$ are essentially ``safely" contained in $\Omega$, away from the boundary.

Our treatment of Harnack's inequality is inspired by \cite{Granlund:Harnack}, where Harnack's inequality is proved for $W^{1,n}$-minimizers of certain variational integrals in $\R^n$. The method used there is a purely variational argument, which does not rely on a differential equation or a representation formula for minimizers, and this allows us to apply it in our discrete setting. The proof will be done in several steps. First we show:

\begin{prop}\label{harmonic:Harnack-Proposition}
Let $u\colon S\to \R$ be a positive carpet-harmonic function. Then for any non-negative test function $\zeta\colon S\to \R$ we have
\begin{align*}
\sum_{i\in \N} \frac{m_{Q_i}(\zeta)^2\osc_{Q_i}(u)^2}{M_{Q_i}(u)m_{Q_i}(u)} \leq C D_\Omega(\zeta),
\end{align*} 
where $C>0$ is a universal constant, not depending on $u,\zeta,S$.
\end{prop}

Note that by continuity $u$ is bounded below away from $0$ on each individual peripheral circle $\partial Q_i$, and this shows that $M_{Q_i}(u)\geq m_{Q_i}(u)>0$. Hence, all quantities above make sense.

\begin{proof}
As in the proof of Theorem \ref{harmonic:Caccioppoli}, we may assume that $\zeta$ is bounded. Replacing $u$ with $u+\delta$ for a small $\delta>0$, and noting that $\osc_{Q_i}(u+\delta)=\osc_{Q_i}(u)$, $M_{Q_i}(u+\delta)=M_{Q_i}(u)+\delta$, $m_{Q_i}(u+\delta)=m_{Q_i}(u)+\delta$, we see that it suffices to prove the statement assuming that $u\geq \delta>0$.

Fix a bounded test function $\zeta$, supported in $S\cap V$, where $V\subset \subset \Omega$. For a small $\varepsilon>0$  consider the variation $h= u+ \varepsilon \zeta^2/u $. We remark that $u^{-1}$ is a Sobolev function, since 
\begin{align*}
|u(x)^{-1}-u(y)^{-1}|\leq |u(x)-u(y)|/\delta^2, 
\end{align*}
so $u^{-1}$ inherits its upper gradient inequality from $u$. Moreover, $\zeta^2$ and $\zeta^2u^{-1}$ are  Sobolev functions by Lemma \ref{harmonic:3-Properties Sobolev}(d). Hence, $h$ is a Sobolev function.

Observe that the function $x\mapsto x+\varepsilon y^2/x$ is increasing as long as $x^2>\varepsilon y^2$. Since $\zeta$ is bounded and $u\geq \delta$, it follows that for sufficiently small $\varepsilon>0$ we have
\begin{align*}
M_{Q_i}(h) &\leq M_{Q_i}(u) + \varepsilon \frac{M_{Q_i}(\zeta^2)}{M_{Q_i}(u)}, \quad \textrm{and}\\
m_{Q_i}(h) &\geq m_{Q_i}(u) + \varepsilon \frac{m_{Q_i}(\zeta^2)}{m_{Q_i}(u)}
\end{align*}
for all $i\in \N$. Hence,
\begin{align*}
\osc_{Q_i}(h)&\leq \osc_{Q_i}(u)+ \varepsilon \left( \frac{M_{Q_i}(\zeta^2)}{M_{Q_i}(u)}-\frac{m_{Q_i}(\zeta^2)}{m_{Q_i}(u)}\right)\\
&=\left(1-\varepsilon \frac{m_{Q_i}(\zeta^2)}{M_{Q_i}(u)m_{Q_i}(u)} \right) \osc_{Q_i}(u) +\varepsilon \frac{M_{Q_i}(\zeta)+m_{Q_i}(\zeta)}{M_{Q_i}(u)} \osc_{Q_i}(\zeta),
\end{align*}
where we used the equalities $\osc_{Q_i}(\zeta)=M_{Q_i}(\zeta)-m_{Q_i}(\zeta)$ and $\osc_{Q_i}(\zeta^2)=\osc_{Q_i}(\zeta)\cdot (M_{Q_i}(\zeta)+m_{Q_i}(\zeta))$.

Since $u$ is carpet-harmonic, and $h$ is equal to $u$ outside $V$ we have $D_V(u)\leq D_V(h)$. Working as in the proof of Theorem \ref{harmonic:Caccioppoli}, and letting $\varepsilon\to 0$, we obtain
\begin{align*}
\sum_{i\in I_V} \frac{m_{Q_i}(\zeta^2)\osc_{Q_i}(u)^2}{M_{Q_i}(u)m_{Q_i}(u)} &\leq \sum_{i\in I_V} \osc_{Q_i}(u) \frac{M_{Q_i}(\zeta)+m_{Q_i}(\zeta)}{M_{Q_i}(u)} \osc_{Q_i}(\zeta).
\end{align*}
Writing $M_{Q_i}(\zeta)=m_{Q_i}(\zeta)+\osc_{Q_i}(\zeta)$ and applying the Cauchy-Schwarz inequality we obtain
\begin{align*}
\sum_{i\in I_V} \frac{m_{Q_i}(\zeta^2)\osc_{Q_i}(u)^2}{M_{Q_i}(u)m_{Q_i}(u)}&\leq  \sum_{i\in I_V} \frac{\osc_{Q_i}(u)}{M_{Q_i}(u)} (2m_{Q_i}(\zeta)+\osc_{Q_i}(\zeta))\osc_{Q_i}(\zeta)\\
&\leq 2\left(\sum_{i\in I_V} \frac{m_{Q_i}(\zeta)^2 \osc_{Q_i}(u)^2}{M_{Q_i}(u)^2}\right)^{1/2}D_V(\zeta)^{1/2}\\
&\quad \quad + \sum_{i\in I_V}\frac{\osc_{Q_i}(u)}{M_{Q_i}(u)}\osc_{Q_i}(\zeta)^2.
\end{align*} 
Noting that $M_{Q_i}(u)^2\geq M_{Q_i}(u)m_{Q_i}(u)$, $\osc_{Q_i}(u)=M_{Q_i}(u)-m_{Q_i}(u)\leq M_{Q_i}(u)$, and defining $A=\sum_{i\in I_V} \frac{m_{Q_i}(\zeta )^2\osc_{Q_i}(u)^2}{M_{Q_i}(u)m_{Q_i}(u)}$, we have
\begin{align*}
A\leq 2A^{1/2}D_V(\zeta)^{1/2}+D_V(\zeta).
\end{align*}
If $A\leq D_V(\zeta)$, then there is nothing to show. Otherwise, we have $D_V(\zeta)^{1/2}\leq A^{1/2}$, hence
\begin{align*}
A\leq 2A^{1/2}D_V(\zeta)^{1/2}+A^{1/2}D_V(\zeta)^{1/2},
\end{align*}
which implies that
\begin{align*}
A\leq 9 D_V(\zeta)
\end{align*}
and this concludes the proof.
\end{proof}

This proposition already has a strong Liouville theorem as a corollary:

\begin{corollary}[Liouville's Theorem]\label{harmonic:Liouville-Strong}\index{Liouville's theorem!strong}
Let $(S,\C)$ be a relative Sierpi\'nski carpet, and $u\colon S\to \R$ a carpet-harmonic function that is bounded above or below. Then $u$ is constant.
\end{corollary}
\begin{proof}
We first reduce the statement to the case that $u>0$. If $u$ is bounded above, then we can replace it with the carpet-harmonic function $1+\sup_{z\in S}u(z)-u$, and showing that this is constant will imply that $u$ is constant. Similarly, if $u$ is bounded below, then we use $1+u-\inf_{z\in S} u(z)$. Hence, from now on we assume that $u>0$.

We fix a ball $B(0,R_0)$ and consider a test function $\zeta$ such that $0\leq \zeta\leq 1$, $\zeta=1$ on $S\cap B(0,R_0)$, and $D_{\C}(\zeta)<\varepsilon$ where $\varepsilon>0$ can be arbitrarily small. Such a function is constructed in the proof of Theorem \ref{harmonic:Liouville Weak}. Then for all $Q_i\subset B(0,R_0)$ we have $m_{Q_i}(\zeta)=1$. Hence, Proposition \ref{harmonic:Harnack-Proposition} yields
\begin{align*}
\sum_{i:Q_i\subset B(0,R_0)} \frac{\osc_{Q_i}(u)^2}{M_{Q_i}(u)m_{Q_i}(u)} \leq CD_{\C}(\zeta)<C\varepsilon.
\end{align*}
Letting $\varepsilon\to 0$ we obtain $\osc_{Q_i}(u)=0$ for all $Q_i\subset B(0,R_0)$. Since $R_0$ was arbitrary, it follows that $\osc_{Q_i}(u)=0$ for all $i\in \N$, thus $u$ is constant by Lemma \ref{harmonic:4-zero oscillation lemma}(a).
\end{proof}

We continue our preparation for the proof of Harnack's inequality. From Proposition \ref{harmonic:Harnack-Proposition} we derive the next lemma:
\begin{lemma}\label{harmonic:Harnack-Lemma}
Let $u\colon S\to \R$ be a positive carpet-harmonic function. Consider a ball $B_1\subset c_1B_1 \subset\subset \Omega$ for some $c_1>1$. Then 
\begin{align*}
\sum_{i:Q_i\subset B_1} \frac{\osc_{Q_i}(u)^2}{M_{Q_i}(u)m_{Q_i}(u)}\leq C,
\end{align*}
where the constant $C>0$ depends only on the data of the carpet $S$ and on $c_1$, but not on $u,B_1$.
\end{lemma}
\begin{proof}
We apply Proposition \ref{harmonic:Harnack-Proposition} to a test function $\zeta$ defined as follows. We set $\zeta=1$ on $B_1$, $\zeta=0$ outside $c_1B_1$, and $\zeta$ is radial with slope $\frac{1}{(c_1-1)r}$ on the annulus $A_1\coloneqq c_1B_1\setminus \br B_1 $, where $r$ is the radius of $B_1$. Then $\zeta$ is Lipschitz so by Example \ref{harmonic:3-Example-Lipschitz} it restricts to a Sobolev function. We only have to show that $D_\Omega(\zeta)$ is bounded by a constant depending only on $c_1$ and the data of $S$. Our computation is very similar to the proof of Theorem \ref{harmonic:Liouville Weak}. We introduce the notation $d(Q_i)=\mathcal H^1(\{s\in [r,c_1r]:\gamma_s\cap Q_i\neq \emptyset\})$ and note that $d(Q_i)^2 \leq K\mathcal H^2(Q_i\cap c_1B_1)$ for $i\in \N$ by the fatness condition and Remark \ref{harmonic:Remark:Fatness implication}. Then
\begin{align*}
D_{\Omega}(\zeta)&= \sum_{i\in I_{A_1}} \osc_{Q_i}(\zeta)^2 \leq \frac{1}{(c_1-1)^2r^2} \sum_{i\in I_{A_1}} d(Q_i)^2\\
&\leq \frac{K}{(c_1-1)^2r^2} \mathcal H^2(c_1B_1)= \frac{ K c_1^2\pi}{(c_1-1)^2},
\end{align*}
as claimed. Note that this constant blows up as $c_1\to 1$.
\end{proof}

Next we prove a version of Gehring's oscillation lemma (see e.g.\ \cite[Lemma 3.5.1, p.~65]{AstalaIwaniecMartin:quasiconformal}). A function $v\colon S\to \R$ is said to be \textit{monotone}\index{monotone function} if it satisfies the maximum and minimum principles as in the statement of Theorem \ref{harmonic:Maximum Principle}.

\begin{lemma}\label{harmonic:Harnack-Oscillation}\index{Gehring's oscillation lemma}
Let $v\colon S\to \R$ be a continuous monotone function on the relative Sierpi\'nski carpet $(S,\Omega)$, lying in the Sobolev space $\mathcal W^{1,2}_{\loc}(S)$. Consider a ball $B_0\subset \Omega$ with $B_0\subset c_0B_0\subset \subset \Omega$, where $c_0>1$. Then
\begin{align*}
\sup_{z\in S\cap B_0}v(z)-\inf_{z\in S\cap B_0}v(z) \leq C \left(\sum_{i\in I_{c_0B_0}} \osc_{Q_i}(v)^2 \right)^{1/2},
\end{align*}
where the constant $C>0$ depends only on the data of the carpet $S$ and on $c_0$, but not on $v,B_0$.
\end{lemma}
\begin{proof}
For any $x,y\in S\cap B_0$, by monotonicity we have
\begin{align}\label{harmonic:Gehring1}
v(x)-v(y) &\leq \sup_{z\in \partial_*(sB_0)}v(z)- \inf_{z\in \partial_*(sB_0)}v(z)
\end{align}
for all $s\in [1,c_0]$. For a.e.\ $s\in [1,c_0]$ the upper gradient inequality for $v$ yields
\begin{align}\label{harmonic:Gehring2}
\sup_{z\in \partial_*(sB_0)}v(z)- \inf_{z\in \partial_*(sB_0)}v(z) \leq \sum_{i:Q_i\cap \partial(sB_0)\neq \emptyset} \osc_{Q_i}(v).
\end{align}
Here we used the fact that the circular path $\partial(sB_0)$ is non-exceptional for a.e.\ $s\in [1,c_0]$, which follows from the proof of Lemma \ref{harmonic:Paths joining continua}. If we write $B_0=B(x_0,r)$, then \eqref{harmonic:Gehring1} and \eqref{harmonic:Gehring2} imply that
\begin{align*}
\sup_{z\in S\cap B_0}v(z)-\inf_{z\in S\cap B_0}v(z) \leq \sum_{i:Q_i\cap \partial B(x_0,s)\neq \emptyset} \osc_{Q_i}(v)
\end{align*}
for a.e.\ $s\in [r,c_0r]$. We now integrate over $s\in [r,c_0r]$ so we obtain
\begin{align*}
r(c_0-1)\left(\sup_{z\in S\cap B_0}v(z)-\inf_{z\in S\cap B_0}v(z)\right) &\leq \int_{r}^{c_0r} \sum_{i:Q_i\cap \partial B(x_0,s)\neq \emptyset} \osc_{Q_i}(v) \,ds\\
&= \sum_{i\in \N} \osc_{Q_i}(v) \int_{r}^{c_0r} \x_{Q_i\cap \partial B(x_0,s)} \,ds\\
&\leq \sum_{i: Q_i\cap c_0B_0\neq \emptyset} \osc_{Q_i}(v) d(Q_i),
\end{align*}
where $d(Q_i)=\mathcal H^1(\{s\in [r,c_0r]: \gamma_s\cap Q_i\neq \emptyset\})$ and $\gamma_s$ is a circular path around $x_0$ with radius $s$. As usual, by the fatness of the peripheral disks (see Remark \ref{harmonic:Remark:Fatness implication}) there exists a uniform constant $K$ such that $d(Q_i)^2\leq K\mathcal H^2(Q_i\cap c_0B_0)$ for all $i\in \N$. Now, applying Cauchy-Schwarz we obtain
\begin{align*}
\sup_{z\in S\cap B_0}v(z)-\inf_{z\in S\cap B_0}v(z) &\leq \frac{1}{r(c_0-1)} \left(\sum_{i\in I_{c_0B_0}} \osc_{Q_i}(v)^2 \right)^{1/2} \left(\sum_{i\in I_{c_0B_0}} d(Q_i)^2 \right)^{1/2}\\
&\leq \frac{1}{r(c_0-1)} \left(\sum_{i\in I_{c_0B_0}} \osc_{Q_i}(v)^2 \right)^{1/2} \mathcal (K\mathcal H^2(c_0B_0))^{1/2}\\
&\leq \frac{K^{1/2}\pi^{1/2}c_0}{c_0-1}\left(\sum_{i\in I_{c_0B_0}} \osc_{Q_i}(v)^2 \right)^{1/2}.
\end{align*}
This completes the proof.
\end{proof}

Finally we proceed with the proof of Harnack's inequality.
\begin{proof}[Proof of Theorem \ref{harmonic:Harnack}]
Replacing $u$ with $u+\delta$ for a small $\delta>0$ and noting that the conclusion of the theorem persists if we let $\delta\to 0$, we may assume that $u\geq \delta>0$.

The function $v\coloneqq  \log u$ is a continuous monotone function, since $u$ has these properties and $\log$ is an increasing function. Also, $v$ lies in the Sobolev space $\mathcal W^{1,2}_{\loc}(S)$, since
\begin{align*}
|v|&\leq \max\{ |u|, |\log \delta|\}, \quad \textrm{and}\\
|v(x)-v(y)|&= \left| \log \frac{u(x)}{u(y)} \right| \leq |u(x)-u(y)|/\delta.
\end{align*} 
The latter inequality shows that $v$ inherits its upper gradient inequality from $u$. In fact, we have
\begin{align}\label{harmonic:Harnack-logu inequality}
\osc_{Q_i}(v)=\osc_{Q_i}(\log u) \leq  \frac{\osc_{Q_i}(u)}{M_{Q_i}(u)^{1/2}m_{Q_i}(u)^{1/2}},
\end{align}
as one can see from the elementary inequality $\log(a/b)\leq \frac{a-b}{(ab)^{1/2}}$ for $a\geq b>0$.

Now, applying Lemma \ref{harmonic:Harnack-Oscillation}, and then \eqref{harmonic:Harnack-logu inequality} and Lemma \ref{harmonic:Harnack-Lemma} one has
\begin{align*}
\sup_{z\in S\cap B_0}v(z)-\inf_{z\in S\cap B_0}v(z) &\leq C \left(\sum_{i\in I_{c_0B_0}} \osc_{Q_i}(v)^2 \right)^{1/2}\\
&\leq C \left(\sum_{i:Q_i\subset B_1} \osc_{Q_i}(v)^2 \right)^{1/2}\\
&\leq C \left(\sum_{i:Q_i\subset B_1} \frac{\osc_{Q_i}(u)^2}{M_{Q_i}(u)m_{Q_i}(u)} \right)^{1/2}\\
&\leq C'.
\end{align*}
Here we used the assumption that 
\begin{align*}
\bigcup_{i\in I_{c_0B_0}}Q_i \subset  B_1 \subset c_1B_1\subset \subset \Omega.
\end{align*}
Therefore
\begin{align*}
\log\left( \frac{\sup_{z\in S\cap B_0}u(z)}{\inf_{z\in S\cap B_0}u(z)}\right)=\sup_{z\in S\cap B_0}v(z)-\inf_{z\in S\cap B_0}v(z)\leq C',
\end{align*}
thus
\begin{align*}
\sup_{z\in S\cap B_0}u(z)\leq e^{C'}\inf_{z\in S\cap B_0}u(z).
\end{align*}
The constant $e^{C'}$ depends only on the data of the carpet $S$ and on $c_0,c_1$.
\end{proof}

We record an application of the oscillation Lemma \ref{harmonic:Harnack-Oscillation}.
\begin{corollary}\label{harmonic:Liouville corollary}
Let $(S,\C)$ be a relative Sierpi\'nski carpet, and $u\colon S\to \R$ a carpet-harmonic function with finite energy, i.e., $D_\C (u)<\infty$. Then $u$ is constant.
\end{corollary}
\begin{proof}
By Lemma \ref{harmonic:Harnack-Oscillation}, for any ball $B_0\subset \C$ we have 
\begin{align*}
\sup_{z\in S\cap B_0}u(z)-\inf_{z\in S\cap B_0}u(z)\leq C \left( \sum_{i\in I_{2B_0}} \osc_{Q_i}(u)^2 \right)^{1/2}\leq C D_\C(u)^{1/2}.
\end{align*}
The ball $B_0$ is arbitrary, so it follows that $u$ is bounded, and therefore it is constant by Liouville's Theorem \ref{harmonic:Liouville-Strong}.
\end{proof}

\subsection{Strong maximum principle}
Using Harnack's inequality we prove a strong maximum principle:
\begin{theorem}\label{harmonic:Strong maximum principle}\index{maximum principle!strong}
Let $u\colon S\to \R$ be a carpet-harmonic function. Assume that $u$ attains a maximum or a minimum at a point $x_0\in S$. Then $u$ is constant.
\end{theorem}
\begin{proof}Using $-u$ instead of $u$ if necessary, we assume that $x_0$ is a point of maximum. First assume that $x_0\in S^\circ$, i.e., $x_0$ does not lie on any peripheral circle. Let $v\coloneqq u(x_0)-u$ which is a non-negative carpet-harmonic function. Then using Lemma \ref{harmonic:Fatness consequence} one can find small balls $B_0,B_1$ centered at $x_0$ such that 
\begin{align*}
\bigcup_{i:Q_i\cap 2B_0 \neq \emptyset}Q_i\subset B_1\subset 2B_1 \subset \subset \Omega.
\end{align*}
Applying Harnack's inequality inside $B_0$ we obtain
\begin{align*}
\sup_{z\in S\cap B_0}v(z)\leq C\min_{z\in S\cap B_0}v(z)=0.
\end{align*}
Thus, $u(z)=u(x_0)$ for $z\in S\cap B_0$.

Now, if $y_0\in S^\circ $ is arbitrary, then by Lemma \ref{harmonic:Paths in S^o} one can find a path $\gamma \subset S^\circ$ that connects $y_0$ to $x_0$. For each point $y\in \gamma$ there exists a small ball $B_y\subset \Omega$ where Harnack's inequality can be applied. By compactness, there are finitely many balls $B_{y_i}$, $i=1,\dots,N$, that cover the path $\gamma$ and form a \textit{Harnack-chain}: $B_{y_i}$ intersects $B_{y_{i+1}}$ and Harnack's inequality can be applied to each ball $B_{y_i}$. The argument in the previous paragraph yields that $u$ is constant on each ball, thus $u(y_0)=u(x_0)$. Since $S^\circ $ is dense in $S$, by continuity it follows that $u$ is constant.

Now, we treat the case that $x_0\in \partial Q_{i_0}$ for some $i_0\in \N$. Then by Lemma \ref{harmonic:Paths in S^o} there exists Jordan curve $\gamma\subset S^\circ$ ``surrounding" a Jordan region $V$ containing $Q_{i_0}$. By the maximum principle (Theorem \ref{harmonic:Maximum Principle}), there exists a point $x_0'\in \gamma =\partial_*V$ such that $u(x_0')=\sup_{z\in S\cap \br{V}}u(z) =u(x_0)$. Since $x_0'\in S^\circ$, it follows that $u$ is constant by the previous case. 
\end{proof}

\section{Equicontinuity and convergence}\label{harmonic:Section Equicontinuity}\index{equicontinuity}
Finally, we establish the local equicontinuity of carpet-harmonic functions and ensure that limits of harmonic functions are harmonic. In all statements the underlying relative Sierpi\'nski carpet is $(S,\Omega)$.

\begin{theorem}\label{harmonic:Equicontinuity}
Let $V\subset \subset U\subset\subset \Omega$ be open sets and $M>0$ be a constant. For each $\varepsilon>0$ there exists $\delta>0$ such that if $u$ is a carpet-harmonic function with $D_U(u) \leq M$, then for all points $x,y\in S\cap V$ with $|x-y|<\delta$ we have $|u(x)-u(y)|<\varepsilon$. The value of $\delta$ depends only on the carpet $S$ and on $\varepsilon,M$, but not on $u$.
\end{theorem}

Usually, equicontinuity for minimizers in potential theory follows from the local H\"older continuity of the energy minimizers. However, in our setting we were not able to establish the H\"older continuity, mainly due to the lack of self-similarity of the carpets. 

\begin{proof}
By compactness, it suffices to show that for each $\varepsilon>0$, each point $x_0\in S\cap \br V$ has a neighborhood $V_0$ such that for all $z,w\in S\cap V_0$ we have $|u(z)-u(w)|<\varepsilon$, whenever $u$ is a carpet-harmonic function with $D_U(u)\leq M$. The proof is based on the arguments we used in Theorem \ref{harmonic:Liouville Weak} and Lemma \ref{harmonic:Harnack-Oscillation}.

Suppose first that $x_0\in S^\circ$, and let $N\in \N$ be a large number to be chosen. For each $k\in \{1,\dots,N\}$ consider an annulus $A_k\coloneqq A(x_0;r_k,2r_k)\subset \subset U$ such that the annuli are nested, all of them surrounding $B(x_0,r_N)$, and they intersect disjoint sets of peripheral disks. Let $V_0=B(x_0,r_N)$ and note that for $z,w \in S\cap V_0$ and $k\in \{1,\dots,N\}$ we have
\begin{align*}
|u(z)-u(w)|^2\leq C \sum_{i\in I_{A_k}} \osc_{Q_i}(u)^2 
\end{align*}
by the computations in the proof of Lemma \ref{harmonic:Harnack-Oscillation}, where $C>0$ depends only on the data. Summing over $k$, we obtain
\begin{align*}
N|u(z)-u(w)|^2 \leq C\sum_{i\in I_U} \osc_{Q_i}(u)^2 =CD_U(u)\leq CM.
\end{align*}
Hence, $|u(z)-u(w)|\leq C'\sqrt{M}/\sqrt{N}$, which can be made smaller than $\varepsilon$, if $N$ is sufficiently large, independent of $u$.

If $x_0\in \partial Q_{i_0}$ for some $i_0\in \N$, we have to modify the argument as usual. We consider again the annuli $A_k$, all of which intersect $Q_{i_0}$, but otherwise they intersect disjoint sets of peripheral disks. We set $V_0$ to be the component of $B(x_0,r_N) \setminus \br Q_{i_0}$ containing $x_0$ in its boundary. This component is bounded by a subarc $\gamma$ of $\partial B(x_0,r_N)$, which defines a crosscut separating $x_0$ from $\infty$ in $\R^2\setminus Q_{i_0}$. The arc $\gamma$ has its endpoints on $\partial Q_{i_0}$. We consider an open Jordan arc $\alpha\subset Q_{i_0}$, having the same endpoints as $\gamma$. Then $\gamma\cup \alpha$ bounds a Jordan region $V_1\supset V_0$. 

For $z,w\in S\cap V_1$ the maximum principle in Theorem \ref{harmonic:Maximum Principle} and the upper gradient inequality yield 
\begin{align*}
|u(z)-u(w)|\leq \sum_{\substack{i:Q_i\cap \partial B(x_0,sr_k)\\i\neq i_0}}\osc_{Q_i}(u)
\end{align*}
for all $k\in \{1,\dots,N\}$ and a.e.\ $s\in (1,2)$. Using this in the proof of Lemma \ref{harmonic:Harnack-Oscillation} we obtain the exact same inequality as in the conclusion, without the term corresponding to $i=i_0$. Now, the proof continues as in the case $x_0\in S$.
\end{proof}

\begin{theorem}\label{harmonic:Convergence}
Suppose that $u_n$, $n\in \N$, is a sequence of carpet-harmonic functions converging locally uniformly to a function $u:S\to \R$. Then $u$ is carpet-harmonic.
\end{theorem}
\begin{proof}
We fix an open set $V\subset \subset \Omega$. Then $u_n$ converges uniformly to $u$ in $V$, so in particular, $u_n$ is uniformly bounded in $V$. By the Caccioppoli inequality in Theorem \ref{harmonic:Caccioppoli} we obtain that $D_V(u_n)$ is uniformly bounded in $n\in \N$. The Caccioppoli inequality ``for differences" in Theorem \ref{harmonic:Caccioppoli2} implies that $\sum_{i\in I_V}(\osc_{Q_i}(u_n)-\osc_{Q_i}(u_m))^2$ is uniformly small for sufficiently large $n$ and $m$. This shows that the tails of the sum
\begin{align*}
D_V(u_n)=\sum_{i\in I_V} \osc_{Q_i}(u_n)^2
\end{align*}
are small, uniformly in $n$. Using the uniform convergence one can show that $\osc_{Q_i}(u_n)\to \osc_{Q_i}(u)$ for each $i\in \N$. These imply that $\{\osc_{Q_i}(u_n)\}_{i\in V}$ converges to $\{\osc_{Q_i}(u)\}_{i\in I_V}$ in $\ell^2$. Hence, $u$ lies in $\mathcal W^{1,2}_{\loc}(S)$; recall Remark \ref{harmonic:Remark:convergence}. 

Moreover, $D_{V}(u_n+\zeta)\to D_V(u+\zeta)$ for each test function $\zeta\in \mathcal W^{1,2}(S)$ vanishing outside $V$. Indeed, it is straightforward from uniform convergence to see that $\osc_{Q_i}(u_n+\zeta)\to \osc_{Q_i}(u+\zeta)$ for each $i\in \N$. Moreover, using the inequality $\osc_{Q_i}(u_n+\zeta)\leq \osc_{Q_i}(u_n)+\osc_{Q_i}(\zeta)$ (see the proof of Proposition \ref{harmonic:3-Properties Sobolev}) and the convergence of  $\{\osc_{Q_i}(u_n)\}_{i\in V}$ in $\ell^2$ one sees that the tails of the sum 
\begin{align*}
\sum_{i\in I_V} \osc_{Q_i}(u_n+\zeta)^2
\end{align*}
are small, uniformly in $n$. 

Finally, the above imply that the inequality $D_V(u_n)\leq D_V(u_n+\zeta)$ from harmonicity passes to the limit, to yield $D_V(u)\leq D_V(u+\zeta)$. This shows that $u$ is carpet-harmonic, as desired.
\end{proof}

\begin{corollary}\label{harmonic:Arzela}
Let $u_n$, $n\in \N$, be a sequence of carpet-harmonic functions that are locally uniformly bounded. Then there exists a subsequence of $u_n$ that converges locally uniformly to a carpet-harmonic function $u:S\to \R$.
\end{corollary}
\begin{proof}
By the Caccioppoli inequality in Theorem \ref{harmonic:Caccioppoli}, it follows that $D_{V}(u_n)$ is uniformly bounded in $n$, for each $V\subset \subset \Omega$. Theorem \ref{harmonic:Equicontinuity} implies that $\{u_n\}_{n\in \N}$ is a locally equicontinuous family. Since the functions $u_n$ are locally uniformly bounded, by the Arzel\`{a}-Ascoli theorem we conclude that they subconverge locally uniformly to a function $u:S\to \R$. This function has to be carpet-harmonic by Theorem \ref{harmonic:Convergence}.
\end{proof}

%% file: accessible.tikz
\begin{tikzpicture}

			//accessible point
			\begin{scope}[shift={(-2,0)}, rotate=30,scale=1.5]

			//draw big peripheral circle	
			
			\def\a{1.25};\def\A{1};
			\def\b{0.05};\def\B{20};
			\def\c{1};\def\C{1};
			\def\d{0.1};\def\D{34};
            	 	\draw [fill=black!5,smooth,domain=0:360] plot 
            		(
            		{\a*cos( \A*\x )+\b*sin(\B*\x)}, 
            		{\c*sin( \C*\x ) + \d*cos(\D*\x )}
            		) ;
            \end{scope}
			\node[draw, shape=circle, fill=black, scale=.15, label=$x$] (x) at (-2,1.62){};

           	//control points
            \node[] (a1) at (0,3){};
            \node[label=$\gamma_0$] at (a1){};
            \node[] (a2) at (-1.5,2.5){};
            \node[] (a3) at (-2,1){};
            \node[] (a4) at (-2.5,2){};
            \node[] (a5) at (-2.5,1){};
            \node[] (a6) at (-3,1.5){};
            \node[] (a7) at (-3,0.5){};
            \node[label=$Q_{i_0}$] (Q) at (-2,0){};
            
            //connections
            \draw[thick,color=purple] (a1.center)..controls (-1,4) and (-1.5,2)..(x.center);
            \draw[thick,color=purple] (x.center)..controls (-2.4,1) and (-2.5,3) ..(-4,1);

		//non-accessible point
		\begin{scope}[scale=1, shift={(5,0)}]
		
				//grid  
			//grid	

			\begin{scope}[shift={(-2,-0.045)}, rotate=21,scale=1.5]

			//draw big peripheral circle	
			
			\def\a{1.25};\def\A{1};
			\def\b{0.05};\def\B{20};
			\def\c{1};\def\C{1};
			\def\d{0.1};\def\D{34};
            	 	\draw [fill=black!5,smooth,domain=0:360] plot 
            		(
            		{\a*cos( \A*\x )+\b*sin(\B*\x)}, 
            		{\c*sin( \C*\x ) + \d*cos(\D*\x )}
            		) ;
            \end{scope}
			\node[draw, shape=circle, fill=black, scale=.15, label=$x$] (x) at (-2,1.62){};

           	//control points
            \node[] (a1) at (0,3){};
            \node[label=$\gamma_0$] at (-0.8,1.8){};
            \node[] (a2) at (-1.5,2.5){};
            \node[] (a3) at (-2,1){};
            \node[] (a4) at (-2.5,2){};
            \node[] (a5) at (-2.5,1){};
            \node[] (a6) at (-3,1.5){};
            \node[] (a7) at (-3,0.5){};
            \node[label=$Q_{i_0}$] (Q) at (-2,0){};
            
            //connections
            \draw[thick,color=purple] (x.center)..controls (-2.4,1) and (-2.5,3) ..(-4,1);
			\begin{scope}[shift={(x)}, rotate=-30]            
            \draw[color=purple,thick, smooth,domain=0.01:1] plot (\x,{\x*sin(1/(\x/2) r)});
            \end{scope}
		\end{scope}
\end{tikzpicture}

%% file: boundaryV.tikz

\definecolor{cqcqcq}{rgb}{0.75,0.75,0.75}
\begin{tikzpicture}[line cap=round,line join=round,>=triangle 45,x=1.0cm,y=1.0cm, scale=.4]

\clip(-2,-4) rectangle (8.5,8);

//big circle
\draw[fill=black!5](3.16,1.92) circle (5.03cm);

//set V
\draw[purple, line width=0.3mm] plot [smooth cycle] coordinates { (5,-1.5) (5,0.5) (3.6,1.2)(2.8,0.8) (2.2, 0.2) (1,0)  (1,-1.5) (2.1,-2.6) (4,-2.5) (4.6,-2.4)};

\draw[fill=white](3.032,3.74) circle (2.33cm);
\draw[fill=white](6.62,2.02) circle (1.26cm);
\draw[fill=white](0.45,0.66) circle (1.48cm);
\draw[fill=white](2.91,-1.023) circle (0.94cm);
\draw[fill=white](4.98,0.64) circle (0.69cm);
\draw[fill=white](4.98,-1.28) circle (0.71cm);
\draw[fill=white](6.06,4.61) circle (0.58cm);
\draw[fill=white](-0.53,3.67) circle (0.56cm);
\draw[fill=white](1.10,-1.65) circle (0.51cm);
\draw[fill=white](3.83,-2.26) circle (0.43cm);
\draw[fill=white](2.23,-2.57) circle (0.35cm);
\draw[fill=white](2.84,0.79) circle (0.37cm);
\draw[fill=white](3.78,0.30) circle (0.30cm);
\draw[fill=white](3.73,1.17) circle (0.22cm);
\draw[fill=white](6.32,0.10) circle (0.41cm);
\draw[fill=white](6.22,-1.02) circle (0.29cm);
\draw[fill=white](5.66,-0.36) circle (0.20cm);
\draw[fill=white](4.19,-0.31) circle (0.17cm);
\draw[fill=white](4.65,-2.39) circle (0.23cm);
\draw[fill=white](2.19,0.27) circle (0.18cm);
\draw[fill=white](7.30,0.51) circle (0.20cm);
\draw[fill=white](7.,-0.5) circle (0.19cm);
\draw[fill=white](-1.29,2.59) circle (0.26cm);
\draw[fill=white](-1.50,1.82) circle (0.26cm);
\draw[fill=white](-0.24,2.68) circle (0.30cm);
\draw[fill=white](0.15,-1.07) circle (0.16cm);
\draw[fill=white](0.071,-1.61) circle (0.17cm);
\draw[fill=white](-0.11,5.17) circle (0.23cm);
\draw[fill=white](0.15,4.49) circle (0.25cm);
\draw[fill=white](0.64,5.70) circle (0.28cm);
\draw[fill=white](1.38,6.22) circle (0.22cm);
\draw[fill=white](2.27,6.45) circle (0.23cm);
\draw[fill=white](3.19,6.52) circle (0.23cm);
\draw[fill=white](4.32,6.31) circle (0.25cm);
\draw[fill=white](5.10,5.83) circle (0.42cm);
\draw[fill=white](6.88,3.83) circle (0.28cm);
\draw[fill=white](5.75,3.66) circle (0.22cm);

\fill[color=pink, opacity=.5] plot [smooth cycle] coordinates { (5,-1.5) (5,0.5) (3.6,1.2)(2.8,0.8) (2.2, 0.2) (1,0)  (1,-1.5) (2.1,-2.6) (4,-2.5) (4.6,-2.4)};
//grid  
			//grid	
\end{tikzpicture}

%% file: uniformization.tex
\chapter{Uniformization of Sierpi\'nski carpets by square carpets}\label{Chapter:Uniformization}

\section{Introduction}\label{unif:Section Introduction}
In this chapter we prove a uniformization result for planar Sierpi\'nski carpets by square Sierpi\'nski carpets, by minimizing some kind of energy. For the convenience of the reader, we include here some definitions, some of which are also given in Chapter \ref{Chapter:Harmonic}. We will also point out, whenever necessary, any discrepancies in the notation between the two chapters. However, for the most part, this chapter is independent of Chapter \ref{Chapter:Harmonic} and we will only use certain results from there that we quote again here. 

Before proceeding to the results, we mention some important discrepancies in the notation between the two chapters that the reader should be aware of: 
\begin{enumerate}
\item A Sierpi\'nski carpet is denoted here by $S=\br \Omega \setminus \bigcup_{i\in \N}Q_i$, in contrast to Chapter \ref{Chapter:Harmonic}, where the letter $S$ was used to denote a relative Sierpi\'nski carpet $S= \Omega \setminus \bigcup_{i\in \N}Q_i$. 
\item $\partial_{*}$ has slightly different meaning; compare its definition in Section \ref{harmonic:1-Basic Assumptions} to the remarks after Theorem \ref{unif:Solution to free boundary problem}.
\item \textit{Carpet modulus} here is going to be a variant of the strong carpet modulus defined in Section \ref{harmonic:2-Carpet modulus}. Here, we also include the unbounded peripheral disk in the sums, whenever we have a path family in $\C$ and not necessarily in $\Omega$.
\item We will define a notion of modulus called \textit{weak carpet modulus}, which is slightly different than the  weak carpet modulus defined in Section \ref{harmonic:2-Carpet modulus}.
\end{enumerate}

\subsection{Results}
In what follows, all the distances are in the Euclidean metric of $\C$.

A {planar Sierpi\'nski carpet}\index{Sierpi\'nski carpet} $S\subset \C$ is constructed by removing from a Jordan region $\Omega \subset \C$ a countable collection $\{Q_i\}_{i\in \N}$ of open Jordan regions, compactly contained in $\Omega$, with disjoint closures, such that $\diam(Q_i)\to 0$ as $i\to \infty$ and such that $S\coloneqq \br \Omega \setminus \bigcup_{i\in \N} Q_i$ has empty interior. The condition $\diam(Q_i)\to 0$ is equivalent to saying that $S$ is locally connected. According to a fundamental result of Whyburn \cite{Whyburn:theorem} all Sierpi\'nski carpets are homeomorphic to each other. We remark that $S$ is a closed set here, in contrast to Chapter \ref{Chapter:Harmonic}, where $S\subset \Omega$; if we used the notation of Chapter \ref{Chapter:Harmonic}, our carpet here would correspond to $\br S$.   

The Jordan regions $Q_i$, $i\in \N$, are called the \textit{(inner) peripheral disks}\index{peripheral disk!inner} of the carpet $S$, and $Q_0\coloneqq\C\setminus \br \Omega$ is the \textit{outer peripheral disk}\index{peripheral disk!outer}. The Jordan curves $\partial Q_i$, $i\in \N\cup \{0\}$, are the \textit{peripheral circles}\index{peripheral circle} of the carpet. Again, we distinguish between the \textit{inner} peripheral circles\index{peripheral circle!inner} $\partial Q_i$, $i\in \N$, and the outer peripheral circle\index{peripheral circle!outer} $\partial Q_0$. A \textit{square Sierpi\'nski carpet}\index{Sierpi\'nski carpet!square} is a Sierpi\'nski carpet for which $\Omega$ is a rectangle and all peripheral disks $Q_i$, $i\in \N$, are squares such that the sides of $\partial Q_i$, $i\in \N\cup \{0\}$, are parallel to the coordinate axes of $\R^2$.

We say that the peripheral disks $Q_i$ are \textit{uniform quasiballs}\index{quasiball}, if there exists a constant $K_0\geq 1$ such that for each $Q_i$, $i\in \N$, there exist concentric balls 
\begin{align}\label{unif:Quasi-balls}
B(x,r) \subset Q_i \subset B(x,R),
\end{align}
with $R/r\leq K_0$. In this case, we also say that the peripheral disks are $K_0$-quasiballs. We say that the peripheral disks are \textit{uniformly fat sets}\index{fat set} if there exists a constant $K_1>0$ such that for every $Q_i$, $i\in \N$, and for every ball $B(x,r)$ centered at some $x\in Q_i$ with $r<\diam (Q_i)$ we have 
\begin{align}\label{unif:Fat sets}
\mathcal H^2( B(x,r)\cap Q_i) \geq K_1 r^2,
\end{align}
where by $\mathcal H^n$ we denote the $n$-dimensional Hausdorff measure, normalized so that it agrees with the $n$-dimensional Lebesgue measure, whenever $n\in \N$. In this case, the peripheral disks are $K_1$-fat sets. A Jordan curve $J\subset \C$ is a \textit{$K_2$-quasicircle}\index{quasicircle} for some $K_2>0$, if for any two points $x,y\in J$ there exists an arc $\gamma \subset J$ with endpoints $x,y$ such that 
\begin{align}\label{unif:Quasicircle}
|x-y|\leq K_2 \diam (\gamma).
\end{align}
Note that if the peripheral circles $\partial Q_i$ are uniform quasicircles (i.e., $K_2$-quasicircles with the same constant $K_2$), then they are both uniform quasiballs and uniformly fat sets, quantitatively. The first claim is proved in \cite[Proposition 4.3]{Bonk:uniformization} and the second in \cite[Corollary 2.3]{Schramm:transboundary}, where the notion of a fat set appeared for the first time in the study of conformal maps. It is clear that for a square carpet the inner peripheral circles, also called peripheral squares, are uniform quasicircles.

In order to describe our main result we need to introduce a  notion of quasiconformality, suitable for the carpet setting. For this purpose, we introduce \textit{carpet modulus with respect to the carpet} $S$. Let $\Gamma$ be a family of paths in $\C$. A sequence of non-negative numbers $\{\lambda(Q_i)\}_{i\in \N\cup \{0\}}$ is admissible for the \textit{carpet modulus}\index{modulus!carpet modulus} $\md(\Gamma)$ if 
\begin{align}\label{unif:Carpet modulus strong}
\sum_{i:Q_i\cap \gamma\neq \emptyset} \lambda(Q_i)\geq 1
\end{align}
for all $\gamma\in \Gamma$ with $\mathcal H^1(\gamma\cap S)=0$. Then $\md(\Gamma)\coloneqq \inf_\lambda \sum_{i\in \N\cup \{0\}} \lambda(Q_i)^2$ where the infimum is taken over all admissible weights $\lambda$. Because of technical difficulties we also consider a very similar notion of modulus denoted by $\br{\md}(\Gamma)$, which is called \textit{weak carpet modulus}\index{modulus!weak carpet modulus}; see Definition \ref{unif:Carpet-modulus weak}. The notation $\md(\Gamma)$ and $\br{\md}(\Gamma)$ does not incorporate the underlying carpet $S$, but this will be implicitly understood. We now state one of the main results.

\begin{theorem}\label{unif:Main theorem}\index{quasiconformal!discrete}
Let $S\subset \br \Omega$ be a planar Sierpi\'nski carpet of  area zero whose peripheral disks $\{Q_i\}_{i\in \N}$ are uniform quasiballs and uniformly fat sets, and whose outer peripheral circle is $\partial Q_0=\partial \Omega$. Then there exists $D>0$ and a homeomorphism $f\colon  \C \to \C$ such that $f(\br \Omega)=[0,1]\times [0,D]$, and $\mathcal R\coloneqq f(S)\subset [0,1]\times [0,D]$ is a square Sierpi\'nski carpet with inner peripheral squares $\{S_i\}_{i\in \N}$ and outer peripheral circle $\partial S_0\coloneqq \partial ([0,1]\times [0,D])$. Furthermore, for any disjoint, non-trivial continua $E,F\subset S$ and for the family $\Gamma$ of paths in $\C$  that join them we have
\begin{align*}
\br{\md}(\Gamma) \leq \md(f(\Gamma)) \quad \textrm{and} \quad \br{\md}(f(\Gamma)) \leq \md(\Gamma).
\end{align*}
\end{theorem}

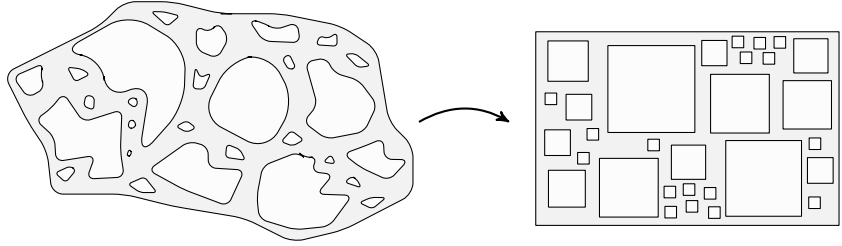
\begin{figure}
	\centering
	\begin{tikzpicture}
	\begin{scope}[scale=.8]
	\clip (-1,-2) rectangle (14,2);
	\input{carpetmap.tikz}	
	\end{scope}
	\end{tikzpicture}
	\caption{Illustration of the uniformizing map in Theorem \ref{unif:Main theorem}.}
\end{figure}
This map is highly non-unique, as will be clear from the construction, since by making some choices one can obtain such maps onto different square carpets. The modulus inequalities in the conclusion say that $f$ is a ``quasiconformal" map in a discrete sense, and in the sense of the so-called geometric definition of quasiconformality, which is also employed in \cite{Rajala:uniformization}.

If the geometric assumptions on the peripheral disks $Q_i$ of $S$ are strengthened, then we obtain a stronger version of the previous theorem with improved regularity for the map $f$. In particular, we consider the following geometric condition:
we say that the peripheral circles $\{\partial Q_i\}_{i\in \N\cup \{0\}}$ of the carpet $S$ are \textit{uniformly relatively separated}\index{uniform relative separation}, if there exists a constant $K_3>0$ such that
\begin{align}\label{unif:Uniformly relatively separated}
\Delta(\partial Q_i,\partial Q_j)\coloneqq \frac{\dist(\partial Q_i,\partial Q_j)}{\min\{ \diam(\partial Q_i),\diam (\partial Q_j)\}} \geq K_3
\end{align}
for all distinct $i,j\in \N\cup \{0\}$. In this case, we say that the peripheral circles are $K_3$-relatively separated. We now have an improvement of the previous theorem:

\begin{theorem}\label{unif:Main theorem-quasisymmetric}
Let $S$ be a Sierpi\'nski carpet of area zero with peripheral circles $\{\partial Q_i\}_{i\in \N\cup \{0\}}$ that are $K_2$-quasicircles and $K_3$-relatively separated. Then there exists an $\eta$-quasi\-symmetric map $f$ from $S$ onto a square Sierpi\'nski carpet $\mathcal R$ such that the distortion function $\eta$ depends only on $K_2$ and $K_3$. Furthermore, the map $f$ maps the outer peripheral circle of $S$ onto the outer peripheral circle of $\mathcal R$.
\end{theorem}

This theorem, together with extension results from \cite[Section 5]{Bonk:uniformization} yield immediately:

\begin{corollary}\label{unif:Main corollary}
Let $S$ be a Sierpi\'nski carpet of  area zero with peripheral circles $\{\partial Q_i\}_{i\in \N\cup \{0\}}$ that are $K_2$-quasicircles and $K_3$-relatively separated. Then there exists a $K$-quasi\-conformal map $f\colon  \widehat\C\to \widehat \C$ such that $\mathcal R\coloneqq f(S)$ is a square Sierpi\'nski carpet, where $K$ depends only on $K_2$ and $K_3$.
\end{corollary}

Theorem \ref{unif:Main theorem-quasisymmetric} and its corollary should be compared to the result of Bonk:

\begin{theorem}[Corollary 1.2, \cite{Bonk:uniformization}]\label{unif:Theorem Bonk}
Let $S$ be a Sierpi\'nski carpet of  area zero whose peripheral circles $\{\partial Q_i\}_{i\in \N\cup \{0\}}$ are $K_2$-quasicircles and they are $K_3$-relatively separated. Then $S$ can be mapped to a round Sierpi\'nski carpet $T\subset \C$ with peripheral circles $\{\partial C_i\}_{i\in \N\cup\{0\}}$ by an $\eta$-quasisymmetric homeomorphism $f\colon S\to T$ that maps the outer peripheral circle $\partial Q_0$ of $S$ to the outer peripheral circle $\partial C_0$ of $T$. The distortion function $\eta$ depends only on $K_2$ and $K_3$, and the map $f$ is unique up to post-composition with M\"obius transformations. 
\end{theorem}
Here a round Sierpi\'nski carpet\index{Sierpi\'nski carpet!round} is a Sierpi\'nski carpet all of whose peripheral circles are round circles.

In our theorem, as already remarked, we do not have such a strong uniqueness statement. Nevertheless, we can obtain a uniqueness result.
\begin{prop}\label{unif:Proposition Uniqueness}\index{Sierpi\'nski carpet!rigidity}
Assume that $f$ is an orientation-preserving quasisymmetry from a Sierpi\'nski carpet $S$ of measure zero onto a square Sierpi\'nski carpet $\mathcal R$ that maps the outer peripheral circle of $S$ to the outer peripheral circle of $\mathcal R$, which is a rectangle $\partial([0,1]\times [0,D])$. Let $\Theta_1= f^{-1}(\{0\}\times[0,D])$ and $\Theta_3= f^{-1}( \{1\}\times [0,D])$ be the preimages of the two vertical sides of $\partial([0,1]\times [0,D])$. Assume also that $g$ is another orientation-preserving quasisymmetry from $S$ onto some other square carpet $\mathcal R'$, whose outer peripheral circle is a rectangle $\partial([0,1]\times [0,D'])$. If $g$ maps $\Theta_1$ onto $\{0\}\times [0,D']$ and $\Theta_3$ onto $\{1\}\times [0,D']$, then $g=f$ and $\mathcal R'=\mathcal R$.
\end{prop}

The proof of this proposition follows from \cite[Theorem 1.4]{BonkMerenkov:rigidity}, or from Theorem \ref{harmonic:Rigidity} in Chapter \ref{Chapter:Harmonic} using the theory of carpet-harmonic functions.

Partial motivation for the current work was the desire to understand necessary conditions for the geometry of the peripheral disks of a carpet in order to obtain a quasisymmetric uniformization result. If one aims to obtain a uniformization result by round or square carpets, then a necessary assumption is that one of uniform quasicircles, since this quality is preserved under quasisymmetries, and both circles and squares share it. Hence, the question is whether the condition of uniform relative separation can be relaxed. However, the next  result implies that the condition is optimal in some sense.

\begin{prop}\label{unif:Proposition Equivalence of square carpet}
A square Sierpi\'nski carpet $\mathcal R$ of  area zero is quasisymmetrically equivalent to a round Sierpi\'nski carpet $T$ if and only if its peripheral circles are uniformly relatively separated. Conversely, a round Sierpi\'nski carpet $T$ of  area zero is quasisymmetrically equivalent to a square Sierpi\'nski carpet $\mathcal R$ if and only if its peripheral circles are uniformly relatively separated.
\end{prop}

Here, $S\subset \C$ is \textit{quasisymmetrically equivalent} to $T\subset \C$ if there exists a quasisymmetry $f\colon S\to T$.

Bonk's uniformizing map in Theorem \ref{unif:Theorem Bonk} is constructed as a limit of conformal maps of finitely connected domains onto finitely connected circle domains, using Koebe's uniformization theorem. The first finitely connected domains converge in some sense to the carpet $S$ and the latter finitely connected circle domains converge to a round carpet. Then, modulus estimates are used to study the properties of the limiting map, and show that it is a quasisymmetry.

In our approach, we construct an ``extremal" map from the carpet to a square carpet by working directly on the carpet, without employing a limiting argument based on finitely connected domains. This has the advantage that we can control the necessary assumptions on the peripheral circles at each step of the construction. Although we impose the assumptions of uniformly fat, uniform quasiballs for the peripheral disks $\{Q_i\}_{i\in \N}$, one can actually obtain Theorem \ref{unif:Main theorem} even if no assumptions at all are imposed on finitely many peripheral circles. To simplify the treatment we do not pursue this here, but we support our claim by remarking that no assumptions are imposed on the outer peripheral circle $\partial Q_0$, in the statement of Theorem \ref{unif:Main theorem}.

\subsection{Organization of the chapter}
In Section \ref{unif:Section Notation} we introduce some notation, and in Section \ref{unif:Section Preliminaries} we discuss preliminaries on quasisymmetric maps, quasiconformal maps, modulus, and exceptional families of paths. The latter will be used extensively throughout the chapter. 

The proof of the Main Theorem \ref{unif:Main theorem} will be given in Sections \ref{unif:Section the function u}--\ref{unif:Section Regularity}. First, in Section \ref{unif:Section the function u} we introduce the real part $u$ of the uniformizing map $f$. To obtain the map $u$ we use the theory of carpet-harmonic functions, developed in Chapter \ref{Chapter:Harmonic}. In fact, $u$ will be a solution to a certain boundary value problem. In Section \ref{unif:Section Level sets} we study the geometry of the level sets of the function $u$. The maximum principle is used in combination with an upper gradient inequality to deduce that almost every level set of $u$ is the intersection of a curve with the carpet $S$, and in fact this intersection has  Hausdorff $1$-measure zero. The latter is the most technical result of the section.

By ``integrating" the ``gradient" of $u$ along each level set, we define the conjugate function $v$ of $u$ in Section \ref{unif:Section Conjugate}. The proof of continuity and regularity properties of $v$ occupies the section. In Section \ref{unif:Section Definition of F} we define the uniformizing function $f\coloneqq (u,v)$ that maps the carpet $S$ into a rectangle $[0,1]\times [0,D]$  and the peripheral circles into squares. We prove that the squares are disjoint and ``fill up" the entire rectangle $[0,1]\times [0,D]$. In Section \ref{unif:Section Injectivity}  the injectivity of $f$ is established and one of the main lemmas in the section is to show that $f$ cannot ``squeeze" a continuum $E\subset S$ to a single point. Finally, in Section \ref{unif:Section Regularity} we prove regularity properties for $f$ and $f^{-1}$ and, in particular, the properties claimed in the Main Theorem \ref{unif:Main theorem}.

Theorem \ref{unif:Main theorem-quasisymmetric} is proved in Section \ref{unif:Section Quasisymmetric} with the aid of Loewner-type estimates for carpet modulus that are quoted in Section \ref{unif:Section Modulus Estimates} and were proved (in some other form) by Bonk in \cite{Bonk:uniformization}. Proposition \ref{unif:Proposition Equivalence of square carpet} is proved in Section \ref{unif:Section Square Round carpets}. Finally, in Section \ref{unif:Appendix} we construct a ``test function" that will be used frequently in variational arguments against the carpet-harmonic function $u$.

\section{Notation and terminology}\label{unif:Section Notation}
We denote $\widehat \R=\R \cup \{-\infty,+\infty\}$, $\widehat \C= \C \cup \{\infty\}$. A function that attains values in $\widehat \R$ is called an extended function. We use the standard open ball notation $B(x,r)=\{y\in \C: |x-y|<r\}$ and $\br B(x,r)$ is the closed ball. If $B=B(x,r)$ then $cB=B(x,cr)$. Also, $A(x;r,R)$ denotes the annulus $B(x,R)\setminus \br B(x,r)$, for $0<r<R$. All the distances will be in the Euclidean distance of $\C\simeq \R^2$. A point $x$ will denote most of the times a point of $\R^2$ and rarely we will use the notation $(x,y)$ for coordinates of a point in $\R^2$, in which case $x,y\in \R$. Each case will be clear from the context. 

The notation $V\subset \subset \Omega$ means that $\br V$ is compact and is contained in $\Omega$. For a set $K\subset \C$ and $\varepsilon >0$ we use the notation 
$$N_\varepsilon(K)=\{x\in \C: \dist(x,K)<\varepsilon\}$$
for the open $\varepsilon$-neighborhood of $K$. The symbols $\br V$, $\inter(V)$, and $\partial V$ denote the closure, interior and boundary, respectively, of a set $V$ with respect to the planar topology. If the reference space is a different set $U$, then we will write instead \textit{the closure of $V$ rel.\ $U$}, or use subscript notation $\inter_{U}(V)$ etc.

A \textit{path}\index{path} or \textit{curve}\index{curve} $\gamma$ is a continuous function $\gamma \colon I\to \C$, where $I$ is a bounded interval such that $\gamma$ has a continuous extension $\br \gamma:\br I\to \C$, i.e., $\gamma$ has endpoints. An \textit{open path} $\gamma$ is a path with $I=(0,1)$. We will also use the notation $\gamma \subset \C$ for the image of the path as a set. A \textit{subpath} of a path $\gamma:I\to \C$ is the restriction of $\gamma$ to a subinterval of $I$. A \textit{path $\gamma$ joins} two sets $E,F\subset \C$ if $\br \gamma \cap E\neq \emptyset$ and $\br \gamma\cap F\neq \emptyset$. More generally, a connected set $\alpha\subset \C$ joins two sets $E,F\subset \C$ if $\br \alpha \cap E\neq \emptyset$ and $\br \alpha\cap F\neq \emptyset$. A Jordan curve\index{Jordan curve} is a homeomorphic image of the unit circle $S^1$, and a Jordan arc\index{Jordan arc} is a homeomorphic image of $[0,1]$. Jordan curves and Jordan arcs are \textit{simple} curves, i.e., they have no self-intersections. 

We denote by $S^\circ$ the points of the Sierpi\'nski carpet $S$ that do not lie on any peripheral circle $\partial Q_i$ or on $\partial \Omega$. For a set $V$ that intersects a Sierpi\'nski carpet $S$ with (inner) peripheral disks $\{Q_i\}_{i\in \N}$ define $I_V=\{i\in \N: Q_i\cap V\neq \emptyset\}$.

In the proofs we will denote constants by $C,C',C'',\dots$, where the same symbol might  denote a different constant if there is no ambiguity. For visual purposes, the closure of a set $U_1$ is denoted by $\br U_1$, instead of $\br {U_1}$.

\section{Preliminaries}\label{unif:Section Preliminaries}

\subsection{Quasisymmetric and quasiconformal maps}

A map $f\colon X\to Y$ between two metric spaces $(X,d_X)$ and $(Y,d_Y)$ is an \textit{$\eta$-quasisymmetry}\index{quasisymmetry} if it is a homeomorphism and there exists an increasing homeomorphism $\eta\colon (0,\infty)\to (0,\infty)$ such that for all triples of distinct points $x,y,z\in X$ we have
\begin{align*}
\frac{d_Y(f(x),f(y))}{d_Y(f(x),f(z))} \leq \eta \left(\frac{d_X(x,y)}{d_X(x,z)} \right).
\end{align*}
The homeomorphism $\eta$ is also called the distortion function associated to the quasisymmetry $f$. If $X=Y=\C$, then it is immediate to see that the quasisymmetric property of $f\colon \C \to \C$ and the distortion function $\eta$ are not affected by compositions with M\"obius transformations of $\C$ (i.e., homotheties). A Jordan curve $J\subset \C$ is a quasicircle in the sense of \eqref{unif:Quasicircle} if and only if there exists a quasisymmetry $f\colon S^1 \to J$. The quasicircle constant and the distortion function of $f$ are related quantitatively. See \cite[Chapters 10--11]{Heinonen:metric} for background on quasisymmetric maps.  

Let $U,V\subset \C$ be open sets. An orientation-preserving homeomorphism $f\colon U\to V $ is is \textit{$K$-quasiconformal}\index{quasiconformal} for some $K>0$ if $f\in W^{1,2}_{\loc}(U)$ and
\begin{align*}
\|Df(z)\|^2\leq K J_f(z)
\end{align*}   
for a.e.\ $z\in U$. An orientation-preserving homeomorphism $f\colon  \widehat{\C} \to \widehat{\C}$ is $K$-quasiconformal if $f$ is $K$-quasiconformal in local coordinates as a map between planar open sets, using the standard conformal charts of $\widehat{\C}$. We direct the reader to \cite[Section 3]{AstalaIwaniecMartin:quasiconformal} for more background.

\subsection{Modulus}
We recall the definition of \textit{conformal modulus}\index{modulus!conformal modulus} or \textit{$2$-modulus}\index{modulus!$2$-modulus}. Let $\Gamma$ be a path family in $\C$. A non-negative Borel function $\lambda\colon \C \to \R$ is \textit{admissible} for the conformal modulus $\md_2(\Gamma)$ if 
\begin{align*}
\int_\gamma \lambda \, ds\geq 1
\end{align*}
for all locally rectifiable paths $\gamma\in \Gamma$. If a path $\gamma$ is not locally rectifiable, we define $\int_\gamma \lambda \, ds =\infty$, even when $\lambda \equiv 0$. Hence, we may require the above inequality for all $\gamma\in \Gamma$. Then $\md_2(\Gamma)\coloneqq \inf_\lambda \int \lambda^2 \, d\mathcal H^2$, where the infimum is taken over all admissible functions $\lambda$.

Let us mention a connection between conformal modulus and the carpet modulus defined in the introduction. Let $S\subset \br \Omega$ be a carpet with (inner) peripheral disks $\{Q_i\}_{i\in \N}$ and outer peripheral disk $Q_0=\C\setminus \br \Omega$. Suppose that $\mathcal H^2(S)=0$ and the peripheral disks are uniformly fat, uniform quasiballs. Consider a family $\Gamma$ of paths contained in $\C$.
\begin{lemma}\label{unif:Zero modulus lemma}
If $\md(\Gamma)=0$ then $\md_2(\Gamma)=0$.
\end{lemma} 
For the proof see Lemma \ref{harmonic:2-weak modulus less than strong} and Lemma \ref{harmonic:2-weak modulus zero implies two modulus zero}. 

If a property (A) holds for all paths $\gamma$ in $\C$ lying outside an exceptional family of $2$-modulus zero, we say that (A) \textit{holds for $\md_2$-a.e.\ $\gamma$}. Furthermore, if a property (A) holds for $\md_2$-a.e.\ path, then $\md_2$-a.e.\ path has the property that all of its subpaths also share property (A). Equivalently, the family of paths having a subpath for which property (A) fails has $2$-modulus zero. The reason is that the family of admissible functions for this curve family contains the admissible functions for the family of curves for which property (A) fails; see also \cite[Theorem 6.4]{Vaisala:quasiconformal}. The same statement is true for the carpet modulus.

A version of the next lemma can be found in \cite[Lemma 2.2]{BonkMerenkov:rigidity} and \cite{Bojarski:inequality}.

\begin{lemma}\label{unif:Bojarski}
Let $\kappa\geq 1$ and $I$ be a countable index set. Suppose that $\{B_i\}_{i\in I}$ is a collection of balls in the plane $\C$, and $a_i$, $i\in I$, are non-negative real numbers. Then there exists a constant $C>0$ depending only on $\kappa$ such that
\begin{align*}
\biggl \| \sum_{i\in I} a_i\x_{\kappa B_i} \biggr\|_2 \leq C \biggl\| \sum_{i\in I} a_i \x_{B_i} \biggr\|_2.
\end{align*}
\end{lemma}
Here $\|\cdot\|_2$ denotes the $L^2$-norm with respect to planar Lebesgue measure.

\subsection{Existence of paths}\label{unif:Section Existence of paths}
Here we mention some results that provide us with paths that avoid exceptional path families. These paths will be very useful in the proof of injectivity of the uniformizing function $f$ in Section \ref{unif:Section Injectivity}, and also in the proof of the regularity of $f^{-1}$ in Section \ref{unif:Section Regularity}. A proof of the next proposition can be found in \cite[Theorem 3]{Brown:distancesets}.

\begin{prop}\label{unif:Paths-distance function}
Let $\beta \subset \C$ be a path that joins two non-trivial continua $E,F \subset \C$. Consider the distance function $\psi(x)=\dist(x,\beta )$. Then there exists $\delta>0$ such that for a.e.\ $s\in (0,\delta)$ there exists a simple path $\beta_s\subset \psi^{-1}(s)$ joining $E$ and $F$.  
\end{prop}

For the next lemma we assume that the carpet $S$ and $\Omega$ are as in the assumptions of Theorem \ref{unif:Main theorem}, so in particular, the peripheral disks of $S$ are uniformly fat, uniform quasiballs.

\begin{lemma}\label{unif:Paths joining continua}
Suppose $\beta\subset \Omega$ is a path joining two non-trivial continua $E,F\subset \br \Omega$. Consider the distance function $\psi(x)=\dist(x,\beta)$ and let $\Gamma$ be a given family in $\Omega$ of carpet modulus or conformal modulus equal to zero. Then, there exists $\delta>0$ such that for a.e.\ $s\in (0,\delta)$ there exists a simple open path $\beta_s\subset \psi^{-1}(s)$ that lies in $\Omega$, joins the continua $E$ and $F$, and lies outside the family $\Gamma$.
\end{lemma}

For a proof see Lemma \ref{harmonic:Paths joining continua} and Lemma \ref{harmonic:Paths boundary}. We also include some topological facts.

\begin{lemma}[Lemma \ref{harmonic:Paths in S^o}]\label{unif:Paths in S^o}
Let $S\subset \C$ be a Sierpi\'nski carpet.
\begin{enumerate}[\upshape(a)]
\item For any $x,y\in S$ there exists an open path $\gamma \subset S^\circ$ that joins $x$ and $y$. Moreover, for each $r>0$, if $y$ is sufficiently close to $x$, the path $\gamma$ can be taken so that $\gamma \subset B(x,r)$.
\item For any two peripheral disks there exists a Jordan curve $\gamma\subset S^\circ$ that separates them. Moreover, $\gamma$ can be taken to be arbitrarily close to one of them.
\end{enumerate}
\end{lemma}
In other words, the conclusion of the first part is that $x,y\in \br \gamma$, but $\gamma$ does not intersect any peripheral circle $\partial Q_i$, $i\in \N\cup \{0\}$. As a corollary of the second part of the lemma, we obtain:
\begin{corollary}\label{unif:cor:paths in S^o}
Let $S\subset \C$ be a Sierpi\'nski carpet and $\gamma\subset \C$ be a path that connects two distinct peripheral disks of $S$. Then $\gamma$ has to intersect $S^\circ$. 
\end{corollary}

We finish the section with a technical lemma:

\begin{lemma}\label{unif:lemma:paths zero hausdorff}
Let $S\subset \C$ be a Sierpi\'nski carpet and $\gamma \subset \C$ be a non-constant path with $\mathcal H^1(\gamma\cap S)=0$. 
\begin{enumerate}[\upshape(a)]
\item If $x\in \gamma\cap S^\circ$, then arbitrarily close to $x$ we can find peripheral disks $Q_i$ with $Q_i\cap \gamma\neq \emptyset$. 
\item If $\gamma$ is an open path that does not intersect a peripheral disk $ Q_{i_0}$, $i_0\in \N\cup \{0\}$, and $x\in \br \gamma\cap \partial Q_{i_0}$, then arbitrarily close to $x$ we can find peripheral disks $Q_i$, $i\neq i_0$, with $Q_i\cap \gamma\neq \emptyset$. 
\end{enumerate}
\end{lemma}
The proof is the same as the proof of Lemma \ref{harmonic:lemma:paths zero hausdorff}, which contains the analogous statement for relative Sierpi\'nski carpets.

\section{The function \texorpdfstring{$u$}{u}}\label{unif:Section the function u}
From this section until Section \ref{unif:Section Regularity} the standing assumptions are that we are given a carpet $S\subset \br \Omega$ of  area zero with peripheral disks $\{Q_i\}_{i\in \N}$ that are uniformly fat, uniform quasiballs, and with outer peripheral circle $\partial Q_0=\partial \Omega$. This is precisely the setup in Chapter \ref{Chapter:Harmonic}, where the theory of carpet-harmonic functions is developed. We will use this theory in order to define the real part $u$ of the uniformizing function $f$.

\subsection{Background on carpet-harmonic functions}
Here we include some definitions and background on carpet-harmonic functions. More details can be found in Chapter \ref{Chapter:Harmonic}. 

\begin{definition}\label{unif:Definition Upper gradient}
Let $g\colon S\cap \Omega\to \widehat \R$ be an extended function. We say that the sequence of non-negative weights $\{\lambda(Q_i)\}_{i\in \N}$ is an \textit{upper gradient}\index{upper gradient} for $g$ if there exists an exceptional family $\Gamma_0$ of paths in ${\Omega}$ with $\md(\Gamma_0)=0$ such that for all paths $\gamma \subset \Omega$ with $\gamma \notin \Gamma_0$ and $x,y\in \gamma\cap S$ we have $g(x),g(y)\neq \pm\infty$ and
\begin{align*}
|g(x)-g(y)|\leq \sum_{i:Q_i\cap \gamma\neq \emptyset}\lambda (Q_i).
\end{align*}
\end{definition}

Note that this definition differs from the classical definition of upper gradients in metric spaces, treated, for example, in \cite{HeinonenKoskelaShanmugalingamTyson:Sobolev}. Here, the presence of ambient space is important, since ``most" of the paths do not lie in the carpet $S$, but meet infinitely many peripheral disks $Q_i$. We also remark that our notation here differs slightly from Definition \ref{harmonic:3-Def Sobolev space}, where a \textit{relative carpet} $S$ does not contain $\partial \Omega$, and functions are defined on $S$.  Since here $S$ contains $\partial \Omega$, we write $S\cap \Omega$ here as the domain of $g$.

For a function $g\colon S\cap \Omega \to \widehat \R$ and a peripheral disk $Q_i$, $i\in \N$, we define 
\begin{align*}
M_{Q_i}(g)&\coloneqq \sup_{x\in \partial Q_i}g(x),\\
m_{Q_i}(g)&\coloneqq \inf_{x\in \partial Q_i}g(x), \textrm{ and}\\
\osc_{Q_i}(g)&\coloneqq  M_{Q_i}(g)-m_{Q_i}(g)
\end{align*}
whenever the latter makes sense. Note here that we do \textit{not} define the above quantities for the outer peripheral circle $\partial Q_0 =\partial \Omega$, which is regarded as the ``boundary" of the carpet.

\begin{definition}\label{unif:Definition Sobolev function}\index{Sobolev space}
Let $g\colon S\cap \Omega \to \widehat \R$ be an extended function. We say that $g$ lies in the \textit{Sobolev space} $\mathcal W^{1,2}_{\loc}(S)$ if for every ball $B\subset\subset \Omega$ we have 
\begin{align}
\label{unif:Definition Sobolev L2 integrability}
&\sum_{i\in I_B} M_{Q_i}(g)^2 \diam(Q_i)^2<\infty,\\
\label{unif:Definition Sobolev L2 gradient}
&\sum_{i\in I_B } \osc_{Q_i}(g)^2 <\infty,
\end{align}
and $\{\osc_{Q_i}(g)\}_{i\in \N}$ is an upper gradient for $g$. If the the above conditions hold for the full sums over $i\in \N$ then we say that $g$ lies in the Sobolev space $\mathcal W^{1,2}(S)$.
\end{definition}

Recall here that $I_B= \{i\in \N: Q_i\cap B\neq \emptyset\}$. Part of the definition is that $\osc_{Q_i}(g)$ is defined for all $i\in \N$, and in particular $M_{Q_i}(g),m_{Q_i}(g)$ are finite. The space $\mathcal W^{1,2}(S)$ contains Lipschitz functions on $S$, and also coordinate functions of restrictions on $S$ of quasiconformal maps $g\colon  \C \to \C$; see Section \ref{harmonic:examples}.

For a set $V\subset \Omega$ and $g\in \mathcal W^{1,2}_{\loc}(S)$ define the \textit{Dirichlet energy functional}\index{Dirichlet energy}
\begin{align*}
D_V(g)= \sum_{i\in I_V} \osc_{Q_i}(g)^2 \in [0,\infty].
\end{align*}
We remark here that the outer peripheral disk $Q_0$ is never used in Dirichlet energy calculations, and the summations are always over subsets of $\{Q_i\}_{i\in \N}$. If $V=\Omega$ we will often omit the subscript and write $D(g)$ instead of $D_\Omega(g)$.

\begin{definition}\label{unif:Definition Carpet harmonic}
A function $u\in \mathcal W^{1,2}_{\loc} (S)$ is \textit{carpet-harmonic}\index{carpet-harmonic} if for every open set $V \subset \subset \Omega$ and every $\zeta\in \mathcal W^{1,2}(S)$ that is supported on $V$ we have
\begin{align*}
D_V(u)\leq D_V(u+\zeta).
\end{align*}
\end{definition}

For each $g\in \mathcal W^{1,2}(S)$ there exists a family of \textit{good paths}\index{good paths} $\mathcal G$ in ${\Omega}$ that contains almost every path (i.e., the paths of $\Omega$ that do not lie in $\mathcal G$ have carpet modulus equal to zero) with the following properties
\begin{enumerate}[\upshape(1)]
\item $\mathcal H^1(\gamma\cap S)=0$,
\item $\sum_{i:Q_i\cap \gamma\neq \emptyset} \osc_{Q_i}(g)<\infty$, and
\item the upper gradient inequality as in Definition \ref{unif:Definition Upper gradient} holds along every subpath of $\gamma$.
\end{enumerate}
A point $x\in S$ is \textit{``accessible" by a path}\index{``accessible'' point} $\gamma_0\in \mathcal G$ if there exists an open subpath $\gamma$ of $\gamma_0$ with $x\in \br{\gamma}$ and $\gamma$ does not meet the peripheral disk $Q_{i_0}$ whenever $x\in \partial Q_{i_0}$, $i_0\in \N$; see Figure \ref{harmonic:fig:accessible}. Note that $x$ can lie on $\partial \Omega$. See Section \ref{harmonic:Subsection-Discrete Sobolev} for a more detailed discussion on good paths and ``accessible" points.  


Finally we require a lemma that allows the ``gluing" of Sobolev functions and will be useful for variational arguments; see Proposition \ref{harmonic:3-Properties Sobolev} and Lemma \ref{harmonic:3-Lemma-removability} for the proof.
\begin{lemma}\label{unif:Gluing lemma}
If $\phi,\psi \in \mathcal W^{1,2}(S)$ and $a,b\in \R$, then the following functions also lie in the Sobolev space $\mathcal W^{1,2}(S)$:
\begin{enumerate}[\upshape(a)]
\item $a\phi+b\psi$, with $\osc_{Q_i}(a\phi+b\psi)\leq |a|\osc_{Q_i}(\phi)+|b|\osc_{Q_i}(\psi)$,
\item $|\phi|$, with $\osc_{Q_i}(|\phi|)\leq \osc_{Q_i}(\phi)$,
\item $\phi \lor \psi \coloneqq  \max(\phi,\psi)$, with $\osc_{Q_i}(\phi\lor \psi) \leq \max \{ \osc_{Q_i}(\phi),\osc_{Q_i}(\psi)\}$,
\item $\phi \land \psi\coloneqq \min(\phi,\psi)$, with $\osc_{Q_i}(\phi\land \psi) \leq \max \{ \osc_{Q_i}(\phi),\osc_{Q_i}(\psi)\}$, and
\item $\phi\cdot \psi$, provided that $\phi$ and $\psi$ are bounded.
\end{enumerate}
Furthermore, if $V\subset \Omega$ is an open set with $ S\cap\Omega\cap \partial V \neq \emptyset$, and $\phi=\psi$ on $S\cap\Omega\cap \partial V$, then $\phi\x_{S\cap V} + \psi\x_{S\setminus V} \in \mathcal W^{1,2}(S)$.
\end{lemma}

\subsection{The free boundary problem}\label{unif:Section:Free boundary problem}\index{Dirichlet problem!free boundary problem}\index{free boundary problem}
We mark four points on $\partial \Omega$ that determine a quadrilateral, i.e., a homeomorphic image of a rectangle, with closed sides $\Theta_1,\dots, \Theta_4$, enumerated in a counter-clockwise fashion. Here $\Theta_1$ and $\Theta_3$ are opposite sides. 

Consider a function $g\in \mathcal W^{1,2}(S)$. Recall from Definition \ref{unif:Definition Sobolev function} that $g$ is only defined in $S\cap \Omega$ and not in $\partial \Omega$. However, one can always define  boundary values of $g$ on $\partial \Omega$; see Section \ref{harmonic:Section Dirichlet} for more details. In this chapter, all functions $g\in \mathcal W^{1,2}(S)$ that we are going to use will actually be continuous up to $\partial \Omega$, so their boundary values are unambiguously defined and we do not need to resort to the theory of Chapter \ref{Chapter:Harmonic}. If $g(x)=0$ for all points $x \in \Theta_1$ and also $g(x)=1$ for all  points $x\in \Theta_3$, we say that $g$ is \textit{admissible (for the free boundary problem)}.

\begin{theorem}[Theorem \ref{harmonic:Free boundary problem}]\label{unif:Solution to free boundary problem}
There exists a unique carpet-harmonic function $u\colon S\to \R$ that minimizes the Dirichlet energy $D_\Omega(g)$ over all admissible functions $g\in \mathcal W^{1,2}(S)$. The function $u$ is continuous up to the boundary $\partial \Omega$ and has boundary values $u=0$ on $\Theta_1$ and $u=1$ on $\Theta_3$.
\end{theorem} 

For an open set $V\subset \C\setminus (\Theta_1\cup\Theta_3)$ define
\begin{align*}
\partial_* V\coloneqq  \partial V \cap S.
\end{align*}
The open arcs $\Theta_2,\Theta_4\subset \partial \Omega$ are not considered as boundary arcs for the free boundary problem since there is no boundary data present on them. With this in mind, we now state the maximum principle for the minimizer $u$:

\begin{theorem}[Theorem \ref{harmonic:Free boundary problem - Maximum Principle}]\label{unif:Maximum principle}\index{maximum principle}
Let $V$ be an open set with $V\subset \C\setminus (\Theta_1\cup\Theta_3)$. Then 
\begin{align*}
\sup_{x\in S\cap \br V} u(x)= \sup_{x\in \partial_* V}u(x) \quad \textrm{and} \quad \inf_{x\in S\cap \br V} u(x)= \inf_{x\in \partial_* V}u(x).
\end{align*}
\end{theorem}

The standard maximum principle would state that the extremal values of $u$ on an open set $V$ are attained on $\partial V$. Our stronger statement states that the extremal values could be attained at the part of $\partial V$ that is disjoint from the interiors of the ``free" arcs $\Theta_2$ and $\Theta_4$. However, extremal values could still be attained at $\Theta_1$ or $\Theta_3$, and this is the reason that we look at sets $V\subset \C\setminus (\Theta_1\cup \Theta_3)$.

Next, we consider a variant of Lemma \ref{harmonic:Harnack-Oscillation}, whose proof follows immediately from an application of the upper gradient inequality, together with the maximum principle.

\begin{lemma}\label{unif:Maximum principle circular arcs}
Consider a ball $B(x,r)\subset \Omega$, with $B(x,cr)\subset\Omega$ for some $c>1$. Then for a.e.\ $s\in [1,c]$ we have
\begin{align*}
\diam (u(B(x,r)\cap S)) \leq \diam (u(B(x,sr)\cap S))\leq \sum_{i: Q_i\cap \partial B(x,sr)\neq \emptyset} \osc_{Q_i}(u).
\end{align*}
\end{lemma}

The function $u$ will be the real part of the uniformizing function $f$. It will be very convenient to have a continuous extension of $u$ to $\br \Omega$ that satisfies the maximum principle:

\begin{prop}\label{unif:Maximum principle for tilde u}
There exists a continuous extension $\tilde u\colon  \br\Omega\to \R$ of $u$ such that for every open set $V\subset \C\setminus (\Theta_1\cup\Theta_3)$ we have
\begin{align*}
\sup_{x\in \br V\cap \br \Omega} \tilde u(x)= \sup_{x\in \partial V \cap \br \Omega}\tilde u(x) \quad \textrm{and} \quad \inf_{x\in \br V\cap \br \Omega} \tilde u(x)= \inf_{x\in \partial  V \cap \br \Omega}\tilde u(x).
\end{align*}
In fact, $\tilde u$ can be taken to be harmonic in the classical sense inside each peripheral disk $Q_i$, $i\in \N$.
\end{prop}
\begin{proof}
For each peripheral disk $Q_i$, $i\in \N$, we consider the Poisson extension $\tilde u\colon \br Q_i \to \R$ of $u$. This is obtained by mapping conformally the Jordan region $Q_i$ to the unit disk $\D$ and taking the Poisson extension there. The function $\tilde u$ is harmonic in $Q_i$ and continuous up to the boundary $\partial Q_i$. Furthermore, 
\begin{align}\label{unif:Maximum principle for tilde u- oscillation}
\diam( \tilde u(Q_i)) = \diam(\tilde u(\partial Q_i)) =\osc_{Q_i}(u),
\end{align}
where the latter is defined after Definition \ref{unif:Definition Upper gradient}.

To show that the extension $\tilde u\colon  \br \Omega \to \R$ is continuous, we argue by contradiction and suppose that there exists a sequence $\{x_n\}_{n\in \N} \subset \br \Omega $ with $x_n\to x\in \br{\Omega}$, but $|\tilde u(x_n)- \tilde u(x)|\geq \varepsilon $ for some $\varepsilon>0$ and all $n\in \N$. If $x_n\in S $ for infinitely many $n$, then we obtain a contradiction, by the continuity of $u$ in $S$. If $x_n$ lies in some peripheral disk $\br Q_{i_0} $ for infinitely many $n$ then we also get a contradiction, by the continuity of $\tilde u$ on $\br Q_{i_0}$. We, thus, assume that $x_n \in \br Q_{i_n} $ where $Q_{i_n}$ are distinct peripheral disks, so we necessarily have $x\in S$. Let $y_n\in \partial Q_{i_n}$. By the local connectedness of $S$ we have $\diam(Q_{i_n})\to 0$, thus $y_n\to x$ and $u(y_n)\to u(x)$. Since $\sum_{i\in \N} \osc_{Q_i}(u)^2<\infty$ it follows that $\osc_{Q_{i_n}}(u) \to 0$. Combining these with \eqref{unif:Maximum principle for tilde u- oscillation} we obtain
\begin{align*}
\varepsilon\leq  |\tilde u(x_n)-\tilde u(x)|&\leq |\tilde u(x_n)-\tilde u(y_n)| + |u(y_n)-u(x)|\leq \osc_{Q_{i_n}}(u) + |u(y_n)-u(x)|
\end{align*}
Letting $n\to\infty$ yields, again, a contradiction.

Finally, we check the maximum principle. Trivially, we have 
$$\sup_{x\in \br V\cap \br \Omega}\tilde u(x)\geq \sup_{x\in \partial V\cap \br \Omega } \tilde u(x) \eqqcolon M,$$
so it suffices to show the reverse inequality. If there exists $z\in V\cap \br \Omega$ with $\tilde u(z)>M$, then we claim that  there actually exists $w\in S\cap V$ with $u(w)=\tilde u(w)>M$. We assume this for the moment. Since $\partial V \cap \br \Omega \supset \partial_*V$, we have
\begin{align*}
\sup_{x\in S\cap \br{V}} u(x) \geq u(w) >M=\sup_{x\in \partial V\cap \br \Omega }\tilde u(x) \geq \sup_{x\in \partial_*V} \tilde u(x)=  \sup_{x\in \partial_*V} u(x),
\end{align*}
which contradicts the maximum principle in Theorem \ref{unif:Maximum principle}. The statement for the infimum is proved similarly.

We now prove our claim. If $z\in S$ then we set $w=z$ and there is nothing to show, so we assume that $z\in Q_{i_0}$ for some $i_0\in \N$. The maximum principle of the harmonic function $\tilde u\big|_{Q_{i_0}}$ implies that there exists 
$$w\in \partial (Q_{i_0}\cap V) \subset (\partial Q_{i_0}\cap V) \cup \partial V \subset (S\cap V)\cup \partial V$$
with $\tilde u(w)>M$; here it is crucial that $Q_{i_0}$ and $V$ are open sets. However, we cannot have $w\in \partial V$, since $\tilde u(w)> M=\sup_{x\in \partial V\cap \br \Omega} \tilde u(x)$. It follows that $w\in S\cap V$, as desired.
\end{proof}

\section{The level sets of \texorpdfstring{$u$}{u}}\label{unif:Section Level sets}\index{level sets}
We study the level sets of $u$ and of its extension $\tilde u$. One of our goals is to show that for a.e.\ $t\in [0,1]$ the level set $\tilde u^{-1}(t)$ is a simple curve that joins $\Theta_2$ to $\Theta_4$. Using these curves we will define the ``conjugate"  function $v$ of $u$ in the next section.

For $0\leq s<t\leq 1$ define 
\begin{align*}
\quad A_{s,t}= \tilde u^{-1}((s,t))
\end{align*}
and for $0\leq t\leq 1$ define $\alpha_t=\tilde u^{-1}(t)$.
For these level sets we have:

\begin{prop}\label{unif:Level sets}
For all $0\leq s<t\leq 1$ the sets $\alpha_t$ and $A_{s,t}$ are connected, simply connected and they join the sides $\Theta_2$ and $\Theta_4$ of the quadrilateral $\Omega$. Furthermore, the intersections of $\alpha_t$, $ A_{s,t}$ with $\Theta_2,\Theta_4$ are all connected. Finally, $\alpha_t$ does not separate the plane and $\br A_{s,t}$ does not separate the plane if $\alpha_t$ and $\alpha_s$ have empty interior.
\end{prop}

This is proved in the same way as \cite[Lemma 6.3]{Rajala:uniformization}, but we include a sketch of it here for the sake of completeness.

\begin{proof}
We prove the statement for the set $A\coloneqq  A_{s,t}$, which is rel.\ open in $\br{\Omega}$. The claims for $\alpha_t$ are proved very similarly, observing also that $\alpha_t= \bigcap_{h>0} A_{t-h,t+h}$ for $0<t<1$.

We first show that each component $V$ of $A$ is simply connected. If there exists a simple loop $\gamma\subset V$ that is not null-homotopic in $V$, then $\gamma$ bounds a region $W$ that is not contained in $ A$. However, $\partial W\cap \br{\Omega}=\gamma$ is contained in $ A$, so we have $s<\tilde u(x)<t$ for all $x\in \partial W\cap \br \Omega$. The maximum principle in Theorem \ref{unif:Maximum principle for tilde u} implies that this also holds in $W$, a contradiction.

Let $V$ be a component of $A$. Then $V$ has to intersect at least one of the sides $\Theta_2$ and $\Theta_4$. Indeed, if this was not the case, then on the connected set $\partial V \subset \Omega$ we would either have $\tilde u\equiv s$ or $\tilde u \equiv t$. The maximum principle in Theorem \ref{unif:Maximum principle for tilde u} implies that $\tilde u$ is a constant on $V$, equal to either $s$ or $t$, but this clearly contradicts the fact that $V\subset A=\tilde u^{-1}((s,t))$. Without loss of generality we assume that $V \cap \Theta_2\neq \emptyset$.

The intersection $V\cap \Theta_2$ must be a connected set. Indeed, if this failed, then we would be able to find a simple arc $\gamma \subset V$ that connects two distinct components of $V\cap \Theta_2$. Since $\tilde u \in (s,t)$ on $\gamma$, it follows again by the maximum principle that the same is true in the region bounded by $\gamma$ and $\Theta_2$. This is a contradiction. Note that here the maximum principle is applied to an open set $W\subset \C \setminus (\Theta_1\cup \Theta_3)$ bounded by the concatenation of $\gamma$ with an arc $\beta\subset \C\setminus \br \Omega$ that connects the endpoints of $\gamma$.

Our next claim is that $V$ intersects $\Theta_4$. We argue by contradiction, so suppose that the boundary of $V$ rel.\ $\br \Omega$ consists of a single component $Y$. On $Y$ we must have $\tilde u\equiv s$ or $\tilde u\equiv t$ and only one of them is possible by the connectedness of $Y$. In either case, $\tilde u$ would have to be constant in $V$ by the maximum principle, and this is a contradiction as in the previous paragraph.

Suppose now there exist two distinct components $V_1,V_2 \subset A$. Since both of them separate the sides $\Theta_1$ and $\Theta_3$, there exists some $x\in \br\Omega \setminus A$ ``between" $V_1$ and $V_3$, i.e., $A$ separates $x$ from both $\Theta_1$ and $\Theta_3$. The maximum principle applied to the region containing $x$, bounded by parts of the boundaries of $V_1$ and $V_3$, is again contradicted. 

For our final claim, suppose that $\C\setminus \br A$ has a bounded component $V$. Note that $V$ is simply connected by the connectedness of $\br A$, and hence the boundary $\partial V$ of $V$ is connected. The set $\partial V$ cannot contain an arc of $\partial \Omega$. Indeed, otherwise we would be able to connect $V$ with the unbounded component of $\C\setminus \br A$ outside $\br A$, a contradiction. Hence,  $\partial V\cap \partial \Omega$ is a totally disconnected set, and each point of $\partial V\cap \partial \Omega$ can be approximated by points in $\partial V\cap \Omega$. On each component of $\partial V \setminus \partial \Omega \subset \partial A\cap \Omega$ we necessarily have $\tilde u\equiv s$ or $\tilde u\equiv t$. By continuity it follows that on each point of $\partial V\cap \partial \Omega$ the function $\tilde u$ has the value $s$ or $t$. Since $\partial V$ is connected, we have $\tilde u\equiv s$ or $\tilde u\equiv  t$ on $\partial V$. The maximum principle implies that $V\subset \alpha_s$ or $V\subset \alpha_t$, but this contradicts the assumption that the level sets $\alpha_s$ and $\alpha_t$ have empty interior.  
\end{proof}

Next, we show:

\begin{theorem}\label{unif:Level sets-sums equal mass}
For a.e.\ $t\in (0,1)$ we have
\begin{align*}
\sum_{i:Q_i\cap \alpha_t\neq \emptyset} \osc_{Q_i}(u) =  D  (u) =\sum_{i\in \N}\osc_{Q_i}(u)^2.
\end{align*}
\end{theorem}

The proof will follow from Propositions \ref{unif:Level sets-sums geq} and \ref{unif:Level sets-sums leq} below.

\begin{prop}\label{unif:Level sets-sums geq}
For a.e.\ $t\in (0,1)$ we have
\begin{align*}
\sum_{i:Q_i\cap \alpha_t\neq \emptyset} \osc_{Q_i}(u) \geq  D  (u).
\end{align*}
\end{prop}
\begin{proof}
Let $t\in (0,1)$ be a value that is not the maximum or minimum value of $u$ on any $\partial Q_i$, $i\in \N$. There are countably many such values that we exclude. Then $Q_i\cap \alpha_{t\pm h}\neq \emptyset$ for all small $h>0$, whenever $Q_i\cap \alpha_t\neq \emptyset$. Indeed, if $x\in Q_i\cap \alpha_t$, and $\tilde u$ is non-constant in $Q_i$, then by harmonicity $\tilde u$ must attain, near $x$, values larger than $t$ and smaller than $t$. We fix a small $h>0$ and define $F_h$ to be the family of indices $i\in \N$ such that $Q_i\cap \alpha_{t+ h}\neq \emptyset$ and $Q_i\cap \alpha_{t- h}\neq \emptyset$. Note that $F_h$ is contained in $\{i\in \N: Q_i\cap \alpha_t\neq \emptyset\}$ by the connectedness of $Q_i$ and the continuity of $u$. In fact, by our previous remark, $F_h$ increases to $\{i\in \N: Q_i\cap \alpha_t\neq \emptyset\}$ as $h\to 0$. Hence, we have 
\begin{align}\label{unif:Level sets-sums F_h}
\sum_{i\in F_h} \osc_{Q_i}(u) \to \sum_{i:Q_i\cap \alpha_t\neq \emptyset} \osc_{Q_i}(u)
\end{align} 
as $h\to 0$. Also, define $N_h= \{i\in \N : Q_i\cap A_{t-h,t+h}\neq \emptyset\}\setminus F_h$.

Now, consider the function
\begin{align*}
g(x)= \begin{cases}
0, & u(x)\leq t-h \\
\frac{u(x)-(t-h)}{2h}, & t-h<u(x)< t+h\\
1, & u(x)\geq t+h .
\end{cases}
\end{align*}
The function $g$ lies in the Sobolev space $\mathcal W^{1,2}(S)$ as follows from Lemma \ref{unif:Gluing lemma}, and furthermore we have
\begin{align*}
\osc_{Q_i}(g) \begin{cases}
=1, & i\in F_h\\
\leq \osc_{Q_i}(u)/2h , & i\in N_h\\
=0, & i\notin F_h\cup N_h.
\end{cases}
\end{align*}
Since $g=0$ on $\Theta_1$ and $g=1$ on $\Theta_3$, the function $g$ is admissible for the free boundary problem. 

Hence, for all $s\in [0,1]$ the function $(1-s)u+sg$ is also admissible for the free boundary problem, so $D (u)\leq D ((1-s)u+sg)$ by the harmonicity of $u$. Lemma \ref{unif:Gluing lemma}(a) implies that 
\begin{align*}
D (u) &\leq \sum_{i\in \N}(( 1-s) \osc_{Q_i}(u)+s\osc_{Q_i}(g))^2\\
&= (1-s)^2D (u)+ 2(1-s)s \sum_{i\in \N} \osc_{Q_i}(u)\osc_{Q_i}(g) +s^2 D (g).
\end{align*}
This simplifies to
\begin{align*}
D (u) \leq (1-s) \sum_{i\in \N} \osc_{Q_i}(u)\osc_{Q_i}(g) +\frac{s}{2}(D (u)+D (g)).
\end{align*}
Letting $s\to 0$, we obtain 
\begin{align}\label{unif:Level sets-Optimization in s}
D (u) \leq \sum_{i\in \N} \osc_{Q_i}(u)\osc_{Q_i}(g),
\end{align}
thus,
\begin{align}\label{unif:Level sets-Inequality with N_h}
D (u) \leq \sum_{i\in F_h} \osc_{Q_i}(u) +\frac{1}{2h} \sum_{i\in N_h} \osc_{Q_i}(u)^2.
\end{align}
By \eqref{unif:Level sets-sums F_h}, it suffices to prove that $\frac{1}{2h} \sum_{i\in N_h} \osc_{Q_i}(u)^2\to 0$ as $h\to 0$, for a.e.\ $t\in [0,1]$. This will follow from the next lemma.

\begin{lemma}\label{unif:Level sets-Measure Derivative lemma}
Let $\{h(Q_i)\}_{i\in \N}$ be a sequence of non-negative numbers that is summable, i.e., $\sum_{i\in \N}h(Q_i)<\infty$. For each $i\in \N$ consider points $p_i,q_i\in \partial Q_i$ such that $u(p_i)= m_{Q_i}(u)=\min_{\partial Q_i}u$ and $u(q_i)= M_{Q_i}(u)=\max_{\partial Q_i}u$. Define a measure $\mu$ on $\C$ by 
\begin{align*}
\mu= \sum_{i\in \N} h(Q_i) ( \delta_{p_i}+\delta_{q_i}) ,
\end{align*}
where $\delta_x$ is a Dirac mass at $x$. Then for the pushforward measure $\lambda\coloneqq  u_*\mu$ on $\R$ we have
\begin{align*}
\lim_{h\to 0} \frac{\lambda((t-h,t+h))}{2h} =\lim_{h\to 0} \frac{1}{2h} \sum_{i\in \N}h(Q_i)(\x_{p_i \in A_{t-h,t+h}}+ \x_{q_i \in A_{t-h,t+h}})=0
\end{align*}
for a.e.\ $t\in [0,1]$.
\end{lemma}
The proof of the lemma is an immediate consequence of the fact that the measure $\lambda$ has no absolutely continuous part; see \cite[Theorem 3.22, p.~99]{Folland:real}. We now explain how to use the lemma in order to derive that the term $\frac{1}{2h} \sum_{i\in N_h} \osc_{Q_i}(u)^2$ in \eqref{unif:Level sets-Inequality with N_h} converges to $0$ as $h\to 0$, for a.e.\ $t\in [0,1]$.

We first refine our choice of $t\in (0,1)$ such that the conclusion of  Lemma \ref{unif:Level sets-Measure Derivative lemma} is true for $h(Q_i)\coloneqq \osc_{Q_i}(u)^2$. Recall that initially we only excluded countably many values of $t$. If $i\in N_h$ then we have $u(p_i) \in (t-2h,t+2h)$ or $u(q_i)\in (t-2h,t+2h)$. To see this first note that $\partial Q_i\cap A_{t-h,t+h}\neq \emptyset$ in this case by the connectedness of $A_{t-h,t+h}$ from Proposition \ref{unif:Level sets}. Thus $M_{Q_i}(u)=u(q_i) \geq t-h$. Since $i\notin F_h$, without loss of generality assume that $Q_i\cap \alpha_{t+h}=\emptyset$. By continuity, the maximum of $u$ on $\partial Q_i$ cannot exceed $t+h$, so $u(q_i)\in [t-h,t+h]\subset (t-2h,t+2h)$. Therefore, $\sum_{i\in N_h} \osc_{Q_i}(u)^2 \leq \lambda ((t-2h,t+2h))$ and this completes the proof.
\end{proof}

\begin{prop}\label{unif:Level sets-sums leq}
For a.e.\ $t\in (0,1)$ we have
\begin{align*}
\sum_{i:Q_i\cap \alpha_t\neq \emptyset} \osc_{Q_i}(u) \leq  D  (u).
\end{align*}
\end{prop}

\begin{proof}
Again we choose a $t\in (0,1)$ that is not a maximum or minimum value of $u$ on any $\partial Q_i$, $i\in \N$.

We fix $\delta,h>0$ and we define a Sobolev function $g\in \mathcal W^{1,2}(S)$ with $g=0$ on $\Theta_1$ and $g=1+\delta(1-2h)$ on $\Theta_3$ as follows:
\begin{align*}
g(x)= \begin{cases}
(1+\delta)u(x), & u(x)\leq t-h\\
u(x)+c_1, &  t-h<u(x)<t+h\\
(1+\delta)u(x)+c_2, & u(x)\geq t+h
\end{cases}
\end{align*}
where the constants $c_1,c_2$ are chosen so that $g$ is continuous. It easy to see that $c_1= \delta(t-h)$ and $c_2=-2\delta h$. Note that $g$ can be written as
\begin{align}\label{unif:Level sets-g definition}
g= \left[((1+\delta)u)\land (u+c_1)\right] \lor ((1+\delta)u+c_2),
\end{align}
which shows that $g\in \mathcal W^{1,2}(S)$, according to Lemma \ref{unif:Gluing lemma}. Consider the index sets $F_h=\{i\in \N: Q_i\cap \alpha_{t+ h}\neq \emptyset \textrm{ and } Q_i\cap \alpha_{t- h}\neq \emptyset\}$ and $N_h=\{i\in \N : Q_i \not\subset A_{t-h,t+h} \}\setminus F_h$. Observe that 
\begin{align*}
\osc_{Q_i}(g) \begin{cases}
\leq (1+\delta)\osc_{Q_i}(u) -2\delta h, & i\in F_h\\
\leq (1+\delta)\osc_{Q_i}(u), & i\in N_h \\
=\osc_{Q_i}(u), & i\notin F_h\cup N_h.
\end{cases}
\end{align*}
Indeed, for the first inequality note that for $i\in F_h$, the maximum of $g$ on $\partial Q_i$ has to be attained at a point $x$ with $u(x)\geq t+h$, and the minimum is attained at a point $y$ with $u(y)\leq t-h$. Hence, $g(x)-g(y) =(1+\delta)(u(x)-u(y)) -2\delta h \leq (1+\delta)\osc_{Q_i}(u)-2\delta h$. The second inequality holds for all $i\in \N$ and is a crude estimate, based on \eqref{unif:Level sets-g definition} and Lemma \ref{unif:Gluing lemma}; here it is crucial that $\delta>0$ so that $\osc_{Q_i}(u)\leq (1+\delta)\osc_{Q_i}(u)$. The third equality is immediate, since $t-h\leq u\leq t+h$ and thus $g=u+c_1$ on $\partial Q_i$, whenever $i\notin F_h\cup N_h$.

The function $g/(1+\delta(1-2h))$ is admissible for the free boundary problem, and testing the minimizing property of $u$ against $g/(1+\delta(1-2h))$ as in the proof of Proposition \ref{unif:Level sets-sums geq} (see \eqref{unif:Level sets-Optimization in s}) we obtain
\begin{align*}
D (u) \leq \frac{1}{1+\delta(1-2h)} \sum_{i\in \N} \osc_{Q_i}(u) \osc_{Q_i}(g).
\end{align*}
This implies that
\begin{align*}
D (u) \leq \frac{1}{1+\delta(1-2h)}&\bigg( (1+\delta)\sum_{i\in F_h} \osc_{Q_i}(u)^2- 2\delta h \sum_{i\in F_h}\osc_{Q_i}(u) \\
&\quad \quad +(1+\delta)\sum_{i\in N_h}\osc_{Q_i}(u)^2 + \sum_{i\notin F_h\cup N_h} \osc_{Q_i}(u)^2 \bigg).
\end{align*}
Manipulating the expression yields 
\begin{align}\label{unif:Level sets-sums leq inequality}
\sum_{i\in F_h} \osc_{Q_i}(u) +\frac{1}{2h} \sum_{i\notin F_h\cup N_h} \osc_{Q_i}(u)^2 \leq D (u).
\end{align}
By the choice of $t$ we have 
$$\sum_{i\in F_h} \osc_{Q_i}(u)\to \sum_{i:Q_i\cap \alpha_t\neq \emptyset}\osc_{Q_i}(u)$$
as $h\to 0$. Moreover, if $i\notin F_h \cup N_h$, then $\br Q_i\subset A_{t-2h, t+2h}$, so 
\begin{align*}
\sum_{i\notin F_h\cup N_h} \osc_{Q_i}(u)^2\leq \lambda((t-2h,t+2h)),
\end{align*}
where $\lambda$ is as in Lemma \ref{unif:Level sets-Measure Derivative lemma} and $h(Q_i)\coloneqq \osc_{Q_i}(u)^2$. By the lemma, it follows that 
\begin{align*}
\frac{1}{2h} \sum_{i\notin F_h\cup N_h} \osc_{Q_i}(u)^2  \to 0
\end{align*}
as $h\to 0$ for a.e.\ $t\in [0,1]$. This, together with \eqref{unif:Level sets-sums leq inequality} yields the conclusion.
\end{proof}

Another important topological property of the level sets of $u$ is the following:

\begin{lemma}\label{unif:Level sets-at most two points of intersection}
For a.e.\ $t\in [0,1]$ and for all $i\in \N$ the intersection $u^{-1}(t)\cap \partial Q_i$ contains at most two points. Furthermore, for a.e.\ $t\in [0,1]$ the intersection $u^{-1}(t)\cap \partial \Omega$ contains exactly two points, one in $\Theta_2$ and one in $\Theta_4$.
\end{lemma}
\begin{proof}
The proof is based on the following elementary lemma, which is the 1-dimensional version of Sard's theorem: 
\begin{lemma}\label{unif:Sard}
Let $f\colon \R \to \R$ be an arbitrary function. Then the set of local maximum and local minimum \textit{values} of $f$ is at most countable.
\end{lemma}
We will use the lemma now and provide a proof right afterwards.

Note that $\partial \Omega ,\partial Q_i \simeq \R/\Z$. We consider a level $t\in (0,1)$ that is not a local maximum or local minimum value of $u$ on $\partial \Omega$ and on any peripheral circle $\partial Q_i$, $i\in \N$. This implies that for each point $x\in u^{-1}(t)\cap \partial Q_i$ there exist, arbitrarily close to $x$, points $x_+,x_- \in \partial Q_i$ with $u(x_+)>t$ and $u(x_-)<t$. 

By Proposition \ref{unif:Level sets}, $u^{-1}(t)$ intersects $\Theta_2$ at a connected set. Since $t$ is not a local maximum or local minimum value of $u$ on $\partial \Omega$, it follows that $u^{-1}(t)\cap \Theta_2$ cannot be an arc, so it has to be a point. Similarly, $u^{-1}(t)\cap \Theta_4$ is a singleton. Furthermore, with the same reasoning, for each $i\in \N$ the intersection $u^{-1}(t)\cap \partial Q_i$ is a totally disconnected set. Let $\partial Q_i$, $i\in \N$, be an arbitrary peripheral circle, intersected by $u^{-1}(t)$. We now split the proof in two parts.

\textbf{Step 1:} There exist continua $C_2,C_4\subset \tilde u^{-1}(t)$ that connect $\Theta_2,\Theta_4$ to $\partial Q_i$, respectively, with $C_2\cap \partial Q_i=\{x_2\}$ and  $C_4\cap \partial Q_i=\{x_4\}$ for some points $x_2,x_4\in \partial Q_i$. We provide details on how to obtain these continua.

Let $C_2\subset \tilde u^{-1}(t)$ be a minimal continuum that connects $\partial Q_i$ to $\Theta_2$, and $C_4\subset \tilde u^{-1}(t)$ be a minimal continuum that connects $\partial Q_i$ to $\Theta_4$. Here, a continuum $C$ joining two sets $E$ and $F$ is minimal if any compact proper subset of $C$ is either disconnected or it does not connect $E$ and $F$. The existence of the continua $C_2$ and $C_4$ follows from Zorn's lemma, because the intersection of a chain of continua connecting two compact sets is again a continuum connecting these compact sets; see \cite[Theorem 28.2]{Willard:topology} and the proof of \cite[Theorem 28.4]{Willard:topology}. We next show that $C_2$ intersects $\partial Q_i$ at a single point $x_2$ and $C_4$ intersects $\partial Q_i$ at a single point $x_4$. 

First note that $C_2\cap Q_i=\emptyset$. Otherwise, $C_2\setminus Q_i$ has to be disconnected by the minimality of $C_2$, so there exists a compact component $W$ of $C_2\setminus Q_i$ that intersects $\Theta_2$. By the minimality of $C_2$, $W$ cannot intersect $\partial Q_i$, so it has a positive distance from it. The component $W$ is the intersection of all rel.\ clopen subsets of $C_2\setminus Q_i$ that contain it; see \cite[Corollary 1.34]{Burckel:complex} or \cite[p.~304]{Remmert:complex}. Let $U\supset W$ be such a clopen set, very close to $W$, so that $U\cap \partial Q_i=\emptyset$. Then $C_2= U \cup (C_2\setminus U)$, where $U$ and $C_2\setminus U$ are non-empty and rel.\ closed in $C_2$. This contradicts the connectedness of $C_2$, and completes the proof that $C_2\cap Q_i=\emptyset$.

Now, assume that $C_2\cap \partial Q_{i}$ contains two points, $x$ and $y$. We connect these points by a simple path $\gamma \subset  Q_i$. Then, we claim that $\br \Omega \setminus (C_2\cup \gamma)$ has a component $V\subset \Omega$ such that $\partial V \subset C_2\cup \gamma$, and such that $\br V$ contains an arc $\beta\subset \partial Q_i$ between $x$ and $y$. Assume the claim for the moment. On $\partial_* V=\partial V\cap S$ we have $u\equiv t$, so by the maximum principle in Theorem \ref{unif:Maximum principle} we obtain that $u\equiv t$ on $\br V\cap S$. However, this implies that $u\equiv t$ on $\beta$ and this contradicts the fact that $u^{-1}(t)\cap \partial Q_i$ is totally disconnected. Thus, indeed $C_2\cap \partial Q_i\subset u^{-1}(t)$ contains precisely one point, $x_2$. The same is true for $C_4\cap \partial Q_i =\{x_4\}$.

Now we prove the claim.  Note that $\gamma$ separates $Q_i$ in two open ``pieces" $A_i$, $i=1,2$. Each of these two pieces lies in a component $V_{i}$, $i=1,2$, of $\br \Omega \setminus (C_2\cup \gamma)$, respectively. Note that $\Theta_1,\Theta_3$ are disjoint from $C_2\cup \gamma$, since $C_2\subset \tilde u^{-1}(t)$ and  $t\neq 0,1$. Hence, $\Theta_1$ and $\Theta_3$ lie in components of $\br \Omega \setminus (C_2\cup \gamma)$. We claim that one of the components $V_1$ and $V_2$, say $V_1$, contains neither $\Theta_1$ nor $\Theta_3$. This will be the desired component $V$ with the claimed properties in the previous paragraph. In particular, the arc $\beta$ is an arc contained in $\partial A_1\cap \partial Q_i$.

To prove our latter claim, we first observe that the components of the set $\br \Omega\setminus (C_2\cup \gamma)$, which is rel.\ open in $\br \Omega$, are pathwise connected and rel.\ open in $\br \Omega$. This is because $\br \Omega$ is a Jordan region and thus it is locally pathwise connected; see \cite[Theorem 27.5 and Theorem 27.9]{Willard:topology}. The components $V_1$ and $V_2$ that contain $A_1$ and $A_2$, respectively, are  necessarily distinct. Otherwise, there is a path in $\br \Omega \setminus (C_2\cup \gamma)$ that connects a point  $a_1\in A_1$ to a point $a_2\in A_2$. Concatenating this path with a path inside $Q_i$ that connects $a_1$ to $a_2$ would provide a loop in $\br \Omega\setminus C_2$ that separates $C_2$, a contradiction to the connectedness of $C_2$; for the construction of that loop it is crucial that $C_2\cap Q_i=\emptyset$. Suppose now that $V_1$ contains $\Theta_1$ and $V_2$ contains $\Theta_3$. Then we can similarly construct path from $\Theta_1$ to $\Theta_3$ passing through $Q_i$ that disconnects $C_2$, a contradiction.

\textbf{Step 2:} The points $x_2$ and $x_4$ are the only points lying in $u^{-1}(t)\cap \partial Q_i$.

Now, we show that there can be no third point in $ u^{-1}(t)\cap \partial Q_i$. In case $x_2\neq x_4$, we join the points $x_2,x_4$ with an arc inside $Q_i$ and we obtain a continuum $C$ that separates $\Theta_1$ and $\Theta_3$, and intersects $\partial Q_i$ in  two points. If $x_2=x_4$ we just let $C=C_2\cup C_4$. The set $\br \Omega \setminus C$ has at least two components, one containing $\Theta_1$ and one containing $\Theta_3$. If $V$ is one of the components of $\br \Omega\setminus C$, then we have $u\equiv t$ on $\partial_*V$, which implies that $u\geq t$ or $u\leq t$ on $V\cap S$ by the maximum principle in Theorem \ref{unif:Maximum principle}.

Assume that there exists another point $x\in \partial Q_i\cap u^{-1}(t)$, $x\neq x_2,x_4$. The point $x$ lies on an open arc $\beta \subset \partial Q_{i}$ with endpoints $x_2,x_4$. This arc lies in one of the components of $\br\Omega \setminus C$, so assume it lies in a component $V$ on which $u\leq t$. However, by the choice of $t$, arbitrarily close to $x$ we can find a point $x_+\in \beta$ with $u(x_+)>t$, a contradiction.
\end{proof}

\begin{proof}[Proof of Lemma \ref{unif:Sard}]
The set of local maximum values of $f\colon \R \to \R$ is the set 
\begin{align*}
E&=\{ y\in \R : \textrm{there exist $x\in \R$ and $\varepsilon>0$ such that} \\
&\quad \quad \quad \textrm{$y=f(x)$ and $f(z)\leq y$ for all $|z-x|<\varepsilon$} \}.
\end{align*}
We will show that this set is at most countable. The claim for the local minimum values follows by looking at $-f$. Note that $E=\bigcup_{n=1}^\infty E_n$, where 
\begin{align*}
E_n&=\{ y\in \R : \textrm{there exists $x\in \R$  such that} \\
&\quad \quad \quad \textrm{$y=f(x)$ and $f(z)\leq y$ for all $|z-x|<1/n$} \}.
\end{align*}
Hence, it suffices to show that $E_n$ is at most countable for each $n\in \N$. For each $y\in E_n$ there exists $x\in \R$ such that $f(x)=y$, and there exists an interval $I=(x-1/n,x+1/n)$ such that $f(z)\leq y$ for all $z\in I$. If $y_1,y_2\in E_n$ are distinct with $y_1=f(x_1)$, $y_2=f(x_2)$, and $I_1,I_2$ are the corresponding intervals, then $\frac{1}{2}I_1\coloneqq (x_1-1/2n, x_1+1/2n)$ and  $\frac{1}{2}I_2\coloneqq (x_2-1/2n, x_2+1/2n)$ are necessarily disjoint intervals. This implies that $E_n$ is in one-to-one correspondence with a family of disjoint open subintervals of $\R$, and hence $E_n$ is at most countable.
\end{proof}

The next corollary is immediate:

\begin{corollary}\label{unif:Level sets-exactly two points}
For each peripheral disk $Q_i$, $i\in \N$, and for a.e.\ level $t\in [m_{Q_i},M_{Q_i}]$ the intersection $u^{-1}(t)\cap \partial Q_i$ contains precisely two points.
\end{corollary}
\begin{proof}
Assume that $m_{Q_i}(u)<M_{Q_i}(u)$ (i.e., $\osc_{Q_i}(u)\neq 0$), and choose a $t\in (m_{Q_i}(u),M_{Q_i}(u))$ so that the conclusion of Lemma \ref{unif:Level sets-at most two points of intersection} is true. Consider two points $p_i,q_i\in \partial Q_i$ such that $u(p_i)=m_{Q_i}(u)$ and $u(q_i)=M_{Q_i}(u)$. Applying the intermediate value theorem on each of the two arcs between the points $p_i,q_i$, it follows that $u^{-1}(t)\cap \partial Q_i$ contains at least two points.
\end{proof}

\begin{remark}It is clear from the proof of Lemma \ref{unif:Level sets-at most two points of intersection} that we only need to exclude at most countably many $t\in [m_{Q_i}(u),M_{Q_i}(u)]$ for the conclusion of Corollary \ref{unif:Level sets-exactly two points}.
\end{remark}

We continue with an absolute continuity lemma. This is the most technical part of the section. We first observe the following consequence of the fatness of the peripheral disks:

\begin{remark}\label{unif:Fatness consequence}
If a peripheral disk $Q_i$, $i\in \N$, intersects two circles $\partial B(x,r)$ and $\partial B(x,R)$ with $0<r<R$, then 
\begin{align*}
\mathcal H^2(Q_i\cap (B(x,R)\setminus B(x,r)))\geq C (R-r)^2,
\end{align*}
where $C>0$ is a constant depending only on the fatness constant of condition \eqref{unif:Fat sets}. To see this, by the connectedness of $Q_i$ there exists a point $y\in Q_i \cap \partial B(x, (r+R)/2)$. Then $B(y, (R-r)/2) \subset B(x,R)\setminus B(x,r)$, so
\begin{align*}
\mathcal H^2(Q_i \cap (B(x,R)\setminus B(x,r)))\geq \mathcal H^2(Q_i\cap B(y,(R-r)/2)) \geq K_1 \frac{(R-r)^2}{4},
\end{align*}
by the fatness condition \eqref{unif:Fat sets}.
\end{remark}

\begin{lemma}\label{unif:Level sets-Hausdorff measure}\index{level sets!Hausdorff $1$-measure of}
For a.e.\ $t\in [0,1]$ we have $\mathcal H^1(u^{-1}(t))=0$.
\end{lemma}

The proof is very technical so we provide first a rough sketch of the argument. For a fixed $\varepsilon>0$ we will find an effective cover (up to a small set) of $S= \bigcup_{t\in [0,1]} u^{-1}(t)$ by balls $B_j$ of radius $r_j<\varepsilon$. Then for each $t$ the quantity $\mathcal H^1_\varepsilon ( u^{-1}(t))$ is bounded by $\sum 2r_j$, where the sum is over the balls intersecting $u^{-1}(t)$. We wish to show that $\mathcal H^1_\varepsilon ( u^{-1}(t))$ converges to $0$ as $\varepsilon\to 0$ for a.e.\ $t\in [0,1]$. One way of proving this is by integrating $\sum 2r_j$ over $t\in[0,1]$, and showing that the integral converges to $0$. 

Upon integrating, one obtains an expression of the form 
\begin{align*}
\sum_j r_j \diam(u(B_j\cap S)), 
\end{align*}
so we wish to find good bounds for $\diam(u(B_j\cap S))$. Thus, we produce bounds using the upper gradient inequality, in combination with the maximum principle (see Lemma \ref{unif:Maximum principle circular arcs}):
\begin{align*}
\diam(u(B_j\cap S))\leq \sum_{i:Q_i\cap \partial B_j\neq \emptyset} \osc_{Q_i}(u).
\end{align*}
This is where technicalities arise, because the right hand side is not a good enough bound for all balls $B_j$.

The bound turns out to be good, in case the ball $B_j$ intersects only ``small" peripheral disks $Q_i$ of diameter $\lesssim r_j$, or in case the ``large" peripheral disks that are possibly intersected by $B_j$ do not have serious contribution to the upper gradient inequality and can be essentially ignored:
 \begin{align*}
\diam(u(B_j\cap S))\lesssim \sum_{\substack{i:Q_i\cap \partial B_j\neq \emptyset\\ Q_i\,\,\,\textrm{``small"}}} \osc_{Q_i}(u).
\end{align*}
We call ``good" the balls $B_j$ satisfying the above.

However, there is a ``bad" subcollection of the balls $B_j$ for which the above estimate fails. Namely, these are the balls that intersect some relatively large peripheral disk $Q_i$, but the latter also has a serious contribution to the upper gradient inequality and cannot be ignored. We amend this by essentially discarding these ``bad" balls $B_j$ from our effective cover of the set $\bigcup_{t\in [0,1]} u^{-1}(t)$, and then replacing each of them (in the cover) with a corresponding ``large" peripheral disk $Q_i$ (after enlarging it slightly so that we still obtain a cover). 

Then, $\mathcal H^1_\varepsilon ( u^{-1}(t))$ is bounded by 
\begin{align*}
\sum 2r_j + \sum \diam(Q_i),
\end{align*}  
where the first sum is over the ``good" balls intersecting $u^{-1}(t)$ and the second sum is over the ``large" peripheral disks $Q_i$ corresponding to ``bad" balls that intersect $u^{-1}(t)$. One now integrates over $t\in [0,1]$ as before, and treats separately the terms corresponding to the ``good" and ``bad" balls. We proceed with the details.

\begin{proof}
By Lemma \ref{unif:Level sets-at most two points of intersection}, for a.e.\ $t\in (0,1)$ the level set $u^{-1}(t)$ intersects $\partial \Omega$ and $\partial Q_i$ in at most two points, for all $i\in \N$. Hence, it suffices to show that for a.e.\ $t\in (0,1)$ the set $u^{-1}(t)\cap S^\circ$ has $\mathcal H^1$-measure equal to zero. Recall that $S^\circ$ contains the points of the carpet not lying on any peripheral circle.

For a fixed $\varepsilon >0$ consider the finite set $E=\{i\in \N: \diam(Q_i)>\varepsilon \}$. We cover $\Omega \setminus \bigcup_{i\in E} \br Q_i$ by balls $B_j$ of radius $r_j<\varepsilon$ such that $2B_j \subset \Omega \setminus \bigcup_{i\in E} \br Q_i$ and such that $\frac{1}{5}B_j$ are disjoint. The existence of this collection of balls is justified by a basic covering lemma; see e.g.\ \cite[Theorem 1.2]{Heinonen:metric}.

Let $J$ be the family of indices $j$ such that for each $s\in [1,2]$ we have 
\begin{align}\label{unif:Level sets-Hausdorff-Q_{i_j's}}
\diam (u(sB_j\cap S))\geq k\osc_{Q_i}(u)
\end{align}
for all peripheral disks $Q_i$ with $\diam(Q_i)>8r_j$ that intersect $\partial(sB_j)$, where $k\geq 1$ is a constant to be determined. It follows from Remark \ref{unif:Fatness consequence} that for each $j\in J$ there can be at most $N_0$ such peripheral disks $Q_i$, where $N_0$ depends only on the fatness constant. Indeed, each such $Q_i$ must intersect both $\partial (2B_j)$ and $\partial(4B_j)$, so it follows that 
\begin{align*}
\mathcal H^2( Q_i\cap 4B_j)\geq \mathcal H^2(Q_i\cap (4B_j\setminus 2B_j)) \geq Cr_j^2
\end{align*} 
for a uniform constant $C>0$, depending only on the fatness constant. Comparing the area of $4B_j$ with $\mathcal H^2(Q_i\cap 4B_j)$ we arrive at the conclusion.

We fix a ball $B_j$, $j\in J$, and peripheral disks $Q_{i_1},\dots,Q_{i_N}$ as above, where $N\leq N_0$. These are the peripheral disks with diameter bigger than $8r_j$, each of which intersects $\partial (sB_j)$ for some $s\in [1,2]$ and satisfies \eqref{unif:Level sets-Hausdorff-Q_{i_j's}}. Our goal is to show that there exists a uniform constant $C>0$, such that for a.e.\ $s\in (1,2)$ we have
\begin{align}\label{unif:Level sets-Hausdorff measure Basic inequality}
\diam(u(sB_j\cap S))\leq  C\sum_{\substack{i:Q_i\cap \partial (sB_j)\neq \emptyset\\ i\neq i_1,\dots,i_N}} \osc_{Q_i}(u),
\end{align}
provided that we choose $k$ suitably, depending only on the data. In other words, the contribution of $\osc_{Q_{i_1}}(u),\dots,\osc_{Q_{i_N}}(u)$ in the upper gradient inequality is negligible. We fix $s\in (1,2)$ such that the conclusion of Lemma \ref{unif:Maximum principle circular arcs} holds, i.e., 
\begin{align*}
\diam ( u(sB_j\cap S)) \leq \sum_{{i:Q_i\cap \partial (sB_j)\neq \emptyset}} \osc_{Q_i}(u).
\end{align*}
If none of $Q_{i_1},\dots,Q_{i_N}$ intersects $\partial (sB_j)$, then \eqref{unif:Level sets-Hausdorff measure Basic inequality} follows immediately, so we assume that this is not the case. After reordering, suppose that $Q_{i_1},\dots,Q_{i_M}$, $M\leq N$, are the peripheral disks intersecting $\partial (sB_j)$, among $Q_{i_1},\dots,Q_{i_N}$. We have
\begin{align*}
\diam ( u(sB_j\cap S)) &\leq \sum_{\substack{i:Q_i\cap \partial (sB_j)\neq \emptyset\\ i\neq i_1,\dots,i_M}} \osc_{Q_i}(u) + \sum_{l=1}^M \osc_{Q_{i_l}}(u)\\
&\leq  \sum_{\substack{i:Q_i\cap \partial (sB_j)\neq \emptyset\\ i\neq i_1,\dots,i_M}} \osc_{Q_i}(u) +\frac{M}{k} \diam ( u(sB_j\cap S)).
\end{align*}
We consider $k\coloneqq 2N_0\geq 2N \geq 2M$. Then 
\begin{align*}
\diam ( u(sB_j\cap S)) &\leq 2\sum_{\substack{i:Q_i\cap \partial (sB_j)\neq \emptyset\\ i\neq i_1,\dots,i_M}} \osc_{Q_i}(u)
\end{align*}
and this completes the proof of \eqref{unif:Level sets-Hausdorff measure Basic inequality}.

If we write $B_j=B(x_j,r_j)$, then \eqref{unif:Level sets-Hausdorff measure Basic inequality} implies that
\begin{align*}
\diam(u(B_j\cap S)) \leq \diam(u(sB_j\cap S))\leq  C\sum_{\substack{i:Q_i\cap \partial B(x_j,s)\neq \emptyset\\ i\neq i_1,\dots,i_N}} \osc_{Q_i}(u)
\end{align*} 
for a.e.\ $s\in (r_j,2r_j)$. Integrating over $s\in (r_j,2r_j)$ and applying Fubini's theorem yields
\begin{align*}
r_j\diam(u(B_j\cap S)) &\leq C\sum_{\substack{i:Q_i\cap 2B_j\neq \emptyset \\ i\neq i_1,\dots,i_N} } \osc_{Q_i}(u) \int_{r_j}^{2r_j} \x_{Q_i\cap \partial  B(x_j,s)} \,  ds \\
&\leq C \sum_{\substack{i:Q_i\cap 2B_j\neq \emptyset \\ i\neq i_1,\dots,i_N} } \osc_{Q_i}(u) \diam(Q_i).
\end{align*}
We note that if $Q_i\cap 2B_j\neq \emptyset$ and $i\neq i_1,\dots,i_N$, then $\diam (Q_i) \leq 8r_j$ (by the definition of $i_1,\dots,i_N$), so $Q_i\subset 11B_j$. Therefore,
\begin{align}\label{unif:Level sets-Hausdorff measure-Good bound}
r_j\diam(u(B_j\cap S)) \leq C \sum_{i:Q_i\subset 11B_j } \osc_{Q_i}(u)\diam(Q_i).
\end{align}
For each $j\in J$ now consider the smallest interval $I_j$ containing $u(B_j\cap S)$, and define $g_\varepsilon(t)=\sum_{j\in J} 2r_j\x_{I_j}(t)$, $t\in [0,1]$.

On the other hand, for each $j\notin J$ there exists $s=s_j\in [1,2]$ and there exists a peripheral disk $Q_i$ that intersects $\partial(sB_j)$ with $\diam(Q_i)>8r_j$, but $\diam (u(sB_j\cap S))< k\osc_{Q_i}(u)$. Note that some $Q_i$ might correspond to multiple balls $B_j$, $j\notin J$. Consider the family $\{Q_i\}_{i\in I}$ of all peripheral disks that correspond to balls $B_j$, $j\notin J$, and let for each $i\in I$
\begin{align*}
\widetilde Q_i &\coloneqq  \br{Q}_i\cup \bigcup \{s_jB_j: Q_i\cap \partial (s_jB_j)\neq \emptyset, \, \diam(Q_i)>8r_j,\\ &\qquad\qquad\qquad \textrm{and}\, \diam(u(s_jB_j\cap S)) <k\osc_{Q_i}(u)\}.
\end{align*}
It is easy to see that for every $\eta>0$ there exist $s_{j_1}B_{j_1},s_{j_2}B_{j_2} \subset \widetilde Q_i$ such that 
\begin{align*}
 \diam (u(\widetilde Q_i\cap S)) &\leq \diam(u(\partial Q_i))+\diam(u(s_{j_1}B_{j_1}\cap S))\\
 &\quad\quad +\diam(u(s_{j_2}B_{j_2}\cap S))+\eta\\
 &\leq \osc_{Q_i}(u) +2k\osc_{Q_i}(u) +\eta.
\end{align*} 
Hence, 
\begin{align}\label{unif:Level sets-Hausdorff measure-Bad bound}
\diam (u(\widetilde Q_i\cap S))\leq C \osc_{Q_i}(u)
\end{align}
for all $i\in I$, where $C=1+2k$ and depends only on the data. Also, observe that $\diam(\widetilde Q_i)< 2\diam(Q_i)$, since $\diam(Q_i)>8r_j$ whenever $s_jB_j\subset \widetilde Q_i$ and $s_j\leq 2$. For each $i\in I$ consider the smallest interval $I_i$ that contains $u(\widetilde Q_i\cap S)$, and define $b_\varepsilon(t) = \sum_{i\in I} 2\diam(Q_i) \x_{I_i}(t)$, $t\in [0,1]$. We remark that $I\cap E=\emptyset$ since the balls $s_jB_j\subset 2B_j$ do not intersect peripheral disks $Q_i$, $i\in E$.

For each $t\in [0,1]$ the set $u^{-1}(t)\cap S^\circ$ is covered by the balls $B_j$, $j\in J$, and the sets $\widetilde Q_i$, $i\in I$. Since $r_j<\varepsilon$ for $j\in J$ and $\diam(Q_i)<\varepsilon$ for $i\in I$, we have 
\begin{align*}
\mathcal H^1_{\varepsilon}(u^{-1}(t)\cap S^\circ) \leq g_\varepsilon(t) +b_\varepsilon(t).
\end{align*}
It suffices to show that $g_\varepsilon(t) \to 0$ and $b_\varepsilon(t) \to 0$ for a.e.\ $t\in[0,1]$, along some sequence of $\varepsilon \to 0$.

Note first that by \eqref{unif:Level sets-Hausdorff measure-Bad bound} we have
\begin{align*}
\int_0^1 b_\varepsilon(t) \, dt &= 2\sum_{i\in I} \diam (Q_i) \diam(u(\widetilde{Q}_i\cap S))\\
&\leq 2C\sum_{i\in \N \setminus E} \diam(Q_i) \osc_{Q_i}(u)\\
&\leq 2C\left(\sum_{i\in \N \setminus E} \diam(Q_i)^2 \right)^{1/2}\left( \sum_{i\in \N \setminus E} \osc_{Q_i}(u)^2 \right)^{1/2}.
\end{align*}
The first sum is finite by the quasiballs assumption \eqref{unif:Quasi-balls}, and the second is also finite since $u\in \mathcal W^{1,2}(S)$. As $\varepsilon \to 0$, the set $E$ increases to $\N$, hence the sums converge to zero. This shows that $b_\varepsilon\to 0$ in $L^1$. In particular $b_\varepsilon(t) \to 0$ a.e., along a subsequence. 

Finally, we show the same conclusion for $g_\varepsilon(t)$. By \eqref{unif:Level sets-Hausdorff measure-Good bound} we have
\begin{align}\label{unif:Level sets-Hausdorff dimension-Good bound 1}
\int_0^1 g_\varepsilon(t) \, dt= \sum_{j\in J} 2r_j \diam(u(B_j\cap S)) \leq  2C \sum_{j\in J} \sum_{i:Q_i\subset 11B_j} \osc_{Q_i}(u) \diam(Q_i).
\end{align}
We define $h(x)=\sum_{i\in \N} (\osc_{Q_i}(u)/\diam(Q_i) )\cdot \x_{Q_i}(x)$. By the quasiballs assumption we have $\diam(Q_i)^2 \simeq \mathcal H^2(Q_i)$ for all $i\in \N$, hence the right hand side in \eqref{unif:Level sets-Hausdorff dimension-Good bound 1} can be bounded up to a constant by
\begin{align*}
\sum_{j\in J} \int_{11B_j} h(x) \, d\mathcal H^2(x).
\end{align*}
If $Mh$ denotes the uncentered Hardy-Littlewood maximal function of $h$, the above is bounded, up to a constant by 
\begin{align*}
\sum_{j\in J} \int_{\frac{1}{5}B_j}Mh(x) \, d\mathcal H^2(x) &= \int_{\bigcup_{j\in J} \frac{1}{5}B_j} Mh(x) \, d\mathcal H^2(x) \\
&\leq \int_{\Omega\setminus \bigcup_{i\in E}  Q_i} Mh(x) \, d\mathcal H^2(x),
\end{align*}
where we used the fact that the balls $\frac{1}{5}B_j$ are disjoint. We wish to show that the latter converges to $0$ as $\varepsilon\to 0$. Then, we will indeed have $\int_0^1 g_\varepsilon(t) \, dt \to 0$ as $\varepsilon\to 0$, as desired.

Since $\mathcal H^2(S)=0$, it follows that $\mathcal H^2(\Omega\setminus \bigcup_{i\in E}  Q_i) \to 0$ as $\varepsilon \to 0$. Note now that $h\in L^2(\Omega)$, with
\begin{align*}
\int_{\Omega} h^2(x)d\mathcal H^2(x) \simeq \sum_{i\in \N} \osc_{Q_i}(u)^2 =D(u)<\infty.
\end{align*}
By the $L^2$-maximal inequality we have $Mh \in L^2(\Omega) \subset L^1(\Omega)$, and this implies that
\begin{align*}
\int_{\Omega\setminus \bigcup_{i\in E} Q_i} Mh(x) \, d\mathcal H^2(x) \to 0
\end{align*} 
as $\varepsilon\to 0$. 
\end{proof}

\begin{remark}The fatness of the peripheral disks was crucial in the preceding argument, and it would be interesting if one could relax the assumption of fatness to e.g.\ a H\"older domain (see \cite{SmithStegenga:HolderPoincare} for definition) assumption on the peripheral disks.
\end{remark}

Next, we wish to show that for a.e.\ $t\in [0,1]$ the level set $\alpha_t =\tilde u^{-1}(t)$ is a simple curve that joins $\Theta_2$ and $\Theta_4$. We include a topological lemma.

\begin{lemma}\label{unif:Level sets-Topological Lemma}
Let $C\subset \R^2$ be a locally connected continuum with the following property: it is a minimal continuum that connects two distinct points $a,b\in \R^2$. Then $C$ can be parametrized by a simple curve $\gamma$.
\end{lemma}

The minimality of $C$ is equivalent to saying that any compact proper subset of $C$ is either disconnected, or it does not contain one of the points $a,b$.

\begin{proof}
It is a well-known fact that a connected, locally connected compact metric space (also known as a \textit{Peano space}) is \textit{arcwise connected}, i.e., any two points in the space can be joined by a homeomorphic image of the unit interval; see \cite[Theorem 31.2, p.~219]{Willard:topology}. In our case, there exists a homeomorphic copy $\gamma \subset C$ of the unit interval that connects $a$ and $b$. The minimality of $C$ implies that $C=\gamma$.
\end{proof}

Next, we show that for a.e.\ $t\in(0,1)$ the level set $\alpha_t=\tilde u^{-1}(t) $ satisfies the assumptions of the previous lemma. 

\begin{lemma}\label{unif:Level sets-curves}\index{level sets!topology of}
For a.e.\ $t\in (0,1)$ the level set $\alpha_t$ is a simple curve that connects $\Theta_2$ to $\Theta_4$. Moreover, if $\alpha_t\cap \br{Q}_i\neq \emptyset$ for some $i\in \N$, then $\alpha_t\cap \br Q_i$ is precisely an arc with two distinct endpoints on $\partial Q_{i}$.  
\end{lemma}
\begin{proof}
By Lemma \ref{unif:Level sets}, for each $t\in (0,1)$ the level set $\alpha_t$ is simply connected and connects $\Theta_2$ to $\Theta_4$. We choose a $t\in(0,1)$ that is not a local maximum or local minimum value of $u$ on $\partial \Omega$ or on any peripheral circle $\partial Q_i$, $i\in \N$. Note that this implies that $\tilde u$ is non-constant in $Q_i$ whenever $\alpha_t\cap Q_i\neq \emptyset$. There are countably many such values $t$ that we exclude; see proof of Lemma \ref{unif:Level sets-at most two points of intersection}.  Restricting our choice even further, we assume that $t$ is not a critical value of $\tilde u$ on any $Q_i$, $i\in \N$; a non-constant planar harmonic function has at most countably many critical values. Finally, suppose that we also have $\mathcal H^1(u^{-1}(t))=0$, which holds for a.e.\ $t\in (0,1)$ by Lemma \ref{unif:Level sets-Hausdorff measure}.

Using an argument similar the proof of Lemma \ref{unif:Level sets-at most two points of intersection} we show that $\alpha_t$ is a minimal continuum connecting $\Theta_2$ to $\Theta_4$. Consider a minimal continuum $C\subset \alpha_t$ joining $\Theta_2$ and $\Theta_4$. The maximum principle from Theorem \ref{unif:Maximum principle for tilde u} implies that on each of the components of $\br\Omega \setminus C$ we have $\tilde u\geq t$ or $\tilde u\leq t$. If $C\neq \alpha_t$ then there exists a point $x\in \alpha_t\setminus C$, lying in the interior (rel.\ $\br \Omega$) of one of these components. If $x\in \br Q_i$ for some $i\in \N$ or $x\in \partial \Omega$, by the choice of $t$, arbitrarily close to $x$ we can find points $x_+,x_- \in \br\Omega$ with $\tilde u(x_+)>t$ and $\tilde u(x_-)<t$. This is a contradiction. 

If $x\in S^\circ$, then we claim that arbitrarily close to $x$ there exists a point $y\in  \alpha_t \cap \br Q_i\setminus C$ for some $i\in \N$. In this case we are led to a contradiction by the previous paragraph. Now we prove our claim; see also Lemma \ref{unif:lemma:paths zero hausdorff}. If it failed, then there would exist a small ball $B(x,\varepsilon)$ such that all points $y\in \alpha_t\cap B(x,\varepsilon)$ lie in $S^\circ$. Since $\alpha_t$ is connected and it exits $B(x,\varepsilon)$, there exists a continuum $\beta\subset \alpha_t \cap B(x,\varepsilon) \cap S^\circ$ with $\diam(\beta)\geq \varepsilon/2$. Then $\mathcal H^1(u^{-1}(t))=\mathcal H^1(\alpha_t\cap S) \geq \diam(\beta)>0$, a contradiction to the choice of $t$.

Next, we show that $\alpha_t$ is locally connected. We will use the following lemma that we prove later:
\begin{lemma}\label{unif:Locally connected lemma}
Suppose that a compact connected metric space $X$ is not locally connected. Then there exists an open set $U$ with infinitely many connected components having diameters bounded below.  
\end{lemma}
Assume that $\alpha_t$ is not locally connected. Then there exists an open set $U$ and $\varepsilon>0$ such that $U\cap \alpha_t$ contains infinitely many components $C_n$, $n\in \N$, of diameter at least $\varepsilon$. By passing to a subsequence, we may assume that the continua $\br C_n$ converge to a continuum $C$ in the Hausdorff sense, with $\diam (C)\geq \varepsilon$. By the continuity of $\tilde u$, it follows that $C\subset \alpha_t$. We claim that $C\subset S$, so $C\subset u^{-1}(t)$. Assuming the claim, we obtain a contradiction, since $\varepsilon\leq \mathcal H^1(C) \leq \mathcal H^{1}(u^{-1}(t))=0$.

If $C \cap  Q_i \neq \emptyset$, then by shrinking $C$ and $C_n$ we may assume that $C,C_n\subset\subset Q_i$ for all $n\in \N$. By our choice, $t$ is not a critical value of $\tilde u$ on $Q_i$. Finding a local harmonic conjugate $\tilde v$, we see that the pair $G\coloneqq (\tilde u, \tilde v)$ yields a conformal map on a neighborhood of $C$. Thus for infinitely many $n\in \N$ the continua $C_n$ have large diameter and lie on the preimage under $G$ of a vertical line segment. This contradicts, e.g., the rectifiability of the preimage of this segment.

Our last assertion in the lemma follows from the fact that $\alpha_t$ is a simple curve and the maximum principle in Theorem \ref{unif:Maximum principle}. Indeed, by Corollary \ref{unif:Level sets-exactly two points}, for a.e.\ $t\in [m_{Q_i}(u),M_{Q_i}(u)]$ the intersection $\alpha_t\cap \partial Q_{i}$ contains precisely two points $x,y$. If $\alpha_t\cap Q_{i}= \emptyset$ then $\alpha_t$ connects $x,y$ ``externally", and there exists a region $V\subset \Omega$ bounded by $\alpha_t$ and an arc inside $Q_i$ joining $x,y$. However, by the maximum principle $u$ has to be constant in $V\cap S$. Then $V\cap S \subset \alpha_t$, which contradicts the fact that $\alpha_t$ is a simple curve. Therefore, $\alpha_t\cap Q_i \neq \emptyset$, and since $\alpha_t$ is a simple curve, the conclusion follows.  
\end{proof}

\begin{proof}[Proof of Lemma \ref{unif:Locally connected lemma}]
We will use the fact that a space $X$ is locally connected if and only if each component of each open set is open; see \cite[Theorem 27.9]{Willard:topology}.

Suppose that $X$ is not locally connected. Then there exists an open set $U$ and a component $C_0$ of $U$ that is not open. This implies that there exists a point $x\in C_0$ and $\varepsilon>0$ such that for each $\delta<\varepsilon$ we have $B(x,\delta)\subset \br B(x,\varepsilon) \subset U$ and $B(x,\delta)$ intersects a component $C_\delta$ of $U$, distinct from $C_0$. We claim that the component $C_\delta$ has to meet $\partial B(x,\varepsilon)$, and in particular $\diam(C_\delta)\geq  \varepsilon-\delta$. Repeating the argument for a sequence of $\delta\to 0$ yields eventually distinct components $C_\delta$ and leads to the conclusion.

Now we prove our claim. Suppose for the sake of contradiction that $C_\delta$ does not intersect $\partial B(x,\varepsilon)$. Then $C_\delta$ is compact, because it is closed in $U$ (as a component of $U$) and all of its limit points (in $X$) are contained in $\br B(x,\varepsilon) \subset U$.  Since $X$ is a continuum, there exists a continuum $K\subset X$ that connects $C_\delta$ to $\partial B(x,\varepsilon)$ with $K\subset \br B(x,\varepsilon)$; see also the proof of Lemma \ref{unif:Level sets-at most two points of intersection}. In particular, $K\cup C_\delta$ is a connected subset of $U$, which contradicts the fact that $C_\delta$ is a component of $U$.
\end{proof}

\section{The conjugate function \texorpdfstring{$v$}{v}}\label{unif:Section Conjugate}\index{conjugate function}
In order to define the conjugate function $v\colon  S \to \R$ we introduce some notation. For $i\in \N$ let $\rho(Q_i)\coloneqq  \osc_{Q_i}(u)$ and for a path $\gamma\subset \br \Omega$ define 
\begin{align*}
\ell_\rho(\gamma)= \sum_{i:Q_i\cap \gamma \neq \emptyset} \rho(Q_i).
\end{align*}
We would like to emphasize that we have \textit{not} defined $\osc_{Q_0}(u)$ (where $Q_0=\C\setminus \br \Omega$), and also terms corresponding to $i=0$ are not included in the summations, as above. We will first define a coarse version $\hat v $ of the conjugate function  which is only defined on the set $\{Q_i\}_{i\in \N}$ of peripheral disks, and then we will define the conjugate function $v$ by taking limits of $\hat v$.

Let $\mathcal T$ be the family of $t\in (0,1)$ for which the conclusions of Theorem \ref{unif:Level sets-sums equal mass}, Lemma \ref{unif:Level sets-at most two points of intersection}, Corollary \ref{unif:Level sets-exactly two points}, Lemma \ref{unif:Level sets-Hausdorff measure} and Lemma \ref{unif:Level sets-curves} hold. Furthermore, we assume that Lemma \ref{unif:Level sets-Measure Derivative lemma} can be applied for the sequence $h(Q_i)=\rho(Q_i)^2$. Finally, we assume that each $t\in \mathcal T$ is not a local extremal value of $u$ on $\partial \Omega$ or on any peripheral circle $\partial Q_i$, $i\in \N$. This, in particular, implies that $\rho(Q_i)>0$ whenever $\alpha_t\cap Q_i\neq \emptyset$, since otherwise $u$ is constant on $\partial Q_i$. It also implies that if $\alpha_t$ intersects $Q_i$, $i\in \N$, then for all sufficiently small $h>0$ the level set $\alpha_{t\pm h}$ intersects $Q_i$.

If $\rho(Q_i)>0$, then for $t\in  \mathcal T\cap [m_{Q_i}(u), M_{Q_i}(u)]$ the path $\alpha_t=\tilde u^{-1}(t)$ intersects $\partial Q_i$ at two points. We parametrize $\alpha_t\colon [0,1]\to \br\Omega$ so that $\alpha_t(0)\in \Theta_2$ and $\alpha_t(1) \in \Theta_4$. We consider an open subpath $\alpha_t^i$ of $\alpha_t$, terminated at the first entry point of $\alpha_t$ in $\partial Q_i$. This is to say that $\br{\alpha ^i_t}$ starts at $\Theta_2$ and terminates at $\partial Q_i$, while $\alpha_t^i\cap \partial Q_i=\emptyset$. We then define
\begin{align*}
\hat v(Q_i) = \inf \{\ell_\rho( \alpha_t^i) :t\in \mathcal T\cap [m_{Q_i}(u),M_{Q_i}(u)] \}.
\end{align*}  
Note that we define $\hat v(Q_i)$ whenever $\rho(Q_i)>0$. By Theorem \ref{unif:Level sets-sums equal mass} we have $0\leq \hat v(Q_i) \leq D (u)$ for all $i\in \N$. In fact, the infimum is not needed:

\begin{lemma}\label{unif:Conjugate-Infimum not needed}
Fix a peripheral disk $Q_{i_0}$, $i_0\in \N$. If $s,t \in \mathcal T\cap [m_{Q_{i_0}}(u),M_{Q_{i_0}}(u)]$, then $\ell_\rho(\alpha_s^{i_0})=\ell_\rho(\alpha_t^{i_0})$. In particular, we have $\hat v(Q_{i_0})= \ell_{\rho}(\alpha_s^{i_0})$.
\end{lemma}
\begin{proof}
Suppose that $s\neq t$. For simplicity, denote $\gamma_s=\alpha_s^{i_0}$ and define $\tilde \gamma_t$ to be the smallest open subpath of $\alpha_t$ that connects $\partial Q_{i_0}$ to $\Theta_4$. By Lemma \ref{unif:Level sets-curves} the curves $\alpha_t^{i_0}$ and $\tilde \gamma_t$ meet disjoint sets of peripheral disks. Thus, by Theorem  \ref{unif:Level sets-sums equal mass}
\begin{align}\label{unif:Conjugate-Infimum-ell}
\ell_\rho(\alpha_t^{i_0})+ \rho(Q_{i_0})+ \ell_{\rho}(\tilde \gamma_t)= \ell_\rho(\alpha_t) =D(u).
\end{align}

For small $h>0$ so that $s\pm h\in \mathcal T$, each of the disjoint curves $\alpha_{s+h}, \alpha_{s-h}$ intersects $\partial Q_{i_0}$ at two points. The strip $\br A_{s-h,s+h}$ is a (closed) Jordan region bounded by the curves $\alpha_{s-h},\alpha_{s+h}$, and subarcs of $\Theta_2$ and $\Theta_4$; see Proposition \ref{unif:Level sets} and Lemma \ref{unif:Level sets-curves}. Also, this Jordan region contains $\alpha_s\supset \gamma_s$. Since $\partial Q_{i_0}$ meets both boundary curves $\alpha_{s-h}$ and $\alpha_{s+h}$, it follows that $\br A_{s-h,s+h} \setminus Q_{i_0}$ has two components, which are (closed) Jordan regions, one intersecting $\Theta_2$ and one intersecting $\Theta_4$. Let $\Omega_{s,h}\subset \br A_{s-h,s+h}\setminus Q_{i_0}$ be the (closed) Jordan region intersecting $\Theta_2$, so $\gamma_s\subset \Omega_{s,h}$. In a completely analogous way we define a closed Jordan region $\widetilde \Omega_{t,h}$ that intersects $\Theta_4$ and contains $\tilde \gamma_t$. Here, we first need to refine our choice of $h>0$ so that we also have $t\pm h\in \mathcal T$. Finally, since $s\neq t$, if $h>0$ is chosen to be sufficiently small, we may have that the strips $\br A_{s-h,s+h}$ and $\br A_{t-h,t+h}$ are disjoint, and hence, so are the regions $\Omega_{s,h}$ and $\widetilde \Omega_{t,h}$.

Following the notation in the proof of Theorem \ref{unif:Level sets-sums geq}, we define $F_{s,h}$ to be the family of indices $i\in \N\setminus \{i_0\}$ such that $Q_i\cap \alpha_{s- h}\neq \emptyset$, $Q_i\cap \alpha_{s+ h}\neq \emptyset$, and $Q_i\cap \gamma_s\neq \emptyset$. In an analogous way we define $F_{t,h}$ that corresponds to $\tilde \gamma_t$. Then as $h\to 0 $ we have
\begin{align}\label{unif:Conjugate-Infimum-F sets}
\sum_{i\in F_{s,h}}\rho(Q_i) \to \sum_{i:Q_i\cap \gamma_s \neq \emptyset}\rho(Q_i) \quad \textrm{and} \quad \sum_{i\in F_{t,h}}\rho(Q_i) \to \sum_{i:Q_i\cap \tilde\gamma_t \neq \emptyset}\rho(Q_i).
\end{align} 
Also, we define $N_h=\{i\in \N: Q_i\cap (A_{s-h,s+h}\cup A_{t-h,t+h}) \neq \emptyset \} \setminus (F_{s,h}\cup F_{t,h} \cup \{i_0\})$. 

The set $\br\Omega \setminus (\Omega_{s,h} \cup \widetilde \Omega_{t,h} \cup Q_{i_0})$ has precisely a component $V_1$ containing $\Theta_1$ and a component $V_3$ containing $\Theta_3$. Now, consider the function
\begin{align*}
g(x)=
\begin{cases}
0, & x\in  {V_1}\cap S\\
\frac{u(x)-(s-h)}{2h}, & x\in \Omega_{s,h}\cap S\\
\frac{u(x)-(t-h)}{2h}, & x\in \widetilde \Omega_{t,h}\cap S\\
1, & x\in V_3\cap S.
\end{cases}
\end{align*}
Using Lemma \ref{unif:Gluing lemma}, one can show that $g\in \mathcal W^{1,2}(S)$. We provide some details. Let $U_s=\frac{u(x)-(s-h)}{2h} \lor 0$, and consider a bump function $\phi_s:\C\to [0,1]$ that is identically equal to $1$ on $\Omega_{s,h}$, and vanishes on $\widetilde\Omega_{t,h}$. Similarly consider $U_t$ and $\phi_t$, which is a bump function equal to $1$ on $\widetilde \Omega_{t,h}$ but vanishes on $\Omega_{s,h}$. Then $U_s\phi_s+U_t\phi_t$ lies in $\mathcal W^{1,2}(S)$. We now have
\begin{align*}
g(x)=\begin{cases}
U_s(x)\phi_s(x)+U_t(x)\phi(x), & x\in S\setminus V_3 \\
1, & x\in V_3\cap S.
\end{cases}
\end{align*} 
On $\partial V_3\cap S$ the two alternatives agree, so by Lemma \ref{unif:Gluing lemma} we conclude that $g\in \mathcal W^{1,2}(S)$.

Also, we have
\begin{align*}
\osc_{Q_i}(g) 
\begin{cases}
=1, & i\in F_{s,h}\cup F_{t,h}\cup \{i_0\}\\
\leq \rho(Q_i)/2h, & i\in N_h\\
=0, & i\notin F_{s,h}\cup F_{t,h}\cup \{i_0\}\cup N_h.
\end{cases}
\end{align*}  
We only need to justify the middle inequality. If $Q_i$, $i\in N_h$, intersects only one of $\Omega_{s,h}, \widetilde\Omega_{t,h}$, then it is clear by the definition of $g$. If $Q_i$ intersects both $\Omega_{s,h}$ and  $\widetilde\Omega_{t,h}$ then $Q_i$ has to meet either $V_1$ or $V_3$. Suppose that it meets $V_1$. Then the minimum of $g$ on $\partial Q_i$ is $0$, and the maximum is attained at a point of $\Omega_{s,h}$ or $\widetilde \Omega_{t,h}$. Suppose that the maximum is attained at a point of $\Omega_{s,h}$. It follows that both the minimum and the maximum  of $g$ on $\partial Q_i$ are attained in $\Omega_{s,h}$, and the definition of $g$ implies that the oscillation is bounded by $\rho(Q_i)/2h$, as desired. The other cases yield the same conclusion. 

Since every convex combination of $u$ and $g$ is admissible for the free boundary problem, it follows that 
\begin{align*}
D(u)\leq \sum_{i\in \N} \rho(Q_i) \osc_{Q_i}(g).
\end{align*}
See also the proof of Proposition \ref{unif:Level sets-sums geq} and \eqref{unif:Level sets-Optimization in s}. Thus,
\begin{align*}
D(u)\leq \sum_{i\in F_{s,h}}\rho(Q_i) +\sum_{i\in F_{t,h}}\rho(Q_i) +\rho(Q_{i_0}) +\frac{1}{2h} \sum_{i\in N_h} \rho(Q_i)^2.
\end{align*}
Letting $h\to 0$ and using \eqref{unif:Conjugate-Infimum-F sets}, together with Lemma \ref{unif:Level sets-Measure Derivative lemma}, we obtain
\begin{align*}
D(u)\leq \ell_{\rho}(\gamma_s) +\ell_{\rho}(\tilde \gamma_t)+\rho(Q_{i_0})= \ell_{\rho}(\alpha_s^{i_0})+\ell_{\rho}(\tilde \gamma_t)+\rho(Q_{i_0}).
\end{align*}
By \eqref{unif:Conjugate-Infimum-ell} we obtain $\ell_{\rho}(\alpha_t^{i_0}) \leq \ell_{\rho}(\alpha_s^{i_0})$. The roles of $s,t$ is symmetric, so the conclusion follows.
\end{proof}

If a point $x\in S$ has arbitrarily small neighborhoods that contain some $Q_i$ with $\rho(Q_i)>0$, then we define\index{conjugate function!definition of}
\begin{align}\label{unif:Conjugate-Definition}
v(x)= \liminf_{Q_i\to x,x\notin Q_i} \hat v(Q_i),
\end{align}
where in the limit we are only using peripheral disks for which $\rho(Q_i)>0$, since for the other peripheral disks $\hat v$ is not defined. Observe that $0\leq v(x)\leq D(u)<\infty$. If $x \in \alpha_t\cap S$ for some $t\in \mathcal T$, then we can approximate $x$ by peripheral disks $Q_i$ that intersect $\alpha_t$; this follows from Lemma \ref{unif:lemma:paths zero hausdorff} because $\mathcal H^1(\alpha_t\cap S)=0$ by Lemma \ref{unif:Level sets-Hausdorff measure}. All these peripheral disks have $\rho(Q_i)>0$ and thus $v(x)$ can be defined by the preceding formula. 

If $v(x_0)$ cannot be defined, this means that there exists a neighborhood of $x_0$ that \textit{contains} only peripheral disks with $\rho(Q_i)=0$; note that $x_0$ could lie on some peripheral circle $\partial Q_{i_0}$ with $\rho(Q_{i_0})>0$. The continuity of $u$ and the upper gradient inequality imply that $u$ is constant in this neighborhood (see Lemma \ref{harmonic:4-zero oscillation lemma}), and in particular it takes some value $t_0$, where $t_0\notin \mathcal T$. We define $U\coloneqq \inter_{\br \Omega}(\alpha_{t_0})$ (i.e., the interior of $\alpha_{t_0}$ rel.\ $\br \Omega$) and observe that if $x_0$ does not lie on a peripheral circle $\partial Q_{i_0}$ with $\rho(Q_{i_0})>0$, then $x_0\in U$. In particular, this is true if $x_0\in S^\circ$. If $x_0$ lies on some peripheral circle with $\rho(Q_{i_0})>0$, then $x_0 \in \br V$, where $V$ is a component of $U$. The set $V$ is rel.\ open in $\br \Omega$ (see Theorem \cite[Theorem 27.9]{Willard:topology}) and if it intersects a peripheral disk $Q_i$, then $\rho(Q_i)=0$ and $Q_i\subset V$. Indeed, if $V\cap Q_i\neq \emptyset$ then the (classical) harmonic function $\tilde u\big|_{Q_i}$ is constant, equal to $t_0$, since it is constant on an open set. Another observation is that $V$ is pathwise connected. This is because it is a component of the open subset $U$ of $\br \Omega$, and $\br \Omega$ is locally pathwise connected; see \cite[Theorem 27.5 and Theorem 27.9]{Willard:topology}. In fact, $V\cap \Omega$ is pathwise connected, because any path in $V$ that connects two points of $V\cap \Omega$ is homotopic to a path in $V\cap \Omega$ that connects the same points.

The following lemma allows us to define $v$ on all of $S$. 

\begin{lemma}\label{unif:Conjugate-Define when rho=0}
Let $V$ be a component of $\inter_{\br\Omega}(\alpha_{t_0})$ for some $t_0\notin \mathcal T$. Then the function $v$ has the same value on the points $x \in \partial_{\br \Omega} V\cap S$ for which the formula \eqref{unif:Conjugate-Definition} is applicable, and there exists at least one such point.
\end{lemma}

Here $\partial_{\br \Omega} V$ is the boundary of $V$ rel.\ $\br \Omega$. Hence, the lemma allows us to define $v$ to be constant on $\br V\cap S$.

For this proof and other proofs in this section we will use repeatedly Lemma \ref{unif:Zeta lemma} from Section \ref{unif:Appendix}.

\begin{proof}
We will split the proof in four cases. We give  details in the proof of Case 0 below, and for the rest of the cases we will describe the variational argument that has to be used, skipping some of the details.

\textbf{Case 0:} $V$ contains only one peripheral disk $Q_{i_0}$, hence $V=Q_{i_0}$. The latter implication holds since $V$ is rel.\ open in $\br \Omega$ and if it contained points outside $Q_{i_0}$ it would also contain other peripheral disks. By the discussion preceding the statement of the lemma, it follows that on every point of $\partial Q_{i_0}$ the formula \eqref{unif:Conjugate-Definition} is applicable. Also, $\tilde u \equiv t_0$ on $Q_{i_0}$, thus $\rho(Q_{i_0})=0$. By a variational argument we now show that $v$ is constant on $\partial Q_{i_0}$. Let $a,b\in \partial Q_{i_0}$ be arbitrary points, and assume there exists $\varepsilon>0$ such that $v(b)-v(a)\geq 10\varepsilon$. Using Lemma \ref{unif:Zeta lemma}, for every $\eta>0$ we can find a test function $\zeta \in \mathcal W^{1,2}(S)$ that vanishes on $\partial \Omega$ with $0\leq \zeta\leq 1$, such that $\zeta \equiv 1$ on small disjoint balls $B(a,r)\cup B(b,r) \subset \Omega$, and
\begin{align}\label{unif:Conjugate-Case0-Define when rho=0-  zeta}
 D(\zeta)-\osc_{Q_{i_0}}(\zeta)^2<\eta.
\end{align}
Using \eqref{unif:Conjugate-Definition}, we can find peripheral disks $Q_{i_a} \subset B(a,r)$ and $Q_{i_b} \subset B(b,r)$ with $\rho(Q_{i_a})>0$, $\rho(Q_{i_b})>0$ and 
\begin{align}\label{unif:Conjugate-Case0-Define when rho=0-  Q_ia Q_ib}
\hat v(Q_{i_b}) -\hat v(Q_{i_a}) >9\varepsilon.
\end{align}
By Lemma \ref{unif:Conjugate-Infimum not needed}, we can consider $s,t\in \mathcal T$, $s\neq t$, such that for the smallest open subpaths $\gamma_s,\gamma_t$ of $\alpha_s,\alpha_t$, respectively, that connect $\Theta_2$ to $Q_{i_a},Q_{i_b}$, respectively, we have
\begin{align}\label{unif:Conjugate-Case0-Define when rho=0-   ell gamma_a gamma_b}
\hat v(Q_{i_a}) = \ell_{\rho}(\gamma_s) \quad \textrm{and} \quad \hat v(Q_{i_b}) = \ell_{\rho}(\gamma_t).
\end{align}
We also denote by $\tilde \gamma_t$ the smallest open subpath of $\alpha_t$ that connects $Q_{i_b}$ to $\Theta_4$. By Theorem \ref{unif:Level sets-sums equal mass} we have
\begin{align*}
\ell_{\rho}(\gamma_t)+\ell_{\rho}(\tilde \gamma_t) \leq \ell_{\rho}(\alpha_t)=D(u).
\end{align*}
Thus, by \eqref{unif:Conjugate-Case0-Define when rho=0-   ell gamma_a gamma_b}
\begin{align*}
\hat v(Q_{i_a}) -\hat v(Q_{i_b}) \geq \ell_\rho(\gamma_s)+\ell_\rho(\tilde \gamma_t) -D(u)
\end{align*}
We claim that 
\begin{align}\label{unif:Conjugate-Case0-Define when rho=0   claim}
\ell_\rho(\gamma_s)+\ell_\rho(\tilde \gamma_t) \geq D(u)-\varepsilon.
\end{align}
This, together with the previous inequality, contradicts \eqref{unif:Conjugate-Case0-Define when rho=0-  Q_ia Q_ib}.

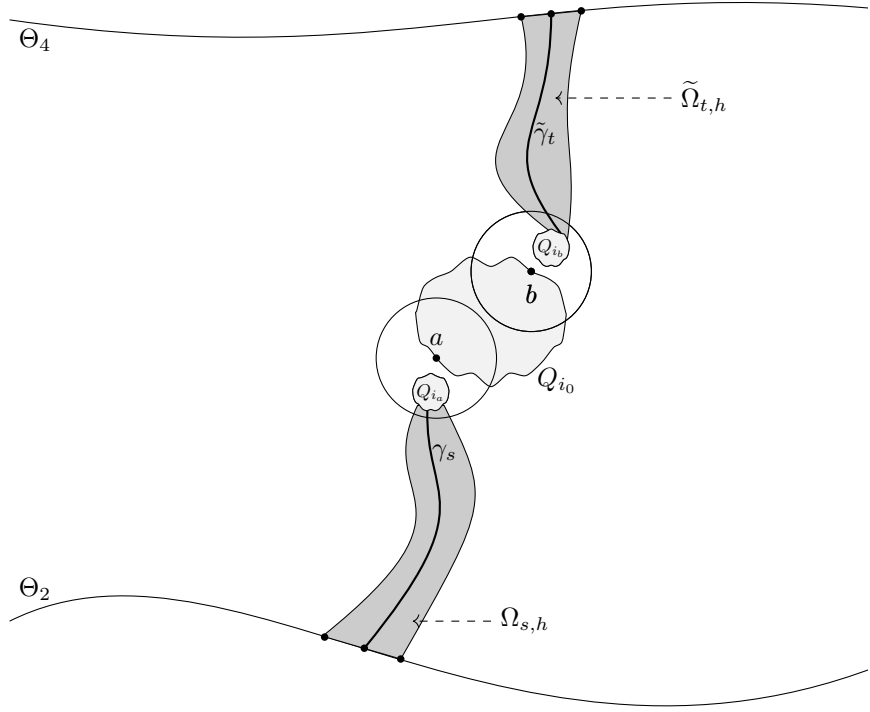
\begin{figure}
	\centering
	\input{conjugate.tikz}
	\caption{The curves $\gamma_s$ and $\tilde \gamma_t$, and the corresponding Jordan regions $\Omega_{s,h}$ and $\widetilde\Omega_{t,h}$.}\label{unif:Conjugate-Figure}
\end{figure}

We now focus on proving \eqref{unif:Conjugate-Case0-Define when rho=0   claim}. As in the proof of Lemma \ref{unif:Conjugate-Infimum not needed}, we can consider a small $h>0$ and disjoint closed Jordan regions $\Omega_{s,h}$ containing $\gamma_s$ and $\widetilde \Omega_{t,h}$ containing $\tilde \gamma_t$ such that $\Omega_{s,h}$ connects $Q_{i_a}$ to $\Theta_2$ and $\widetilde \Omega_{t,h}$ connects $Q_{i_b}$ to $\Theta_4$; see Figure \ref{unif:Conjugate-Figure}. Define $F_{s,h}$ to be the set of indices $i\in \N$ such that $Q_i\cap \gamma_{s} \neq \emptyset $, $Q_{i}\cap \alpha_{s- h}\neq \emptyset$, and $Q_{i}\cap \alpha_{s+ h}\neq \emptyset$. Define similarly $F_{t,h}$ that corresponds to $\widetilde \Omega_{t,h}$, and set $N_h=\{i\in \N: Q_i\cap (A_{s-h,s+h}\cup A_{t-h,t+h}) \neq \emptyset \} \setminus (F_{s,h}\cup F_{t,h} )$.

Now, we define carefully an admissible function that lies in $\mathcal W^{1,2}(S)$. Let $\phi_s\colon \br\Omega \to [0,1]$ be a smooth function that is equal to $1$ on $\Omega_{s,h}$ but vanishes outside $\Omega_{s,2h}\cup Q_{i_a}$ (choose a smaller $h>0$ if necessary so that the Jordan region $\Omega_{s,2h}$ still connects $\Theta_2$ to $Q_{i_0}$). Similarly, define $\phi_t$ to be $1$ on $\widetilde \Omega_{t,h}$, and $0$ outside $\widetilde \Omega_{t,2h}\cup Q_{i_b}$. Then define
\begin{align}\label{unif:Conjugate-U_s}
U_s(x)= \begin{cases}\frac{u(x)-(s-h)}{h}\lor 0 , & u(x)<s\\
\frac{s+h-u(x)}{h} \lor 0, & u(x)\geq s 
\end{cases}= \left(1- \frac{|u(x)-s|}{h}\right) \lor 0
\end{align}
for $x\in S$, and similarly define $U_t(x)$ where $s$ is replaced by $t$ in the previous definition. These functions are supported on the sets $\{s-h<u<s+h\}=A_{s-h,s+h}\cap S$ and $\{t-h<u<t+h\}=A_{t-h,t+h}\cap S$, respectively. Also, they lie in the Sobolev space $\mathcal W^{1,2}(S)$ by Lemma \ref{unif:Gluing lemma}. We consider their truncation $U_s  \phi_s + U_t  \phi_t$, and then take $(U_s  \phi_s + U_t  \phi_t)\lor \zeta$. This function again lies in $\mathcal W^{1,2}(S)$, but it vanishes on $\Theta_1$ and $\Theta_3$, so it is not yet admissible for the free boundary problem. To turn it into an admissible function, we would like to set it equal to $1$ near $\Theta_3$.

Consider the union of the following paths: $\gamma_{s}$ that joins $\Theta_2$ to $Q_{i_a}$, a line segment in $B(a,r)$ connecting the endpoint of $\gamma_{s}$ to $a$, an arc inside $Q_{i_0}$ connecting $a$ to $b$, a line segment in $B(b,r)$ connecting $b$ to the endpoint of $\tilde \gamma_{t}$, and $\tilde \gamma_{t}$, where the latter connects $\Theta_4$ to $Q_{i_b}$. This union separates $\Theta_1$ from $\Theta_3$, so it contains a simple path $\gamma$ that connects $\Theta_2$ to $\Theta_4$ and separates $\Theta_1$ from $\Theta_3$ (see e.g.\ the proof of Lemma \ref{unif:Level sets-Topological Lemma}). Let $W$ be the component of $\br\Omega\setminus \gamma$ that contains $\Theta_3$. We define
\begin{align*}
g(x)= \begin{cases}
(U_s(x)  \phi_s(x) + U_t(x) \phi_t(x))\lor \zeta(x), & x\in  S\setminus   W \\
1, & x\in  W\cap S.
\end{cases}
\end{align*} 
On $\gamma\cap S$ the function  $g$ is equal to $1$ and the two alternatives agree, so by Lemma \ref{unif:Gluing lemma} we conclude that $g\in \mathcal W^{1,2}(S)$ and $g$ is admissible. Furthermore, by construction and Lemma \ref{unif:Gluing lemma} we have
\begin{align*}
\osc_{Q_i}(g)\leq \osc_{Q_i}(\zeta)+
\begin{cases}
1, & i\in F_{s,h}\cup F_{t,h}\cup \{i_0\}\\
\rho(Q_i)/2h, & i\in N_h\\
0, & i\notin F_{s,h}\cup F_{t,h}\cup \{i_0\}\cup N_h.
\end{cases}
\end{align*}
Testing the minimizing property of $u$ against $g$, as usual (see the proof of Proposition \ref{unif:Level sets-sums geq} and \eqref{unif:Level sets-Optimization in s}), we obtain
\begin{align*}
D(u)&\leq \sum_{i\in \N}\rho(Q_i)\osc_{Q_i}(g)\\
&\leq \sum_{i\in \N }\rho(Q_i) \osc_{Q_i}(\zeta)+\sum_{i\in F_{s,h}}\rho(Q_i) +\sum_{i\in F_{t,h}}\rho(Q_i) +\rho(Q_{i_0}) +\frac{1}{2h} \sum_{i\in N_h} \rho(Q_i)^2.
\end{align*}
Letting $h\to 0$ and using Lemma \ref{unif:Level sets-Measure Derivative lemma} we obtain
\begin{align}\label{unif:Conjugate:Case0}
\begin{aligned}
D(u)&\leq \sum_{i\in \N \setminus \{i_0\}}\rho(Q_i) \osc_{Q_i}(\zeta) +\ell_{\rho}(\gamma_s) +\ell_{\rho}(\tilde \gamma_t)+\rho(Q_{i_0})\\
&\leq D(u)^{1/2}(D(\zeta)- \osc_{Q_{i_0}}(\zeta)^2)^{1/2}+\ell_{\rho}(\gamma_s) +\ell_{\rho}(\tilde \gamma_t) +\rho(Q_{i_0})\\
&\leq D(u)^{1/2}\eta^{1/2}+\ell_{\rho}(\gamma_s) +\ell_{\rho}(\tilde \gamma_t) +\rho(Q_{i_0}),
\end{aligned}
\end{align}
where we used \eqref{unif:Conjugate-Case0-Define when rho=0-  zeta}. Note that $\rho(Q_{i_0})=0$. If $\eta$ is chosen to be sufficiently small, depending on $\varepsilon$, then \eqref{unif:Conjugate-Case0-Define when rho=0   claim} is satisfied and our claim is proved. We have completed the proof of Case 0.

\begin{figure}
	\centering
	\input{conjugate_cases.tikz}
	\caption{Case 1 (left) and Case 2b (right).}\label{unif:fig:cases}
\end{figure}
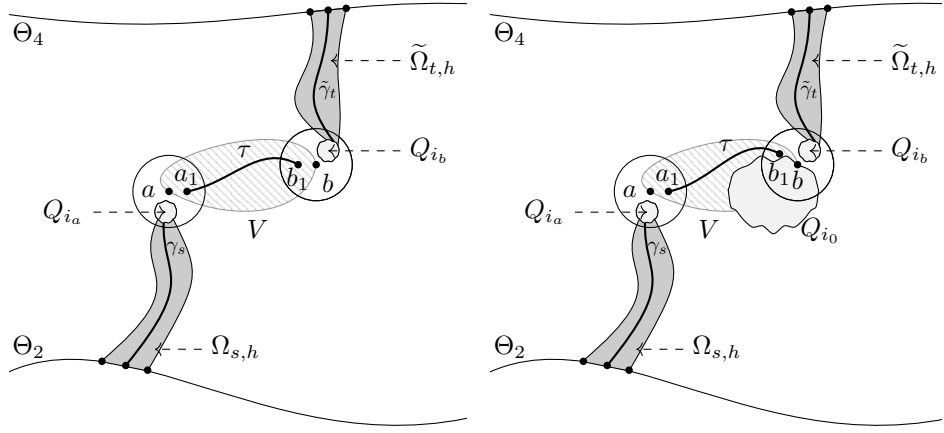

Now, we analyze the remaining cases. If $V$ contains more than one peripheral disk, then it also contains a point $x_0\in S^\circ$, as one can see by connecting two peripheral disks of $V$ with a path inside $V$ (recall that $V$ is pathwise connected by the discussion preceding Lemma \ref{unif:Conjugate-Define when rho=0}) and then applying Corollary \ref{unif:cor:paths in S^o}. Using Lemma \ref{unif:Paths in S^o}, we can find a path $\beta\subset S^\circ $ that connects $x_0$ to one of $\Theta_1$ or $\Theta_3$ that is not intersected by $V\subset \alpha_{t_0}$ (recall that $\Theta_1\subset \alpha_0$ and $\Theta_3\subset \alpha_1$). Then $\beta$ has to exit $V$, so it meets $\partial_{\br \Omega}V\cap S^\circ $ at a point $a$. Since $a\in S^\circ,$ it follows that $v(a)$ is defined by \eqref{unif:Conjugate-Definition}. Indeed, by the discussion preceding Lemma \ref{unif:Conjugate-Define when rho=0}, if $a\in S^\circ$ and $v(a)$ cannot be defined by \eqref{unif:Conjugate-Definition}, then we necessarily have $a\in \inter_{\br\Omega}(\alpha_{t_0})$ so $a\notin \partial_{\br \Omega} V$. We claim that for any point $b \in \partial_{\br\Omega} V\cap S$ for which definition \eqref{unif:Conjugate-Definition} applies we have $v(b)=v(a)$. 

\textbf{Case 1:} $b\in S^\circ$; see Figure \ref{unif:fig:cases}. We keep the same notation as in Case 0. Since $a,b\in S^\circ$, by Lemma \ref{unif:Zeta lemma}, for every $\eta>0$ we can find a test function $\zeta$ that vanishes on $\partial \Omega$ with  $0\leq \zeta \leq 1$, such that $\zeta\equiv 1$ on disjoint balls $B(a,r),B(b,r)\subset \Omega$ and $D(\zeta)<\eta$. We choose $Q_{i_a} \subset B(a,r)$ and $Q_{i_b} \subset B(b,r)$ as before, and heading for a contradiction, we wish to show again \eqref{unif:Conjugate-Case0-Define when rho=0   claim} using a variational argument. 

Consider, as in Case 0 the functions $U_s,U_t,\phi_s,\phi_t$ with exactly the same definition, and the function $(U_s\phi_s+U_t\phi_t) \lor \zeta$ (we choose a sufficiently small $h>0$ so that this function vanishes on $\Theta_1$ and $\Theta_3$). Also, consider the union of the following paths: $\gamma_{s}$ that joins $\Theta_2$ to $Q_{i_a}$, a line segment in $B(a,r)$ connecting  the endpoint of $\gamma_{s}$ to a point $a_1 \in V\cap B(a,r)$, a simple path $\tau$ inside the pathwise connected open set $V\cap \Omega$ joining $a_1$ to a point $b_1\in V\cap B(b,r)$, a line segment in $B(b,r)$ joining $b_1$ to the endpoint of $\tilde \gamma_{t}$, and $\tilde \gamma_{t}$. This union contains a simple path $\gamma$ that separates $\Theta_1$ from $\Theta_3$. However, the function $(U_s\phi_s+U_t\phi_t) \lor \zeta$ need not be equal to $1$ on $\gamma\cap \tau \cap S$. We amend this as follows.

Let $\delta>0$ be so small that $N_\delta(\tau)\subset\subset V\cap \Omega$. Consider the Lipschitz function $\psi (x)= \max\{1-\delta^{-1}\dist(x,\tau),0\}$ on $\br \Omega$. Since $\psi\equiv 1$ on $\tau$, it follows that the function $(U_s\phi_s+U_t\phi_t) \lor \zeta \lor \psi $ is equal to $1$ on $\gamma\cap S$. Let $W$ be the component of $\br\Omega\setminus \gamma$ containing $\Theta_3$, and define
\begin{align*}
g(x)= \begin{cases}
(U_s(x)  \phi_s(x) + U_t(x) \phi_t(x))\lor \zeta(x) \lor \psi(x), & x\in  S\setminus   W \\
1, & x\in  W\cap S.
\end{cases}
\end{align*} 
This function lies in $\mathcal W^{1,2}(S)$ by Lemma \ref{unif:Gluing lemma} and it is admissible. By construction, using notation from Case 0 we have
\begin{align*}
\osc_{Q_i}(g)\leq \osc_{Q_i}(\zeta)+
\begin{cases}
1, & i\in F_{s,h}\cup F_{t,h}\\
\rho(Q_i)/2h, & i\in N_h\\
0, & i\notin F_{s,h}\cup F_{t,h}\cup N_h\cup I_V\\
\delta^{-1}{\diam(Q_i)}, & i\in I_V.
\end{cases}
\end{align*}
Recall that $I_V=\{i\in \N: Q_i\cap V\neq \emptyset\}$ and that $Q_i\subset V$ whenever $Q_i\cap V\neq \emptyset$; see remarks before Lemma \ref{unif:Conjugate-Define when rho=0}. Testing the minimizing property of $u$ against $g$ we obtain
\begin{align}\label{unif:Conjugate-Case1-variation argument} 
\begin{aligned}
D(u)&\leq \sum_{i\in \N}\rho(Q_i)\osc_{Q_i}(g)\\
&\leq \sum_{i\in \N }\rho(Q_i) \osc_{Q_i}(\zeta)+\sum_{i\in F_{s,h}}\rho(Q_i) +\sum_{i\in F_{t,h}}\rho(Q_i) \\
&\quad \quad+\frac{1}{2h} \sum_{i\in N_h} \rho(Q_i)^2 + \sum_{i:Q_i\subset V} \delta^{-1}\rho(Q_i)\diam(Q_i).
\end{aligned}
\end{align}
Note that $\rho(Q_i)=0$ for all $Q_i\subset V$, so the last term vanishes identically. By letting $h\to 0$ and choosing a sufficiently small $\eta>0$, we again arrive at \eqref{unif:Conjugate-Case0-Define when rho=0   claim} and this completes the proof.

\textbf{Case 2:} $b\in \partial Q_{i_0}$ for some $i_0\in \N$. Since $b\in \partial Q_{i_0}$ (as in Case 0), the difficulty is that for the test function $\zeta$ (equal to $1$ near $b$) we have $\osc_{Q_{i_0}}(\zeta)=1$. Thus, in the variational argument the term $\sum_{i\in \N} \rho(Q_i) \osc_{Q_i}(\zeta)$ is not small, unless $\rho(Q_{i_0})=0$, which was true in Case 0. We assume here that $\rho(Q_{i_0})>0$, since otherwise the argument is similar, but simpler. Then $Q_{i_0}\cap V=\emptyset$, by the remarks before the statement of Lemma \ref{unif:Conjugate-Define when rho=0}.

As in the other cases, we assume $v(b)-v(a)\geq 10\varepsilon$ and consider a test function $\zeta\in \mathcal W^{1,2}(S)$ that vanishes on $\partial \Omega$ with $\zeta\equiv 1$ on $B(a,r)\cup B(b,r)$, and $D(\zeta)-\osc_{Q_{i_0}}(\zeta)^2<\eta$. Furthermore, by the definition \eqref{unif:Conjugate-Definition}, we may take $Q_{i_a}\subset B(a,r)$ and $Q_{i_b}\subset B(b,r)$ such that $\hat v(Q_{i_a})-\hat v(Q_{i_b})>9\varepsilon$, and consider as in Case 0 the open paths $\gamma_s$ and $\tilde \gamma_t$. Again, we are aiming for \eqref{unif:Conjugate-Case0-Define when rho=0   claim}. We now split into two sub-cases:

\textbf{Case 2a:} $\tilde \gamma_t \cap Q_{i_0}\neq \emptyset$. In this case, we use exactly the same function $g\in \mathcal W^{1,2}(S)$ that we used in Case 1, with the observation that $i_0 \in F_{t,h}$, and $\osc_{Q_{i_0}}(g)\leq 1$. Hence, (for all sufficiently small $h>0$) we obtain the inequality \eqref{unif:Conjugate-Case1-variation argument} with the first sum replaced by a sum over $i\in \N \setminus \{i_0\}$, because the term $i=i_0$ is included in the third sum. The fact that $\sum_{i\in \N \setminus \{i_0\}} \rho(Q_i)\osc_{Q_i}(\zeta)$ can be made arbitrarily small, leads to the conclusion.

\textbf{Case 2b:} $\tilde \gamma_t\cap Q_{i_0}=\emptyset$; see Figure \ref{unif:fig:cases}. In this case, we may assume that $\gamma_s\cap Q_{i_0}=\emptyset$ too, since otherwise we are essentially reduced to the previous case, where the summation index $i=i_0$ is included in the second sum appearing in \eqref{unif:Conjugate-Case1-variation argument}.

We construct a path $\gamma$ that separates $\Theta_1$ from $\Theta_3$ and does not intersect $Q_{i_0}$. Consider the union of the following continua: $\gamma_{s}$, a line segment in $B(a,r)$ joining the endpoint of $\gamma_{s}$ to a point $a_1\in V\cap B(a,r)$, a simple path $\tau \subset V$ connecting $a_1$ to a point $b_1\in V\cap B(b,r)$, a  \textit{simple path $\gamma_b$ in} $B(b,r)\setminus Q_{i_0}$ connecting $b_1$ to the endpoint of $\tilde \gamma_{t}$, and $\tilde \gamma_{t}$. For the existence of $\gamma_b$ note that if $Q_{i_b}$ is sufficiently close to $b$ then it has to lie in the component of $B(b,r)\setminus Q_{i_0}$ that contains $b$ in its boundary; here one uses the local connectedness property of the Jordan curve $\partial Q_{i_0}$. Now we let $\gamma$ be a simple path contained in this union, and separating $\Theta_1$ from $\Theta_3$. Since $Q_{i_0}\cap V=\emptyset$, it follows that $\tau \cap Q_{i_0}=\emptyset$. To ensure that $\gamma\cap Q_{i_0}=\emptyset$, one only has to take a possibly smaller ball $B(a,r)$, so that the line segment we considered there does not meet $Q_{i_0}$; recall that $a\in S^\circ$.

Let $W$ be the component of $\br \Omega \setminus \gamma$ that contains $\Theta_3$, and by construction $Q_{i_0}\subset W$ or $Q_{i_0}\subset \br\Omega\setminus W$. If $Q_{i_0} \subset W$, then we consider the variation $g$ as in Case 1, which is admissible for a sufficiently small $h>0$. Since $g=1$ on $\partial Q_{i_0}$, we have $\osc_{Q_{i_0}}(g)=0$. Hence, we may have \eqref{unif:Conjugate-Case1-variation argument} with the first sum replaced  by $\sum_{i\in \N \setminus \{i_0\}} \rho(Q_i)\osc_{Q_i}(\zeta)$ and this completes the proof.

Assume now that $Q_{i_0} \subset \br\Omega\setminus W$. Consider the function $g_0\coloneqq (U_s\phi_s+U_t\phi_t) \lor \zeta \lor \psi$ as in Case 1, and define $g_1= 1 $ on $S \setminus W$, and $g_1=g_0$ on $S\cap W$. Then $g_2\coloneqq 1-g_1$ is admissible and satisfies $\osc_{Q_{i_0}}(g_2)=0$, which yields the result.

\textbf{Case 3:} $b\in \partial \Omega$. The complication here is that our test function $\zeta$ does not vanish at $\partial \Omega$. It turns out though that we can always construct an admissible function using the procedure in the Case 2b, as we explain below.

We consider a path $\gamma$ as before that separates $\Theta_1$ and $\Theta_3$ such that $\gamma$ contains a simple path $\tau\subset V$ as in Case 2. The function $g_0\coloneqq  (U_s\phi_s+U_t\phi_t) \lor \zeta \lor \psi$ is equal to $1$ on $\gamma\cap S$. Let $W$ be the component of $\br\Omega \setminus \gamma$ that contains $\Theta_3$. If $g_0=1$ on points of $\Theta_3$, then $g_0=0$ on $\Theta_1$ (by choosing a sufficiently small $h>0$ and a $\zeta$ with small support), so we can set $g=1$ on $ S\cap W$ and $g=g_0$ on $S\setminus W$. If $g_0=1$ on points of $\Theta_1$ then we ``flip" the function $g$. We set $g_1=1$ on $S\setminus W$ and $g_1=g_0$ on $S\cap W$. Then $g_2\coloneqq 1-g_1$. 

The energy $D(\zeta)$ can be made arbitrarily small because $b\in \partial \Omega$ (recall Lemma \ref{unif:Zeta lemma}), so $\sum_{i\in \N} \rho(Q_i)\osc_{Q_i}(\zeta)$ can be made arbitrarily small in \eqref{unif:Conjugate-Case1-variation argument}. In either case, running the variational argument with the admissible function $g$ or $g_2$ will yield the conclusion. 
\end{proof}

Now we can define $v$ on all of $S$ as follows. Let $x\in S$. From the analysis preceding the statement of Lemma \ref{unif:Conjugate-Define when rho=0} we see that either $v(x)$ can be defined by the formula \eqref{unif:Conjugate-Definition}, or $x\in \br V\cap S$, where $V$ is a component of $U\coloneqq \inter_{\br \Omega}(\alpha_{t_0})$ for some $t_0\notin \mathcal T$. In the latter case, we define $v(x)$ to be the constant value of $v$ on the points of $\partial_{\br \Omega}V\cap S$ for which \eqref{unif:Conjugate-Definition} is applicable; there exists at least one such point by Lemma \ref{unif:Conjugate-Define when rho=0}.

Following a similar argument to the proof of Lemma \ref{unif:Conjugate-Define when rho=0}, we show that $v$ is continuous on $S$.

\begin{lemma}\label{unif:Conjugate-Continuous on S}\index{conjugate function!continuity of}
The function $v\colon S\to \R$ is continuous.
\end{lemma}
\begin{proof}
The proof uses essentially the same variational arguments as in the proof of Lemma \ref{unif:Conjugate-Define when rho=0}, so we skip most of the details. We will also use the same notation as in Lemma \ref{unif:Conjugate-Define when rho=0}. This time we do not need to use the function $\psi$, and we only need to consider a test function $\zeta$ that is supported in one small ball around the point that we wish to show continuity, rather than having two disjoint balls (so the variational arguments here are simpler). By the definition of $v$ we only need to prove continuity for points $x_0$ such that $v(x_0)$ is defined by \eqref{unif:Conjugate-Definition}. Indeed, if $x_0$ cannot be defined by \eqref{unif:Conjugate-Definition}, then $x_0$ has a neighborhood in $S$ where $v$ is constant. This is because this neighborhood is contained in $\br V$, where $V$ is a component of some set $U=\inter_{\br \Omega}(\alpha_{t_0})$ for a level  $t_0\notin \mathcal T$; see also the comments before the statement of Lemma \ref{unif:Conjugate-Define when rho=0}.

For $\eta>0$ consider a function $\zeta\in \mathcal W^{1,2}(S)$ supported around $x_0$, with $D(\zeta)<\eta$ and $\zeta\equiv 1$ in $B(x_0,r)$. If continuity at $x_0$ fails, then there exists $\varepsilon>0$ and a point $y_0\in B(x_0,r)\cap S$ arbitrarily close to $x_0$ such that, say, $v(y_0)-v(x_0)>10\varepsilon$. We claim that there exist peripheral disks $Q_{i_a},Q_{i_b} \subset B(x_0,r)$ arbitrarily close to $x_0$ such that $\rho(Q_{i_a}),\rho(Q_{i_b})>0$ and 
\begin{align}\label{unif:Conjugate-Continuous-assumption}
\hat v(Q_{i_b})-\hat v(Q_{i_a})>9\varepsilon.
\end{align}
Note that by the definition \eqref{unif:Conjugate-Definition} there exists a peripheral disk $Q_{i_a}$ arbitrarily close to $x_0$ with $\rho(Q_{i_a})>0$ and $|v(x_0)-\hat v(Q_{i_a})|<\varepsilon/2$. If $v(y_0)$ can be defined by \eqref{unif:Conjugate-Definition}, then we can find a peripheral disk $Q_{i_b}$ arbitrarily close to $y_0$ (and thus close to $x_0$) such that $|v(y_0)-\hat v(Q_{i_b})|<\varepsilon/2$. We now easily obtain \eqref{unif:Conjugate-Continuous-assumption}. If $y_0$ cannot be defined by \eqref{unif:Conjugate-Definition}, then $y_0\in \br V$, where $V$ is a component of $U\coloneqq \inter_{\br \Omega}(\alpha_{t_0})$ for some $t_0\notin \mathcal T$. Note that $x_0\notin \br V$, otherwise we would have $v(x_0)=v(y_0)$ by Lemma \ref{unif:Conjugate-Define when rho=0}. If $y_0$ is sufficiently close to $x_0$, then by Lemma \ref{unif:Paths in S^o}(a) we may consider a path $\gamma\subset B(x,r)\cap S^\circ$ that connects $x_0$ and $y_0$. Parametrizing $\gamma$ as it runs from $x_0$ to $y_0$, we consider its first entry point $z_0$ in $\partial_{\br \Omega}V\cap S^\circ$. Then we necessarily have that $v(y_0)=v(z_0)$ and $v(z_0)$ can be defined by \eqref{unif:Conjugate-Definition}; see comments before Lemma \ref{unif:Conjugate-Define when rho=0}. Thus, there exists a small $Q_{i_b}$ near $z_0$ such that $|v(z_0)-\hat v(Q_{i_b})|<\varepsilon/2$. If $y_0$ is very close to $x_0$ then $\gamma\ni z_0$ is also very close to $x_0$ (by Lemma \ref{unif:Paths in S^o}) and thus $Q_{i_b}$ is close to $x_0$. Now \eqref{unif:Conjugate-Continuous-assumption} follows from the above and from the assumption that $v(z_0)-v(x_0)=v(y_0)-v(x_0)>10\varepsilon$.

By Lemma \ref{unif:Conjugate-Infimum not needed} we consider $s,t\in \mathcal T$ such that $\hat v(Q_{i_a})=\ell_\rho(\gamma_s)$ and $\hat v(Q_{i_b})=\ell_\rho(\gamma_t)$ where $\gamma_s$ and $\gamma_t$ are the smallest open subpaths of $\alpha_s,\alpha_t$ that connect $\Theta_2$ to $Q_{i_a},Q_{i_b}$, respectively. Also, denote by $\tilde \gamma_t$ the smallest open subpath of $\alpha_t$ that connects $Q_{i_b}$ to $\Theta_4$. Arguing as in the beginning of the proof of Lemma \ref{unif:Conjugate-Define when rho=0}, it suffices to show \eqref{unif:Conjugate-Case0-Define when rho=0   claim} to obtain a contradiction. We will split again in cases, and sketch the variational argument that has to be used. Recall that will use the notation from the proof of Lemma \ref{unif:Conjugate-Define when rho=0}.

\textbf{Case 1:} $x_0\in S^\circ$.  We consider the function $g_0\coloneqq (U_s\phi_s+U_t\phi_t)\lor \zeta$, and a simple path $\gamma$ that separates  $\Theta_1$ from $\Theta_3$, such that $g_0=1$ on $\gamma\cap S$. Let $W$ be the component of $\br \Omega\setminus \gamma$ that contains $\Theta_3$. We define $g=1$ on $S\cap W$ and $g=g_0$ on $S\setminus W$. This yields an admissible function, as long as $h>0$ is sufficiently small so that $g_0$ vanishes on $\Theta_1$. Then by testing the minimizing property of $u$ against $g$ we obtain the conclusion.

\textbf{Case 2:} $x_0\in \partial Q_{i_0}$ for some $i_0\in \N$. 

\textbf{Case 2a:} At least one of the sets $\gamma_s,\tilde \gamma_t$ meets $Q_{i_0}$. This is treated similarly to Case 2a in Lemma \ref{unif:Conjugate-Define when rho=0}. The observation is that we have either $i_0\in F_{s,h}$ or $i_0\in F_{t,h}$.

\textbf{Case 2b:} None of $\gamma_s,\tilde \gamma_t$ meets $Q_{i_0}$. This is similar to the Case 2b in Lemma \ref{unif:Conjugate-Define when rho=0}, since we can construct a simple path $\gamma$ that separates $\Theta_1$ from $\Theta_3$ so that the function $g_0=(U_s\phi_s+U_t\phi_t)\lor \zeta$ is equal to $1$ on $\gamma\cap S$ and vanishes on $\Theta_1$ and $\Theta_3$. 

\textbf{Case 3:} $x_0\in \partial \Omega$. We consider a path $\gamma$ as before such that $g_0=1$ on $\gamma\cap S$, and let $W$ be the component of $\br\Omega \setminus \gamma$ that contains $\Theta_3$. If $g_0=1$ on points of $\Theta_3$, then $g_0=0$ on $\Theta_1$ (by choosing a $\zeta$ with small support), so we can set $g=1$ on $ S\cap W$ and $g=g_0$ on $S\setminus W$. If $g_0=1$ on points of $\Theta_1$ then we ``flip" the function $g$. We set $g_1=1$ on $S\setminus W$ and $g_1=g_0$ on $S\cap W$. Then $g_2\coloneqq 1-g_1$. In either case, running the variational argument with the admissible function $g$ or $g_2$ will yield the conclusion.
\end{proof}

\begin{remark}\label{unif:Conjugate-Remark defintion}
A conclusion of the proof is that
\begin{align}\label{unif:Conjugate-Definition NEW}
v(x)=\lim_{Q_i\to x, x\notin Q_i} \hat v(Q_i)
\end{align}
whenever $x$ can be approximated by $Q_i$ with $\rho(Q_i)>0$. If $Q_i \subset \inter_{\br\Omega}(\alpha_{t_0})$, where $t_0\notin \mathcal T$, 
by Lemma \ref{unif:Conjugate-Define when rho=0} we can define $\hat v(Q_i)$ to be equal to the constant value of $v$ on $\partial_{\br \Omega}V\cap S$, where $V$ is a component of $\inter_{\br\Omega}(\alpha_{t_0})$. Thus, \eqref{unif:Conjugate-Definition NEW} can be used to define $v$ at all points $x\in S$.
\end{remark}

In the proofs of Lemma \ref{unif:Conjugate-Define when rho=0} and Lemma \ref{unif:Conjugate-Continuous on S} we were always heading towards proving \eqref{unif:Conjugate-Case0-Define when rho=0   claim}. We see that the following general statement holds (using the notation of the proof of Lemma \ref{unif:Conjugate-Define when rho=0}):

\begin{lemma}\label{unif:Conjugate:general variation}
\begin{enumerate}[\upshape(a)]
\item Let $x\in S$ and $\varepsilon>0$. Then there exists a small $r>0$ such that the following hold.  Suppose that $Q_{i_a},Q_{i_b}$ are peripheral disks contained in $B(x,r)$ with $\rho(Q_{i_a})>0$, $\rho(Q_{i_b})>0$. Moreover, let $\gamma_s\subset \alpha_s$ be a path that connects $\Theta_2$ to $Q_{i_a}$ and $\tilde \gamma_t \subset \alpha_t$ be a path that connects $\Theta_4$ to $Q_{i_b}$, where $s,t\in \mathcal T$. Then
\begin{align*}
\ell_\rho(\gamma_s)+ \ell_\rho(\tilde \gamma_t)\geq D(u)-\varepsilon.
\end{align*} 
In the particular case that $x\in \Theta_2$, we obtain the stronger conclusion $\ell_\rho(\tilde \gamma_t)\geq D(u)-\varepsilon$ and in case $x\in \Theta_4$ we have $\ell_\rho(\gamma_s)\geq D(u)-\varepsilon$.
\item Let $\varepsilon>0$, $i_0\in \N$, and $x,y\in \partial Q_{i_0}$. Then there exists a small $r>0$ such that the following hold. Suppose that $Q_{i_a}\subset B(x,r)$, $Q_{i_b}\subset B(y,r)$ with $\rho(Q_{i_a})>0$, $\rho(Q_{i_b})>0$. Moreover, let $\gamma_s\subset \alpha_s$ be a path that connects $\Theta_2$ to $Q_{i_a}$ and $\tilde \gamma_t \subset \alpha_t$ be a path that connects $\Theta_4$ to $Q_{i_b}$, where $s,t\in \mathcal T$. Then
\begin{align*}
\ell_\rho(\gamma_s)+\rho(Q_{i_0})+ \ell_\rho(\tilde \gamma_t)\geq D(u)-\varepsilon.
\end{align*}
\end{enumerate}
\end{lemma}
\begin{proof}
Part (b) is proved with a slight modification of the proof of Case 0 in Lemma \ref{unif:Conjugate-Define when rho=0}; see \eqref{unif:Conjugate:Case0} and Figure \ref{harmonic:fig:accessible}. The first claim in part (a) of the lemma follows essentially from the arguments in the proof of Lemma \ref{unif:Conjugate-Continuous on S}. We will only justify the latter claim in (a) for $x\in \Theta_2$, since the case $x\in \Theta_4$ is analogous. The notation that we use is the same as in the proof of Lemma \ref{unif:Conjugate-Define when rho=0}. Suppose that $x\in \Theta_2$ and consider a test function $\zeta \in \mathcal W^{1,2}(S)$ with $\zeta\equiv 1$ in a small ball $B(x,r)$ and with small Dirichlet energy $D(\zeta)$. This is provided as usual by Lemma \ref{unif:Zeta lemma}. Consider the function $g_0\coloneqq (U_t\phi_t)\lor \zeta$. There exists a simple path $\gamma\subset \br \Omega$ that separates $\Theta_1$ and $\Theta_3$ with $g_0=1$ on $\gamma\cap S$. We let $W$ be the component of $\br \Omega \setminus \gamma$ that contains $\Theta_3$. If $g_0=1$ on points of $\Theta_3$, then $g_0=0$ on $\Theta_1$ (by choosing a $\zeta$ with a sufficiently small support), so we can set $g=1$ on $ S\cap W$ and $g=g_0$ on $S\setminus W$. If $g_0=1$ on points of $\Theta_1$ then we ``flip" the function $g$. We set $g_1=1$ on $S\setminus W$ and $g_1=g_0$ on $S\cap W$. Then $g_2\coloneqq 1-g_1$. In either case, running the standard variational argument with the admissible function $g$ or $g_2$ will yield the conclusion.
\end{proof}

\begin{lemma}\label{unif:Conjugate-Values on boundaries}
We have $v\equiv 0$ on $\Theta_2$ and $v\equiv D(u)$ on $\Theta_4$. Furthermore, for each $i\in \N$ we have $\osc_{Q_i}(v)=\osc_{Q_i}(u)=\rho(Q_i)$. In fact, 
\begin{align*}
\diam( v(\partial Q_i \cap \alpha_t))=\osc_{\partial Q_i\cap \alpha_t}(v)=\rho(Q_i)
\end{align*}
for all $t\in \mathcal T$ with $Q_i\cap\alpha_t\neq \emptyset$.
\end{lemma}
Recall here that $\osc_{Q_i}(v)\coloneqq \sup_{\partial Q_{i}}(v)-\inf_{\partial Q_i}(v)$.

\begin{proof}
If $a\in \Theta_2 \cap \alpha_t$ for $t\in \mathcal T$, then $\rho(Q_i)>0$ whenever $Q_i\cap \alpha_t\neq \emptyset$. By Lemma \ref{unif:lemma:paths zero hausdorff}(b) we can find $Q_i\to a$ with $Q_i\cap \alpha_t\neq \emptyset$. Using \eqref{unif:Conjugate-Definition NEW} and Lemma \ref{unif:Conjugate-Infimum not needed} we obtain
\begin{align}\label{unif:Conjugate-Values on boundaries limits}
v(a)= \lim_{Q_i\to a, Q_i\cap \alpha_t\neq \emptyset} \hat v(Q_i) =\lim_{Q_i \to a, Q_i\cap \alpha_t\neq \emptyset} \sum_{j:Q_j\cap \alpha_t^i\neq \emptyset}\rho(Q_j).
\end{align}
This limit is equal to zero, because the sum $\sum_{j:Q_j\cap {\alpha_t}\neq \emptyset}\rho(Q_j)$ is convergent and $\{j:Q_j\cap {\alpha_t^i}\neq \emptyset\}$ ``converges" to the emptyset as $Q_i\to a$, $Q_i\cap \alpha_t\neq \emptyset$; recall that $\alpha_t$ is a path. Hence, $v(a)=0$ in this case. If $t\notin \mathcal T$ and $a\in \Theta_2\cap \alpha_t$, by Lemma \ref{unif:Conjugate-Define when rho=0}, it suffices to show that $v(a)=0$ whenever $a$ can be approximated by $Q_i$ with $\rho(Q_i)>0$. If $v(a)=10\varepsilon >0$ then we can find a small ball $B(a,r)$ and a peripheral disk $Q_{i_a}\subset B(a,r)$ with $\rho(Q_{i_a})>0$ and $\hat v(Q_{i_a}) >9\varepsilon$. Let $\gamma_t,\tilde \gamma_t \subset \alpha_t$ for $t\in \mathcal T$ be paths as in the proof of Lemma \ref{unif:Conjugate-Define when rho=0} that connect $Q_{i_a}$ to $\Theta_2,\Theta_4$, respectively. We have 
\begin{align*}
\hat v(Q_{i_a})+\rho(Q_{i_a})+ \ell_\rho(\tilde \gamma_t)=\ell_\rho(\gamma_t)+\rho(Q_{i_a})+ \ell_\rho(\tilde \gamma_t)=D(u),
\end{align*}
which implies that $\hat v(Q_{i_a}) \leq D(u) -\ell_\rho(\tilde \gamma_t)$. However, by Lemma \ref{unif:Conjugate:general variation}(a) we have 
\begin{align*}
\ell_\rho(\tilde \gamma_t) \geq D(u)-\varepsilon
\end{align*}
and this leads to a contradiction.

Now, if $a\in \Theta_4\cap \alpha_t$ for $t\in \mathcal T$, again \eqref{unif:Conjugate-Definition NEW} and Lemma \ref{unif:Conjugate-Infimum not needed} we have
\begin{align*}
v(a)=\lim_{Q_i \to a, Q_i\cap \alpha_t\neq \emptyset} \hat v(Q_i) = \lim_{Q_i \to a, Q_i\cap \alpha_t\neq \emptyset} \sum_{j:Q_j\cap \alpha_t^i\neq \emptyset}\rho(Q_j)=D(u),
\end{align*}
where the last equality follows from Theorem \ref{unif:Level sets-sums equal mass}. If $t\notin \mathcal T$ and $a\in \Theta_4\cap \alpha_t$ then an application of Lemma \ref{unif:Conjugate:general variation}(a) proves that $v(a)=D(u)$ as before.

Next, if $\rho(Q_{i_0})=0$ then $Q_{i_0}\subset \alpha_t$ for some $t\notin \mathcal T$ and Lemma \ref{unif:Conjugate-Define when rho=0} implies that $v$ is constant on $\partial Q_{i_0}$, so $\osc_{Q_{i_0}}(u)=\osc_{Q_{i_0}}(v)=0$. We assume that $\rho(Q_{i_0})>0$. If $\partial Q_{i_0}\cap \alpha_t\neq \emptyset$ for some $t\in \mathcal T$, then $\partial Q_{i_0} \cap \alpha_t$ contains precisely two points: $x_2$, which is the entry point of $\alpha_t$ into $Q_{i_0}$, and $x_4$, which is the exit point, as $\alpha_t$ travels from $\Theta_2$ to $\Theta_4$. Using \eqref{unif:Conjugate-Definition NEW} and Lemma \ref{unif:Conjugate-Infimum not needed} as in \eqref{unif:Conjugate-Values on boundaries limits}, it is easy to see that $v(x_2)=\hat v(Q_{i_0})$ and $v(x_4)= \hat v(Q_{i_0})+ \rho(Q_{i_0})$. This shows the last part of the lemma.

Now, if $a \in \partial Q_{i_0} \cap \alpha_{t_0}$ for some $t_0\notin \mathcal T$ then we need to show that $\hat v(Q_{i_0})\leq v(a) \leq \hat v(Q_{i_0})+\rho(Q_{i_0})$. This will complete the proof of the statement that $\osc_{Q_{i_0}}(v)=\rho(Q_{i_0})$. Assume that $v(a)< \hat v(Q_{i_0})-10\varepsilon$ for some $\varepsilon>0$, and without loss of generality assume that $v(a)$ is defined as in \eqref{unif:Conjugate-Definition}, thanks to Lemma \ref{unif:Conjugate-Define when rho=0}. Then we can find some $Q_{i_a} \subset  B(a,r)$ (where $r>0$ is arbitrarily small) with $\rho(Q_{i_a})>0$ such that $\hat v(Q_{i_a}) < \hat v(Q_{i_0})-9\varepsilon$. Consider the smallest open path $\gamma_s\subset \alpha_s$ that connects $\Theta_2$ to $Q_{i_a}$ for some $s\in \mathcal T$, so $\hat v(Q_{i_a})=\ell_{\rho}(\gamma_s)$. Also, consider the smallest open path $\tilde \gamma_t \subset \alpha_t$ that connects $Q_{i_0}$ to $\Theta_4$, for some $t\in \mathcal T$ (this exists because $\rho(Q_{i_0})>0$). In particular, $\hat v(Q_{i_0})+\rho(Q_{i_0})+ \ell_\rho(\tilde \gamma_t)=D(u)$. The path $\tilde \gamma_t$ lands at a point $y\in \partial Q_{i_0}$, and we can find peripheral disks $Q_{i_b}$ intersecting $\tilde \gamma_t$ and lying arbitrarily close to $y$, by Lemma \ref{unif:lemma:paths zero hausdorff}(b). 
Then by Lemma \ref{unif:Conjugate:general variation}(b) we have
\begin{align*}
\ell_{\rho}(\gamma_s)+\rho(Q_{i_0})+\ell_{\rho}(\tilde \gamma_t)\geq D(u)-\varepsilon
\end{align*}
and this contradicts as usual the assumption $\hat v(Q_{i_a}) < \hat v(Q_{i_0})-9\varepsilon$.
\end{proof}

Let us record a corollary:

\begin{corollary}\label{unif:Conjugate-Oscillation-Increasing}
If $t\in \mathcal T$, then $v$ is strictly increasing on $S\cap \alpha_t$, in the sense that if $x,y \in S\cap \alpha_t$ and $\alpha_t$ hits $x$ before hitting $y$ as it travels from $\Theta_2$ to $\Theta_4$, then $v(x)<v(y)$.
\end{corollary}
\begin{proof}
Observe that $\rho(Q_i)>0$ for all $Q_i\cap \alpha_t \neq \emptyset$, and that ``between" any two points $x,y\in S\cap \alpha_t$ there exists some peripheral disk $Q_i$ with $Q_i\cap \alpha_t\neq \emptyset$. To make this more precise, denote by $[x,y]$ the arc of $\alpha_t$ from $x$ to $y$. Then  $\mathcal H^1([x,y]\cap S)=0$ by Lemma \ref{unif:Level sets-Hausdorff measure}, so there exists a peripheral disk $Q_i$ with $Q_i\cap [x,y]\neq \emptyset$; see also Lemma \ref{unif:lemma:paths zero hausdorff}. Using Remark \ref{unif:Conjugate-Remark defintion} and taking limits along peripheral disks that intersect $\alpha_t$ we obtain 
\[
v(x)\leq \hat v(Q_i) \leq v(y)-\rho(Q_i) <v(y). \qedhere
\]
\end{proof}

The function $v$ also satisfies a version of an upper gradient inequality. Recall that $u$ satisfies the upper gradient inequality in Definition \ref{unif:Definition Upper gradient}, where we exclude a family of curves $ \Gamma_0$ with $\md(\Gamma_0)=0$, i.e., vanishing carpet modulus.  

\begin{lemma}\label{unif:Conjugate-Defect upper gradient}\index{conjugate function!upper gradient inequality}
There exists a family of paths $\Gamma_0$ with $\md_2(\Gamma_0)=0$ such that for every path $\gamma \subset \Omega$ with $\gamma\notin \Gamma_0$ and for every open subpath $\beta$ of $\gamma$ we have
\begin{align}\label{unif:Conjugate-Defect upper gradient-To prove}
|v(a)-v(b)|\leq \sum_{i:\br{Q}_i\cap \beta\neq \emptyset} \rho(Q_i).
\end{align}
for all $a,b\in \br \beta\cap S$.
\end{lemma}
We remark that in the sum we are using peripheral disks whose \textit{closure} intersects $\beta$, in contrast to Definition \ref{unif:Definition Upper gradient}. Also, here we are excluding a path family of conformal modulus zero, instead of carpet modulus. This is only a technicality and does not affect the ideas used in the proof.

\begin{proof}
Consider a path $\gamma \subset  \Omega$ with $\mathcal H^1(\gamma\cap S)=0$. This holds for $\md_2$-a.e.\ $\gamma \subset  \Omega$ because $\mathcal H^2(S)=0$. Let $\beta $ be an open subpath of $\gamma$ and assume that $a,b$ are its endpoints. We wish to show that
\begin{align}\label{unif:Conjugate-Defect upper gradient-To prove-beta}
|v(a)-v(b)|\leq \sum_{i:\br{Q}_i\cap \beta\neq \emptyset} \rho(Q_i).
\end{align}
Recall that $\rho(Q_i)=\osc_{Q_i}(v)$ by Lemma \ref{unif:Conjugate-Values on boundaries}. We suppose that 
$$\sum_{i:\br Q_i\cap {\beta}\neq \emptyset}\rho(Q_i) <\infty,$$
otherwise the statement is trivial. The statement is also trivial if $a$ and $b$ lie on the same peripheral circle $\partial Q_i$ and $\beta$ intersects $\br Q_i$, so we assume that this is not the case. If $a\in \partial Q_{i_a}$ for some $i_a\in \N$, let $a'\in \partial Q_{i_a}$ be the last exit point of $\beta$ from $\partial Q_{i_a}$, assuming that it is parametrized to run from $a$ to $b$. Similarly, consider the point $b'\in \partial Q_{i_b}$ of first entry of $\beta$ in $\partial Q_{i_b}$, in case $b\in \partial Q_{i_b}$. Note that $|v(a)-v(a')|\leq \osc_{Q_{i_a}}(v)$ and $|v(b)-v(b')| \leq \osc_{Q_{i_b}}(v)$, so it suffices to prove the statement for the open subpath of $\beta$ that connects $a'$ and $b'$. This subpath has the property that it does not intersect the peripheral disks that possibly contain $a$ and $b$ on their boundary. For simplicity we denote $a'$ by $a$, $b'$ by $b$ and the subpath by $\beta$.

By Lemma \ref{unif:lemma:paths zero hausdorff} we can find arbitrarily close to $a$ peripheral disks $Q_i$ with $Q_i\cap \beta\neq \emptyset$. Using now Lemma \ref{unif:Paths in S^o}(b), one can see that arbitrarily close to $a$ there exist points $a'\in \beta\cap S^\circ$.  Similarly, arbitrarily close to $b$ there exist points $b'\in \beta\cap S^\circ$. By the continuity of $v$, it suffices to prove the statement for $a',b'$ and the subpath of $\beta$ that connects them instead. 

Summarizing, we have reduced the statement to points $a,b\in S^\circ$ and a subpath $\beta$ of $\gamma$ that connects them. We will prove that using a variational argument, very similar to the one used in Lemma \ref{unif:Conjugate-Define when rho=0}. Since the technical details are similar and we only wish to demonstrate the new idea in this proof we assume that the points $a,b$ can be both approximated by peripheral disks $Q_i$ with $\rho(Q_i)>0$, which corresponds to Case 1 in the proof of Lemma \ref{unif:Conjugate-Define when rho=0}. (If this is not the case, and e.g.\ $a$ lies in a component $V$ of $\inter_{\br\Omega}(\alpha_{t_0})$ for some $t_0\notin \mathcal T$, then one can use a ``bridge" $\tau\subset V$,  to connect $a$ to a point on $\partial_{\br\Omega}V\cap S$ that can be approximated by $Q_i$ with $\rho(Q_i)>0$. Such a bridge $\tau$ was also employed in  Case 1 in the proof of Lemma \ref{unif:Conjugate-Define when rho=0}.)

Let $\ell$ denote the sum in the right hand side of \eqref{unif:Conjugate-Defect upper gradient-To prove-beta}. If the conclusion fails, then there exists $\varepsilon>0$ such that, say, $v(b)-v(a)\geq 10\varepsilon +\ell$. Using Lemma \ref{unif:Zeta lemma}, for a small $\eta>0$ consider a function $\zeta\in \mathcal W^{1,2}(S)$ that vanishes on $\partial \Omega$ with $0\leq \zeta\leq 1$, such that  $\zeta\equiv 1$ on small disjoint balls $B(a,r)\cup B(b,r) \subset \Omega$, and $D(\zeta)<\eta$. Then we can find peripheral disks $Q_{i_a}\subset B(a,r)$, $Q_{i_b}\subset B(b,r)$ with $\rho(Q_{i_a}),\rho(Q_{i_b})>0$ such that
\begin{align}\label{unif:Conjugate-Defect upper gradient- hat v inequality}
\hat v(Q_{i_b})-\hat v(Q_{i_a})>9\varepsilon+\ell.
\end{align}
Consider $s,t\in \mathcal T$ such that for the smallest open subpaths of $\gamma_s,\gamma_t$ of $\alpha_s,\alpha_t$ that connect $\Theta_2$ to $Q_{i_a},Q_{i_b}$, respectively, we have $\hat v(Q_{i_a})=\ell_\rho(\gamma_s)$ and $\hat v(Q_{i_b})=\ell_\rho(\gamma_t)$; see Lemma \ref{unif:Conjugate-Infimum not needed}. If $\tilde \gamma_t$ denotes the smallest open subpath of $\alpha_t$ that connects $\Theta_4$ to $Q_{i_b}$ we have 
\begin{align*}
\hat v(Q_{i_a})-\hat v(Q_{i_b})\geq \ell_{\rho}(\gamma_s)+\ell_\rho(\tilde \gamma_t)-D(u)
\end{align*} 
by Theorem \ref{unif:Level sets-sums equal mass}. Hence, in order to obtain a contradiction to \eqref{unif:Conjugate-Defect upper gradient- hat v inequality}, it suffices to prove
\begin{align}\label{unif:Conjugate-Defect upper gradient- Claim}
\ell_\rho(\gamma_s)+\ell_\rho(\tilde \gamma_t)\geq D(u)-\varepsilon -\ell.
\end{align}

We will construct an admissible function $g$ with the same procedure and notation as in Case 1 of Lemma \ref{unif:Conjugate-Define when rho=0}. Recall the definition of $U_s$ and $U_t$ in \eqref{unif:Conjugate-U_s}. Moreover, $\phi_s$ is a bump function supported in a neighborhood of a strip $\Omega_{s,h}\subset A_{s-h,s+h}$ that connects $\Theta_2$ to $Q_{i_a}$, and $\phi_t$ is a bump function supported in a neighborhood of a strip $\widetilde \Omega_{t,h}\subset A_{t-h,t+h}$ that connects $\Theta_4$ to $Q_{i_b}$. Also, for small $\delta>0$ consider the function $\psi(x)\coloneqq  \max\{ 1-\delta^{-1}\dist(x,\beta),0\}$. Now, we define 
\begin{align*}
g_0= (U_s\phi_s +U_t\phi_t) \lor \zeta\lor \psi
\end{align*} 
on $S$. As before, we can find a simple path on which we have $g_0\equiv 1$ such that it separates $\Theta_1$ from $\Theta_3$. If $W$ is the component that contains $\Theta_3$, we define $g=1$ on $S\cap W$ and $g=g_0$ on $S\setminus W$.

Recall the definitions of the index sets $F_{s,h},F_{t,h}$ and $N_h$. We have 
\begin{align*}
\osc_{Q_i}(g)&\leq \osc_{Q_i}(\zeta)+\osc_{Q_i}(\psi) +
\begin{cases}
1, & i\in F_{s,h}\cup F_{t,h}\\
\rho(Q_i)/2h, & i\in N_h\\
0, & i\notin F_{s,h}\cup F_{t,h}\cup N_h.\\
\end{cases}
\end{align*} 
Since $\psi$ is $(1/\delta)-$Lipschitz, we have 
\begin{align}
\label{unif:Conjugate-Defect upper gradient-osc psi}\osc_{Q_i}(\psi)\leq \min\{ \delta^{-1}\diam(Q_i \cap N_\delta(\beta)),1\}.
\end{align}
Testing the minimizing property of $u$ against $g$ (see also \eqref{unif:Level sets-Optimization in s}) we obtain
\begin{align*}
D(u)&\leq \sum_{i\in \N}\rho(Q_i)\osc_{Q_i}(g)\\
&\leq \sum_{i\in \N}\rho(Q_i) \osc_{Q_i}(\zeta)+ \sum_{i\in F_{s,h}\cup F_{t,h}}\rho(Q_i)+ \frac{1}{2h}\sum_{i\in N_h}\rho(Q_i)^2\\
&\quad \quad + \sum_{i: Q_i\cap N_\delta(\beta)\neq \emptyset} \rho(Q_i)\osc_{Q_i}(\psi)
\end{align*}
Letting $h\to 0$ and choosing a small $\eta$ so that $\sum_{i\in \N}\rho(Q_i) \osc_{Q_i}(\zeta)<\varepsilon$ we obtain
\begin{align*}
D(u)\leq \varepsilon+\ell_\rho(\gamma_s)+\ell_\rho(\tilde \gamma_t)+ \sum_{i: Q_i\cap N_\delta(\beta)\neq \emptyset} \rho(Q_i)\osc_{Q_i}(\psi).
\end{align*}
To prove our claim in \eqref{unif:Conjugate-Defect upper gradient- Claim}, it suffices to show that the limit of the latter term as $\delta\to 0$ stays below $\ell$. Note that we can split this term as
\begin{align*}
\sum_{i: \br Q_i\cap \beta\neq \emptyset} \rho(Q_i)\osc_{Q_i}(\psi)+ \sum_{\substack{i: \br Q_i\cap \beta= \emptyset \\ Q_i\cap N_\delta(\beta)\neq \emptyset}} \rho(Q_i)\osc_{Q_i}(\psi).
\end{align*}
Using \eqref{unif:Conjugate-Defect upper gradient-osc psi} we see that the first term is already bounded by $\ell$, so it suffices to show that the second term converges to $0$ as $\delta\to 0$. Define $D_\delta(\beta)$ to be the family of indices $i\in \N$ such that $\diam(Q_i)\geq \delta$, $Q_i\cap N_\delta(\beta)\neq \emptyset$, and $\br {Q}_i\cap \beta =\emptyset$. Observe that if $Q_i\cap N_\delta(\beta)\neq \emptyset$ and $\diam(Q_i)<\delta$, then $Q_i\subset N_{2\delta}(\beta)$. Hence, using \eqref{unif:Conjugate-Defect upper gradient-osc psi} it suffices to show that 
\begin{align*}
\frac{1}{\delta}\sum_{i:Q_i\subset N_{2\delta}(\beta)} \rho(Q_i)\diam(Q_i) \to 0 \quad \textrm{and}\quad \sum_{i\in D_\delta(\beta)} \rho(Q_i) \to 0
\end{align*}
as $\delta\to 0$. This will follow from the next general lemma. 

\begin{lemma}\label{unif:Lebesgue differentiation Lemma}
Let $\{\lambda(Q_i)\}_{i\in \N}$ be a sequence in $\ell^2(\N)$. Then there exists an exceptional path family $\Gamma$ in $\Omega$ with $\md_2(\Gamma)=0$ such that for all non-constant paths $\gamma\subset \Omega$ with endpoints in $S^\circ$ and $\gamma\notin \Gamma$ we have
\begin{align*}
\frac{1}{\delta}\sum_{i:Q_i\subset N_{\delta}(\gamma)} \lambda (Q_i)\diam(Q_i) \to 0 \quad \textrm{and} \quad \sum_{i\in D_\delta(\gamma)} \lambda (Q_i) \to 0
\end{align*}
as $\delta\to 0$.
\end{lemma} 

Since the family of paths $\gamma$ that have a subpath $\beta$ for which the conclusion of the lemma fails also has conformal modulus zero, this completes the proof of the upper gradient inequality.
\end{proof}

\begin{proof}[Proof of Lemma \ref{unif:Lebesgue differentiation Lemma}]
By the subadditivity of modulus, we can treat each of the claims separately. Fix a curve $\gamma$ with $\mathcal H^1(\gamma\cap S)=0$. The latter holds for $\md_2$-a.e.\ $\gamma\subset \Omega$. 

We can cover $N_\delta(\gamma)$ by balls $B_{j,\delta}$ of radius $2\delta$ centered at $\gamma$ such that $\frac{1}{20} B_{j,\delta}$ are disjoint. To do this, one can cover $\gamma$ by balls of radius $\delta/10$ and extract a disjoint subcollection $\{B_l\}_l$ such that the balls $5B_l$ still cover $\gamma$; see for instance \cite[Theorem 1.2]{Heinonen:metric}. We now define $\{B_{j,\delta}\}_j$ to be the collection of balls $\{20B_l\}_l$, each of which has radius $2\delta$. Then any point $x\in N_\delta(\gamma)$ is $\delta$-far from $\gamma$ and thus $\delta+\delta/2$-far from the center of a ball $5B_l$ (of radius $\delta/2$), which is equal to $\frac{1}{4}B_{j,\delta}$ for some $j$. It follows that $x\in B_{j,\delta}$.

Next, consider the subfamily $\{B_{j,\delta}'\}_j$ of the balls $\{B_{j,\delta}\}_j$ that are not entirely contained in any peripheral disk. We note that $\{B_{j,\delta}'\}_j$ covers $\bigcup_{i:Q_i\subset N_\delta(\gamma)} Q_i$ and $\gamma\cap S$. In fact, as $\delta\to 0$ along a sequence one can construct covers $\{B_{j,\delta}'\}_j$ as above such that $\bigcup_{j}B_{j,\delta}'$ is decreasing. Furthermore, as $\delta\to 0$ we have that 
\begin{align}\label{unif:Lebesgue differentiation lemma-B_j cover}
\gamma\cap (\bigcup_{j}B_{j,\delta}')   \to  \gamma \cap S.
\end{align} 
Indeed, if this failed, then there would exist $i_0\in \N$ and some $x\in Q_{i_0}\cap \gamma$ which belongs to $\bigcup_{j}B_{j,\delta}'$ infinitely often as $\delta\to 0$. This contradicts the construction of the cover $\{B_{j,\delta}'\}$, since $\dist(x,\partial Q_{i_0})>0$, and a ball $B_{j,\delta}'$ that contains $x$ would be entirely contained in $Q_{i_0}$ for small $\delta$, so it would have been discarded during the construction.

Now consider the function $\Lambda(x)=\sum_{i\in \N} \frac{\lambda(Q_i)}{\diam(Q_i)} \x_{Q_i}(x)$, and observe that by the quasiballs assumption \eqref{unif:Quasi-balls} there exists a constant $C>0$ such that
\begin{align}\label{unif:Lebesgue differentiation lemma-Lambda}
\frac{1}{C}\int_{Q_i} \Lambda(x) \, d\mathcal H^2(x)\leq  \lambda(Q_i)\diam(Q_i)\leq C\int_{Q_i} \Lambda(x) \, d\mathcal H^2(x) \
\end{align}
for all $i\in \N$. Using the properties of the cover $\{B_{j,\delta}'\}_j$ and the uncentered maximal function $M\Lambda$ we have:
\begin{align*}
\sum_{i:Q_i\subset N_\delta(\gamma)} \lambda(Q_i)\diam(Q_i)&\leq C \int_{{\bigcup_{i:Q_i\subset N_\delta(\gamma)}}Q_i} \Lambda(x) \, d\mathcal H^2(x)\leq C\int_{\bigcup_j B_{j,\delta}'} \Lambda(x) \, d\mathcal H^2(x)\\
&\leq C\sum_{j} \int_{B_{j,\delta}'}\Lambda(x) \, d\mathcal H^2(x) \leq C' \delta \sum_j \delta\inf_{x\in \frac{1}{20}B_{j,\delta}'} M\Lambda(x)\\
&\leq C'' \delta \sum_j \int_{\gamma \cap (\frac{1}{20}B_{j,\delta}')} M\Lambda(x) \, d\mathcal H^1(x)\\
&=C'' \delta \int_{\gamma \cap (\bigcup_{j}\frac{1}{20}B_{j,\delta}')}    M\Lambda(x) \, d\mathcal H^1(x)\\
&\leq  C'' \delta \int_{\gamma \cap (\bigcup_j B_{j,\delta}')}  M\Lambda(x) \, d\mathcal H^1(x).
\end{align*}
The first part of the lemma will follow, if we show that 
\begin{align}\label{unif:Lebesgue differentiation lemma-DCT}
\int_{\gamma \cap (\bigcup_j B_{j,\delta}')}  M\Lambda(x) \, d\mathcal H^1(x) \to 0
\end{align}
as $\delta\to 0$ for $\md_2$-a.e.\ curve $\gamma$. First observe that $M\Lambda \in L^2(\Omega)$ since $\Lambda \in L^2(\Omega)$ by \eqref{unif:Lebesgue differentiation lemma-Lambda}. Hence, $\int_\gamma M\Lambda(x) \, d\mathcal H^1(x) <\infty$ for $\md_2$-a.e.\ $\gamma\subset \Omega$. By construction, $\gamma \cap (\bigcup_j B_{j,\delta}') $  decreases to $\gamma\cap S$ with $\mathcal H^1(\gamma\cap S)=0$. The dominated convergence theorem now immediately implies \eqref{unif:Lebesgue differentiation lemma-DCT}.

Now, we show the second part of the lemma. Recall that $D_\delta(\gamma)$ contains all indices $i\in \N$ for which $\diam(Q_i)\geq \delta$, $Q_i\cap N_\delta(\gamma)\neq \emptyset$, and $\br Q_i\cap \gamma=\emptyset$. We wish to show that the family of paths for which the conclusion fails has conformal modulus equal to zero. We first remark that this family contains no constant paths, by assumption. By the subadditivity of conformal modulus it suffices to show that for every $d>0$, $\varepsilon_0>0$ the family $\Gamma$ of paths $\gamma$, having endpoints in $S^\circ$, with $\diam(\gamma)\geq d$ and 
\begin{align*}
\limsup_{\delta\to 0}\sum_{i\in D_\delta(\gamma)} \lambda(Q_i)\geq \varepsilon_0,
\end{align*} 
has  conformal modulus zero. Let $\{\lambda_0(Q_i)\}_{i\in \N}$ be a finitely supported sequence with $\lambda_0(Q_i)=\lambda(Q_i)$ or $\lambda_0(Q_i)=0$, for each $i\in \N$. Then 
\begin{align*}
\sum_{i\in D_\delta(\gamma)} \lambda_0(Q_i)\to 0
\end{align*}
as $\delta\to 0$, since $\br Q_i \cap \gamma=\emptyset$ and thus $\dist(\br Q_i,\gamma)>0$ for all $i\in D_\delta(\gamma)$. Here, it is crucial that $\gamma$ has endpoints in $S^\circ$, and the preceding statement would fail if $\gamma$ was an open path and one of its endpoints was on a peripheral circle. Consequently, if $\gamma\in \Gamma$, then 
\begin{align}\label{unif:Lebesgu differentiation lemma-part2}
\limsup_{\delta\to 0}\sum_{i\in D_\delta(\gamma)} h(Q_i)\geq \varepsilon_0,
\end{align}
where $h(Q_i)\coloneqq \lambda(Q_i)-\lambda_0(Q_i)$. 

We will construct an admissible function $\tilde h$ for $\md_2(\Gamma)$ with arbitrarily small mass. By the summability assumption on $\lambda$, for each $\eta>0$ we can find $\lambda_0$ as above such that $\sum_{i\in \N}h(Q_i)^2 <\eta$.  For each $i\in \N$ consider balls $B(x_i,r_i)\subset Q_i\subset B(x_i,R_i)$ as in the quasiballs assumption \eqref{unif:Quasi-balls} with $R_i=\diam(Q_i)$. We define 
$$\tilde h= c_0\sum_{i\in \N} \frac{h(Q_i)}{R_i}\x_{B(x_i,4R_i)} $$
where $c_0$ is a constant to be determined, independent of $\eta$. Note that by Lemma \ref{unif:Bojarski} and the fact that the balls $B(x_i,R_i/K_0)\subset B(x_i,r_i)$ are disjoint we have 
\begin{align*}
\int \tilde h(x)^2 \, d\mathcal H^2(x) \leq C c_0^2\sum_{i\in \N} \frac{h(Q_i)^2}{R_i^2}  R_i^2  \leq Cc_0^2\eta.
\end{align*} 
Since this can be made arbitrarily small, it remains to show that $\tilde h$ is admissible for $\Gamma$.

Fix a curve $\gamma \in \Gamma$, so \eqref{unif:Lebesgu differentiation lemma-part2} holds. Observe that $\max_{i\in D_\delta(\gamma)} \diam(Q_i) \to 0$ as $\delta\to 0$, by the definition of $D_\delta(\gamma)$ and the fact that there are only finitely many peripheral disks with ``large" diameter. Now, let $\delta$ be sufficiently small, so that $8\max_{i\in D_\delta(\gamma)} \diam(Q_i)<d \leq \diam(\gamma)$ and $\sum_{i\in D_\delta(\gamma)}h(Q_i) >\varepsilon_0/2$. If  $i\in D_\delta(\gamma)$ then $R_i=\diam(Q_i)\geq \delta$ and $Q_i\cap N_\delta(\gamma)\neq \emptyset$, thus $B(x_i,2R_i)$ meets $\gamma$. By the choice of $\delta$, $\gamma$ has to exit $B(x_i,4R_i)$. Hence, $\mathcal H^1(\gamma\cap B(x_i,4R_i))\geq 2R_i$, which implies that
\begin{align*}
\int_\gamma \tilde h(x) \, d\mathcal H^1(x) \geq c_0\sum_{i\in D_\delta(\gamma)} \frac{h(Q_i)}{R_i} \mathcal H^1(\gamma\cap B(x_i,4R_i)) \geq c_0\varepsilon_0.
\end{align*}
We choose $c_0=1/\varepsilon_0$ and this completes the proof.
\end{proof}

\begin{remark}\label{unif:Lebesgue differentiation remark}
It is clear from the proof that this general lemma holds for carpets $S$ of  area zero for which the peripheral disks are uniform quasiballs; the fatness assumption was not used here.
\end{remark}

\begin{remark}
If the function $v$ satisfied the same upper gradient inequality as $u$ (see Definition \ref{unif:Definition Upper gradient}) then $x\mapsto v(x)/D(u)$ would be admissible for the free boundary problem with respect to $\Theta_2$ and $\Theta_4$. One then could show that $v/D(u)$ is carpet-harmonic and thus $v$ is carpet-harmonic. We believe that the form of the upper gradient inequality of $v$ depends on the geometry of the peripheral disks and their separation, and without any extra assumptions the harmonicity of $v$ is far from being clear.
\end{remark}

\section{Definition of \texorpdfstring{$f$}{f}}\label{unif:Section Definition of F}\index{uniformizing map!definition}
Let  $D\coloneqq D(u)=\sum_{i\in \N}\rho(Q_i)^2$, and consider the continuous function
$$f\coloneqq (u,v)\colon  S\to [0,1] \times [0,D].$$
The fact that the range of $f$ is $[0,1]\times [0,D]$ is justified by Lemma \ref{unif:Conjugate-Values on boundaries}. The same lemma also implies that $f( \partial \Omega)= \partial ([0,1]\times [0,D])= \partial S_0$, where $S_0\coloneqq \C \setminus [0,1]\times[0,D]$. If $\rho(Q_i)=0$, then $u$ and $v$ are constant on $\partial Q_i$, so $f(\partial Q_i)$ is a singe point, denoted by $\br S_i$ or $\partial S_i$. If $\rho(Q_i)>0$, again by Lemma \ref{unif:Conjugate-Values on boundaries} we have
\begin{align}\label{unif:Definition-F(Q_i) contained}
f(\partial Q_i) \subset [m_{Q_i}(u), M_{Q_i}(u)] \times [ m_{Q_i}(v), M_{Q_i}(v)]\eqqcolon \br S_i,
\end{align}
where $M_{Q_i}(u)-m_{Q_i}(u)= M_{Q_i}(v)-m_{Q_i}(v)=\rho(Q_i)$. Thus, the image of $\partial Q_i$ is contained in a square of sidelength $\rho(Q_i)$. We define $S_i$ to be the open square $(m_{Q_i}(u), M_{Q_i}(u)) \times ( m_{Q_i}(v), M_{Q_i}(v))= \inter(\br S_i)$, or the empty set in case $\rho(Q_i)=0$. If $\rho(Q_i)>0$, then we will call $\br S_i$ a \textit{non-degenerate} square. We claim that these squares have disjoint interiors:

\begin{lemma}\label{unif:Definition-Squares disjoint interior}
The (open) squares $S_i$, $i\in \N$, are disjoint. Furthermore, for each $i\in \N$ we have $S_i\cap f(S)=\emptyset$, and $f(\partial Q_i)\subset \partial S_i$.
\end{lemma}
\begin{proof}
Assume that $\rho(Q_i),\rho(Q_j)>0$ for some $i,j\in \N$, $i\neq j$, and that $S_i\cap S_j\neq \emptyset$. Since the $x$-coordinates of the squares intersect at an interval of positive length, there exists some $t\in \mathcal T$ such that $\br Q_i\cap \alpha_t\neq \emptyset$ and $\br Q_j\cap \alpha_t\neq \emptyset$. Assume that the path $\alpha_t$ meets ordered points $x_1,x_2\in \partial Q_i$ and then $y_1,y_2\in \partial Q_j$ as it travels from $\Theta_2$ to $\Theta_4$. This can be justified using the properties of $\alpha_t$ from Lemma \ref{unif:Level sets-curves}. By Lemma \ref{unif:Conjugate-Values on boundaries} and Corollary \ref{unif:Conjugate-Oscillation-Increasing} we have
\begin{align*}
v(x_1)+\rho(Q_i)=v(x_2)< v(y_1) = v(y_2)-\rho(Q_j).
\end{align*}
This clearly contradicts the assumption that $S_i\cap S_j\neq \emptyset$.

For our second claim, assume that there exists some $x\in S$ with $f(x)\in S_{i_0}$ for some $i_0\in \N$. If $x\in \alpha_t$ for some $t\in \mathcal T$, then $x$ can be approximated by peripheral disks $Q_i$ with $\rho(Q_i)>0$; see Lemma \ref{unif:lemma:paths zero hausdorff} and recall that all peripheral disks intersecting $\alpha_t$ satisfy $\rho(Q_i)>0$. By the continuity of $f$ the diameter $ \diam(S_i)$ is arbitrarily small as $Q_i\to x$, so $S_i\subset S_{i_0}$. This contradicts the first part of the lemma.

If $x\in \alpha_t$ for some $t\notin \mathcal T$ then there exists a point $y \in \alpha_t$ such that $v(x)=v(y)$ and $y$ can be approximated by $Q_i$ with $\rho(Q_i)>0$; see Lemma \ref{unif:Conjugate-Define when rho=0} and the discussion that precedes it. Thus $f(x)=f(y)=(t,v(y))$, and the previous case applies to yield a contradiction.

The final claim follows from \eqref{unif:Definition-F(Q_i) contained} and the previous parts of the lemma.
\end{proof}

In fact, we have the following:

\begin{corollary}\label{unif:Definition-Range-Extension}
For each $i\in \N\cup \{0\}$ we have $f(\partial Q_i)= \partial S_i$. Moreover, 
\begin{align*}
f(S)= [0,1]\times [0,D]\setminus \bigcup_{i\in \N} S_i \eqqcolon \mathcal R, \quad \mathcal H^2(\mathcal R)=0,
\end{align*}
and the intersection of a non-degenerate square $\br S_i$, $i\in \N$, with $\partial S_0= \partial ([0,1]\times [0,D])$ or with another non-degenerate square $\br S_j$, $j\in \N$, $j\neq i$,  is either the empty set or a singleton.  
\end{corollary}

\begin{proof}
By the preceding lemma we know that $f(\partial Q_i)\subset \partial S_i$ for each $i\in \N$. Consider a continuous extension $\widetilde f\colon  \br\Omega\to [0,1]\times [0,D]$ such that $\widetilde f (Q_i) \subset S_i$, whenever $S_i$ is a non-degenerate square. One way to find such an extension is to consider a Poisson extension $\tilde v $ of $v$ in each peripheral disk as in Lemma \ref{unif:Maximum principle for tilde u}. Then $\tilde v(Q_i) \subset (m_{Q_i}(v),M_{Q_i}(v))$, and also $\tilde u(Q_i) \subset (m_{Q_i}(u),M_{Q_i}(u))$, by the maximum principle. Hence, if we define $\widetilde f\coloneqq (\tilde u,\tilde v)$ we have the desired property $\widetilde f(Q_i)\subset S_i$ whenever $S_i$ is non-degenerate. Combining this with Lemma \ref{unif:Definition-Squares disjoint interior}, we see that for each non-degenerate $S_i$ we have 
\begin{align}\label{unif:Definition:homotopy}
\widetilde f(\br \Omega \setminus Q_i) \subset [0,1]\times [0,D] \setminus S_i.
\end{align}

First we show that $f(\partial \Omega)=f(\partial Q_0)=\partial S_0$. Recall that $u\equiv 0$ on $\Theta_1$ and $u\equiv 1$ on $\Theta_3$, so these sets are mapped into the left and right vertical sides of the rectangle $\partial S_0$, respectively; see Theorem \ref{unif:Solution to free boundary problem}. Also, by Lemma \ref{unif:Conjugate-Values on boundaries}, $v\equiv 0$ on $\Theta_2$ and $v\equiv D$ on $\Theta_4$, so these sets are mapped to the bottom and top sides of $\partial S_0$, respectively. By continuity, we must have $f(\partial Q_0)=\partial S_0$. 

Proposition \ref{unif:Level sets} shows that the functions $f\big|_{\Theta_2},f\big|_{\Theta_4}$ are ``increasing" from $0$ to $1$ if $\Theta_2,\Theta_4$ are parametrized as arcs from $\Theta_1$ to $\Theta_3$. Thus, $\widetilde f \big|_{\partial \Omega}$ winds once around every point of $(0,1)\times(0,D)$. By homotopy, using \eqref{unif:Definition:homotopy}, we see that $\widetilde f \big|_{\partial Q_i}$ winds once around every point of $\inter (S_i)$. Thus $\widetilde f(\partial Q_i)=\partial S_i$, and $\widetilde f(Q_i)=S_i$; see \cite[Chapter II]{RadoReichelderfer:topology}.

Note that the area of $[0,1]\times [0,D]$ is equal to $D=\sum_{i\in \N}\rho(Q_i)^2$, which is the sum of the areas of the squares $S_i$. By Lemma \ref{unif:Definition-Squares disjoint interior}, the squares have disjoint interiors, so in some sense they ``tile" $[0,1]\times [0,D]$, and this already shows that $\mathcal H^2(\mathcal R)=0$. Furthermore, we obtain that the boundaries $\partial S_i=f(\partial Q_i)$ are dense in $[0,1]\times [0,D]\setminus \bigcup_{i\in \N} S_i$. Since the sets $\partial Q_i$, $i\in \N$, are also dense in the carpet $S$ and $f$ is continuous, we obtain $f(S)=\mathcal R$. 

For the last claim, assume that two squares $\br S_i,\br S_j$, $i\neq j$, share part of a vertical side, i.e., there exists a non-degenerate vertical line segment $\tau\coloneqq  \{t\} \times [s_1,s_2]\subset \br S_i \cap \br S_j$. Let $s\in (s_1,s_2)$ and consider, by surjectivity, a point $x\in \partial Q_i$ such that $f(x)=(t,s)\in \tau$. We claim that $x$ cannot be approximated by $Q_k$, $k\neq i$, with $\rho(Q_k)>0$. Indeed, if this was the case, then $f(\partial Q_k) =\partial S_k$ would be non-degenerate distinct squares that approximate the point $f(x)$ by continuity. Obviously, this cannot happen since $f(x)\in \tau$, and $S_i,S_k$ have to be disjoint by Lemma \ref{unif:Definition-Squares disjoint interior}. Hence, $\rho(Q_k)=0$ for all $Q_k$ contained in a neighborhood of $x$. The upper gradient inequality of $u$ along with continuity imply that $u$ is constant in a neighborhood of $x$; see Lemma \ref{harmonic:4-zero oscillation lemma} for a proof. By Lemma \ref{unif:Conjugate-Define when rho=0} and the definition $v$ we conclude that $v$ is also constant in a neighborhood of $x$; see also the comments before Lemma \ref{unif:Conjugate-Define when rho=0}. In particular, $f$ is constant in some arc of $\partial Q_i$ containing $x$. However, there are countably many such subarcs of $\partial Q_i$, and uncountably many preimages $f^{-1}((t,s))$, $s\in (s_1,s_2)$. This is a contradiction. The same argument applies if $\br S_i,\br S_j$ share a horizontal segment or if $\br S_i \cap \partial S_0 \neq \emptyset$.
\end{proof}

Combining the upper gradient inequalities of $u$ and $v$ we obtain:

\begin{prop}\label{unif:Upper gradient for F}
There exists a family of paths $\Gamma_0$ with $\md_2(\Gamma_0)=0$ such that for every path $\gamma \subset  \Omega$ with $\gamma\notin \Gamma_0$ and for every open subpath $\beta$ of $\gamma$ we have
\begin{align*}
|f(x)-f(y)|\leq \sqrt{2}\sum_{i: \br Q_i\cap \beta\neq \emptyset} \rho(Q_i)=\sum_{i: \br Q_i\cap \beta \neq \emptyset} \diam(\br S_i) 
\end{align*}
for all $x,y\in \br \beta\cap S$.
\end{prop} 

Observe that the paths $\beta$ are contained in $\Omega$, so $\br Q_0 \cap \beta=\emptyset$, and we never include the term $\diam(\br S_0)=\diam(\C \setminus (0,1)\times(0,D))=\infty$ in the above summations.

\begin{proof}
This type of upper gradient inequality holds for the function $v$ by Lemma \ref{unif:Conjugate-Defect upper gradient}, so we only have to argue for $u$. Recall that $u\in \mathcal W^{1,2}(S)$ so it satisfies the upper gradient inequality in Definition \ref{unif:Definition Upper gradient} with an exceptional family $\Gamma_0$ that has carpet modulus equal to $0$. Lemma \ref{unif:Zero modulus lemma} now implies that $\md_2(\Gamma_0)=0$. To complete the proof, note that the sum over $\{i:\br Q_i \cap \beta\neq \emptyset\}$ is larger than the sum over $\{i:Q_i\cap \beta\neq \emptyset\}$ (which was used in Definition \ref{unif:Definition Upper gradient}).
\end{proof}

The upper gradient inequality has the next important corollary. 

\begin{corollary}\label{unif:Upper gradient-projection}
There exists a family of paths $\Gamma_0$ with $\md_2(\Gamma_0)=0$ such that the following holds: 

\begin{flushleft}
For every open path $\gamma \subset  \Omega$, $\gamma \notin \Gamma_0$ with endpoints $x,y\in S$, and every Lipschitz map $\pi\colon   \R^2 \to \R$ we have 
$$|\pi(f(x))-\pi(f(y))| \leq \mathcal H^1\biggl( \bigcup_{i: \br Q_i\cap \gamma\neq \emptyset}\pi( S_i ) \biggr) \leq \sum_{i: \br Q_i\cap \gamma \neq \emptyset} \mathcal H^1( \pi(S_i)).$$ 
\end{flushleft}
\end{corollary}

\begin{remark}
The corollary does \textit{not} imply that for almost every $\gamma$ we have $\mathcal H^1(f(\gamma\cap S))=0$; cf.\ Lemma \ref{unif:Regularity-Path hausdorff measure zero}. To interpret its meaning, let $\gamma$ be a non-exceptional path joining $x$ and $y$, and $\pi$ be the orthogonal projection to the line passing through $f(x),f(y)$. Then the corollary says that the projection of the squares intersected by $f(\gamma)$ ``covers" the entire distance from $f(x)$ to $f(y)$.
\end{remark}

\begin{remark}\label{unif:Upper gradient-projection remark}
If $\{\lambda(Q_i)\}_{i\in \N}$ is a sequence in $\ell^2(\N)$, it can be proved that the family of paths $\Gamma$ in $\C$ satisfying 
\begin{align*}
\sum_{i:\br Q_i \cap \gamma \neq \emptyset}\lambda(Q_i)=\infty
\end{align*} 
has $2$-modulus zero. If we were using $Q_i$ instead of $\br Q_i$ under the summation, then this would follow immediately from Lemma \ref{unif:Zero modulus lemma}. The proof in our case is in fact exactly the same as the proof of Lemma \ref{unif:Zero modulus lemma}; see the proof of Lemma \ref{harmonic:2-weak modulus zero implies two modulus zero}.
\end{remark}

\begin{proof}
Let $\gamma \subset \Omega$ be a path such that the upper gradient inequality of Proposition \ref{unif:Upper gradient for F} holds for all open subpaths $\beta$ of $\gamma$ and such that $\sum_{i:\br Q_i\cap \gamma\neq \emptyset} \rho(Q_i) <\infty$. The latter holds for all curves outside a family conformal modulus zero, by the preceding remark.

If $x\in \partial Q_{i_x}$ for some $i_x\in \N$, then we let $x'\in \partial Q_{i_x}$ be the point of last exit of $\gamma$ from $Q_{i_x}$, and we similarly consider the point $y'\in \partial Q_{i_y}$ of first entry of $\gamma$ into $Q_{i_y}$ (after $x'$), in case $y\in \partial Q_{i_y}$. Since the differences $|\pi(f(x))-\pi(f(x'))|$ and $|\pi(f(y))-\pi(f(y'))|$ are controlled by $\mathcal H^1( \pi(S_{i_x}))$ and $\mathcal H^1(\pi(S_{i_y}))$, respectively, it suffices to prove the statement with $x$ and $y$ replaced by $x'$ and $y'$, respectively, and with $\gamma$ replaced by its open subpath connecting $x'$ and $y'$.

Hence, from now on we assume that $\gamma$ is an open path with endpoints $x,y$ and we suppose that $\gamma$ does not intersect the (closed) peripheral disks that possibly contain $x$ and $y$ in their boundary. Fix $\varepsilon>0$ and consider a finite index set $J\subset \N$ such that 
\begin{align}\label{unif:Upper gradient-projection-J bound}
\sum_{\substack{i:\br Q_i\cap \gamma\neq \emptyset \\i\notin J }} \rho(Q_i)<\varepsilon.
\end{align}
Assume $\gamma$ is parametrized as it runs from $x$ to $y$. Using the reformulation of Lemma \ref{harmonic:3-Split path} in Remark \ref{harmonic:remark:split path}, we may obtain finitely many open subpaths $\gamma_1,\dots,\gamma_m$ of $\gamma$ with the following properties:
\begin{enumerate}[(1)]
\item the upper gradient inequality of $f$ holds along each path $\gamma_k$, $k=1,\dots,m$,
\item the paths $\gamma_k$ intersect disjoint sets of peripheral disks $\br Q_i$, $i\in \N \setminus J$, and
\item the path $\gamma_1$ starts at $x_1=x$, the path $\gamma_m$ terminates at $y_m=y$, and in general the path $\gamma_k$ has endpoints $x_k,y_k\in S$ such that for $1\leq k\leq m-1$ we either have
\begin{itemize}
\item $y_k=x_{k+1}$, or
\item $y_k,x_{k+1}\in \partial Q_{j_k}$ for some $j_k\in \N$. The peripheral disks $\br Q_{j_k}$ are distinct and they are intersected by $\gamma$. 
\end{itemize}
We denote by $I\subset \{1,\dots,m\}$ the set of indices $k$ for which the second alternative holds. 
\end{enumerate}
The assumption that $\gamma$ does not intersect the (closed) peripheral disks possibly containing $x$ and $y$ in their boundary is essential, otherwise property (2) could fail.

For each $S_{j_k}$, $k\in I$, the image $\pi(S_{j_k})\subset \R$ is an interval. Thus, $\bigcup_{k\in I} \pi(S_{j_k})$ is a union of finitely many intervals that contain the points $\pi(f(y_{k}))$ and $ \pi(f(x_{k+1}))$, $k\in I$, in their closure. Without loss of generality, assume that $\pi (f(x))\leq \pi(f(y))$. The interval $ [\pi(f(x)), \pi(f(y))]$ is covered by the union of the set $\bigcup_{k\in I} \pi(S_{j_k})$ together with the closed intervals between $\pi (f(x_k))$ and $\pi(f(y_k))$, $k=1,\dots,m$. If $L$ is the Lipschitz constant of $\pi$, then we have
\begin{align*}
|\pi(f(x)) -\pi(f(y))|& \leq \mathcal H^1\biggl(\bigcup_{k\in I} \pi(S_{j_k})\biggr) + \sum_{k=1}^m |\pi(f(x_k)) -\pi(f(y_k))|\\
&\leq \mathcal H^1\biggl( \bigcup_{i: \br Q_i\cap \gamma\neq \emptyset}\pi( S_i )\biggr) + L \sum_{k=1}^m |f(x_k)-f(y_k)|\\
&\leq \mathcal H^1\biggl( \bigcup_{i: \br Q_i\cap \gamma\neq \emptyset}\pi( S_i )\biggr) + L \sum_{k=1}^m \sum_{i:\br Q_i\cap \gamma_k\neq \emptyset} \sqrt{2}\rho(Q_i)
\end{align*}
Since the paths $\gamma_k$ intersect disjoint sets of peripheral disks $\br Q_i$, $i\in \N\setminus J$, the latter term is bounded by 
\begin{align*}
\sqrt{2} L \sum_{\substack{i:\br Q_i\cap \gamma\neq \emptyset \\i\notin J }} \rho(Q_i)<\sqrt{2} L\varepsilon,
\end{align*}
where we used \eqref{unif:Upper gradient-projection-J bound}. Letting $\varepsilon\to 0$ finishes the proof.
\end{proof}

\section{Injectivity of \texorpdfstring{$f$}{f}}\label{unif:Section Injectivity}\index{uniformizing map!injectivity}

We will prove that $f$ is injective in two steps. First we show that $f\colon S \to \mathcal R$ is a \textit{light map}\index{light map}, i.e., the preimage of every point $z\in \mathcal R$ contains no non-trivial continua. Then we show that the preimage of every point $z\in \mathcal R$ is actually a single point. 

From now on, $f$ will denote the continuous extension $\widetilde f\colon \br \Omega\to [0,1]\times [0,D]$, as in the proof of  Corollary \ref{unif:Definition-Range-Extension}, which has the property that $\widetilde f(Q_i)=S_i$ whenever $S_i$, $i\in \N$, is a non-degenerate square. Also, recall that the coordinates $\tilde u$ and $\tilde v$ of $f$ are harmonic in the classical sense inside each $Q_i$, $i\in \N$, and that $\alpha_t=\tilde u^{-1}(t)$, $t\in [0,1]$.

\begin{lemma}\label{unif:Injectivity-light}
For every $z\in \mathcal R$ the set $f^{-1}(z)$ contains no non-trivial continua.
\end{lemma}

The proof will follow from a modulus-type argument. Essentially, the carpet modulus of a family of curves passing through the point $z$ is zero; however, this would not be the case for the curves passing through $f^{-1}(z)$ if the latter contains a continuum. We will use the upper gradient inequality for $f$ in Proposition \ref{unif:Upper gradient for F} and its corollary to compare the modulus in the image and the preimage. In some sense, the map $f$ preserves carpet modulus, and this prevents a curve family of positive modulus from being mapped to a curve family of  modulus zero.

\begin{proof}
Assume that there exists a non-trivial continuum $E \subset f^{-1}(z)$ for some $z\in \mathcal R$.  Note that the preimage $ f^{-1}(z)$ cannot intersect both $\Theta_1$ and $\Theta_3$, so let $F$ be one of the two sets with $E\cap F=\emptyset$. We consider a small $R>0$ such that $ f^{-1}(B(z,R)) \cap F=\emptyset$. Such an $R$ exists because of the following general fact: If $g:X\to Y$ is a continuous map between metric spaces $X,Y$ and $X$ is compact, then for each $z\in Y$ and $\varepsilon>0$ there exists $R>0$ such that $g^{-1}(B(z,R)) \subset N_\varepsilon (g^{-1}(z))$.

By the structure of $\mathcal R$ (see Corollary \ref{unif:Definition-Range-Extension}), the point $z$, among other possibilities, might lie on a square $\partial S_{i_1}$, $i_1\in \N$, or it could be the common vertex of two intersecting squares $\partial S_{i_1},\partial S_{i_2}$, $i_1,i_2\in \N$. We construct closed annuli $A_j$ centered at $z$ in the following way. Let $A_1=\br A(z;r_1,R_1)=\br B(z,R_1)\setminus B(z,r_1)$, where $R_1=R$ and $r_1=R_1/2$. Let $R_2<r_1$ be so small that no square $S_i$ (except possibly for $S_{i_1},S_{i_2}$) intersects both $A_1$ and $A_2\coloneqq \br A(z;r_2,R_2)$, where $r_2=R_2/2$. This can be achieved because the squares near $z$ are arbitrarily small, with the possible exception of $S_{i_1}$ and $S_{i_2}$. We proceed inductively to obtain annuli $A_1,\dots,A_N$, for some fixed large $N$; cf.\ proof of Lemma \ref{unif:Zeta lemma}. We then replace each $A_j$ with $A_j\cap ([0,1]\times [0,D])$.

For $Q_i\cap f^{-1}(A_j)\neq \emptyset$ (equiv.\ $S_i\cap A_j\neq \emptyset$) we set $\lambda(Q_i)=\frac{1}{Nr_j} d_j(S_i)$, where $d_j(S_i)\coloneqq  \mathcal H^1(\{r \in[r_j,R_j]: S_i \cap B(z,r)\neq \emptyset \} )$. In other words, $d_j(S_i)$ is the the ``radial" diameter of the intersection $S_i\cap A_j$. If $S_i\cap A_j =\emptyset$ we set $\lambda(Q_i)=0$. Note that there exists a constant $C>0$ such that
\begin{align}\label{unif:Injectivity-light-d_j}
d_j(S_i)^2 \leq C \mathcal H^2 (S_i \cap A_j)
\end{align}
for all $i\in \N$ and $j=1,\dots,N$. This is true because, for example, the squares are uniformly fat sets; see also Remark \ref{unif:Fatness consequence}.

We now wish to construct a Lipschitz family of non-exceptional open paths in $\Omega$ that connects $E$ to $F$, but avoids the peripheral disk $\br Q_{i_1}$ in case $z\in \partial S_{i_1}$, or avoids the peripheral disks $\br Q_{i_1}$ and $\br Q_{i_2}$ in case $z\in \partial S_{i_1}\cap \partial S_{i_2}$. Here non-exceptional means that the paths, as well as their subpaths, avoid a given path family of $2$-modulus zero. In particular, require that for all open subpaths of these non-exceptional paths  the conclusion of Corollary \ref{unif:Upper gradient-projection} holds.

If none of the aforementioned two scenarios occur (i.e., $z\in \partial S_{i_1}$ or $z\notin \partial S_{i_1}\cap \partial S_{i_2}$), then we consider an open path $\tau$ joining $E$ to $F$. We can make now direct use of Lemma \ref{unif:Paths joining continua} to obtain a small $\delta>0$ and non-exceptional paths $\beta_s$, $s\in (0,\delta)$, that connect $E$ to $F$. If $z \in \partial S_{i_1}$ (i.e., we are in the first of the two scenarios), then we split into two cases. In case $E\setminus \partial Q_{i_1}\neq \emptyset$, we consider a point $x\in E\setminus \partial Q_{i_1}$  and connect it to $F$ with a path $\tau \subset \Omega\setminus \br Q_{i_1}$. Then for a small $\delta>0$ the perturbations $\beta_s$ given by Lemma \ref{unif:Paths joining continua} do not intersect $\br Q_{i_1}$ and have the desired properties. In case $E\subset \partial Q_{i_1}$, $E$ has to contain an arc; we choose $x$ to be an ``interior" point of this arc. We connect $E\ni x$ to $F$ with an open path $\tau \subset \Omega \setminus \br Q_{i_1}$; the latter region is just a topological annulus. Lemma \ref{unif:Paths joining continua} yields non-exceptional paths $\beta_s$, $s\in (0,\delta)$, but this time the paths are not necessarily disjoint from $\br Q_{i_1}$. To amend this, we consider a possibly smaller $\delta>0$ and open subpaths of $\beta_s$ that we still denote by $\beta_s$, which connect $E$ to $F$ without entering $\br Q_{i_1}$. Finally, one has to treat the case $z\in \partial S_{i_1}\cap \partial S_{i_2}$, but this is done exactly as the case $z\in \partial S_{i_1}$.

For the moment, we fix a path $\beta_s$ and we consider subpaths of $\beta_s$ as follows. Assume $\beta_s$ is parametrized as it runs from $F$ to $E$. Let $\gamma_j$ be the open subpath from the point of last entry of $\beta_s$ into $f^{-1}(A_j)$ until the point of first entry into $f^{-1}(B(z,r_j))$, $j=1,\dots,N$. Then $\gamma_j$ intersects only peripheral disks meeting $ f^{-1}(A_j)$. Hence, by construction of the annuli $A_j$, the paths $\gamma_j$ for distinct indices $j$ intersect disjoint sets of peripheral disks. 

Let $\pi \colon \R^2\to \R$ be the ``projection" $w=z+re^{i\theta} \mapsto r$, so $d_j(S_i)= \mathcal H^1(\pi (S_i))$. If the endpoints of $\gamma_j$ lie in $S$, then by Corollary \ref{unif:Upper gradient-projection} we have:
\begin{align}\label{unif:Injectivity-light-lambda gamma_j}
\sum_{i: \br Q_i\cap \gamma_j\neq \emptyset} \lambda(Q_i)= \frac{1}{Nr_j} \sum_{i: \br Q_i\cap \gamma_j\neq \emptyset}d_j(S_i) \geq \frac{1}{Nr_j} r_j =\frac{1}{N}.
\end{align}  
If this is not the case, then $\gamma_j$ enters or exits $ f^{-1}(A_j)$ through some peripheral disks $Q_{k},Q_{l}$. Applying Corollary \ref{unif:Upper gradient-projection} to a subpath of $\gamma_j$ that has its endpoints on $S$, and considering the contribution of $d_j(Q_{k}), d_j(Q_l)$ we also obtain \eqref{unif:Injectivity-light-lambda gamma_j} in this case. 

From now on, to fix our notation, we assume that the exceptional squares $S_{i_1},S_{i_2}$ that wish to exclude actually exist (if not, then one just has to ignore the indices $i_1,i_2$ in what follows). Summing in \eqref{unif:Injectivity-light-lambda gamma_j} over $j$ we obtain
\begin{align*}
1\leq \sum_{\substack{i:\br Q_i\cap \beta_s\neq \emptyset \\i\neq i_1,i_2} } \lambda(Q_i)\leq \sum_{\substack{i:\br Q_i\cap \psi^{-1}(s)\neq \emptyset \\i\neq i_1,i_2} } \lambda(Q_i).
\end{align*}
Here $\psi$ is as in Lemma \ref{unif:Paths joining continua}. Observe that for all $i\in \N$ the functions $s\mapsto \x_{\br Q_i \cap \psi^{-1}(s)}$ are upper semi-continuous, thus measurable. We integrate over $s\in (0,\delta)$ and we obtain using Fubini's theorem:
\begin{align*}
\delta \leq  \sum_{j=1}^N \frac{1}{Nr_j} \sum_{\substack{i:S_i\cap A_j\neq \emptyset \\i\neq i_1,i_2}}d_j(S_i) \int_{0}^\delta \x_{\br Q_i\cap \psi^{-1}(s)}\, ds.
\end{align*}
The fact that $\psi$ is 1-Lipschitz yields  
\begin{align}\label{unif:Injectivity-light-1}
\delta\leq \sum_{j=1}^N \frac{1}{Nr_j} \sum_{\substack{i:S_i\cap A_j\neq \emptyset \\i\neq i_1,i_2}}d_j(S_i) \diam(Q_i).
\end{align}
If we interchange the sums and apply the Cauchy-Schwarz inequality in the right hand side we have:
\begin{align*}
&\frac{1}{N} \sum_{i\in \N\setminus \{i_1,i_2\}} \diam(Q_i) \sum_{j:S_i\cap A_j\neq \emptyset} \frac{d_j(S_i)}{r_j} \\
&\quad\quad \leq  \frac{1}{N} \left( \sum_{i\in \N} \diam(Q_i)^2 \right)^{1/2} \left[ \sum_{i\in \N\setminus \{i_1,i_2\}} \left( \sum_{j:S_i\cap A_j\neq \emptyset} \frac{d_j(S_i)}{r_j}  \right)^2 \right]^{1/2}.
\end{align*}
Observe that the first sum is a finite constant $C^{1/2}$ by the quasiballs assumption \eqref{unif:Quasi-balls}. In the second sum, note that each $S_i$, $i\neq i_1,i_2$, intersects only one annulus $A_j$ so the inner sum actually is a sum over a single term. Combining these observations with \eqref{unif:Injectivity-light-1}, we have
\begin{align*}
\delta^2 N^2 \leq C \sum_{i\in \N\setminus \{i_1,i_2 \}} \left( \sum_{j:S_i\cap A_j\neq \emptyset} \frac{d_j(S_i)^2}{r_j^2}\right) =C \sum_{j=1}^N \frac{1}{r_j^2}\sum_{\substack{i:S_i\cap A_j\neq \emptyset \\i\neq i_1,i_2}} d_j(S_i)^2.
\end{align*}  
By \eqref{unif:Injectivity-light-d_j}, $d_j(S_i)^2\leq C\mathcal H^2(S_i\cap A_j)$. Also, 
$$\sum_{i:S_i\cap A_j \neq \emptyset} \mathcal H^2(S_i\cap A_j) \leq  \mathcal H^2(A_j)=\pi\cdot (R_j^2-r_j^2)= 3\pi r_j^2.$$
Hence,
\begin{align*}
\delta^2 N^2\leq 3\pi C \sum_{j=1}^N \frac{1}{r_j^2} r_j^2 =3\pi C N.
\end{align*}
This is a contradiction as $N\to \infty$.
\end{proof}

We have the following strong conclusion:

\begin{corollary}\label{unif:Injectivity-rho non zero}
For all $i\in \N$ we have $\rho(Q_i)=\osc_{Q_i}(u)=\osc_{Q_i}(v)\neq 0$. In particular, no peripheral circle $\partial Q_i$ is mapped under $f$ to a point, and all squares $S_i$ are non-degenerate.
\end{corollary}

\begin{remark}
The non-vanishing of the ``gradient" of a non-constant carpet-harmonic function is not clear in general. The difficulty in obtaining such a result in non-linear potential theory is also reflected by the fact that in dimension $n\geq 3$ it is not known whether the gradient of a non-constant $p$-harmonic function can vanish on an open set.
\end{remark}

\begin{remark}
In \cite{Schramm:Tiling} Schramm shows that finite triangulations of quadrilaterals can be transformed to square tilings of rectangles by a similar method. However, non-degeneracies cannot be avoided in his case, namely a vertex of the triangulation might correspond to a degenerate square under the correspondence. It is quite a surprise that, if one regards the carpet as an ``infinite triangulation", these degeneracies disappear.
\end{remark}

\begin{corollary}\label{unif:Injectivity:empty interior level sets}
For each $t\in [0,1]$ the level set $\alpha_t=\tilde u^{-1}(t)$ has empty interior.
\end{corollary}
\begin{proof}
Suppose that $\alpha_t$ has non-empty interior for some $t\in [0,1]$. Then for some peripheral disk $Q_i$, $i\in \N$, the intersection $Q_i\cap \alpha_t$ contains an open set. Since $\tilde u$ is harmonic in $Q_i$ (in the classical sense), it follows that $\tilde u$ is constant in $\br Q_i$, and thus $\rho(Q_i)=0$. This contradicts Corollary \ref{unif:Injectivity-rho non zero}.
\end{proof}

Recall that if a level set $\alpha_t$, $t\in \mathcal T$, intersects a peripheral circle $\partial Q_i$, then it intersects it in precisely two points. Using Lemma \ref{unif:Injectivity-light} we obtain a better description of the level sets $\alpha_t$ also for $t\notin \mathcal T$.

\begin{lemma}\label{unif:Injectivity-level sets}
Fix $i\in \N$. For $t\in \{m_{Q_i}(u) , M_{Q_i}(u) \}$ the level set $\alpha_t$ intersects $\partial Q_i$ in a connected set, i.e., in an arc. Let $\beta_1$ be the arc corresponding to $t=m_{Q_i}(u)$ and $\beta_3$ be the arc corresponding to $t=M_{Q_i}(u)$. For all $t\in (m_{Q_i}(u),M_{Q_i}(u))$ the level set $\alpha_t$ intersects $\partial Q_i$ in exactly two points, each of which is located on one of the two complementary arcs of $\beta_1$ and $\beta_3$ in $\partial Q_i$. 

Furthermore, $\Theta_1=\alpha_0\cap \partial \Omega$ and $\Theta_3=\alpha_1\cap \partial \Omega$. Finally, for all $t\in (0,1)$ the intersection $\alpha_t\cap \partial \Omega$ contains two points, one on $\Theta_2$ and one on $\Theta_4$.
\end{lemma}

\begin{proof}
Since $\osc_{Q_i}(u)>0$ by Corollary \ref{unif:Injectivity-rho non zero}, it follows that $m_{Q_i}(u)< M_{Q_i}(u)$. Suppose that for $t=m_{Q_i}(u)$ the set $\alpha_t\cap\partial Q_i$ has at least two components $E_1,E_2\subset \partial Q_i$. The sets $E_1$ and $E_2$ are closed arcs or points, they are disjoint, and $\partial Q_i \setminus (E_1\cup E_2)$ has two components $F_1,F_2\subset \partial Q_i$, which are open arcs. The function $u$ is non-constant near the endpoints of $F_j$, on which it has the value $t$ for $j=1,2$. In fact for $j=1,2$ we can find disjoint arcs $I_{j,k} \subset F_j$, $k=1,2$, on which $u$ is non-constant and attains the value $t$ at one of the endpoints of $I_{j,k}$, $k=1,2$. By continuity, it follows that $\bigcap_{j,k=1,2}u(I_{j,k})$ contains an  interval of the form $(t,t+\varepsilon)$ for some $\varepsilon>0$. The intermediate value theorem implies that we can find some $t'\in \mathcal T\cap (t,t+\varepsilon)$ such that $\alpha_{t'}\cap I_{j,k}\neq \emptyset$ for $j,k=1,2$.  Hence, $\alpha_{t'} \cap \partial Q_i$ contains more than two points and this contradicts Corollary \ref{unif:Level sets-exactly two points}. 

Let $\beta_1\subset \partial Q_i$ be the closed arc that is the unique preimage in $\partial Q_i$ under $f$ of the left vertical side of $\partial S_i$. The same conclusion holds for $t=M_{Q_i}(u)$, and we consider the corresponding closed arc $\beta_3$.

Let $\beta_2,\beta_4 \subset \partial Q_{i}$ be the complementary closed arcs of the two extremal arcs, numbered in a counter-clockwise fashion, exactly as we numbered the sides of $\partial \Omega$. Since $f(\partial Q_i)=\partial S_i$ (by Corollary \ref{unif:Definition-Range-Extension}), the images of $\beta_2$ and $\beta_4$ are the bottom and top sides of $\partial S_i$; here it is not necessary that $\beta_2$ is mapped to the bottom side and $\beta_4$ is mapped to the top side, but it could be the other way around. It follows that $v$ is constant on each of the arcs $\beta_2,\beta_4$.

By an application of the intermediate value theorem followed by the use of Corollary \ref{unif:Level sets-exactly two points}, one sees that for all $t\in (m_{Q_i}(u),M_{Q_i}(u))$ the intersection $ \alpha_t\cap \partial Q_i$ must have two components, one in $\beta_2$, and one in $\beta_4$. Since $u$ and $v$ are constant on each component, these components must be singletons by Lemma \ref{unif:Injectivity-light}.

The set $\alpha_0\cap \partial \Omega$ contains $\Theta_1$ and is contained in $\partial \Omega$. If there exists a point $x\in \alpha_0\cap \Theta_2\setminus \Theta_1$, then there has to be an entire arc $E\subset \alpha_0\cap \Theta_2$. Otherwise, by the intermediate value theorem, there would exist level sets $\alpha_t$, $t\in \mathcal T$, that intersect $\Theta_2$ in at least two points, a contradiction to Lemma \ref{unif:Level sets-at most two points of intersection}. However, $v$ is constant equal to $0$ on $E$ by Lemma \ref{unif:Conjugate-Values on boundaries}, which implies that the continuum $E$ is mapped to a point under $f$. This contradicts Lemma \ref{unif:Injectivity-light}. The same argument applies with $\Theta_4$ in the place of $\Theta_2$ and yields that $\Theta_1=\alpha_0\cap \partial \Omega$. Similarly, $\Theta_3=\alpha_1\cap \partial \Omega$.

Finally, we already know from Proposition \ref{unif:Level sets} that for all $t\in (0,1)$ the intersection $\alpha_t\cap \partial \Omega$ has two components, one on $\Theta_2$ and one on $\Theta_4$. Again, by Lemma \ref{unif:Injectivity-light} these components have to be singletons.  
\end{proof}

An immediate corollary is the following. We use notation from the proof of the preceding lemma.

\begin{corollary}\label{unif:Injectivity-top-bottom}
For each peripheral disk $Q_i$, $i\in \N$,  the arcs $\beta_2$ and $\beta_4$ are mapped injectively onto the bottom and top sides of $\partial S_i$, respectively. Furthermore, $\Theta_2$ and $\Theta_4$ are mapped injectively onto the bottom and top sides of $[0,1]\times [0,D]$, respectively.
\end{corollary}
\begin{proof}
Recall that the winding number of $f\big|_{\partial \Omega}$ around every point of $(0,1)\times (0,D)$ is $+1$ (see proof of Corollary \ref{unif:Definition-Range-Extension}). Moreover, we have 
$$f(\br \Omega \setminus Q_i)= [0,1]\times [0,D]\setminus S_i,$$
by Corollary \ref{unif:Definition-Range-Extension} and the fact that $f(Q_j)=S_j$ for all $j\in \N$; see the comments before Lemma \ref{unif:Injectivity-light}. It follows by homotopy that the winding number of $f\big|_{\partial Q_i}$ around each point of $S_i$ is $+1$. Hence, the arcs $\beta_2$ and $\beta_4$ of $\partial Q_i$ must be mapped onto the bottom and top sides of $\partial S_i$, respectively.

Regarding the injectivity claim, note that $f$ is injective when restricted to the ``interior" of the arc $\beta_2$, since for each $t\in (m_{Q_i}(u),M_{Q_i}(u))$ the level set $\alpha_t$ intersects this arc at one point, by the previous lemma. By continuity, the endpoint $\beta_1\cap \beta_2$ (strictly speaking, $\beta_1\cap \beta_2$ is a singleton set containing one point) of $\beta_2$ has to be mapped to the bottom left corner of the square $S_i$, and the endpoint $\beta_3\cap \beta_2$ of $\beta_2$ has to be mapped to the bottom right corner of $S_i$. This shows injectivity on all of $\beta_2$. Similarly, one shows the other claims about injectivity, based on the preceding lemma.
\end{proof}

We have completed our preparation to show the injectivity of $f$.

\begin{lemma}\label{unif:Injectivity}
The map $f\colon S \to \mathcal R$ is injective.
\end{lemma}

The proof of this lemma is slightly technical, so we first provide a sketch of the argument. Note that $f$ is already injective, when restricted to level sets $S\cap \alpha_t$ for $t\in \mathcal T$, by Corollary \ref{unif:Conjugate-Oscillation-Increasing}. In fact, $v$ is ``increasing" in some sense in these level sets. Thus, the only possibility is that injectivity fails at some level set $\alpha_{t_0}\cap S$ for $t_0\notin \mathcal T$. Suppose that $z=(t_0,s_0)\in \mathcal R$ has two preimages $x,y\in S\cap \alpha_{t_0}$. Then we show that for $t\in \mathcal T$ near $t_0$ there exist points $a_x=(t,s_x)$, $a_y=(t,s_y)\in S\cap \alpha_t$ near $x,y$, respectively. Using the fact that $v$ is increasing on $\alpha_t$, we show that there exists a path $\gamma \subset S$ connecting $a_x,a_y$ that is mapped into a small neighborhood of $z$ under $f$. In the limit, as $a_x\to x$ and $a_y\to y$ one obtains a continuum $E\subset S$ that connects $x$ and $y$ and is mapped to $z$. This will contradict Lemma \ref{unif:Injectivity-light}.

\begin{proof}
If $t\in \mathcal T$, then $v$ is ``increasing" on $ \alpha_t\cap S$, in the sense of Corollary \ref{unif:Conjugate-Oscillation-Increasing}. Hence, $f$ is injective on the set $U\coloneqq \bigcup_{t\in \mathcal T}(  \alpha_t\cap S)$.

Assume $z=(t_0,s_0)\in \mathcal R $ and $t_0\notin \mathcal T $. Note that every point $x\in f^{-1}(z)$ can be approximated by $Q_i$ with $\rho(Q_i)>0$; see Corollary \ref{unif:Injectivity-rho non zero}. Since $\rho(Q_i)>0$, it follows that there exist levels $t\in \mathcal T$ with $\alpha_t\cap\partial Q_i \neq \emptyset$. Hence every point $x\in f^{-1}(z)$ can be approximated by points $a\in U$, and so $U$ is dense in $S$. We will split in two cases. The ideas are similar but the technical details are slightly different. We recommend that the reader focus on Case 1, at a first reading of the proof. 

\textbf{Case 1:}
Assume first that $z=(0,s_0)$ or $(1,s_0)$ or $z$ lies in an open vertical side of some square $\partial S_{i_0}$, $i_0\in \N$. In any case, either every preimage $x \in f^{-1}(z)$ can be only approximated by points $a\in U$ with $u(a)<t_0$, or every preimage can be only  approximated  by points $a\in U$ with $u(a)>t_0$. Indeed, this is clear if $z$ lies in a vertical side of $[0,1]\times[0,D]$ (i.e., $t_0=0$ or $t_0=1$), so suppose that $z=(t_0,s_0)$ lies in the open left vertical side of a square $S_{i_0}$, and $x$ is a preimage of $z$. If $a_n \in U$ is a sequence converging to $x$, then $f(a_n)\in S$ converges to $z$ by continuity. Since $a_n\in U$ and $t_0\notin \mathcal T$, we necessarily have $u(a_n)\neq t_0$ for each $n\in \N$, and $f(a_n)$ cannot lie on the vertical line passing through $z$. Since $z$ lies in the interior of the left vertical side of $S_{i_0}$ and $f(a_n)\notin S_{i_0}$ for all $n\in \N$, it follows that $f(a_n)$ lies on the ``left" of the vertical line $u=t_0$ for all sufficiently large $n$. Therefore, $u(a_n)<t_0$ for all sufficiently large $n$, as desired. 

Suppose in what follows that every preimage of $z$ can be only approximated by points $a\in U$ with $u(a)<t_0$, and consider a preimage $x \in f^{-1}(z)$. Consider a small ball $B(x,\delta)$ such that $S\cap B(x,\delta)$  contains points $a\in U$, all of which satisfy $u(a)<t_0$. Fix such a point with $u(a)=t<t_0$. The level sets $\alpha_t,\alpha_{t_0}$ define a simply connected region $A_{t,t_0}$ that contains all level sets $\alpha_{t'}$ for $t'\in (t,t_0)$, by Proposition \ref{unif:Level sets}. Since $B(x,\delta)$ meets both $\alpha_t$ and $\alpha_{t_0}$, it follows that $\alpha_{t'}$ meets $B(x,\delta)$ for every $t'\in (t,t_0)$. The level sets $\alpha_{t'}$ cannot meet $\partial Q_{i_0}$, in case $x\in \partial Q_{i_0}$ and $z\in \partial S_{i_0}$; this is because $f(\partial Q_{i_0})=\partial S_{i_0}$. This implies that $\alpha_{t'}$ either does not meet peripheral disks near $x$ or it meets arbitrarily small peripheral disks near $x$. Thus, for all $t<t_0$ sufficiently close to $t_0$ we have that the intersection $S\cap B(x,\delta)\cap \alpha_t$ is non-empty. 

Suppose now that there exist two points $x,y\in f^{-1}(z)$, and consider a small $\delta>0$ such that $\alpha_t$ intersects both sets $S\cap B(x,\delta),S\cap B(y,\delta)$ for some $t<t_0$, $t\in \mathcal T$, sufficiently close to $t_0$. Let $a_x\in S\cap B(x,\delta)\cap \alpha_t$ and $a_y \in S\cap B(y,\delta)\cap\alpha_t$. By the continuity of $f$, the images $ f(a_x)=(t,s_x)$,  $f(a_y)=(t,s_y)$ will lie in a small ball $B(z,\varepsilon)$. Without loss of generality $s_x<s_y$. Consider the path $\gamma \subset \alpha_t$ that connects $a_x$ to $a_y$. The function $v$ is increasing on $\alpha_t$ by Corollary \ref{unif:Conjugate-Oscillation-Increasing}, hence $f(\gamma\cap S)\subset \{t\} \times [s_x,s_y]  \subset B(z,\varepsilon)$. We alter the path $\gamma$ as follows. Whenever $\gamma \cap Q_i\neq \emptyset$ we replace this intersection with an arc $\beta \subset \partial Q_i$ joining the two points of $\gamma \cap \partial Q_i$.  We call the resulting path $\tilde \gamma$ and note that $\tilde \gamma \subset S$. We claim that $f(\tilde\gamma) \subset B(z,3\varepsilon)$.  To see this, note that the image arcs $f(\beta)$ are contained in squares $\partial S_i$ by Corollary \ref{unif:Definition-Range-Extension}, and the top and bottom sides of these squares intersect the ball $B(z,\varepsilon)$, namely at the endpoints of $\beta$ which were also points of $\gamma$. Our claim is proved.

Finally, observe that as $\delta\to 0$ the path $\tilde \gamma $ subconverges to a non-trivial continuum $E\subset S$ containing $x$ and $y$. Then the images $f(\tilde \gamma)$ converge to the point $z$, so by continuity $f(E)=z$. This contradicts Lemma \ref{unif:Injectivity-light}.

\textbf{Case 2:} Assume that $z=(t_0,s_0)$ does not lie on an open vertical side of any square $\partial S_i$, $i\in \N$, or on a vertical side of the rectangle $[0,1]\times [0,D]$ (so $0<t_0<1$). We claim that there exists a ``distinguished" preimage $x$ of $z$ that can be approximated by points $a,b \in [0,1]\times [0,D]$ with $\tilde u(a)<t_0$ and $\tilde u(b)>t_0$.

To prove that, we will show that there exists a continuum $C\subset \alpha_{t_0}$ connecting $\Theta_2$ and $\Theta_4$ such that every point $x\in C$ can be approximated by points $a,b$ with $\tilde u(a)<t_0$ and $\tilde u(b)>t_0$. The function $\tilde v$ is continuous on $C$ and attains all values between $0$ and $D$. Hence, there exists some $x\in C$ with $\tilde v(x)=s_0$. In fact, $x\in C\cap S \cap f^{-1}(z)$, because $f(Q_i)= S_i$ for all $i\in \N$; see comments before Lemma \ref{unif:Injectivity-light}. 

The existence of $C$ will be justified with the following lemma about planar topology, which reflects the \textit{unicoherence} of the plane:
\begin{lemma}[Theorem 5.28a, p.~65, \cite{Wilder:topology}] \label{unif:Unicoherence}
If $A,B$ are planar continua neither of which separates the plane, then $A\cup B$ does not separate the plane if and only if $A\cap B$ is a continuum.
\end{lemma}
Let $V_1=\tilde u^{-1}((0,t_0))= A_{0,t_0}$ and $V_3=\tilde u^{-1}((t_0,1))= A_{t_0,1}$. Since all level sets $\alpha_t$ have empty interior by Corollary \ref{unif:Injectivity:empty interior level sets}, Proposition \ref{unif:Level sets} implies that the closures $\br V_1,\br V_3$ are continua that do not separate the plane. Moreover, $\br V_1\cup \br V_3=\br \Omega$ so $\br V_1\cup \br V_3$ does not separate the plane. Indeed, a point $w\in \br \Omega \setminus \br V_1\cup \br V_3$ would necessarily satisfy $u(w)=t_0$ and would have a rel.\ open neighborhood $W$ in $\br \Omega\setminus \br V_1\cup \br V_3 $. This would imply that $\alpha_{t_0}$ has non-empty interior, a contradiction. By Lemma \ref{unif:Unicoherence} we conclude that $\br V_1\cap \br V_3\subset \alpha_{t_0}$ is a continuum. Moreover, this continuum has to connect $\alpha_{t_0}\cap \Theta_2$ to the point $\alpha_{t_0}\cap \Theta_4$. This is because $\alpha_{t_0}\cap \Theta_2,\alpha_{t_0}\cap \Theta_4$ are singletons by Lemma \ref{unif:Injectivity-level sets}, and they lie in $\br V_1\cap \br V_3$ since their complement in $\Theta_2,\Theta_4$, respectively, consists of two arcs, one contained in $V_1$ and one in $V_3$. Summarizing, $C\coloneqq \br V_1\cap \br V_3$ is the desired continuum connecting $\Theta_2$ to $\Theta_4$.

Our next claim is that for every small ball $B(x,\delta)$, where $x\in C$ is a preimage of $z$, the set $S\cap B(x,\delta)\cap \alpha_t$ is non-empty for all $t<t_0$ sufficiently close to $t_0$, and for all $t>t_0$ sufficiently close to $t_0$; this will be the crucial property that we need for the ``distinguished" preimage $x$ of $z$. To prove this, we consider three cases: $x\in S^\circ$, $x\in \partial Q_i$ for some $i\in \N$, and $x\in \partial \Omega$.

First, assume that $x\in S^\circ$. We fix a small ball $B(x,\delta)$ and a point $a\in B(x,\delta)$ with $\tilde u(a)=t<t_0$; such a point exists by the properties of $x\in C$. The level sets $\alpha_t,\alpha_{t_0}$ define a simply connected region $A_{t,t_0}$, and all level sets $\alpha_{t'}$ for $t'$ between $t$ and $t_0$ lie in this region and have to intersect $B(x,\delta)$; see Proposition \ref{unif:Level sets}. If $x\in S^\circ $ then $S\cap B(x,\delta)\cap \alpha_{t'}\neq \emptyset$, since the peripheral disks that might be  intersected by sets $\alpha_t'$ for $t'\in (t,t_0)$ are arbitrarily small; see also Case 1. The same holds for level sets $\alpha_t$, $t>t_0$, sufficiently close to $t_0$.

We now prove the claim in case $x\in \partial Q_{i_0}$ for some $i_0\in \N$. Recall by Corollary \ref{unif:Injectivity-top-bottom} that $\partial Q_{i_0}$ is partitioned into four arcs $\beta_1,\dots,\beta_4$, such that $\beta_2, \beta_4$ are mapped injectively onto the top an bottom sides of $\partial S_{i_0}$, respectively, and $\beta_1, \beta_3$ are mapped injectively onto the left and right sides of $\partial Q_i$, respectively, by the Case 1. In particular, $f\big|_{\partial Q_i}$ is injective and $f\big|_{\partial Q_i}^{-1}$ is defined and is continuous on $\partial S_i$. The  only possibility here is that $x\in \beta_2\cup \beta_4$ (recall the assumption of Case 2). If $x$ is an interior point of one of the arcs $\beta_2,\beta_4$, then $\partial Q_{i_0}\cap S\cap B(x,\delta)\cap \alpha_t\neq \emptyset$ for all $t$ near $t_0$ by the continuity of $f\big|_{\partial Q_i}^{-1}$. If $x$ is a ``corner" point lying, e.g., on $\beta_2\cap \beta_1$ then we can approximate $x$ by points $a\in \beta_2$ thus satisfying $u(a)>t_0$, and we can approximate $x$ by points $b\in \beta_1$ with $u(b)=t_0$. In the latter case, $f(b)$ is contained in the left open vertical side of $\partial S_{i_0}$. Arguing as in Case 1, for a small $\delta'>0$ we have $S\cap B(b,\delta')\cap \alpha_t\neq \emptyset$ for all $t<t_0$ sufficiently close to $t_0$. If $b$ is sufficiently close to $x$ then $B(b,\delta')\subset B(x,\delta)$, so $S\cap B(x,\delta)\cap \alpha_t\neq \emptyset$, as desired.

Finally, if $x\in \partial \Omega$, then it has to lie in the interior of the arcs $\Theta_2$ or $\Theta_4$. The map $f$ is injective on these arcs by Corollary \ref{unif:Injectivity-top-bottom}. Hence it is easy to see that $S\cap B(x,\delta)\cap \alpha_t\neq \emptyset$ for all $t$ near $t_0$, and in fact the intersection contains points of $\Theta_2$ or $\Theta_4$.

The preimage $x$ is the ``distinguished" preimage of $z$ that is ``accessible" from both sides $u<t_0$ and $u>t_0$. Now, assume that there exists another preimage $y\in f^{-1}(z)$. Since $U$ is dense in $S$, there exist near $y$ points $a\in \alpha_t\cap S$, $t\in \mathcal T$, with $u(a)<t_0$ or $u(a)>t_0$. Without loss of generality, we assume that arbitrarily close to $y$ we can find such points with $u(a)=t<t_0$. Then for a small $\delta>0$ the intersection $S\cap B(y,\delta)\cap \alpha_t$ is non-empty for $t<t_0$, $t\in \mathcal T$, arbitrarily close to $t_0$. Now, we argue exactly as in the last part of Case 1. Consider a level set $\alpha_t$, $t\in \mathcal T$, that intersects both sets $S\cap B(x,\delta)$ and $S\cap B(y,\delta)$ and let $\gamma \subset \alpha_t$ be a subpath that connects the balls $B(x,\delta),B(y,\delta)$. The image $f(\gamma\cap S)$ is contained in a small ball $B(z,\varepsilon)$. We modify the path $\gamma$ to obtain a path $\tilde \gamma \subset S$ that connects the balls $B(x,\delta),B(y,\delta)$, and such that $f(\tilde \gamma)$ lies in a slightly larger ball $B(z,3\varepsilon)$. As $\delta\to 0$ the path $\tilde \gamma$ subconverges to a continuum in $S$ that connects $x$ and $y$. The images $f(\tilde \gamma)$ converge to $z$, so Lemma \ref{unif:Injectivity-light} is again contradicted.
\end{proof}

\begin{remark}\label{unif:Injectivity:Remark Extension}
Since $f\big|_{\partial Q_i}$ is a homeomorphism and $\tilde u, \tilde v$ are harmonic on $Q_i$, we can use the Rad\'o--Kneser--Choquet theorem \cite[p.~29]{Duren:harmonic}  to conclude that $\widetilde f \big|_{\br{Q}_i}$ is a homeomorphism onto $\br S_i$. Thus, $ f\colon  \br \Omega\to [0,1]\times [0,D]$ is a homeomorphism. Furthermore, the set $\mathcal R=f(S)$ is a Sierpi\'nski carpet, as defined in the Introduction, with $\mathcal H^2(\mathcal R)=0$ and the basic assumptions of quasiballs \eqref{unif:Quasi-balls} and fatness \eqref{unif:Fat sets} are trivially satisfied for the peripheral disks of $\mathcal R$, which are squares. 

In fact, the harmonicity of the extension inside each $Q_i$ is not needed anymore, so one could consider  arbitrary homeomorphic extensions $\widetilde f\colon \br Q_i\to \br S_i$ given, for example, by the Sch\"onflies theorem. Recall that $Q_0=\C\setminus \br \Omega$ and $S_0= \C\setminus [0,1]\times [0,D]$. We also extend $f\big|_{\partial Q_0}$ to a homeomorphism from $\br Q_0$ onto $\br S_0$. Pasting together these homeomorphisms we obtain a homeomorphic extension $\widetilde f\colon \C\to \C$ of $f$; see e.g.\ \cite[Lemma 5.5]{Bonk:uniformization}.
\end{remark}

\section{Regularity of \texorpdfstring{$f$}{f} and \texorpdfstring{$f^{-1}$}{f-1}}\label{unif:Section Regularity}\index{uniformizing map!regularity}
In this section we prove the main result, Theorem \ref{unif:Main theorem}, which will follow from Proposition \ref{unif:Regularity-F modulus} and Proposition \ref{unif:Regularity-F-1 modulus}. We consider an (arbitrary) homeomorphic extension of $f\colon S\to \mathcal R$ to a map $f\colon \C \to \C$, as in Remark \ref{unif:Injectivity:Remark Extension}.  As discussed in this remark, the set $\mathcal R$ is a Sierpi\'nski carpet, as defined in the Introduction.

Recall that by Proposition \ref{unif:Upper gradient for F} we have the upper gradient inequality
\begin{align}\label{unif:Regularity-upper gradient for F}
|f(x)-f(y)|\leq \sqrt{2} \sum_{i:\br Q_i \cap \beta \neq \emptyset} \rho(Q_i) =\sum_{i:\br Q_i\cap \beta\neq \emptyset} \diam(S_i)
\end{align}
for points $x,y\in \br \beta\cap S$ and all open paths $\beta$, which are subpaths of paths $\gamma\subset \Omega$ lying outside an exceptional family of 2-modulus zero. We will use this to show that $f$ preserves, in some sense, carpet modulus. In fact, for technical reasons, we will introduce a slightly different notion of carpet modulus here, which was also used in the statement of Theorem \ref{unif:Main theorem}:

\begin{definition}\label{unif:Carpet-modulus weak}
Let $\Gamma$ be a family of paths in $\C$. A sequence of non-negative numbers $\{\lambda(Q_i)\}_{i\in \N\cup \{0\}}$ is admissible for the \textit{weak carpet modulus}\index{modulus!weak carpet modulus} $\br \md (\Gamma)$ if there exists an exceptional path family $\Gamma_0$ with $\md_2(\Gamma_0)=0$ such that for all $\gamma\in \Gamma\setminus \Gamma_0$ we have
\begin{align*}
\sum_{i:\br Q_i\cap \gamma\neq \emptyset}\lambda(Q_i)\geq 1.
\end{align*}  
We define $\br \md(\Gamma)=\inf_{\lambda}\sum_{i\in \N\cup \{0\}}\lambda(Q_i)^2$ where the infimum is taken over all admissible weights $\lambda$.
\end{definition}

Recall at this point that $Q_0=\C\setminus \br \Omega$, and we \textit{do} include $\lambda(Q_0)$ in the above sums, whenever $\br Q_0\cap \gamma \neq \emptyset$, in contrast to  the previous sections, in which all paths were ``living" in $\Omega$.

First we show a preliminary lemma, in the same spirit as Corollary \ref{unif:Upper gradient-projection}.

\begin{lemma}\label{unif:Regularity-Path hausdorff measure zero}
There exists a path family $\Gamma_1$ in $\C$ with $\md_2(\Gamma_1)=0$ such that for every path $\gamma\subset \C$, $\gamma\notin \Gamma_1$, the following holds:
\begin{flushleft}
For any two points $x,y\in  \br \gamma$, there exists an open path $\tilde \gamma \subset \C$ joining $x$ and $y$, such that $\{i\in \N\cup \{0\}:\br Q_i \cap \tilde \gamma \neq \emptyset\} \subset \{i\in \N\cup \{0\}:\br Q_i\cap \gamma \neq \emptyset\}$ and $\mathcal H^1(f(\tilde \gamma\cap S))=0$. 
\end{flushleft}
\end{lemma}

The curve $\tilde \gamma$ is not necessarily a subcurve of $\gamma$, but it intersects no more (closed) peripheral disks than $\gamma$ does. 

\begin{remark}An important ingredient in the proof will be the following. Although the upper gradient inequality \eqref{unif:Regularity-upper gradient for F} holds - a priori - only for subpaths of paths $\gamma$ \textit{contained} in $\Omega$, lying outside an exceptional family $\Gamma_0$ with $\md_2(\Gamma_0)$, it turns out that this can be extended to paths in $\C$. Namely, the family of paths in $\C$ that have a subpath $\gamma \subset \Omega$ for which the upper gradient inequality fails, has $2$-modulus zero. This implies that \eqref{unif:Regularity-upper gradient for F} holds for all open subpaths $\beta\subset \Omega$ of paths $\gamma\subset \C$, lying outside an exceptional family $\Gamma_0$ of $2$-modulus zero.
\end{remark}

\begin{proof}
Let $\gamma\subset \C$ be a path such that:
\begin{enumerate}[\upshape(1)]
\item \eqref{unif:Regularity-upper gradient for F} holds along all of its subpaths that are contained in $\Omega$, 
\item $\mathcal H^1(\gamma\cap S)=0$, and
\item $$\sqrt{2}\sum_{\substack{i:\br Q_i\cap \gamma \neq \emptyset \\ i\in \N\setminus \{0\}}} \rho(Q_i)=\sum_{\substack{i:\br S_i\cap f(\gamma)\neq \emptyset\\ i\in \N\setminus \{0\} }}\diam (S_i)<\infty.$$ 
\end{enumerate}
This is true for paths outside an exceptional family $\Gamma_1$ with $\md_2(\Gamma_1)=0$. Indeed, (1) holds for $\md_2$-a.e.\ path by the preceding remark, and (2) holds for $\md_2$-a.e.\ path since $\mathcal H^2(S)=0$. The third statement also holds for all paths outside a family of conformal modulus zero by Remark \ref{unif:Upper gradient-projection remark}. By the subadditivity of conformal modulus, all three conditions are simultaneously met by all paths outside a family of conformal modulus zero, as desired.

Let $x,y\in \br \gamma$. Note that if $x,y$ lie on the same peripheral disk $\br Q_{i_0}$, $i_0\in \N\cup \{0\}$, and $\gamma$ intersects $\br Q_{i_0}$, then there is nothing to show, since we can just connect $x,y$ with an open path  $\tilde \gamma \subset Q_{i_0}$, and this will trivially have the desired properties. Hence, we may assume that this is not the case. By replacing $\gamma$ with an open subpath we may suppose that $x,y$ are the endpoints of $\br \gamma$. Assume that $\gamma$ is parametrized as it runs from $x$ to $y$. If $x\in \br Q_{i_x}$ for some $i_x\in \N\cup \{0\}$ and $\br Q_{i_x}\cap \gamma\neq \emptyset$ then we can replace $x$ with the last exit point $x_0\in \partial Q_{i_x}$ of $\gamma$ from $\br Q_{i_x}$, and obtain similarly a point $y_0\in \partial Q_{i_y}$, which is the first entry point of $\gamma$ (after $x_0$) in $\br Q_{i_y}$, in case $y\in \partial Q_{i_y}$ for some $i_y\in \N\cup \{0\}$ (cf.\ the discussion on ``accessible" points following Definition \ref{unif:Definition Carpet harmonic}). If we can find a path $\tilde \gamma$ joining $x_0,y_0$ with the desired properties, then we can concatenate it with arcs inside $ Q_{i_x}, Q_{i_y}$, and these do not contribute to the Hausdorff $1$-measure of $f(\tilde \gamma\cap S)$, so the conclusion holds for the concatenation. Hence, we assume that $x,y\in \br \gamma\cap S$ and $\gamma$ does not intersect the closures of the peripheral disks that possibly contain $x,y$ in their boundary.

Another reduction we can make is that we can assume that $\gamma\subset \Omega$. Indeed, by the previous paragraph we may assume that none of $x,y$ lies in $Q_{0}$ and that $\gamma$ does not intersect $\br Q_0$ in case $x\in \partial Q_0$ or $y\in \partial Q_0$. If $\gamma$ hits $\br Q_0$ (thus $x,y\in \Omega$), we let $x_0$ be the first entry point of $\gamma$ in $\br Q_{0}$, and $y_0$ be the last exit point from $\br Q_0$. We consider the open subpaths $\gamma_x\subset \Omega$ from $x$ to $x_0$ and $\gamma_y\subset \Omega$ from $y$ to $y_0$. If the statement of the lemma is true for $\gamma_x$ and $\gamma_y$, then there exist paths $\tilde \gamma_x,\tilde \gamma_y$ joining $x$ to $x_0$, $y$ to $y_0$, respectively, such that they do not intersect more closed peripheral disks than $\gamma$ does, and such that their image has length zero in the carpet $\mathcal R$. Concatenating $\tilde \gamma_x,\tilde \gamma_y$ with a path inside $Q_0$ that joins $x_0$ to $y_0$ provides the desired path $\tilde \gamma$. 

Assuming now that $\gamma\subset \Omega$ and $\gamma$ does not intersect the closed peripheral disks possibly containing $x,y$ in their boundaries, we will construct the path $\tilde \gamma$ through some iteration procedure. We fix $\varepsilon>0$, and a finite index set $J_1\subset \N$ such that
\begin{align}\label{unif:Regularity-Path hausdorff-J_1 bound}
\sum_{\substack{i:\br Q_i\cap \gamma\neq \emptyset \\ i\in \N \setminus J_1}} \diam(S_i)<\varepsilon.
\end{align}
Using the reformulation of Lemma \ref{harmonic:3-Split path} in Remark \ref{harmonic:remark:split path}, we may obtain open subpaths $\gamma_1,\dots,\gamma_m \subset \Omega$ of $\gamma$ that intersect disjoint sets of peripheral disks $\br Q_i$, $i\notin J_1$, such that the upper gradient inequality of $f$ holds along each of them and they have the following property: the path $\gamma_1$ starts at $x_1=x$, the path $\gamma_m$ terminates at $y_m=y$ and in general the path $\gamma_k$ has endpoints $x_k,y_k\in S$ such that for $1\leq k\leq m-1$ we either have
\begin{itemize}
\item $y_k=x_{k+1}$, or
\item $y_k,x_{k+1}\in \partial Q_{j_k}$ for some $j_k\in \N$. The peripheral disks $\br Q_{j_k}$   are distinct and they are intersected by $\gamma$. 
\end{itemize}
We denote by $I\subset \{1,\dots,m\}$ the set of indices $k$ for which the second alternative holds. By shrinking the open paths $\gamma_k$ (or even discarding some of them), we may further assume that they do not intersect the peripheral disks that possibly contain their endpoints in their boundary. See also the proof of Corollary \ref{unif:Upper gradient-projection}. 

For each $k\in I$ we consider a line segment $f(\alpha_k)$ inside $S_{j_k}=f(Q_{j_k})$ that connects the endpoints of $f(\gamma_{k})$ and $f(\gamma_{k+1})$.  Concatenating all the paths $f(\alpha_k),f (\gamma_k)$ we obtain a path $f(\beta_1)$ that connects $f(x)$ to $f(y)$ and intersects fewer peripheral disks than $f(\gamma)$. We estimate $\mathcal H^1(f(\beta_1\cap S))$ as follows: for each $k\in \{1,\dots,m\}$ the image $f(\br \gamma_k \cap S)$ can be covered by a ball $B_k$ of radius $r_k\coloneqq  \sum_{i:\br Q_i\cap \gamma_k\neq \emptyset} \diam(S_i)$, because by the upper gradient inequality \eqref{unif:Regularity-upper gradient for F} we have
\begin{align*}
\sup_{z,w\in \br \gamma_k\cap S} |f(z)-f(w)| \leq \sum_{i:\br Q_i\cap \gamma_k\neq \emptyset}\diam(S_i).
\end{align*}
By construction, $\gamma_1,\dots,\gamma_m$ intersect disjoint sets of peripheral disks. Since we have $\mathcal H^1(f(\alpha_k\cap S))=0$ for all $k$ (as $\alpha_k\cap S$ contains at most two points), it follows that
\begin{align*}
\mathcal H^1(f(\beta_1\cap S))\leq  2 \sum_{k=1}^m r_k\leq 2\sum_{k=1}^m\sum_{i:\br Q_i\cap \gamma_k\neq \emptyset}\diam(S_i) <2\varepsilon
\end{align*}
by \eqref{unif:Regularity-Path hausdorff-J_1 bound}. Here it is crucial that the paths $\gamma_k$ intersect disjoint sets of peripheral disks.

Now, we iterate this procedure with $\varepsilon$ replaced by $\varepsilon/2$. We consider a finite index set $J_2\supset J_1$ such that 
\begin{align*}
\sum_{\substack{i:\br Q_i\cap \gamma\neq \emptyset \\ i\in \N \setminus J_2}} \diam(S_i)<\varepsilon/2
\end{align*}
and we modify each of the paths $\gamma_1,\dots,\gamma_m$ from the first step, according to the previous procedure. The paths $\alpha_k$ from the first step remain unchanged, but new arcs $\alpha_l\subset Q_l$ will be added from the second step, such that $\br Q_l\cap \gamma_k\neq \emptyset$ for some $k\in \{1,\dots,m\}$, and thus 
\begin{align*}
\diam(f(\alpha_l))\leq \diam(S_l)\leq \sum_{i:\br Q_i\cap \gamma_k\neq \emptyset} \diam(S_i)=r_k.
\end{align*} 
This implies that $f(\alpha_l)\subset S_l\subset 2B_k$. Hence, the new path $f(\beta_2)$ that we obtain as a result of concatenations connects $f(x)$ to $f(y)$ and stays close to the path $f(\beta_1)$. In fact, we can achieve that $f(\beta_2\cap S)$ is covered by balls $B_l$ of radius $r_l$, such that $2B_l$ is contained in one of the balls $2B_k$, and
\begin{align*}
\mathcal H^1(f(\beta_2\cap S))\leq 2\sum_{l} r_l \leq \varepsilon.
\end{align*}
This can be achieved by noting that $f(\beta_2\cap S)\subset f(\beta_1\cap S)\subset \bigcup_k B_k$, and choosing an even larger set $J_2\supset J_1$, so that 
\begin{align*}
r_l\leq \sum_{\substack{i:\br Q_i\cap \gamma\neq \emptyset \\ i\in \N \setminus J_2}} \diam(S_i)< \min_{k\in \{1,\dots,m\}}r_k/2.
\end{align*}

In the $n$-th step we have a path $f(\beta_n)$ connecting $f(x),f(y)$ such that the set $f(\beta_n \cap S)$ admits a cover by balls $\{ 2B_l^n \}_l$ that is decreasing in $n$, and whose radii sum does not exceed $\varepsilon/2^{n-3}$. Moreover, by construction, if $f(\beta_n)\cap S_i\neq \emptyset$ for a square $i\in \N$, then $f(\beta_m)\cap S_i$ is a fixed line segment for all $m\geq n$.

The curves $f(\beta_n)$ subconverge to a continuum $f(\beta)$ in the Hausdorff sense that connects $f(x)$ to $f(y)$. The set $f(\beta\cap S)$ is covered by $\{2B_{l}^n\}_l$ for all $n$ and thus $\mathcal H^1(f(\beta\cap S))=0$. The set $\beta$ can only intersect peripheral disks that can be approximated by points intersected by the paths $\beta_n$, and thus by the path $\gamma$. Hence, $\beta$ cannot intersect more peripheral disks than $\gamma$ does.

We claim that $f(\beta)$ is locally connected, and thus it contains a path connecting $f(x),f(y)$; see the proof of Lemma \ref{unif:Level sets-Topological Lemma}. The argument is similar to the proof of Lemma \ref{unif:Level sets-curves}. If $f(\beta)$ is not locally connected, there exists an open set $U$ and $\varepsilon>0$ such that $U\cap f(\beta)$ contains infinitely many components $C_n$, $n\in \N$, of diameter at least $\varepsilon$. By passing to a subsequence, we may assume that the continua $\br C_n$ subconverge in the Hausdorff sense to a continuum $C\subset f(\beta)$ with $\diam(C)\geq \varepsilon$. We claim that $C\subset \mathcal R$, hence $\mathcal H^1(C)\leq \mathcal H^1(f(\beta\cap S))=0$, which is a contradiction.

If $C\cap S_{i_0}\neq \emptyset$ for some $i_0\in \N$, then by shrinking $C$ and $C_n$ we may assume that $C,C_n\subset \subset S_{i_0}$ for all $n\in \N$. In particular, $f(\beta)$ has infinite length inside $S_{i_0}$. However, by the construction of the paths $f(\beta_n)$, the intersection $f(\beta)\cap S_{i_0}$ is either empty or it is one line segment. This is a contradiction.  
\end{proof}

Using the lemma we show:

\begin{prop}\label{unif:Regularity-F modulus}\index{quasiconformal!discrete}
Let $\Gamma$ be a family of paths in $\C$ joining two continua $E,F \subset S$, but avoiding finitely many peripheral disks $\br Q_i$, $i\in I_0$, where $I_0$ is a finite (possibly empty) subset of $\N\cup \{0\}$. We have
\begin{align*}
\br \md (\Gamma)\leq  \md ( f(\Gamma))
\end{align*}
\end{prop}

Here, of course, the left hand side is weak carpet modulus with respect to the carpet $S$, and the right hand side is carpet modulus with respect to the carpet $\mathcal R$.

\begin{proof}
Consider a weight $\{\lambda'(S_i)\}_{i\in \N\cup \{0\}}$ that is admissible for $\md( f(\Gamma))$, i.e., 
\begin{align}\label{unif:Regularity-F modulus-lambda' admissible}
\sum_{i:S_i\cap \gamma' \neq \emptyset}\lambda'(S_i)\geq 1
\end{align} 
for every path $\gamma'=f(\gamma)\in f(\Gamma)$ with $\mathcal H^1(\gamma'\cap \mathcal R)=0$. Define $\lambda(Q_i)\coloneqq  \lambda'(S_i)$, $i\in \N\cup \{0\}$, and let $\Gamma_0= \{\gamma \in \Gamma: \sum_{i:\br Q_i \cap \gamma\neq \emptyset}\lambda(Q_i) <1\}$. We wish to show that $\md_2(\Gamma_0)=0$. If this is the case, then $\{\lambda(Q_i)\}_{i\in \N}$ is clearly admissible for $\br\md(\Gamma)$, thus
\begin{align*}
\br \md(\Gamma ) \leq \sum_{i\in \N\cup \{0\}} \lambda(Q_i)^2=\sum_{i\in \N\cup \{0\}}\lambda'(S_i)^2.
\end{align*}
Infimizing over $\lambda$ we obtain the desired $\br \md(\Gamma)\leq \md( f(\Gamma)) $.

Now we show our claim. If $\gamma \in \Gamma_0 \setminus \Gamma_1$, where $\Gamma_1$ is as in Lemma \ref{unif:Regularity-Path hausdorff measure zero}, then there exists a path $\tilde \gamma$ given by the lemma that connects points $x\in E$ and $y\in F$. The path $\tilde \gamma$ intersects fewer peripheral disks than $\gamma$, so 
$$\sum_{i:S_i\cap  f(\tilde \gamma)\neq \emptyset} \lambda'(S_i)=\sum_{i:Q_i\cap \tilde \gamma\neq \emptyset} \lambda(Q_i) \leq \sum_{i:\br Q_i \cap \gamma\neq \emptyset}\lambda(Q_i)<1 .$$
Also, $\mathcal H^1(f(\tilde \gamma \cap S))=\mathcal H^1(f(\tilde \gamma)\cap \mathcal R)=0$, $\tilde \gamma$ still joins the continua $E,F$, and avoids $\bigcup_{i\in I_0}Q_{i}$, so $\tilde \gamma\in \Gamma$. This contradicts \eqref{unif:Regularity-F modulus-lambda' admissible}, hence $\Gamma_0\setminus \Gamma_1=\emptyset$, and $\Gamma_0\subset \Gamma_1$. We thus have $\md_2(\Gamma_0) \leq \md_2(\Gamma_1)=0$.
\end{proof}

We also wish to show an analog of Proposition \ref{unif:Regularity-F modulus} for $g\coloneqq  f^{-1}$:
\begin{prop}\label{unif:Regularity-F-1 modulus}\index{quasiconformal!discrete}
Let $\Gamma$ be the family  of paths in $\C$ joining two continua $E,F \subset \mathcal R$, but avoiding finitely many squares $\br S_i$, $i\in I_0$, where $I_0$ is a finite (possibly empty) subset of $\N\cup \{0\}$. We have
\begin{align*}
\br \md (\Gamma)\leq  \md ( g(\Gamma)).
\end{align*}
\end{prop}

The proof of Proposition \ref{unif:Regularity-F modulus} applies without change, if we can establish an analog of Lemma \ref{unif:Regularity-Path hausdorff measure zero} for $g$. In the proof of the latter we only used the fact that $f$ is a homeomorphism, together with \eqref{unif:Regularity-upper gradient for F}. Hence, in order to obtain Proposition \ref{unif:Regularity-F-1 modulus}, it suffices to show that
\begin{align}\label{unif:Regularity-G upper gradient}
|g(x)-g(y)|\leq \sum_{i: \br S_i\cap \gamma\neq \emptyset} \diam(Q_i)
\end{align}
for all paths $\gamma\subset(0,1)\times(0,D)$ outside an exceptional family of $2$-modulus zero, and points $x,y\in \br \gamma\cap \mathcal R$. We will show this in Lemma \ref{unif:Regularity-G a.e.gamma} after we have established two auxiliary results.

Recall the notation $\beta_s \subset \psi^{-1}(s)\cap ((0,1)\times (0,D))$, $s\in (0,\delta)$, for the perturbations of a given path $\beta$ connecting two non-trivial continua $E,F \subset [0,1]\times [0,D]$, as in Proposition \ref{unif:Paths joining continua}. Here $\psi(x)=\dist(x,\beta)$.
\begin{lemma}\label{unif:Regularity-G paths beta_s}
Fix a path $\beta$ as above. For a.e.\ $s\in (0,\delta)$ we have
\begin{align*}
\mathcal H^1(g(\psi^{-1}(s)\cap \mathcal R))=0.
\end{align*}
\end{lemma}

We postpone the proof for the moment. Using this we have a preliminary version of \eqref{unif:Regularity-G upper gradient}:

\begin{lemma}\label{unif:Regularity-G paths beta_s upper gradient}
For a.e.\ $s\in (0,\delta)$ and $x,y\in \br \beta_s\cap \mathcal R$ we have
\begin{align*}
|g(x)-g(y)|\leq \sum_{i: \br S_i\cap \beta_s\neq \emptyset} \diam(Q_i).
\end{align*}
\end{lemma}
\begin{proof}
Consider $s\in (0,\delta)$ such that the conclusion of Lemma \ref{unif:Regularity-G paths beta_s} holds. Without loss of generality, we assume that $\sum_{i: \br S_i\cap \beta_s\neq \emptyset} \diam(Q_i)<\infty$. Let $\varepsilon>0$ and consider a cover of $g(\beta_s\cap \mathcal R) \subset g(\psi^{-1}(s)\cap \mathcal R)$ by finitely many small balls $B_j$ of radius $r_j$ such that $\sum_j r_j<\varepsilon$. The union of $\bigcup_j B_j$ and $\bigcup_{i:\br S_i\cap \beta_s\neq \emptyset}Q_i$ covers $g(\beta_s)$ so we can find a finite subcover. Traveling along this subcover from $g(x)$ to $g(y)$ (as a finite chain) we obtain
\begin{align*}
|g(x)-g(y)|\leq \sum_{j}2r_j +\sum_{i: \br S_i\cap \beta_s\neq \emptyset} \diam(Q_i) \leq 2\varepsilon+\sum_{i: \br S_i\cap \beta_s\neq \emptyset} \diam(Q_i).
\end{align*}
The conclusion follows.
\end{proof}

Finally, we have:

\begin{lemma}\label{unif:Regularity-G a.e.gamma}
There exists an exceptional family of paths $\Gamma_0$ with $\md_2(\Gamma_0)=0$ such that for every path $\gamma\subset (0,1)\times (0,D)$ with $\gamma\notin \Gamma_0$ and every open subpath $\beta$ of $\gamma$ we have
\begin{align}\label{unif:Regularity-G a.e. gamma claim}
|g(x)-g(y)|\leq \sum_{i: \br S_i\cap \beta\neq \emptyset} \diam(Q_i)
\end{align}
for all points $x,y\in \br \beta\cap \mathcal R$.
\end{lemma}

\begin{proof}
Note that Lemma \ref{unif:Lebesgue differentiation Lemma} is also valid for the carpet $\mathcal R$, since it satisfies the basic assumptions \eqref{unif:Quasi-balls} and \eqref{unif:Fat sets} and has area zero (see Corollary \ref{unif:Definition-Range-Extension}). We fix a non-exceptional path $\gamma\subset (0,1)\times (0,D)$ so that the conclusions of Lemma \ref{unif:Lebesgue differentiation Lemma} hold for all subpaths $\beta$ of $\gamma$ having endpoints in $\mathcal R^\circ$, and $\mathcal H^1(\gamma\cap \mathcal R)=0$.

By continuity and the usual reduction to ``accessible" points, it suffices to prove the main claim \eqref{unif:Regularity-G a.e. gamma claim} for a subpath $\beta$ of $\gamma$ and its endpoints $x,y$, with the further assumption that they lie in $\mathcal R^\circ$ (i.e.\ they do not lie on any square or on the boundary rectangle); see the comments in the beginning of the proof of Lemma \ref{unif:Conjugate-Defect upper gradient}. Consider two non-trivial disjoint continua $E,F \subset \mathcal R^\circ$ such that $x\in E$ and $y\in F$; the existence of such continua is implied by Lemma \ref{unif:Paths in S^o}. Let $\psi(x)=\dist(x,\beta)$ and consider the perturbations $\beta_s \subset \psi^{-1}(s)$ of $\beta$ that connect points $x_s \in E$ to $y_s\in F$, $s\in (0,\delta)$. If $\delta$ is sufficiently small, then $\psi^{-1}(s) \subset (0,1)\times(0,D)$ for all $s\in (0,\delta)$. 

For fixed $\varepsilon>0$ we choose an even smaller $\delta>0$ so that $x_s$ and $y_s$ are close to $x$ and $y$, respectively, and 
\begin{align*}
|g(x)-g(y)|&\leq |g(x_s)-g(y_s)|+\varepsilon \\
&\leq\sum_{i: \br S_i\cap \psi^{-1}(s)\neq \emptyset} \diam(Q_i) +\varepsilon
\end{align*}
for a.e.\ $s\in (0,\delta)$, by the continuity of $g$ and Lemma \ref{unif:Regularity-G paths beta_s upper gradient}. The right hand side is a measurable function of $s$. Thus, averaging over $s\in (0,\delta)$ and using Fubini's theorem we obtain
\begin{align*}
|g(x)-g(y)|&\leq \frac{1}{\delta}\sum_{i\in \N} \diam(Q_i) \int_{0}^\delta \x_{\br S_i \cap\psi^{-1}(s)}\, ds +\varepsilon\\
&\leq \frac{1}{\delta}\sum_{i\in \N} \diam(Q_i) \mathcal H^1(\{ s\in (0,\delta):\br S_i\cap N_s(\beta)\neq \emptyset\})+\varepsilon\\
&\leq \sum_{i: \br{S}_i\cap \beta\neq \emptyset }\diam(Q_i) + \frac{1}{\delta} \sum_{\substack{i:\br S_i\cap \beta=\emptyset\\ S_i\cap N_{\delta}(\beta)\neq \emptyset }} \diam(Q_i)\cdot \min\{\diam(S_i),\delta\} +\varepsilon.
\end{align*}
It suffices to show that the second sum converges to $0$ as $\delta\to 0$. Note that we can bound it by
\begin{align*}
\frac{1}{\delta}\sum_{i: S_i\subset N_{2\delta}(\beta)} \diam(Q_i) \diam(S_i) + \sum_{i\in D_\delta(\beta)}\diam(Q_i),
\end{align*}
where $D_\delta(\beta)$ is the family of indices $i\in \N$ such that $\diam(S_i)\geq \delta$, $S_i\cap N_\delta(\beta)\neq \emptyset$, and $\br S_i\cap \beta=\emptyset$. We are exactly in the setting of Lemma \ref{unif:Lebesgue differentiation Lemma}, so the conclusion follows.
\end{proof}

It remains to prove the very first Lemma \ref{unif:Regularity-G paths beta_s}. We assume that we have a path $\beta$ connecting two continua $E,F\subset [0,1]\times [0,D]$, and its perturbations $\beta_s\subset \psi^{-1}(s)\cap ((0,1)\times(0,D))$, $s\in (0,\delta)$. We split the proof in two parts:

\begin{lemma}\label{unif:Regularity-psi finitely many points}
For a.e.\ $s\in (0,\delta)$ the intersection $\psi^{-1}(s)\cap \partial S_i$ contains finitely many points, for all $i\in \N\cup \{0\}$.
\end{lemma}

\begin{lemma}\label{unif:Regularity-G-psi}
For a.e.\ $s\in (0,\delta)$ we have 
\begin{align*}
\mathcal H^{1}( g(\psi^{-1}(s)\cap \mathcal R^\circ)) = \mathcal H^1( g( \psi^{-1}(s)) \cap S^\circ) =0.
\end{align*}
\end{lemma}

\begin{remark}
One easily recognizes the connection to the proof of Lemma \ref{unif:Level sets-Hausdorff measure}, where we first proved that for a.e.\ $t$ the level set $u^{-1}(t)$ intersects the peripheral circles $\partial Q_i$ in finitely many points (in fact in at most two points), and then treated separately the set $u^{-1}(t)\cap S^\circ$. In fact, Lemma \ref{unif:Level sets-Hausdorff measure} is a particular case of what we are about to show.
\end{remark}

\begin{proof}[Proof of Lemma \ref{unif:Regularity-psi finitely many points}]
The proof is based on the fact that $A\coloneqq \partial S_i$ has finite Hausdorff $1$-measure for all $i\in \N\cup\{0\}$. Let $J\subset (0,\delta)$ be the set of $s$ for which $\psi^{-1}(s)\cap A$ contains at least $N$ points, where $N\in \N$ is fixed. We will show that for the outer $1$-measure $m_1^*(J)$ we have
\begin{align}\label{unif:Regularity-psi finitely-m*}
m_1^*(J)\leq \frac{1}{N}\mathcal H^1(A).
\end{align}
Hence, the set of $s\in (0,\delta)$ for which $\psi^{-1}(s)\cap A$ contains infinitely many points has outer measure that is also bounded by $\frac{1}{N}\mathcal H^1(A)$, for all $N\in \N$. If we let $N\to\infty$ the conclusion will follow.

For $n\in \N$ define $J_n \subset J$ to be the set of $s\in (0,\delta)$ for which the intersection $\psi^{-1}(s)\cap A$ contains $N$ points whose mutual distance is at least $1/n$; recall that $N$ is fixed. It is easy to see that $J= \bigcup_{n\in \N}J_n$ and $J_{n} \subset J_{n+1}$, thus $m_1^*(J)=\lim_{n\to\infty} m_1^*(J_n)$; see \cite[Prop.\ 1.5.12, p.\ 23]{Bogachev:measure}. It suffices to show \eqref{unif:Regularity-psi finitely-m*} for $J_n$ in the place of $J$.

Recall that $\psi(x)=\dist(x,\beta)$ for the given path $\beta$. For fixed $n\in \N$ we cover $A$ by finitely many open sets $\{U_i\}_i$ with $\diam(U_i)<1/n$. If $s\in J_n$ then $\psi^{-1}(s)$ meets at least $N$ distinct sets $U_i$. For all $t$ near $s$ with $t\leq s$, the set $\psi^{-1}(t)$ also intersects at least $N$ distinct sets $U_i$. Indeed, note that for each point $x\in \psi^{-1}(s)$ and for every small ball $B(x,r)$ there exists $t_0<s$ such that $\psi^{-1}(t)\cap B(x,r)\neq \emptyset$ for all $t_0\leq t\leq s$. To see the latter, let $[x,y]$ denote the line segment from $x$ to its closest point $y\in \beta$, so $|x-y|=s$. Arbitrarily close to $x$ we can find points $z\in [x,y]$, with $z\in \psi^{-1}(t_0)$ for some $t_0<s$. Then by the intermediate value theorem applied to $\psi$ on the segment $[z,x]$ the claim follows.

Define $W_s$ to be a non-trivial closed interval of the form $[t_0,s] \subset (0,\delta)$, such that $\psi^{-1}(t)$ intersects at least $N$ distinct sets $U_i$ for all $t\in W_s$. Also, let $W= \bigcup_{s\in J_n} W_s \supset J_n$, and $V_i\coloneqq \psi(U_i)$. Then
\begin{align*}
\sum_{i} \x_{V_i}(t)\geq N
\end{align*}
for all $t\in W$. The set $W$ is measurable as it is a  union of non-trivial closed intervals; this is easy to see in one dimension, but a similar statement is true in $\R^n$ \cite[Theorem 1.1]{BalcerzakKharazishvili:measurable}. Hence, we have
\begin{align*}
Nm_1^*(J_n)&\leq N m_1(W) = N\int \x_W(t) \, dt \leq \int \sum_i \x_{V_i}(t) \, dt\\
&= \sum_i m_1(V_i) \leq \sum_i \diam(\psi(U_i)) \leq \sum_i \diam(U_i)
\end{align*}
since $\psi$ is $1$-Lipschitz. The cover $\{U_i\}$ of $A$ was arbitrary, so taking the infimum over all covers completes the proof.
\end{proof}

An ingredient for the proof of Lemma \ref{unif:Regularity-G-psi} is the following {monotonicity} property of $f$; cf.\ Lemma \ref{unif:Maximum principle circular arcs}.

\begin{lemma}\label{unif:Regularity:monotonicity}\index{monotone function}
Let $x\in S$, $r>0$, and $c>1$ such that $B(x,r)\subset B(x,cr)\subset \Omega$. We have
\begin{align}\label{unif:Regularity-G-psi-Upper gradient}
\diam(f(B(x,r)\cap S))\leq \diam (f(B(x,sr)\cap S))\leq \sum_{i:\br Q_i \cap \partial B(x,sr)\neq \emptyset}\diam(S_i) 
\end{align}
for a.e.\ $s\in [1,c]$.
\end{lemma}
\begin{proof}
The map $f$ is a homeomorphism so it has the \textit{monotonicity} property that for any set $U\subset \subset \C$ we have 
$$\sup_{z,w\in U} |f(z)-f(w)|= \sup_{z,w\in \partial U}|f(z)-f(w)|.$$
Thus, for each $s\in (1,c)$ we have
\begin{align*}
\diam(f(B(x,r)\cap S))&\leq \diam (f(B(x,sr)\cap S))\\
&\leq \diam(f(B(x,sr)))= |f(z)-f(w)|
\end{align*}
for some points $z,w\in \partial B(x,sr)$. Suppose first that $z,w\in S$. Since for a.e.\ $s\in [1,c]$ the path $\partial B(x,sr)$ is non-exceptional for $2$-modulus (this follows, for example, from a modification of Lemma \ref{unif:Paths joining continua}), by the upper gradient inequality \eqref{unif:Regularity-upper gradient for F} we obtain
$$|f(z)-f(w)|\leq \sum_{i:\br Q_i \cap \partial B(x,sr)\neq \emptyset}\diam(S_i). $$
If $z$ lies in a peripheral disk $\br Q_{i_z}$ with $\br Q_{i_z}\cap \partial B(x,sr)\neq \emptyset$, then we let $z'$ be the last exit point of the path $\partial B(x,sr)$ from $\br Q_{i_z}$ as it travels from $z$  to $w$. Similarly, we consider a point $w'$, in case $w$ lies in a peripheral disk $\br Q_{i_w}$. Observe now that $|f(z)-f(z')|\leq \diam(S_{i_z})$, $|f(w)-f(w')|\leq \diam(S_{i_w})$, and the open subpath of $\partial B(x,sr)$ from $z'$ to $w'$ does not intersect $\br Q_{i_z},\br Q_{i_w}$. The upper gradient inequality applied to this subpath yields the result.
\end{proof}

\begin{proof}[Proof of Lemma \ref{unif:Regularity-G-psi}]
The proof is very similar to the proof of Lemma \ref{unif:Level sets-Hausdorff measure} so we omit most of the details.

For a fixed $\varepsilon>0$ we consider the set $E_\varepsilon=\{i\in \N: \diam(Q_i)>\varepsilon\}$. We cover $\Omega\setminus \bigcup_{i\in E_\varepsilon} \br Q_i$ by balls $B_j$ of radius $r_j<\varepsilon$ such that $2B_j \subset \Omega \setminus\bigcup_{i\in E_\varepsilon} \br Q_i $, and such that $\frac{1}{5}B_j$ are disjoint.

Let $J$ be the family of indices $j$ such that for each $s\in [1,2]$ we have 
$$\diam( f(sB_j\cap S))\geq k\diam (S_i)$$ 
for all peripheral disks $Q_i$ with $\diam(Q_i)>8r_j$ that intersect $\partial (sB_j)$. The constant $k\geq 1$ can be chosen exactly as in the proof of Lemma \ref{unif:Level sets-Hausdorff measure}. Using \eqref{unif:Regularity-G-psi-Upper gradient} for $B_j$ and integrating over $s\in [1,2]$ one obtains
\begin{align}\label{unif:Regularity-G-psi-good}
r_j\diam (f(B_j\cap S))\leq C \sum_{i:Q_i\subset 11B_j} \diam(S_i)\diam(Q_i)
\end{align} 
for all $j\in J$, where $C>0$ is a uniform constant depending only on the data of the carpet $S$ (cf.\ \eqref{unif:Level sets-Hausdorff measure-Good bound}). For each $j\in J$ consider the smallest interval $I_j$ that contains $\psi(f(B_j\cap S))$ and define $g_\varepsilon(t)=\sum_{j\in J} 2r_j\x_{I_j}(t)$, $t\in (0,\delta)$. 

For $j\notin J$ there exists $s=s_j\in [1,2]$ and there exists a peripheral disk $Q_i$ that intersects $\partial(sB_j)$ with $\diam(Q_i)>8r_j$, but $\diam(f(sB_j\cap S))<k\diam(S_i)$. Let $\{Q_i\}_{i\in I}$ denote the family of such peripheral disks. Some $Q_i$, $i\in I$, might intersect multiple balls $B_j$, $j\notin J$. We define
\begin{align*}
\widetilde Q_i &\coloneqq  \br{Q}_i\cup \bigcup \{s_jB_j: Q_i\cap \partial (s_jB_j)\neq \emptyset, \, \diam(Q_i)>8r_j,\\ &\qquad\qquad\qquad \textrm{and}\, \diam(f(s_jB_j\cap S)) <k\osc_{Q_i}(u)\},
\end{align*}
and note that
\begin{align}\label{unif:Regularity-G-psi-bad}
\diam(f(\widetilde Q_i\cap S))\leq C \diam(S_i)
\end{align}
for all $i\in I$, where $C>0$ depends only on the data. Furthermore, $\diam(\widetilde Q_i) < 2\diam(Q_i)$ since $\diam(Q_i)>8r_j$ whenever $s_jB_j\subset \widetilde Q_i$. Now, let $I_i$ be the smallest interval containing $\psi(f(\widetilde Q_i\cap S))$ and define $b_\varepsilon(t)=\sum_{i\in I} 2\diam(Q_i)\x_{I_i}(t)$, $t\in (0,\delta)$.

Observe that for each $t\in (0,\delta)$ the set $g(\psi^{-1}(s)\cap \mathcal R^\circ)$ is covered by the balls $B_j$, $j\in J$, and the sets $\widetilde Q_i$, $i\in I$. Since $r_j<\varepsilon$ for $j\in J$ and $\diam(Q_i)<\varepsilon$ for $i\in I$, we have
\begin{align*}
\mathcal H^1_{\varepsilon}( g(\psi^{-1}(s) \cap \mathcal R^\circ)) \leq g_\varepsilon(s)+b_\varepsilon(s).
\end{align*}
It suffices to show that $g_\varepsilon(s) \to 0 $ and $b_\varepsilon(s) \to 0$ for a.e.\ $s\in (0,\delta)$, along a sequence of $\varepsilon \to 0$.

The function $\psi$ is $1$-Lipschitz, so we have
\begin{align*}
\diam(I_j)&=\diam(\psi(f(B_j\cap S))) \leq \diam (f(B_j\cap S)) \quad \quad \textrm{and}\\
\diam(I_i)&= \diam(\psi(f(\widetilde Q_i\cap S)))\leq \diam(f(\widetilde Q_i\cap S))
\end{align*}
for all $j\in J$ and $i\in I$. These can be estimated above by \eqref{unif:Regularity-G-psi-good} and \eqref{unif:Regularity-G-psi-bad}, respectively. The proof continues exactly as in Lemma \ref{unif:Level sets-Hausdorff measure} by estimating $\int_0^\delta b_\varepsilon (s) \, ds$ and $\int_0^\delta g_\varepsilon(s) \, ds$ and showing that they converge to $0$ as $\varepsilon \to 0$. 
\end{proof}

\section{Carpet modulus estimates}\label{unif:Section Modulus Estimates}\index{Loewner estimates}\index{modulus!Loewner estimates}

In this section we state some modulus estimates, which were proved in \cite[Section 8]{Bonk:uniformization}. The statements there involve some bounds for the \textit{transboundary modulus}, which is a notion of modulus ``between" the classical conformal modulus and the carpet modulus that we are employing in this chapter. Thus, the proofs can be applied with minor changes and we restrict ourselves to mentioning the results. 

We first recall some definitions from the introduction. Consider a Sierpi\'nski carpet $S\subset \br \Omega$ with its peripheral circles $\{\partial Q_i\}_{i\in \N\cup \{0\}}$, where $\partial Q_0=\partial \Omega$. We say that the peripheral circles of the carpet $S$ are  $K_2$-\textit{quasicircles} for some constant $K_2>0$ if $\partial Q_i$ is a $K_2$-quasicircle for all $i\in \N\cup \{0\}$. This is to say  for any two points $x,y\in \partial Q_i$ there exists an arc $\gamma \subset \partial Q_i$ connecting $x$ and $y$ with $\diam(\gamma)\leq K_3 |x-y|$. Furthermore, the peripheral circles of the carpet $S$ are $K_3$-\textit{relatively separated}\index{uniform relative separation} for a constant $K_3>0$ if 
\begin{align*}
\Delta(\partial Q_i, \partial Q_j)\coloneqq \frac{\dist(\partial Q_i, \partial Q_j)}{\min \{ \diam(\partial Q_i),\diam( \partial Q_j)\}} \geq K_3
\end{align*} 
for all $i,j\in \N\cup \{0\}$ with $i\neq j$. Recall  that if the peripheral circles of $S$ are uniform quasicircles then the inner peripheral disks $Q_i$, $i\in \N$, are uniformly fat and uniform quasiballs.

In the following two propositions the \textit{common assumption} is that we have a Sierpi\'nski carpet $S$ contained in a Jordan region $\br \Omega$ such that $\partial Q_0=\partial \Omega \subset S$ is the outer peripheral circle and $\{Q_i\}_{i\in \N}$ are the inner peripheral disks.

\begin{prop}[Prop.\ 8.1,  \cite{Bonk:uniformization}]\label{unif:Modulus-Loewner lower}
Assume that the peripheral circles $\{\partial Q_i\}_{i\in \N \cup \{0\}}$ of $S$ are $K_2$-quasicircles and they are $K_3$-relatively separated, and fix an integer $N\in \N$. Then there exists a non-increasing function $\phi\colon (0,\infty)\to (0,\infty)$ that can be chosen only depending on $K_2,K_3$, and $N$ with the following property: if $E$ and $F$ are arbitrary disjoint continua in $S$, and $I_0\subset \N\cup \{0\}$ is a finite index set with $\#I_0 =N$, then the family of curves $\Gamma$ in $\C$ joining the continua $E$ and $F$, but avoiding the finitely many peripheral disks $\br Q_i$, $i\in I_0$, has carpet modulus satisfying
\begin{align*}
\md (\Gamma) \geq \phi(\Delta(E,F)).
\end{align*} 
The same conclusion is true if we use instead the weak carpet modulus $\br \md(\Gamma)$.
\end{prop}

\begin{prop}[Prop.\ 8.4, \cite{Bonk:uniformization}]\label{unif:Modulus-Loewner upper}
Assume that the peripheral circles $\{\partial Q_i\}_{i\in \N \cup \{0\}}$ of $S$  are $K_2$-quasicircles and they are $K_3$-relatively separated. Then there exists a non-increasing function $\phi\colon  (0,\infty)\to (0,\infty)$ that can be chosen only depending on $K_2$ and $K_3$ with the following property: if $E$ and $F$ are disjoint continua in $S$, then the family of curves $\Gamma$ in $\C$ joining the continua $E$ and $F$ has carpet modulus satisfying 
\begin{align*}
\md(\Gamma) \leq \phi(\Delta(E,F)).
\end{align*}
The same conclusion is true if we use instead the weak carpet modulus $\br \md(\Gamma)$.
\end{prop}

A  square Sierpi\'nski carpet\index{Sierpi\'nski carpet!square} is by definition a Sierpi\'nski carpet $\mathcal R$ whose inner peripheral disks $\{ S_i\}_{i\in \N}$ are squares, and the outer peripheral circle $\partial S_0$ is a rectangle, where $S_0$ is the unbounded component of $\C\setminus \mathcal R$. 

\begin{prop}[Prop.\ 8.7, \cite{Bonk:uniformization}]\label{unif:Modulus-Remove squares}
Let $\mathcal R \subset [0,1]\times [0,D]$ be a square Sierpi\'nski carpet with inner peripheral squares $\{S_i\}_{i\in \N}$ and outer peripheral rectangle $\partial S_0\coloneqq  \partial ([0,1]\times [0,D])$. There exists a number $N=N(D)\in \N$ and a non-increasing function $\psi\colon (0,\infty)\to (0,\infty)$ with 
\begin{align*}
\lim_{t\to\infty}\psi(t)=0
\end{align*}
that can be chosen only depending on $D$ and satisfies the following: if $E$ and $F$ are arbitrary continua in $\mathcal R$ with $\Delta(E,F)\geq 12$, then there exists a set $I_0 \subset \N\cup \{0\}$ with $\#I_0\leq N$ such that the family of curves $\Gamma$ in $\C$ joining the continua $E$ and $F$, but avoiding the finitely many peripheral disks $\br S_i$, $i\in I_0$, has carpet modulus satisfying
\begin{align*}
\md(\Gamma)\leq \psi(\Delta(E,F)).
\end{align*}
Moreover, if $D\in [D_1,D_2] \subset (0,\infty)$, the number $N$ and the function $\psi$ can be chosen to depend only on $D_1,D_2$.
\end{prop}

As a last remark, (weak) carpet modulus of a path family $\Gamma$ is always considered with respect to a given carpet, although this is not explicitly manifested in the notation $\md(\Gamma)$. It will be clear from the context what the reference carpet each time is, when we use these estimates in the next section.

\section{Quasisymmetric uniformization}\label{unif:Section Quasisymmetric}\index{quasisymmetry}
In this section we prove Theorem \ref{unif:Main theorem-quasisymmetric}, which we restate for the convenience of the reader. 

\begin{theorem}\label{unif:Quasisymmetric-Main theorem}
Let $S$ be a Sierpi\'nski carpet of  area zero with peripheral circles $\{\partial Q_i\}_{i\in \N\cup \{0\}}$ that are $K_2$-quasicircles and $K_3$-relatively separated. Then there exists an $\eta$-quasi\-symmetric map $f$ from $S$ onto a square Sierpi\'nski carpet $\mathcal R$ such that the distortion function $\eta$ depends only on $K_2$ and $K_3$.
\end{theorem}

Before proceeding to the proof we include some lemmas.

\begin{lemma}\label{unif:Quasisymmetric-Three point lemma}
Let $a,b>0$, and $(X,d_X)$ and $(Y,d_Y)$ be metric spaces. Suppose that $x_1,x_2,x_3\in X$ and $y_1,y_2,y_3\in Y$ are points such that
\begin{align*}
d_X(x_i,x_j)\geq a \quad \textrm{and} \quad d_Y(y_i,y_j)\geq b \quad \textrm{for}\quad  i,j=1,2,3,\enskip i\neq j.
\end{align*}
Then for all $x\in X$ and $y\in Y$ there exists an index $l\in \{1,2,3\}$ such that $d_X(x,x_l)\geq a/2$ and $d_Y(y,y_l)\geq b/2$.
\end{lemma}
\begin{proof}
At most one of the points $x_i$ can lie in the ball $B(x,a/2)$, so there are at least two points, say $x_1,x_2$ that have distance at least $a/2$ to $x$. At most one of the points $y_1,y_2$ can lie in $B(y,b/2)$, so one of them, say $y_1$, has to lie outside the ball $B(y,b/2)$. Then the desired statement holds for $l=1$.
\end{proof}

\begin{lemma}\label{unif:Quasisymmetric-LLC}
Let $\mathcal R\subset [0,1]\times [0,D]$ be a square Sierpi\'nski carpet such that $\partial ([0,1]\times [0,D])\subset \mathcal R$ is the outer peripheral circle. Then there exists a constant $C(D)>0$ such that the following two conditions are satisfied:
\begin{enumerate}[\upshape(1)]
\item For all $x,y\in \mathcal R$ there exists a path $\gamma\subset \mathcal R$ connecting $x$ and $y$ with $\diam(\gamma)\leq 2|x-y|$.

\item If $a\in \mathcal R$, $0<r\leq C(D)$, and $x,y\in \mathcal R \setminus B(a,r)$, then there exists a continuum $E$ connecting $x$ and $y$ with $E\subset \mathcal R \setminus B(a,r/2)$.  
\end{enumerate}
In fact, one can take $C(D)= \min \{1,D\}$.
\end{lemma}

\begin{proof}
Let $S_i$, $i\in \N$, denote the inner (open) peripheral squares of $\mathcal R$, and $\partial S_0\coloneqq  \partial ([0,1]\times [0,D])$. For the first statement note that for any two points $x,y$ lying on a square $\partial S_i$ there exists an arc $\gamma \subset \partial S_i$ connecting $x,y$ with length at most $2|x-y|$. Now, if $x,y\in \mathcal R$ are arbitrary, we connect them with a line segment $[x,y]$, and then replace each of the segments $[x_i,y_i]\coloneqq [x,y]\cap \br S_i$ with an arc $\gamma_i\subset \partial S_i$ that connects $x_i,y_i$ and has length at most $2|x_i-y_i|$. The resulting path $\gamma\subset \mathcal R$ connects $x,y$ and has length at most $2|x-y|$. 

For the second claim, let $C(D)=  \min \{1,D\}$, so a ball $B(a,r/2)$ with $r\leq C(D)$ cannot intersect two opposite sides of the rectangle $\partial S_0$. If $x \in \mathcal R \setminus B(a,r)$, then we connect $x$ to $\partial S_0$ with a line segment $\gamma_x$ parallel to one of the coordinate axes that does not intersect $B(a,r)$. We replace each of the arcs $\gamma_x\cap S_i$, $i\in \N$, with an arc in $\partial S_i$ that has the same endpoints and does not intersect $B(a,r)$. To see the existence of such an arc, note that if both of the arcs of $\partial S_i$ with the same endpoints as $\gamma_x\cap S_i$ intersected $B(a,r)$, then $B(a,r)$ would also intersect $\gamma_x\cap S_i$, by convexity. This is a contradiction.

We still call the resulting path $\gamma_x$. We do the same for a point $y\in \mathcal R\setminus B(a,r)$ and obtain a path $\gamma_y$. Then one has to concatenate $\gamma_x$ and $\gamma_y$ with a path in $\partial S_0$ that does not intersect $B(a,r/2)$. If $\partial S_0\setminus B(a,r/2)$ has only one component then this can be clearly done. The other case is that $\partial S_0 \setminus B(a,r/2)$ has two components $E$ and $F$, and thus $B(a,r/2)$ intersects two neighboring sides of $\partial S_0$. The distance of $a$ to these two sides is at most $r/2$, so $B(a,r)$ contains one of the components $E,F$, say it contains $F$. This now implies that the endpoints $\gamma_x\cap \partial S_0,\gamma_y\cap \partial S_0$ have to lie on $E$ and can therefore be connected with a subarc of $E$.
\end{proof}

Now we proceed to the proof of the main result. The candidate for the quasisymmetric map $f\colon S\to \mathcal R$ is the map that we constructed in the previous sections. The principle that we will use is that a ``quasiconformal" map (in our case a map that preserves modulus) between a \textit{Loewner space} and a space that is \textit{linearly locally connected}\index{linear local connectedness} (this is essentially implied by Lemma \ref{unif:Quasisymmetric-LLC}) is quasisymmetric; see \cite[Chapter 11]{Heinonen:metric} for background. Certain complications arise since we do not know in advance that the peripheral squares of $\mathcal R$ are uniformly relatively separated, and we bypass this by employing Proposition \ref{unif:Modulus-Remove squares}. 

\begin{proof}[Proof of Theorem \ref{unif:Quasisymmetric-Main theorem}]
We apply the considerations from Section \ref{unif:Section the function u} to Section \ref{unif:Section Regularity}. So we obtain a homeomorphism $f$ from $S$ onto a square carpet $\mathcal R\subset [0,1]\times[0,D]$, with outer peripheral disk $S_0=\C\setminus [0,1]\times [0,D]$. This homeomorphism maps $\partial \Omega$ to $\partial S_0$ and has the regularity as in Section \ref{unif:Section Regularity}: it satisfies the conclusions of Propositions \ref{unif:Regularity-F modulus} and \ref{unif:Regularity-F-1 modulus}; see also the formulation of Theorem \ref{unif:Main theorem}.

We split the proof into two parts. The first part is to show that if we choose the sides $\Theta_1,\Theta_3\subset \partial \Omega$ suitably, then the height $D$ of the rectangle $[0,1]\times [0,D]$ that contains the square carpet $\mathcal R$ is bounded above and below (away from $0$), depending only on $K_2$ and $K_3$. Recall from Section \ref{unif:Section:Free boundary problem} that the choice of $\Theta_1$ and $\Theta_3$ specifies uniquely the function $u$, and thus it specifies the uniformizing function $f=(u,v)$ and the height $D$ of the rectangle that contains the carpet $\mathcal R=f(S)$. The second step is to prove that the map $f\colon S\to \mathcal R$  that satisfies the conclusions of Proposition \ref{unif:Regularity-F modulus} and Proposition \ref{unif:Regularity-F-1 modulus} is a quasisymmetry.

For the first step, note that $\partial Q_0=\partial \Omega$ is a $K_2$-quasicircle, so it is the quasisymmetric image of the unit circle. It follows that we can choose two disjoint arcs $\Theta_1 ,\Theta_3 \subset \partial \Omega$ with endpoints $a_i$, $i=1,\dots,4$, such that 
\begin{align}\label{unif:Quasisymmetric-Step 1-diameter}
\min_{\substack{i\neq j\\ i,j=1,\dots,4}}|a_i-a_j|\geq  C_0\diam(\partial \Omega)= C_0\diam(S),
\end{align}
where $C_0>0$ is a constant depending only on $K_2$. Using this, and again that $\partial Q_0$ is a quasicircle, one can see that 
$$\frac{1}{C'}\leq \Delta(\Theta_1,\Theta_3)\leq C' $$
for some constant $C'>0$ depending only on $K_2$. Hence, if $\Gamma$ denotes the family of paths in $\Omega$ that connect $\Theta_1$ and $\Theta_3$ (this family avoids $\br Q_0$), by Proposition \ref{unif:Modulus-Loewner lower} and Proposition \ref{unif:Modulus-Loewner upper} we have
\begin{align}\label{unif:Quasisymmetric-Step 1}
\md(\Gamma) \leq C'' \quad \textrm{and} \quad \br \md(\Gamma) \geq C'''
\end{align}
for constants $C'',C'''>0$ depending only on $K_2$ and $K_3$. The family $f(\Gamma)$ is the path family in $(0,1)\times(0,D)$ that connects $\{0\}\times [0,D]$ to $\{1\}\times [0,D]$. Combining  \eqref{unif:Quasisymmetric-Step 1} with Proposition \ref{unif:Regularity-F-1 modulus} and Proposition \ref{unif:Regularity-F modulus} we obtain
\begin{align*}
\br \md(f(\Gamma)) \leq C''  \quad \textrm{and} \quad \md(f(\Gamma)) \geq C''',
\end{align*}
where here (weak) carpet modulus is with respect to the carpet $\mathcal R$. Finally, both of the above moduli are equal to $D$, so the conclusion of the first step follows. To see this, consider the discrete weight $\lambda(S_i)\coloneqq \ell(S_i)$, $i\in \N $, where $\ell(S_i)$ is the sidelength of the square $S_i$, and $\lambda(S_0)\coloneqq 0$. It is immediate that $\{\lambda(S_i)\}_{i\in \N\cup \{0\}}$ is admissible for $f(\Gamma)$ with respect to both notions of modulus, since for any path $\gamma \in \Gamma$ with $\mathcal H^1(\gamma\cap \mathcal R)=0$ we have 
\begin{align*}
\sum_{i:Q_i\cap \gamma\neq \emptyset}\lambda(Q_i)\geq 1.
\end{align*} 
Thus, $\br \md(f(\Gamma))$ and $\md(f(\Gamma))$ are both bounded above by
\begin{align*}
 \sum_{i\in \N} \ell(S_i)^2= \mathcal H^2([0,1]\times [0,D]) =D.
\end{align*} 
Conversely, if $\lambda(S_i)$ is an arbitrary weight that is admissible for $\br \md(f(\Gamma))$ , then we may assume that $\lambda(S_0)=0$, since the path family $f(\Gamma)$ does not hit $\br S_0$. Moreover, for the paths $\gamma_t(r)=(r,t)$, $r\in (0,1)$, and for a.e.\ $t\in [0,1]$ we have
\begin{align*}
1\leq \sum_{i:\br S_i \cap \gamma_t\neq \emptyset}\lambda(S_i).
\end{align*}
This is because the paths $\gamma_t(r)$ are non-exceptional for a.e.\ $t\in [0,1]$; see Lemma \ref{unif:Paths joining continua} and its proof in Lemma \ref{harmonic:Paths joining continua}.

Integrating over $t\in[0,1]$ and applying Fubini's theorem yields
\begin{align*}
1&\leq  \sum_{i\in \N} \lambda(S_i) \int_0^1 \x_{\br S_i\cap \gamma_t}\, dt \leq \sum_{i\in \N} \lambda(S_i)\ell(S_i) \leq \left(\sum_{i\in \N }\lambda(S_i)^2\right)^{1/2} \left(\sum_{i\in \N }\ell(S_i)^2\right)^{1/2}\\
&=  \left(\sum_{i\in \N }\lambda(S_i)^2\right)^{1/2} \mathcal H^2([0,1]\times [0,D]) ^{1/2}.
\end{align*}
Hence $\sum_{i\in \N }\lambda(S_i)^2 \geq D$, which shows that $\br \md(f(\Gamma)) \geq D$. The same computation proves the claim for $\md(f(\Gamma))$. Summarizing, under our choice of $\Theta_1$ and $\Theta_3$, the height $D$ of the rectangle $(0,1)\times (0,D)$ is bounded above and below, depending only on $K_2$ and $K_3$.

Now we move to the second step of the main proof. We remark that among the constants that we introduced in the first step only the constant $C_0$ is used again, and all other constants here are new constants. For simplicity we rescale $S$ so that $\diam(S)=1$. This does not affect the constants $K_2,K_3$, or the quasisymmetry distortion function. We show that the map $f\colon S\to \mathcal R$ is a \textit{weak quasisymmetry}, i.e., there exists a constant $H>0$ depending only on $K_2,K_3$ such that for any three points $x,y,z\in S$ with $|x-y|\leq |x-z|$ we have that the images $x'=f(x)$, $y'=f(y)$, $z'=f(z)$ satisfy
\begin{align*}
|x'-y'|\leq H|x'-z'|.
\end{align*}
By a well-known criterion this implies that $f$ is an  $\eta$-quasisymmetry, where $\eta$ depends only on $K_2,K_3$; see \cite[Theorem 10.19]{Heinonen:metric}.

We argue by contradiction, assuming that there exist points $x,y,z\in S$ with $|x-y|\leq |x-z|$, but $|x'-y'|>H|x'-z'|$ for some  large $H>0$. Then the points $x,y,z$ are distinct. By Lemma \ref{unif:Quasisymmetric-LLC}(1), there exists a continuum $E' \subset \mathcal R$ with $\diam(E')\leq 2|x'-z'|$, connecting $x',z'$. 

Recall that the mutual distance of the endpoints $a_i$, $i=1,\dots,4$, of $\Theta_1$ and $\Theta_3$ is at least 
\begin{align*}
\delta\coloneqq \min_{\substack{i\neq j\\ i,j=1,\dots,4}} |a_i-a_j|  \geq C_0
\end{align*}
by \eqref{unif:Quasisymmetric-Step 1-diameter}. Their images $a_i'$, $i=1,\dots,4$, are the vertices of the rectangle $\partial S_0= \partial ([0,1]\times [0,D])$. So their mutual distance is bounded below by $\delta' = \min \{1,D\}$. By Lemma \ref{unif:Quasisymmetric-Three point lemma}, there exists an index $i=1,\dots,4$ such that for $u=a_i$ and $u'=a_i'$ we have
\begin{align*}
|u-y|\geq \delta/2 \quad \textrm{and} \quad |u'-x'|\geq \delta'/2.
\end{align*}
Since $|x'-y'|\leq 2\max\{1,D\}$, it follows that
\begin{align*}
|u'-x'|\geq \frac{\delta'}{2}\geq \frac{1}{4} \frac{\min\{1,D\}}{\max\{1,D\}}|x'-y'|=  \frac{1}{4} \min \{D,1/D\}|x'-y'| \eqqcolon C_1(D) |x'-y'|.
\end{align*}
Hence, $u'\notin B(x',r)$, where $r\coloneqq C_1(D)|x'-y'|$. The fact that $C_1(D)<1$ implies that we also have $y'\notin B(x',r)$. By Lemma \ref{unif:Quasisymmetric-LLC}(2), we can find a continuum $F' \subset \mathcal R \setminus B(x',r/2)$, connecting $u'$ and $y'$. We have
\begin{align*}
\dist(E',F') &\geq \frac{r}{2}- \diam(E')\geq \frac{C_1(D)}{2}|x'-y'| - 2|x'-z'|\\
&\geq \left( H\frac{C_1(D)}{2} - 2\right)|x'-z'| \\
&\geq \frac{HC_1(D)}{4}|x'-z'|
\end{align*}
for $H\geq \frac{8}{C_1(D)}$. Also,
\begin{align*}
\min \{ \diam(E'),\diam(F')\} \leq \diam (E') \leq 2|x'-z'|.
\end{align*}
Therefore,
\begin{align*}
\Delta( E',F')= \frac{\dist(E',F')}{\min \{\diam(E'),\diam(F')\}} \geq \frac{HC_1(D)}{8} \geq  H \cdot C',
\end{align*}
where the constant $C'>0$ depends only on the lower and upper bounds of $D$, and thus only on $K_2$ and $K_3$, by the first step of the proof.

Now we choose an even larger $H$ so that $H C'\geq 12$ and apply Proposition \ref{unif:Modulus-Remove squares}. There exists an index set $I_0\subset \N\cup \{0\}$ with $I_0\leq N$ such that  the family $\Gamma'$ of curves joining $E'$ and $F'$ in $\C$ but avoiding the peripheral disks $\br S_i$, $i\in I_0$, satisfies
\begin{align}\label{unif:Quasisymmetric-Step 2 psi}
\md(\Gamma')\leq  \psi(\Delta(E',F')) \leq  \psi( H C').
\end{align}
The number $N\in \N$ and the function $\psi$ depend only on $D$ and thus, only on $K_2$ and $K_3$.

Define $E\coloneqq f^{-1}(E')$ and $F\coloneqq f^{-1}(F')$. Then $E$ and $F$ are disjoint continua in $S$ containing the sets $\{x,z\}$ and $\{y,u\}$, respectively. We have
\begin{align*}
\diam(F) \geq |y-u|\geq \frac{\delta}{2} \geq \frac{\delta}{2} \diam(E) \geq \frac{C_0}{2}\diam(E),
\end{align*}
since $\diam(E)\leq \diam(S)=1$. Also,
\begin{align*}
\dist(E,F)&\leq |x-y|\leq |x-z|\leq \diam(E)\\
&\leq \max\biggl\{1, \frac{2}{C_0}\biggr\}\cdot \min\{ \diam(E),\diam(F)\}.
\end{align*}
Therefore,
\begin{align*}
\Delta(E,F) \leq \max\biggl\{1, \frac{2}{C_0}\biggr\}\eqqcolon C'',
\end{align*}
where the latter is a uniform constant. The path family $\Gamma \coloneqq f^{-1}(\Gamma')$ connects the continua $E$ and $F$ in $\C$ but avoids the finitely many peripheral disks $\br Q_i$, $i\in I_0$. Since $\# I_0$ is uniformly bounded, depending only on $K_2$ and $K_3$, by Proposition \ref{unif:Modulus-Loewner lower} we have
\begin{align}\label{unif:Quasisymmetric-Step 2 phi}
\br \md(\Gamma) \geq  \phi( \Delta(E,F)) \geq \phi(C'').
\end{align}
Combining \eqref{unif:Quasisymmetric-Step 2 phi} and \eqref{unif:Quasisymmetric-Step 2 psi} with the modulus inequality for $f$ in Proposition \ref{unif:Regularity-F modulus} we obtain
\begin{align*}
\phi(C'')\leq \br \md(\Gamma)  \leq \md( f(\Gamma)) =\md(\Gamma')\leq \psi( HC').
\end{align*}
Since $\lim_{t\to\infty}\psi(t)=0$, if $H$ is sufficiently large depending only on $K_2$ and $K_3$, we obtain a contradiction.
\end{proof}

\section{Equivalence of square and round carpets}\label{unif:Section Square Round carpets}

In this section we prove Proposition \ref{unif:Proposition Equivalence of square carpet}.

\begin{proof}[Proof of Proposition \ref{unif:Proposition Equivalence of square carpet}]
We denote the peripheral circles of a square carpet $\mathcal R$ by $\partial S_i$, $i\in \N \cup \{0\}$, and the ones corresponding to the round carpet $T$ by $\partial C_i$, $i\in \N\cup \{0\}$.

If the peripheral circles of a square carpet $\mathcal R$ are uniformly relatively separated, then by Bonk's result in Theorem \ref{unif:Theorem Bonk}, $\mathcal R$ is quasisymmetrically equivalent to a round carpet $T$. Conversely if the peripheral circles of a round carpet $T$ are uniformly relatively separated, then our main result Theorem \ref{unif:Main theorem-quasisymmetric} implies that $T$ is quasisymmetrically equivalent to a square carpet. We remark here that the property of uniform relative separation is a quasisymmetric invariant.

Now, assume that we are given a quasisymmetry $F$ from a square carpet $\mathcal R$ onto a round carpet $T$, but the peripheral circles $\partial S_i$, $i\in \N\cup \{0\}$, of $\mathcal R$ are not uniformly relatively separated. This implies that there exists a sequence of pairs $W_n,Z_n$ of rectangles (i.e., $ W_n= \partial S_i$ and $ Z_n=\partial  S_j$ for some $i,j\in \N\cup\{0\}$) such that $\Delta(W_n,Z_n) \to 0$ as $n\to\infty$. Note that the quasisymmetry $F$ maps each peripheral circle of $\mathcal R$ onto a peripheral circle of $T$. We split in two cases.

\textbf{Case 1:} $W_n$ and $Z_n$ ``tend" to share a large segment of a side. More precisely, consider the largest parallel  line segments $\alpha_n \subset W_n$ and $\beta_n\subset Z_n$ so that one of them is either a horizontal \textit{or} vertical translation of the other, $d_n\coloneqq \dist(W_n,Z_n)=\dist(\alpha_n,\beta_n)$, and $\ell_n$ is the length of $\alpha_n$ and $\beta_n$. If such segments do not exist, then we set $\alpha_n\in W_n$ and $\beta_n\in Z_n$ to be the nearest vertices of the two rectangles; here we set $\ell_n=0$. We assume in this case that $\lim_{n\to\infty}d_n/\ell_n=0$. 

Consider a $K$-quasiconformal extension of $F$ on $\widehat{\C}$, as in \cite[Section 5]{Bonk:uniformization}. Let $g_n$ be a Euclidean similarity that maps $U\coloneqq \{0\}\times [0,1]$ onto $\beta_n$. Also, let $h_n$ be a M\"obius transformation that maps the disk bounded by the circle $F(Z_n)$ onto $ \br \D$ such that $f_n\coloneqq h_n\circ F\circ g_n$ satisfies $f_n(0)=1$, $f_n(1)=-1$, and $f_n(1/2)=i$ or $f_n(1/2)=-i$ so that $f$ is orientation-preserving. Note that $f_n\colon \widehat \C\to \widehat \C$ is $K$-quasiconformal. By passing to a subsequence, we assume that $f_n$ converges uniformly in the spherical metric of $\widehat \C$ to a $K$-quasiconformal map $f\colon \widehat \C \to \widehat \C$; see \cite[Theorem 5.1, p.~73]{LehtoVirtanen:quasiconformal} for compactness of quasiconformal maps. 

Since $g_n$ is a scaling by $\ell_n$, it follows that the Euclidean distance of $U=g_n^{-1}(\beta_n)$ and $V_n\coloneqq g_n^{-1}(\alpha_n)$ is $d_n/\ell_n$, and in fact $V_n$ is a vertical line segment of length $1$. In what follows, we will use Hausdorff convergence with respect to the spherical metric of $\widehat \C$, and the fact that the Hausdorff convergence is compatible with the uniform convergence of $f_n$; for instance, $V_n$ converges to $U$ in the Hausdorff sense, and thus $f_n(V_n)$ converges to $f(U)$ in the Hausdorff sense. Observe that any Hausdorff limit of the rectangles $g_n^{-1}(Z_n)\cup g_n^{-1}(W_n)\subset \widehat \C$  is not a circle in $\widehat \C$. On the other hand, $f_n(V_n)$ is an arc of a circle, so its Hausdorff limit will be an arc of a circle $C\subset \widehat \C$. The circle $C$ is distinct from $\partial \D$ and they bound disjoint regions; this is justified because $f$ is a homeomorphism and the Hausdorff limits of $g_n^{-1}(Z_n)\cup g_n^{-1}(W_n)$ are not just a single circle. Thus, $C\cap \partial \D$ can contain at most one point. On the other hand, $f(U)=\lim_{n\to\infty} f_n(U)=\lim_{n\to\infty}f_n(V_n) \subset C \cap \partial \D$, which a contradiction, since $f$ is a homeomorphism and cannot map $U$ to a point. 

\textbf{Case 2:} $W_n$ and $Z_n$ ``tend" to share a corner. Using the notation of Case 1, we assume that there exists a constant $c>0$ such that $d_n/\ell_n\geq c$ for infinitely many $n$. We allow the possibility $d_n/\ell_n =\infty$, which occurs whenever $\alpha_n$ and $\beta_n$ are vertices and $\ell_n=0$. The assumption that $\Delta(W_n,Z_n)\to 0$ implies that in this case neither $W_n$ nor $Z_n$ can be the outer peripheral rectangle $\partial S_0$, so they are both squares, for sufficiently large $n$. By passing to a subsequence we assume that the above hold for all $n\in \N$. Also, assume that $Z_n$ is smaller than $W_n$, so if $m_n$ denotes the sidelength of $Z_n$, then $d_n/m_n\to 0$ as $n\to\infty$ by our assumption on the separation of $Z_n$ and $W_n$. Note that we also have $\ell_n/m_n \to 0$ as $n\to\infty$.

Again, we consider a quasiconformal extension $F\colon \widehat{\C}\to \widehat{\C}$. We precompose $F$ with a Euclidean similarity $g_n$ that maps the unit square $U=\partial([0,1]\times [0,1])$ onto the square $Z_n$, and postcompose with a M\"obius transformation $h_n$ that maps the disk bounded by the circle $F(Z_n)$ onto $\br \D$, with suitable normalizations. After passing to a subsequence we may assume that $f_n\coloneqq h_n\circ F\circ g_n$ converges uniformly in the spherical metric to a $K$-quasiconformal map $f\colon \widehat \C \to \widehat \C$. 

Note that  the Euclidean distance of $ U=g_n^{-1}(Z_n)$ and $V_n\coloneqq g_n^{-1}(W_n) $ is $d_n/m_n$, the square $V_n$ is larger than $U$, and the segments $g^{-1}(\alpha_n)$ and $g^{-1}(\beta_n)$ have length $\ell_n/m_n$, which converges to $0$. After passing to a subsequence, it follows that a Hausdorff limit (in $\widehat \C$) of $V_n$ contains two perpendicular segments of Euclidean length at least $1$ that meet at a corner $x$ of $U$, but otherwise disjoint from $U$. Also, observe that all the arcs that lie in the union of these two segments with $\partial U$ are quasiarcs. 

On the other hand, the image $f_n(V_n)$ converges to a circle $C\subset \widehat{\C}$ that meets $\partial \D$ at one point $x'=f(x)$. It is easy to see that there exist arbitrarily small arcs in $C\cup \partial \D$ passing through $x'$ that are not quasiarcs. This leads to a contradiction, since the quasiconformal map $f\colon \widehat \C\to \widehat \C$ is a quasisymmetry and thus must map quasiarcs to quasiarcs. 

As a final remark, if we are given a quasisymmetry $F$ between a square carpet $\mathcal R$ and a round carpet $T$, but the peripheral circles $\partial C_i$, $i\in \N\cup \{0\}$, of $T$ are not uniformly relatively separated, then the peripheral circles $\partial S_i$, $i\in \N\cup\{0\}$, of $\mathcal R$ are also not uniformly relatively separated, so we are reduced to the previous analysis.
\end{proof}

\section{A test function}\label{unif:Appendix}\index{test function}

Here we include a Lemma that is often used in variational arguments in Section \ref{unif:Section Conjugate}. The assumptions here are that we have a carpet $S\subset \br \Omega$ of  area zero with outer peripheral circle $\partial Q_0=\partial \Omega$ and inner peripheral disks $\{Q_i\}_{i\in \N}$ that are uniformly fat, uniform quasiballs.

\begin{lemma}\label{unif:Zeta lemma}
For each $x\in S$, $r>0$, and $\varepsilon>0$ there exists a function $\zeta\in \mathcal W^{1,2}(S)$, supported in $B(x,r)\cap S$, with $0\leq \zeta\leq 1$ and $\zeta\equiv 1$ in some smaller ball $B(x,r')\cap S$, $r'<r$, such that:
\begin{enumerate}[\upshape(a)]
\item If $x\in S^\circ\cup \partial \Omega$ then 
\begin{align*}
D(\zeta)=\sum_{i\in \N} \osc_{Q_i}(\zeta)^2 <\varepsilon.
\end{align*}
\item If $x\in \partial Q_{i_0}$ for some $i_0\in \N$ then
\begin{align*}
D(\zeta)- \osc_{Q_{i_0}}(\zeta)^2= \sum_{i\in \N\setminus \{i_0\}} \osc_{Q_i}(\zeta)^2 <\varepsilon.
\end{align*}
\end{enumerate}
\end{lemma}

\begin{proof}
The function $\zeta$ will be a discrete version of the logarithm. In fact we will construct a Lipschitz function $\zeta$ defined on all of $\R^2$, and thus it will lie in $\mathcal W^{1,2}(S)$; see comments after Definition \ref{unif:Definition Sobolev function} and also Example \ref{harmonic:3-Example-Lipschitz}.
 
We fix a large integer $N$ which will correspond to the number of annuli around $x$ that we will construct, and $\zeta$ will increase by $1/N$ on each annulus. We set $\zeta=0$ outside $B(x,r)$ and define $R_1\coloneqq r$ and $r_1\coloneqq R_1/2$. In the annulus $A_1\coloneqq A(x;r_1,R_1)$ define $\zeta$ to be a radial function of constant slope $\frac{1}{Nr_1}$, so on the inner boundary of $A_1$ the function $\zeta$ has value $\frac{1}{N}$. Then consider $R_2<r_1$ sufficiently small, and $r_2\coloneqq R_2/2$, so that no peripheral disk intersects both annuli $A_1$ and $A_2\coloneqq A(x;r_2,R_2)$, except possibly for $Q_{i_0}$, in case $x\in \partial Q_{i_0}$; this is possible since the diameters of the peripheral disks converge to $0$. In the ``transition" annulus $A(x;R_2,r_1)$ we define $\zeta$ to be constant, equal to $\frac{1}{N}$, and on $A_2$ we let $\zeta$ be a radial function of slope $\frac{1}{Nr_2}$. We continue constructing annuli $A_{j}=A(x;r_j,R_j)$, $j=1,\dots,N$, and defining the function $\zeta$ in the same way. The last annulus will be $A_N\coloneqq A(x; r_N,R_N)$ and the value of $\zeta$ will be $1$ in the inner boundary of $A_N$. We extend $\zeta$ to be $1$ in the ball $B(x,r_N)$. 

We now compute the Dirichlet energy $D(\zeta)$ of $\zeta$. Assume that $x\in \partial Q_{i_0}$, $i_0\in \N$, since otherwise the details are almost the same, but simpler. For $i\in \N$ and $j\in \{1,\dots,N\}$, let $d_j(Q_i)\coloneqq  \mathcal H^1(\{s\in [r_j,R_j]: \gamma_s\cap Q_i\neq \emptyset \})$, where $\gamma_s$ is the circle of radius $s$ around $x$. Since the peripheral disks $Q_i$, $i\in \N$, are fat, there exists a constant $C>0$ such that $d_j(Q_i)^2\leq C\mathcal H^2(Q_i\cap A_j)$, for all $i\in \N$ and $j\in \{1,\dots,N\}$; see Remark \ref{unif:Fatness consequence}. Also, if $Q_i\cap A_j\neq \emptyset$, $i\neq i_0$, then $\osc_{Q_i}(\zeta) \leq d_j(Q_i) \frac{1}{Nr_j}$. By construction, each peripheral disk $Q_i$, $i\neq i_0$, can only intersect one annulus $A_j$, and if a peripheral disk $Q_i$ does not intersect any annulus $A_j$, then $\zeta$ is constant on $Q_i$, so $\osc_{Q_i}(\zeta)=0$. Combining these observations, we have
\begin{align*}
\sum_{i\in \N\setminus \{i_0\}} \osc_{Q_i}(\zeta)^2&= \sum_{j=1}^N \sum_{\substack{i:Q_i\cap A_j\neq \emptyset\\i\in \N\setminus \{i_0\}}} \osc_{Q_i}(\zeta)^2 \leq \frac{1}{N^2} \sum_{j=1}^N \frac{1}{r_j^2} \sum_{\substack{i:Q_i\cap A_j\neq \emptyset\\ i\in \N\setminus \{i_0\}}} d_j(Q_i)^2 \\
&\leq \frac{C}{N^2}\sum_{j=1}^N \frac{1}{r_j^2} \sum_{\substack{i:Q_i\cap A_j\neq \emptyset \\ i\in \N\setminus \{i_0\}}}\mathcal H^2(Q_i\cap A_j)\\
&\leq \frac{C}{N^2}\sum_{j=1}^N \frac{1}{r_j^2} \mathcal H^2(A_j)= \frac{\pi C}{N^2} \sum_{j=1}^N \frac{4r_j^2-r_j^2}{r_j^2}\\
&=\frac{3\pi C}{N}.
\end{align*}
Making $N$ sufficiently large we can achieve that $\frac{3\pi C}{N}<\varepsilon$, as desired.
\end{proof}

%% file: carpetmap.tikz
\begin{scope}[shift={(0,0)},scale=.4,line cap=round,line join=round,>=triangle 45,x=1.0cm,y=1.0cm]
\clip (-1.5,-10) rectangle (20,10);






\draw[rounded corners=1.5mm, fill=black!5] (1,3)--(-1,2)--(0,0)--(0.8,-1.2)--
			(1,-3)--(4,-3)--(6,-3.5)--(9,-4)--(11,-5)--(13.1,-4.5)--
			(16,-3)--(16,-1)--(15,0)--(14.5,3)--(13,4)--
			(10.5,4.5)--(8.5,4.5)--(8,4.5)--(6,5)--(4,5)--(3,4)--cycle;



	
\draw[rounded corners=3 mm,fill=black!1,fill opacity=1  ] (1.34,2.9) -- (4.22,4.52) -- (6.56,3.1) -- (6.48,0.62) -- (4.6,-1.5) -- (5.1,1.66) -- (3.24,0.96) -- (3.1,2.66) -- cycle;
\draw[rounded corners=3 mm,fill=black!1,fill opacity=1  ] (8.371104449397732,2.572244035034731) -- (7.195226064217023,0.22048726467331042) -- (9.933572988610456,-1.2292258129467435) -- (11.093343450706499,0.8325883418906666) -- (10.046328450203127,2.8299708043894074) -- cycle;
\draw[rounded corners=1 mm,fill=black!1,fill opacity=1  ] (7.,4.) -- (7.,3.) -- (7.871758833773048,2.6366757273734005) -- (8.354996526313066,3.490395650860765) -- (7.887866756857715,4.150820497332123) -- (7.565708295164369,3.69979865096144) -- cycle;
\draw[rounded corners=1 mm,fill=black!1,fill opacity=1  ] (6.695880448592337,-0.7943118896607273) -- (5.,-2.) -- (6.470369525406996,-2.8078022752441356) -- (8.306672757059063,-3.500442967884828) -- (9.,-2.) -- (7.5012766028257,-1.9057585825027687) -- (7.485168679741033,-1.0198228128460691) -- cycle;
\draw[rounded corners=1 mm,fill=black!1,fill opacity=1  ] (1.4769133691601464,1.2030705728380136) -- (2.2984174464781764,0.) -- (3.909209754944902,0.5104298801973212) -- (3.844778062606233,-1.2131178898620762) -- (4.005857293452905,-2.292348736534783) -- (2.443388754240182,-2.372888351958119) -- (2.7977630621028613,-1.3903050437934161) -- (1.1708628305514686,-1.6480318131480924) -- (1.1064311382127996,-0.6332326588140547) -- (0.31714290706410403,0.0916238799959723) -- (0.9614598304507943,0.558753649451323) -- cycle;
\draw[rounded corners=1 mm,fill=black!1,fill opacity=1  ] (0.6715172149267836,2.443380650357393) -- (-0.6171166318465968,1.8796033423940388) -- (0.,1.) -- (0.7198409841807855,1.4124735729386881) -- (0.49433006099544385,1.9440350347327078) -- cycle;
\draw[rounded corners=1 mm,fill=black!1,fill opacity=1  ] (9.176500603631094,4.247468035840127) -- (9.772493757763783,3.3615322661834273) -- (11.31885437389184,3.2971005738447583) -- (11.189990989214502,4.134712574247456) -- (10.2557314503038,3.828662035638778) -- (9.756385834679117,4.150820497332123) -- cycle;
\draw[rounded corners=1 mm,fill=black!1,fill opacity=1  ] (11.528257373992515,2.491704419611395) -- (12.027602989617199,0.7037249572133284) -- (11.447717758569178,0.) -- (12.398085220564546,-0.713772274237391) -- (13.976661682861938,-0.40772173562871294) -- (14.508223144655956,0.7198328802979956) -- (13.751150759676596,1.7346320346320334) -- (12.720243682257891,2.0406825732407117) -- (12.398085220564546,2.668891573542735) -- cycle;
\draw[rounded corners=1 mm,fill=black!1,fill opacity=1  ] (11.238314758468503,-1.2453337360314107) -- (9.756385834679117,-1.8252189670794323) -- (9.514766988409107,-3.081636967683479) -- (9.61141452691711,-4.032004429678848) -- (11.093343450706499,-4.483026276049531) -- (12.526948605241884,-4.338054968287526) -- (13.477316067237252,-3.0494211215141447) -- (12.,-3.) -- (12.688027836088557,-2.1795932749421123) -- (11.785984143347191,-2.066837813349441) -- (12.011495066532532,-1.4225208899627506) -- (11.512149450907847,-1.5191684284707543) -- cycle;
\draw[rounded corners=1 mm,fill=black!1,fill opacity=1  ] (14.,-1.) -- (13.106833836289905,-1.3258733514547472) -- (13.815582452015265,-1.8896506594181015) -- (15.007568760280641,-2.4373200442967886) -- (15.812964914514005,-1.567492197724756) -- (14.814273683264634,-1.5513842746400888) -- (14.765949914010633,-0.6815564280680565) -- cycle;

\draw[rounded corners=0.5 mm,fill=black!1,fill opacity=1  ] (1.637992600006819,2.2823014195107203) -- (1.960151061700164,1.6702003422933644) -- (2.57225213891752,1.9762508809020425) -- (2.2339857541395074,2.201761804087384) -- cycle;
\draw[rounded corners=0.5 mm,fill=black!1,fill opacity=1  ] (2.411172908070847,1.1225309574146773) -- (2.3145253695628436,0.6554011879593266) -- (2.7333313697641923,0.4943219571126539) -- (2.765547215933527,1.041991341991341) -- cycle;
\draw[rounded corners=0.5 mm,fill=black!1,fill opacity=1  ] (6.960686958746813,2.26895278001982) -- (6.85176786951916,1.385497945173306) -- (7.311648468480359,1.3612937031227166) -- (7.602099373087432,2.0269103595139257) -- (7.166423016176823,1.845378544134505) -- cycle;
\draw[rounded corners=0.5 mm,fill=black!1,fill opacity=1  ] (6.5976233279879715,0.054264632390887906) -- (6.149844850052068,-0.18777778811500628) -- (6.5976233279879715,-0.4298202086209005) -- (6.984891200797402,-0.24828839324147983) -- cycle;
\draw[rounded corners=0.5 mm,fill=black!1,fill opacity=1  ] (8.267716029478642,-0.9260071706579835) -- (8.207205424352168,-1.4706026167962454) -- (9.042251775097503,-1.5916238270491925) -- (8.812311475616903,-1.1438453491132883) -- cycle;
\draw[rounded corners=0.5 mm,fill=black!1,fill opacity=1  ] (10.942284776068771,-0.3572074824691322) -- (10.518710540183456,-0.8170880814303312) -- (11.,-1.) -- (11.462675980156444,-0.7444753552785629) -- cycle;
\draw[rounded corners=0.5 mm,fill=black!1,fill opacity=1  ] (10.857569928891708,2.8256503471833763) -- (10.567119024284635,2.6199142897533663) -- (10.83336568684112,2.172135811817462) -- (11.135918712473487,2.2447485379692305) -- cycle;
\draw[rounded corners=0.5 mm,fill=black!1,fill opacity=1  ] (4.842815779320239,-2.281444725490991) -- (4.152994880878441,-2.5839977511233587) -- (4.661283963940819,-2.9349592608569055) -- (5.3632069834079115,-2.9349592608569055) -- cycle;
\draw[rounded corners=0.5 mm,fill=black!1,fill opacity=1  ] (4.23441011736713,0.11806082272927261) -- (4.14693858048472,-0.25369320902096776) -- (4.445799664832952,-0.26827179850136934) -- (4.5,0.) -- cycle;
\draw[rounded corners=0.5 mm,fill=black!1,fill opacity=1  ] (4.372906717430944,1.0510905494749738) -- (4.103202812043516,0.8324117072689502) -- (4.460378254313354,0.5772863913619225) -- (4.591585559636968,0.8834367704503557) -- cycle;
\draw[rounded corners=0.5 mm,fill=black!1,fill opacity=1  ] (4.219831527886728,-1.0919621041440588) -- (4.088624222563114,-1.4564268411540984) -- (4.343749538470141,-1.361666009531488) -- cycle;
\draw[rounded corners=0.5 mm,fill=black!1,fill opacity=1  ] (12.595463534726795,-1.3553434181266313) -- (12.27152237313831,-1.6684865409954999) -- (12.77903019296027,-1.7440728120328128) -- cycle;
\draw[rounded corners=0.5 mm,fill=black!1,fill opacity=1  ] (13.449696956756394,-2.260085810322037) -- (13.164976912939476,-2.506330172542075) -- (13.857539181683332,-2.737184262123361) -- (13.657465637379552,-2.367817718793304) -- cycle;
\draw[rounded corners=0.5 mm,fill=black!1,fill opacity=1  ] (14.265381406610272,-3.1373313507309226) -- (13.76519754585082,-3.3527951676734555) -- (14.311552224526528,-3.614429802532246) -- (14.880992312160366,-3.145026487050299) -- cycle;
\draw[rounded corners=0.5 mm,fill=black!1,fill opacity=1  ] (12.12898394371463,3.580284495832821) -- (11.976565426715537,3.206166317744135) -- (12.849507842255804,3.150741402471737) -- (12.835651613437705,3.704990555195716) -- (12.461533435349018,3.4694346652880252) -- cycle;
\draw[rounded corners=0.5 mm,fill=black!1,fill opacity=1  ] (13.597744198433176,2.8459043684735486) -- (13.168201105072091,2.8181919108373497) -- (13.708594028977972,2.2639427581133704) -- (14.23513072406575,2.333223902203868) -- cycle;
\draw[rounded corners=0.5 mm,fill=black!1,fill opacity=1  ] (5.658125085662174,4.577932970735984) -- (5.117732161756295,4.591789199554083) -- (5.796687373843169,3.982115131557706) -- (6.3093678401128495,4.4116582249187895) -- cycle;
\draw[rounded corners=0.5 mm,fill=black!1,fill opacity=1  ](1.7559484498089948,-1.8326174444905543) -- (1.132786781709373,-2.121399680926964) -- (1.5735596689017883,-2.5165753728925777) -- (2.,-2.) -- cycle;

\end{scope}

\begin{scope}[shift={(4,0.5)}, scale=1, line cap=round,line join=round,>=triangle 45,x=1.0cm,y=1.0cm]


\clip(3.016101984218515,-2.8619821859205565) rectangle (11.994350382410705,1.7950643647627003);
\draw[fill=black!5] (4.44,1.) -- (4.44,-2.22) -- (9.48,-2.22) -- (9.48,1.) -- cycle;
\draw[fill=black!1] (5.3128561834113155,0.8476099037004573) -- (4.641713661839562,0.8476099037004574) -- (4.641713661839561,0.17646738212870403) -- (5.312856183411315,0.17646738212870314) -- cycle;
\draw[fill=black!1] (5.638736545284338,-0.6748716833164445) -- (7.083941270518898,-0.6748716833164445) -- (7.0839412705189,0.7703330419181151) -- (5.638736545284339,0.7703330419181164) -- cycle;
\draw[fill=black!1] (6.693980895803577,-1.454235157882946) -- (7.263263839888179,-1.454235157882946) -- (7.263263839888179,-0.8849522137983439) -- (6.693980895803577,-0.8849522137983437) -- cycle;
\draw[fill=black!1] (7.597490764548813,-2.068070106877798) -- (8.856475898515274,-2.0680701068777987) -- (8.856475898515273,-0.8090849729113381) -- (7.597490764548813,-0.8090849729113379) -- cycle;
\draw[fill=black!1] (8.550011473737635,-0.6173175476086116) -- (9.354524849533938,-0.6173175476086116) -- (9.354524849533936,0.18719582818769082) -- (8.550011473737634,0.18719582818768937) -- cycle;
\draw[fill=black!1] (5.26375097727232,-1.3035371388767152) -- (4.651096791183124,-1.303537138876715) -- (4.651096791183125,-1.9161913249659106) -- (5.26375097727232,-1.9161913249659106) -- cycle;
\draw[fill=black!1] (7.345120041083107,-0.6841466657325486) -- (8.321327277198009,-0.6841466657325488) -- (8.321327277198009,0.2920605703823525) -- (7.345120041083107,0.2920605703823524) -- cycle;
\draw[fill=black!1] (7.198418692032852,0.4356721967476723) -- (7.61835989273344,0.43567219674767255) -- (7.618359892733441,0.8556133974482609) -- (7.198418692032852,0.8556133974482609) -- cycle;
\draw[fill=black!1] (5.4987674220909195,-2.081802906011521) -- (6.474974658205825,-2.0818029060115215) -- (6.474974658205825,-1.1055956698966156) -- (5.4987674220909195,-1.1055956698966152) -- cycle;
\draw[fill=black!1] (8.727393683821624,0.31477649972251365) -- (9.29667662790623,0.3147764997225134) -- (9.29667662790623,0.8840594438071188) -- (8.727393683821626,0.8840594438071199) -- cycle;
\draw[fill=black!1] (8.218441530292814,0.4609501659937956) -- (8.410814360122513,0.4609501659937958) -- (8.410814360122517,0.6533229958234946) -- (8.218441530292816,0.653322995823497) -- cycle;
\draw[fill=black!1] (6.302368231728742,-0.9795186096056598) -- (6.494741061558425,-0.9795186096056598) -- (6.494741061558425,-0.787145779775977) -- (6.302368231728742,-0.7871457797759767) -- cycle;
\draw[fill=black!1] (4.595605450197088,-0.2114753579164143) -- (4.787978280026787,-0.2114753579164143) -- (4.787978280026788,-0.01910252808671542) -- (4.595605450197089,-0.019102528086714823) -- cycle;
\draw[fill=black!1] (5.136080331015442,-1.1999754688868325) -- (5.328453160845142,-1.1999754688868323) -- (5.328453160845143,-1.0076026390571324) -- (5.136080331015442,-1.0076026390571324) -- cycle;
\draw[fill=black!1] (8.976296589461654,-1.9395726742172175) -- (9.168669419291353,-1.9395726742172168) -- (9.168669419291355,-1.7471998443875179) -- (8.976296589461654,-1.7471998443875176) -- cycle;
\draw[fill=black!1] (7.241087761571144,-1.783119419243482) -- (7.433460591400843,-1.7831194192434818) -- (7.433460591400842,-1.5907465894137827) -- (7.241087761571143,-1.5907465894137833) -- cycle;
\draw[fill=black!1] (6.572605672137916,-1.7617848844743356) -- (6.764978501967615,-1.761784884474336) -- (6.764978501967616,-1.5694120546446373) -- (6.572605672137917,-1.5694120546446364) -- cycle;
\draw[fill=black!1] (8.983408101051369,-0.958184074836514) -- (9.175780930881068,-0.958184074836514) -- (9.175780930881066,-0.7658112450068151) -- (8.983408101051367,-0.7658112450068163) -- cycle;
\draw[fill=black!1] (8.954962054692508,-1.5199934904240182) -- (9.382014766982811,-1.519993490424018) -- (9.382014766982811,-1.0929407781337155) -- (8.954962054692508,-1.0929407781337153) -- cycle;
\draw[fill=black!1] (4.9440695180931336,-0.4674897751461628) -- (5.3711222303834365,-0.46748977514616274) -- (5.371122230383437,-0.04043706285585973) -- (4.9440695180931336,-0.04043706285585962) -- cycle;
\draw[fill=black!1] (4.588493938607373,-1.0577452370925275) -- (5.015546650897676,-1.0577452370925275) -- (5.015546650897676,-0.6306925248022246) -- (4.588493938607373,-0.6306925248022245) -- cycle;
\draw[fill=black!1] (5.292533585989176,-0.8017308198627787) -- (5.484906415818861,-0.8017308198627789) -- (5.484906415818862,-0.6093579900330941) -- (5.292533585989177,-0.6093579900330935) -- cycle;
\draw[fill=black!1] (7.838454735107225,0.47122975469624867) -- (8.030827564936912,0.4712297546962486) -- (8.030827564936912,0.6636025845259351) -- (7.838454735107225,0.6636025845259353) -- cycle;
\draw[fill=black!1] (8.40026415069473,0.7272441719259971) -- (8.59263698052441,0.7272441719259971) -- (8.592636980524412,0.9196170017556783) -- (8.400264150694731,0.9196170017556795) -- cycle;
\draw[fill=black!1] (8.066023105978104,0.713021148746567) -- (8.258395935807801,0.7130211487465666) -- (8.258395935807801,0.9053939785762637) -- (8.066023105978104,0.9053939785762639) -- cycle;
\draw[fill=black!1] (7.703336014902633,0.7343556835157125) -- (7.895708844732333,0.7343556835157123) -- (7.895708844732331,0.9267285133454122) -- (7.703336014902632,0.9267285133454114) -- cycle;
\draw[fill=black!1] (6.891355716732111,-1.7204568592811715) -- (7.08372854656181,-1.7204568592811722) -- (7.08372854656181,-1.5280840294514735) -- (6.891355716732111,-1.5280840294514735) -- cycle;
\draw[fill=black!1] (7.312202877468297,-2.1031374407806673) -- (7.504575707297993,-2.1031374407806673) -- (7.504575707297993,-1.9107646109509722) -- (7.312202877468297,-1.9107646109509722) -- cycle;
\draw[fill=black!1] (6.94,-2.) -- (7.132372829829691,-2.) -- (7.132372829829692,-1.807627170170309) -- (6.94,-1.8076271701703086) -- cycle;
\draw[fill=black!1] (6.586828695317346,-2.0960259291909513) -- (6.7792015251470445,-2.0960259291909513) -- (6.779201525147045,-1.9036530993612524) -- (6.5868286953173465,-1.903653099361252) -- cycle;

\end{scope}
\begin{scope}[->,>=stealth',auto,node distance=3cm,
  thick,main node/.style={circle,draw,font=\sffamily\Large\bfseries}]
\path (6.5,0) edge[bend left] node[right] {}(8,0);
\end{scope}

%% file: conjugate.tikz
\begin{tikzpicture}
	\clip (-6.5,-5.2) rectangle (5,5);

		\begin{scope}[scale=0.8]

			\begin{scope}[shift={(0,0)}, rotate=0,scale=1]

			//grid   

			//draw lines
				\draw (-8,5)..controls (-1,4) and (1,6)..(8,5); //top
				\draw (-8,5) node[anchor=north west] {$\Theta_4$};
			
				\draw (-8,-5)..controls (-4,-3) and (1,-9)..(8,-5); //bot       	 
				\draw (-8,-4.8) node[anchor=south west] {$\Theta_2$};

			//draw big peripheral circle	
			
			\def\a{1.25};\def\A{1};
			\def\b{0.05};\def\B{2};
			\def\c{1};\def\C{1};
			\def\d{0.1};\def\D{34};
            	 	\draw [fill=black!5,smooth,domain=0:360] plot 
            		(
            		{\a*cos( \A*\x )+\b*sin(\B*\x)}, 
            		{\c*sin( \C*\x ) + \d*cos(\D*\x )}
            		) ;
            		\draw (1.1,-1) node[] {$Q_{i_0}$};

            //bottom conections

            	//bot peripheral
            	\def\a{0.3};\def\A{1};
				\def\b{0.01};\def\B{2};
				\def\c{0.3};\def\C{1};
				\def\d{0.02};\def\D{40};
				\def\e{-1};\def\f{-1.2};
            	 	\draw [fill=black!5,smooth,domain=0:360] plot 
            		(
            		{\a*cos( \A*\x )+\b*sin(\B*\x)+\e}, 
            		{\c*sin( \C*\x ) + \d*cos(\D*\x )+\f}
            		) ;
            	\draw (\e,\f) node[scale=0.7]{ $Q_{i_a}$};  
           		\def\x{280};

            	//left+right bot connections
            	\def\x{220};
            	\coordinate (OmegasL) at (
            		{\a*cos( \A*\x )+\b*sin(\B*\x)+\e}, 
            		{\c*sin( \C*\x ) + \d*cos(\D*\x )+\f}
            		);
            		
            	\def\x{320};
				\coordinate (OmegasR) at (
            		{\a*cos( \A*\x )+\b*sin(\B*\x)+\e}, 
            		{\c*sin( \C*\x ) + \d*cos(\D*\x )+\f}
            		);	
            		
            	\draw[xshift=0,thin, fill=black!20](OmegasL) ..controls (-2,-3) and (0,-3)..(-2.8,-5.25)-- (-1.5,-5.63)..controls (0,-3) and (0,-3)..(OmegasR)--cycle;
            	\draw[xshift=1,yshift=1] (-2.8,-5.3)
            		 node[circle,fill,inner sep=1pt](a){} ;
            	//mid
				\def\x{280};
            	\draw[xshift=-3,thick] (
            		{\a*cos( \A*\x )+\b*sin(\B*\x)+\e}, 
            		{\c*sin( \C*\x ) + \d*cos(\D*\x )+\f}
            		) ..controls (-1,-3) and (0,-3)..(-2,-5.45);            	
				\draw[xshift=-3] (-2,-5.45)
            		 node[circle,fill,inner sep=1pt](a){} ;
            	//gammas
				\draw (-0.77,-2.2) node[] {$\gamma_s$};  	 
        
				\draw[xshift=-0,yshift=2] (-1.5,-5.7)
            		 node[circle,fill,inner sep=1pt](a){} ;
           		
           		//ball
            	\def\a{1.25};\def\A{1};
				\def\b{0.05};\def\B{2};
				\def\c{1};\def\C{1};
				\def\d{0.1};\def\D{34};
				\def\x{220};
           		//point a
            	 	\draw   
            		(
            		{\a*cos( \A*\x )+\b*sin(\B*\x)}, 
            		{\c*sin( \C*\x ) + \d*cos(\D*\x )}
            		) circle (1cm); 
            		\draw(
            		{\a*cos( \A*\x )+\b*sin(\B*\x)}, 
            		{\c*sin( \C*\x ) + \d*cos(\D*\x )}
            		)
            		 node[circle,fill,inner sep=1pt,label=north:$a$](a){} ;
				         		
           		//bot peripheral
            	\def\a{0.3};\def\A{1};
				\def\b{0.01};\def\B{2};
				\def\c{0.3};\def\C{1};
				\def\d{0.02};\def\D{40};
				\def\e{-1};\def\f{-1.2};
            	 	\draw [fill=black!5,smooth,domain=0:360] plot 
            		(
            		{\a*cos( \A*\x )+\b*sin(\B*\x)+\e}, 
            		{\c*sin( \C*\x ) + \d*cos(\D*\x )+\f}
            		) ;
            	\draw (\e,\f) node[scale=0.7]{ $Q_{i_a}$};  
           		\def\x{280};

				//Omega label
				\draw[->,dashed] (0,-5) node[anchor=west] {$\Omega_{s,h}$} to (-1.3,-5);           		
           			
            //top connections

            `	
            
            	//ball
            	\def\a{1.25};\def\A{1};
				\def\b{0.05};\def\B{2};
				\def\c{1};\def\C{1};
				\def\d{0.1};\def\D{34};
				\def\x{60};
            	 	\draw   
            		(
            		{\a*cos( \A*\x )+\b*sin(\B*\x)}, 
            		{\c*sin( \C*\x ) + \d*cos(\D*\x )}
            		) circle (1cm); 
            		\draw(
            		{\a*cos( \A*\x )+\b*sin(\B*\x)}, 
            		{\c*sin( \C*\x ) + \d*cos(\D*\x )}
            		)
            		 node[circle,fill,inner sep=1pt,label=south:$b$](a){} ;

            	//top peripheral
            	\def\a{0.3};\def\A{1};
				\def\b{0.01};\def\B{2};
				\def\c{0.3};\def\C{1};
				\def\d{0.02};\def\D{40};
				\def\e{1};\def\f{1.2};
            	 	\draw [fill=black!5,smooth,domain=0:360] plot 
            		(
            		{\a*cos( \A*\x )+\b*sin(\B*\x)+\e}, 
            		{\c*sin( \C*\x ) + \d*cos(\D*\x )+\f}
            		) ;
				\draw (\e,\f) node[scale=0.7]{ $Q_{i_b}$};

            	\def\x{60};
		

            	//left+right
            	\def\x{100};
            	\coordinate (OmegatL) at (
            		{\a*cos( \A*\x )+\b*sin(\B*\x)+\e}, 
            		{\c*sin( \C*\x ) + \d*cos(\D*\x )+\f}
            		);
            	\def\x{30};
            	\coordinate (OmegatR) at (
            		{\a*cos( \A*\x )+\b*sin(\B*\x)+\e}, 
            		{\c*sin( \C*\x ) + \d*cos(\D*\x )+\f}
            		);

            	\draw[xshift=0,yshift=0,thin,fill=black!20] (OmegatL) ..controls (-1,3) and (1,3)..(0.5,5.05)--(1.5,5.15) ..controls (1,3) and (1.5,3).. (OmegatR)--cycle;
            	\draw[yshift=0] (0.5,5.05)
            		 node[circle,fill,inner sep=1pt](a){} ;

				
				\draw[xshift=-0,yshift=0] (1.5,5.15)
            		 node[circle,fill,inner sep=1pt](a){} ;
           		//mid
				\def\x{60};
            	\draw[xshift=-0,thick] (
            		{\a*cos( \A*\x )+\b*sin(\B*\x)+\e}, 
            		{\c*sin( \C*\x ) + \d*cos(\D*\x )+\f}
            		) ..controls (0,3) and (1,3)..(1,5.1);            	
				\draw[xshift=-0] (1,5.1)
            		 node[circle,fill,inner sep=1pt](a){} ;
           		
            	//gammat
            	\draw (0.9,3.1) node[] {$\tilde \gamma_t$};
            	//top peripheral
            	\def\a{0.3};\def\A{1};
				\def\b{0.01};\def\B{2};
				\def\c{0.3};\def\C{1};
				\def\d{0.02};\def\D{40};
				\def\e{1};\def\f{1.2};
            	 	\draw [fill=black!5,smooth,domain=0:360] plot 
            		(
            		{\a*cos( \A*\x )+\b*sin(\B*\x)+\e}, 
            		{\c*sin( \C*\x ) + \d*cos(\D*\x )+\f}
            		) ;
				\draw (\e,\f) node[scale=0.7]{ $Q_{i_b}$};  
            	//ball
            	\def\a{1.25};\def\A{1};
				\def\b{0.05};\def\B{2};
				\def\c{1};\def\C{1};
				\def\d{0.1};\def\D{34};
				\def\x{60};
            	 	\draw   
            		(
            		{\a*cos( \A*\x )+\b*sin(\B*\x)}, 
            		{\c*sin( \C*\x ) + \d*cos(\D*\x )}
            		) circle (1cm); 
            		\draw(
            		{\a*cos( \A*\x )+\b*sin(\B*\x)}, 
            		{\c*sin( \C*\x ) + \d*cos(\D*\x )}
            		)
            		 node[circle,fill,inner sep=1pt,label=south:$b$](a){} ;
			            
            	//Omega label
				\draw[->,dashed] (3,3.7) node[anchor=west] {$\widetilde\Omega_{t,h}$} to (1.07,3.7);
				
            \end{scope}

		\end{scope}
		
	\end{tikzpicture}

%% file: conjugate_cases.tikz
\begin{tikzpicture}[scale=.8]

\begin{scope}
	\clip (-4.7,-4) rectangle (2.9,3.2);%
		\begin{scope}[scale=0.6]%
			//grid   
			%
			%
			//draw lines
				\draw (-8,5)..controls (-1,4) and (1,6)..(8,5); //top
				\draw (-8,5) node[anchor=north west] {$\Theta_4$};%
			
				\draw (-8,-5)..controls (-4,-3) and (1,-9)..(8,-5); //bot       	 
				\draw (-8,-4.8) node[anchor=south west] {$\Theta_2$};%
			//draw big peripheral circle	
			
            %
           \coordinate (T) at (-2.5,.7);%
           
           \begin{scope} [shift={(T)}]%
            //bottom conections
            	//bot peripheral
            	\def\a{0.3};\def\A{1};%
				\def\b{0.01};\def\B{2};%
				\def\c{0.3};\def\C{1};%
				\def\d{0.02};\def\D{40};%
				\def\e{-1};\def\f{-1.2};%
            	 	\draw [fill=black!5,smooth,domain=0:360] plot ({\a*cos( \A*\x )+\b*sin(\B*\x)+\e},{\c*sin( \C*\x ) + \d*cos(\D*\x )+\f}) ;%
           		\def\x{280};
            	//left+right bot connections
            	\def\x{220};%
            	\coordinate (OmegasL) at ({\a*cos( \A*\x )+\b*sin(\B*\x)+\e}, {\c*sin( \C*\x )+\d*cos(\D*\x )+\f});%
            	\def\x{320};%
				\coordinate (OmegasR) at ({\a*cos( \A*\x )+\b*sin(\B*\x)+\e}, {\c*sin( \C*\x ) +\d*cos(\D*\x )+\f});%
            	\draw[xshift=0,thin, fill=black!20](OmegasL) ..controls (-2,-3) and (0,-3)..(-2.8,-5.34)-- (-1.5,-5.6)..controls (0,-3) and (0,-3)..(OmegasR)--cycle;%
            	\draw[xshift=1,yshift=1] (-2.8,-5.37)%
            		 node[circle,fill,inner sep=1pt](a){} ;%
            	//mid
				\def\x{280};%
            	\draw[xshift=-3,thick] ({\a*cos( \A*\x )+\b*sin(\B*\x)+\e}, {\c*sin( \C*\x ) + \d*cos(\D*\x )+\f}) ..controls (-1,-3) and (0,-3)..(-2,-5.45);  %
				\draw[xshift=-3] (-2,-5.45)%
            		 node[circle,fill,inner sep=1pt](a){} ;%
            	//gammas
				\draw (-0.7,-2.2) node[scale=.8] {$\gamma_s$};%
				\draw[xshift=-0,yshift=2] (-1.5,-5.65) node[circle,fill,inner sep=1pt](a){} ;%
           		//ball
            	\def\a{1.25};\def\A{1};%
				\def\b{0.05};\def\B{2};%
				\def\c{1};\def\C{1};%
				\def\d{0.1};\def\D{34};%
				\def\x{220};%
				\coordinate (A) at ({\a*cos( \A*\x )+\b*sin(\B*\x)},{\c*sin( \C*\x ) + \d*cos(\D*\x )}) ;%
            	\coordinate (A1) at ($(A)+(0.5,0)$);%
           		//point a
            	 	\draw (A)%
            		circle (1cm); %
            		\draw({\a*cos( \A*\x )+\b*sin(\B*\x)},{\c*sin( \C*\x ) + \d*cos(\D*\x )}) node[circle,fill,inner sep=1pt,label=west:$a$](a){} ;%
				    \draw (A1) node[circle,fill,inner sep=1pt,label={[label distance=-0.1cm]90:$a_1$}](a1){} ;%
           		//bot peripheral
            	\def\a{0.3};\def\A{1};%
				\def\b{0.01};\def\B{2};%
				\def\c{0.3};\def\C{1};%
				\def\d{0.02};\def\D{40};%
				\def\e{-1};\def\f{-1.2};%
            	 	\draw [fill=black!5,smooth,domain=0:360] plot %
            		(%
            		{\a*cos( \A*\x )+\b*sin(\B*\x)+\e}, %
            		{\c*sin( \C*\x ) + \d*cos(\D*\x )+\f}%
            		) ;%
            	\draw[<-,dashed] (\e,\f) to (-3,\f) node[anchor=east]{ $Q_{i_a}$};  %
           		\def\x{280};%
				//Omega label
				\draw[->,dashed] (0,-5) node[anchor=west] {$\Omega_{s,h}$} to (-1.3,-5); %
			\end{scope}%
           	\begin{scope}%
            //top connections
            //ball
            	\def\a{1.25};\def\A{1};%
				\def\b{0.05};\def\B{2};%
				\def\c{1};\def\C{1};%
				\def\d{0.1};\def\D{34};%
				\def\x{60};%
				\coordinate (B) at (%
            		{\a*cos( \A*\x )+\b*sin(\B*\x)}, %
            		{\c*sin( \C*\x ) + \d*cos(\D*\x )}%
            		);%
            	\coordinate (B1) at ($(B)-(0.5,0)$);	%
            	 	\draw (B)%
            		 circle (1cm); %
            		\draw(B)%
            		 node[circle,fill,inner sep=1pt,label={[label distance=-.1cm]{-10}:$b$}](b){} ;%
            		\draw(B1)%
            		 node[circle,fill,inner sep=1pt,label={[label distance=-.14cm]{-90}:$b_1$}](b_1){} ;%
            	//top peripheral
            	\def\a{0.3};\def\A{1};%
				\def\b{0.01};\def\B{2};%
				\def\c{0.3};\def\C{1};%
				\def\d{0.02};\def\D{40};%
				\def\e{1};\def\f{1.2};%
            	 	\draw [fill=black!5,smooth,domain=0:360] plot %
            		(%
            		{\a*cos( \A*\x )+\b*sin(\B*\x)+\e}, %
            		{\c*sin( \C*\x ) + \d*cos(\D*\x )+\f}%
            		) ;%
		%
         %
            	\def\x{60};%
		%
          %
          %
          %
           %
		%
            %
            	//left+right
            	\def\x{100};%
            	\coordinate (OmegatL) at (%
            		{\a*cos( \A*\x )+\b*sin(\B*\x)+\e}, %
            		{\c*sin( \C*\x ) + \d*cos(\D*\x )+\f}%
            		);%
            	\def\x{30};%
            	\coordinate (OmegatR) at (%
            		{\a*cos( \A*\x )+\b*sin(\B*\x)+\e},%
            		{\c*sin( \C*\x ) + \d*cos(\D*\x )+\f}%
            		);%
            	\draw[xshift=0,yshift=0,thin,fill=black!20] (OmegatL) ..controls (-1,3) and (1,3)..(0.5,5.05)--(1.5,5.15) ..controls (1,3) and (1.5,3).. (OmegatR)--cycle;%
            	\draw[yshift=0] (0.5,5.05)%
            		 node[circle,fill,inner sep=1pt](a){} ;%
            	%
            	%
            	%
				%
				\draw[xshift=-0,yshift=0] (1.5,5.15)
            		 node[circle,fill,inner sep=1pt](a){} ;%
           		//mid
				\def\x{60};%
            	\draw[xshift=-0,thick] (%
            		{\a*cos( \A*\x )+\b*sin(\B*\x)+\e}, %
            		{\c*sin( \C*\x ) + \d*cos(\D*\x )+\f}%
            		) ..controls (0,3) and (1,3)..(1,5.1);%
				\draw[xshift=-0] (1,5.1)%
            		 node[circle,fill,inner sep=1pt](a){} ;%
            	//gammat
            	\draw (1,2.9) node[scale=.8] {$\tilde \gamma_t$};%
            	//top peripheral
            	\def\a{0.3};\def\A{1};%
				\def\b{0.01};\def\B{2};%
				\def\c{0.3};\def\C{1};%
				\def\d{0.02};\def\D{40};%
				\def\e{1};\def\f{1.2};%
            	 	\draw [fill=black!5,smooth,domain=0:360] plot %
            		(%
            		{\a*cos( \A*\x )+\b*sin(\B*\x)+\e}, %
            		{\c*sin( \C*\x ) + \d*cos(\D*\x )+\f}%
            		) ;%
				\draw[<-,dashed] (\e,\f) to (3,\f) node[anchor=west] { $Q_{i_b}$};%
            	//ball
            	\def\a{1.25};\def\A{1};%
				\def\b{0.05};\def\B{2};%
				\def\c{1};\def\C{1};%
				\def\d{0.1};\def\D{34};%
				\def\x{60};%
            	 	\draw %
            		(%
            		{\a*cos( \A*\x )+\b*sin(\B*\x)}, %
            		{\c*sin( \C*\x ) + \d*cos(\D*\x )}%
            		) circle (1cm); %
			    %
            	//Omega label
				\draw[->,dashed] (3,3.7) node[anchor=west] {$\widetilde\Omega_{t,h}$} to (1.07,3.7);%
			\end{scope}
			//V region
			\draw[pattern=north west lines,opacity=.4] (A) to [out=150, in=100] (B) to [out=-100, in=-30] node[below,opacity=1] {$V$} (A);%
			%
			\draw[thick] (A1) to [out=10, in =150] node[above]{$\tau$} (B1);%
            \end{scope}%
	\end{scope}%
\begin{scope}[shift={(8,0)}]
	\clip (-4.7,-4) rectangle (2.8,3.2);%
		\begin{scope}[scale=0.6]%
			//grid   
			%
			%
			//draw lines
				\draw (-8,5)..controls (-1,4) and (1,6)..(8,5); //top
				\draw (-8,5) node[anchor=north west] {$\Theta_4$};%
			
				\draw (-8,-5)..controls (-4,-3) and (1,-9)..(8,-5); //bot       	 
				\draw (-8,-4.8) node[anchor=south west] {$\Theta_2$};%
            //draw big peripheral circle	
			
			\def\a{1.25};\def\A{1};
			\def\b{0.05};\def\B{2};
			\def\c{1};\def\C{1};
			\def\d{0.1};\def\D{34};
            	 	\draw [fill=black!5,smooth,domain=0:360] plot 
            		(
            		{\a*cos( \A*\x )+\b*sin(\B*\x)}, 
            		{\c*sin( \C*\x ) + \d*cos(\D*\x )}
            		) ;
            		\draw (1.3,-1) node[] {$Q_{i_0}$};
            		
           \coordinate (T) at (-2.5,.7);%
           
           \begin{scope} [shift={(T)}]%
            //bottom conections
            	//bot peripheral
            	\def\a{0.3};\def\A{1};%
				\def\b{0.01};\def\B{2};%
				\def\c{0.3};\def\C{1};%
				\def\d{0.02};\def\D{40};%
				\def\e{-1};\def\f{-1.2};%
            	 	\draw [fill=black!5,smooth,domain=0:360] plot ({\a*cos( \A*\x )+\b*sin(\B*\x)+\e},{\c*sin( \C*\x ) + \d*cos(\D*\x )+\f}) ;%
           		\def\x{280};
            	//left+right bot connections
            	\def\x{220};%
            	\coordinate (OmegasL) at ({\a*cos( \A*\x )+\b*sin(\B*\x)+\e}, {\c*sin( \C*\x )+\d*cos(\D*\x )+\f});%
            	\def\x{320};%
				\coordinate (OmegasR) at ({\a*cos( \A*\x )+\b*sin(\B*\x)+\e}, {\c*sin( \C*\x ) +\d*cos(\D*\x )+\f});%
            	\draw[xshift=0,thin, fill=black!20](OmegasL) ..controls (-2,-3) and (0,-3)..(-2.8,-5.34)-- (-1.5,-5.6)..controls (0,-3) and (0,-3)..(OmegasR)--cycle;%
            	\draw[xshift=1,yshift=1] (-2.8,-5.37)%
            		 node[circle,fill,inner sep=1pt](a){} ;%
            	//mid
				\def\x{280};%
            	\draw[xshift=-3,thick] ({\a*cos( \A*\x )+\b*sin(\B*\x)+\e}, {\c*sin( \C*\x ) + \d*cos(\D*\x )+\f}) ..controls (-1,-3) and (0,-3)..(-2,-5.45);  %
				\draw[xshift=-3] (-2,-5.45)%
            		 node[circle,fill,inner sep=1pt](a){} ;%
            	//gammas
				\draw (-0.7,-2.2) node[scale=.8] {$\gamma_s$};%
				\draw[xshift=-0,yshift=2] (-1.5,-5.65) node[circle,fill,inner sep=1pt](a){} ;%
           		//ball
            	\def\a{1.25};\def\A{1};%
				\def\b{0.05};\def\B{2};%
				\def\c{1};\def\C{1};%
				\def\d{0.1};\def\D{34};%
				\def\x{220};%
				\coordinate (A) at ({\a*cos( \A*\x )+\b*sin(\B*\x)},{\c*sin( \C*\x ) + \d*cos(\D*\x )}) ;%
            	\coordinate (A1) at ($(A)+(0.5,0)$);%
           		//point a
            	 	\draw (A)%
            		circle (1cm); %
            		\draw({\a*cos( \A*\x )+\b*sin(\B*\x)},{\c*sin( \C*\x ) + \d*cos(\D*\x )}) node[circle,fill,inner sep=1pt,label=west:$a$](a){} ;%
				    \draw (A1) node[circle,fill,inner sep=1pt,label={[label distance=-0.1cm]90:$a_1$}](a1){} ;%
           		//bot peripheral
            	\def\a{0.3};\def\A{1};%
				\def\b{0.01};\def\B{2};%
				\def\c{0.3};\def\C{1};%
				\def\d{0.02};\def\D{40};%
				\def\e{-1};\def\f{-1.2};%
            	 	\draw [fill=black!5,smooth,domain=0:360] plot %
            		(%
            		{\a*cos( \A*\x )+\b*sin(\B*\x)+\e}, %
            		{\c*sin( \C*\x ) + \d*cos(\D*\x )+\f}%
            		) ;%
            	\draw[<-,dashed] (\e,\f) to (-3,\f) node[anchor=east]{ $Q_{i_a}$};  %
           		\def\x{280};%
				//Omega label
				\draw[->,dashed] (0,-5) node[anchor=west] {$\Omega_{s,h}$} to (-1.3,-5); %
			\end{scope}%
           	\begin{scope}%
            //top connections
            //ball
            	\def\a{1.25};\def\A{1};%
				\def\b{0.05};\def\B{2};%
				\def\c{1};\def\C{1};%
				\def\d{0.1};\def\D{34};%
				\def\x{60};%
				\coordinate (B) at (%
            		{\a*cos( \A*\x )+\b*sin(\B*\x)}, %
            		{\c*sin( \C*\x ) + \d*cos(\D*\x )}%
            		);%
            	\coordinate (B1) at ($(B)+(-0.5,.3)$);	%
            	 	\draw (B)%
            		 circle (1cm); %
            		\draw(B)%
            		 node[circle,fill,inner sep=1pt,label={[label distance=-.1cm]{272}:$b$}](b){} ;%
            		\draw(B1)%
            		 node[circle,fill,inner sep=1pt,label={[label distance=-.07cm]{-90}:$b_1$}](b_1){} ;%
            	//top peripheral
            	\def\a{0.3};\def\A{1};%
				\def\b{0.01};\def\B{2};%
				\def\c{0.3};\def\C{1};%
				\def\d{0.02};\def\D{40};%
				\def\e{1};\def\f{1.2};%
            	 	\draw [fill=black!5,smooth,domain=0:360] plot %
            		(%
            		{\a*cos( \A*\x )+\b*sin(\B*\x)+\e}, %
            		{\c*sin( \C*\x ) + \d*cos(\D*\x )+\f}%
            		) ;%
		%
         %
            	\def\x{60};%
		%
          %
          %
          %
           %
		%
            %
            	//left+right
            	\def\x{100};%
            	\coordinate (OmegatL) at (%
            		{\a*cos( \A*\x )+\b*sin(\B*\x)+\e}, %
            		{\c*sin( \C*\x ) + \d*cos(\D*\x )+\f}%
            		);%
            	\def\x{30};%
            	\coordinate (OmegatR) at (%
            		{\a*cos( \A*\x )+\b*sin(\B*\x)+\e},%
            		{\c*sin( \C*\x ) + \d*cos(\D*\x )+\f}%
            		);%
            	\draw[xshift=0,yshift=0,thin,fill=black!20] (OmegatL) ..controls (-1,3) and (1,3)..(0.5,5.05)--(1.5,5.15) ..controls (1,3) and (1.5,3).. (OmegatR)--cycle;%
            	\draw[yshift=0] (0.5,5.05)%
            		 node[circle,fill,inner sep=1pt](a){} ;%
            	%
            	%
            	%
				%
				\draw[xshift=-0,yshift=0] (1.5,5.15)
            		 node[circle,fill,inner sep=1pt](a){} ;%
           		//mid
				\def\x{60};%
            	\draw[xshift=-0,thick] (%
            		{\a*cos( \A*\x )+\b*sin(\B*\x)+\e}, %
            		{\c*sin( \C*\x ) + \d*cos(\D*\x )+\f}%
            		) ..controls (0,3) and (1,3)..(1,5.1);%
				\draw[xshift=-0] (1,5.1)%
            		 node[circle,fill,inner sep=1pt](a){} ;%
            	//gammat
            	\draw (1,2.9) node[scale=.8] {$\tilde \gamma_t$};%
            	//top peripheral
            	\def\a{0.3};\def\A{1};%
				\def\b{0.01};\def\B{2};%
				\def\c{0.3};\def\C{1};%
				\def\d{0.02};\def\D{40};%
				\def\e{1};\def\f{1.2};%
            	 	\draw [fill=black!5,smooth,domain=0:360] plot %
            		(%
            		{\a*cos( \A*\x )+\b*sin(\B*\x)+\e}, %
            		{\c*sin( \C*\x ) + \d*cos(\D*\x )+\f}%
            		) ;%
				\draw[<-,dashed] (\e,\f) to (3,\f) node[anchor=west] { $Q_{i_b}$};%
            	//ball
            	\def\a{1.25};\def\A{1};%
				\def\b{0.05};\def\B{2};%
				\def\c{1};\def\C{1};%
				\def\d{0.1};\def\D{34};%
				\def\x{60};%
            	 	\draw %
            		(%
            		{\a*cos( \A*\x )+\b*sin(\B*\x)}, %
            		{\c*sin( \C*\x ) + \d*cos(\D*\x )}%
            		) circle (1cm); %
			    %
            	//Omega label
				\draw[->,dashed] (3,3.7) node[anchor=west] {$\widetilde\Omega_{t,h}$} to (1.07,3.7);%
			\end{scope}
			//V region
			\begin{scope}[on background layer]
			\draw[pattern=north west lines,opacity=.4] (A) to [out=150, in=100] (B) to [out=-100, in=-30] node[below,xshift=-.4cm,opacity=1] {$V$} (A);
			%
			\draw[thick] (A1) to [out=10, in =150] node[above]{$\tau$} (B1);%
			\end{scope}

            \end{scope}%
	\end{scope}%

	\end{tikzpicture}